\documentclass{article}
\usepackage{psfrag}
\usepackage{float, graphicx}
\usepackage[]{epsfig}
\usepackage{amsmath, amsthm, amssymb,}
\usepackage{epsfig}
\usepackage{verbatim}
\usepackage{multicol}
\usepackage{url}
\usepackage{latexsym}
\usepackage{mathrsfs}
\usepackage{graphicx}
\usepackage{amsmath}
\usepackage{enumerate}
\usepackage[normalem]{ulem}
\usepackage{bm}
\usepackage{dsfont}
\usepackage{stmaryrd}
\usepackage{enumitem}
\usepackage{tikz}
\usepackage{cite}
\usepackage{color}
\usepackage{xcolor}

\usepackage{titlesec}

\titleformat{\subsubsection}[runin]
{\normalfont\normalsize\bfseries}
{\thesubsubsection}
{0.5em}
{}
[.]

\makeatletter

\def\@settitle{\begin{center}%
    \bfseries
 \normalfont\LARGE\@title
  \end{center}%
}
\def\@setauthors{\begin{center}%
 \normalsize\@author
  \end{center}%
}

\makeatother

\numberwithin{equation}{section}

\renewcommand{\cal}{\mathcal}
\newcommand\cA{{\mathcal A}}
\newcommand\cB{{\mathcal B}}
\newcommand{\cC}{{\cal C}}

\newcommand{\cL}{{\cal L}}
\newcommand{\cT}{{\cal T}}

\newcommand{\cP}{{\cal P}}
\newcommand{\cQ}{{\cal Q}}
\newcommand{\cR}{{\mathcal R}}
\newcommand{\cS}{{\mathcal S}}
\newcommand{\cU}{{\mathcal U}}

\newcommand\cW{{\mathcal W}}
\newcommand{\cX}{{\mathcal X}}

\newcommand{\sfX}{{\mathsf X}}
\newcommand{\sfC}{{\mathsf C}}

\newcommand{\sfD}{{\mathsf D}}

\newcommand{\sfS}{{\mathsf S}}

\newcommand{\sfT}{{\mathsf T}}
\newcommand{\sfL}{{\mathsf L}}
\newcommand{\sfR}{{\mathsf R}}

\newcommand{\fa}{{\mathfrak a}}
\newcommand{\fs}{{\mathfrak s}}

\newcommand{\fc}{{\mathfrak c}}

\newcommand{\fR}{{\frak R}}
\newcommand{\fC}{{\frak C}}
\newcommand{\fU}{{\frak U}}
\newcommand{\fV}{{\frak V}}

\newcommand{\bmt} {{\bm t }}
\newcommand{\bms} {{\bm s}}

\newcommand{\bmv}{{\bm{v}}}

\newcommand{\bmx}{{\bm{x}}}

\newcommand{\bmz} {{\bm {z}}}

\newcommand{\rd}{{\rm d}}

\newcommand{\ri}{\mathrm{i}}

\newcommand{\bC}{{\mathbb C}}

\newcommand{\bE}{\mathbb{E}}

\newcommand{\bN}{\mathbb{N}}
\newcommand{\bP}{\mathbb{P}}
\newcommand{\bI}{\mathbb{I}}
\newcommand{\bQ}{\mathbb{Q}}
\newcommand{\bR}{{\mathbb R}}

\newcommand{\bZ}{\mathbb{Z}}

\newcommand{\al}{\alpha}

\newcommand{\la}{\lambda}

\newcommand{\md}{m_d}
\newcommand{\msc}{m_{\rm sc}}

\DeclareMathOperator{\Tr}{Tr}

\DeclareMathOperator{\supp}{supp}

\DeclareMathOperator{\OO}{O}
\DeclareMathOperator{\oo}{o}

\DeclareMathOperator{\Ai}{Ai}

\renewcommand{\Re}{\mathop{\mathrm{Re}}}
\renewcommand{\Im}{\mathop{\mathrm{Im}}}
\renewcommand{\le}{\leq}
\renewcommand{\ge}{\geq}

\renewcommand{\leq}{\leqslant}
\renewcommand{\geq}{\geqslant}

\newcommand{\del}{\partial}

\newcommand{\qq}[1]{[\![{#1}]\!]}

\newcommand{\beq}{\begin{equation}}
\newcommand{\eeq}{\end{equation}}

\newcommand{\abs}[1]{\lvert #1 \rvert}

\usepackage[
  colorlinks=true,
]{hyperref}

\usepackage{aliascnt}

\newtheorem{theorem}{Theorem}[section]
\newtheorem*{theorem*}{Theorem}

\newaliascnt{lemma}{theorem}
\newtheorem{lemma}[lemma]{Lemma}
\aliascntresetthe{lemma}
\newtheorem*{lemma*}{Lemma}

\newaliascnt{corollary}{theorem}
\newtheorem{corollary}[corollary]{Corollary}
\aliascntresetthe{corollary}
\newtheorem*{corollary*}{Corollary}

\newaliascnt{proposition}{theorem}
\newtheorem{proposition}[proposition]{Proposition}
\aliascntresetthe{proposition}
\newtheorem*{proposition*}{Proposition}

\newaliascnt{assumption}{theorem}
\newtheorem{assumption}[assumption]{Assumption}
\aliascntresetthe{assumption}
\newtheorem*{assumption*}{Assumption}

\newaliascnt{claim}{theorem}

\aliascntresetthe{claim}

\newaliascnt{conjecture}{theorem}
\newtheorem{conjecture}[conjecture]{Conjecture}
\aliascntresetthe{conjecture}

\newaliascnt{definition}{theorem}
\newtheorem{definition}[definition]{Definition}
\aliascntresetthe{definition}
\newtheorem*{definition*}{Definition}

\newaliascnt{example}{theorem}

\aliascntresetthe{example}
\newtheorem*{example*}{Example}

\newaliascnt{remark}{theorem}
\newtheorem{remark}[remark]{Remark}
\aliascntresetthe{remark}
\newtheorem*{remark*}{Remark}

\newaliascnt{remarks}{theorem}

\aliascntresetthe{remarks}
\newtheorem*{remarks*}{Remarks}

\usepackage{cleveref}

\crefname{theorem}{Theorem}{Theorems}
\crefname{lemma}{Lemma}{Lemmas}
\crefname{proposition}{Proposition}{Propositions}
\crefname{assumption}{Assumption}{Assumptions}
\crefname{definition}{Definition}{Definitions}
\crefname{conjecture}{Conjecture}{Conjectures}
\crefname{corollary}{Corollary}{Corollaries}
\crefname{remark}{Remark}{Remarks}

\def\author#1{\par
    {\centering{\authorfont#1}\par\vspace*{0.05in}}
}

\def\titlefont{\fontsize{13}{15}\bfseries\boldmath\selectfont\centering{}}
\def\authorfont{\fontsize{13}{15}}

\let\affiliationfont\rhfont

\def\address#1{\par
    {\centering{\affiliationfont#1\par}}\par\vspace*{11pt}
}

\def\body{
\setcounter{footnote}{0}
\def\thefootnote{\alph{footnote}}
\def\@makefnmark{{$^{\rm \@thefnmark}$}}
}

\def\title#1{
    \thispagestyle{plain}
    \vspace*{-14pt}
    \vskip 79pt
    {\centering{\titlefont #1\par}}%
    \vskip 1em
}

\begin{document}

~\vspace{1.2cm}

\title{Loop Equations Characterize Random Matrix Statistics}

\vspace{1.2cm}

\noindent  \begin{minipage}[c]{0.49\textwidth}
 \author{Paul Bourgade}
\address{Courant Institute,  NYU\\
   bourgade@cims.nyu.edu}
\end{minipage}
 \begin{minipage}[c]{0.5\textwidth}
 \author{Jiaoyang Huang}
\address{University of Pennsylvania\\
   huangjy@wharton.upenn.edu}
\end{minipage}

\vspace{1.2cm}

\begin{abstract}
We prove that the universal local point processes of random matrix theory are characterized by their loop equation hierarchies. More precisely, for every rational $\beta>0$, the $\mathrm{Sine}_{\beta}$ point process is the unique solution of the bulk loop equation hierarchy, and the $\mathrm{Airy}_{\beta}$ point process is the unique solution of the edge loop equation hierarchy.

These uniqueness results provide a direct route to universality: it suffices to verify the corresponding approximate loop equations for the ensemble. In many models,  these equations follow from local laws and integration by parts.

\end{abstract}

\vspace{1.2cm}

{
  \hypersetup{linkcolor=black}
\setcounter{tocdepth}{2}
\tableofcontents
}

\newpage

\section{Introduction}

A central theme in random matrix theory is \emph{universality}: after a microscopic scaling, local eigenvalue statistics in
the bulk and at regular edges depend only on the symmetry class of the model,
or more generally on the parameter $\beta$, and not on the microscopic details
of the underlying distribution or potential.

This prediction, originally formulated by Wigner,  was extended to non-spectral correlated systems,  as demonstrated by several paradigmatic integrable models.
Examples include the Tracy--Widom fluctuations of the longest increasing
subsequence of a uniform random permutation \cite{MR1682248} and
integrable models of directed last-passage percolation and random growth
\cite{SFRM}, which belong to the broader one-dimensional KPZ universality
class  \cite{prahofer2002scale,corwin2012kardar,quastel2015one}. 

Random-matrix predictions also play a central role in two major conjectural
settings. The high-lying zeros of the Riemann zeta function are expected 
to exhibit GUE local statistics,  beginning with Montgomery's
pair-correlation conjecture and supported by Odlyzko's numerical computations
\cite{montgomery1973pair,odlyzko1987distribution}. In quantum chaos, the
Bohigas--Giannoni--Schmit conjecture predicts that quantum systems whose
classical dynamics are sufficiently chaotic have local spectral statistics
described by the appropriate Wigner--Dyson ensemble, determined by the
symmetries of the system \cite{bohigas1984characterization}. In contrast with
the preceding random-matrix and integrable-probability examples, these
conjectural directions remain open.\\

 We recall two major approaches to
proving the universality of random matrix statistics.  The first is based on exact algebraic formulas.  For
unitary invariant ensembles, the eigenvalues form determinantal point
processes, and the asymptotic analysis of the associated
Christoffel--Darboux kernels leads to the sine-kernel process in the bulk and
the Airy-kernel process at a regular soft edge.  For orthogonal and symplectic
ensembles, analogous Pfaffian, or quaternion determinantal, formulas play the
same role.  These exact structures, together with orthogonal polynomial and
Riemann--Hilbert steepest-descent methods, have provided one of the most
powerful routes to local random matrix statistics; see, for example,
\cite{mehta2004random,deift2000orthogonal,deift2009random,akemann2011oxford}.

A second major route is comparison with integrable ensembles and their perturbations, such as Gaussian and Gaussian-divisible ensembles.  In the setting of Wigner matrices,  
this approach led to the dynamical three-step scheme developed by
Erd\H{o}s, Schlein, Yau, Yin, and collaborators
\cite{MR2810797,MR2919197,erdHos2017dynamical}: 
\begin{itemize}
    \item[(i)] establish a local semicircle law, together with eigenvalue
    rigidity and eigenvector delocalization estimates
    \cite{erdHos2009local,erdHos2012rigidity};

    \item[(ii)] prove universality after a short Gaussian evolution, equivalently
    for Gaussian divisible ensembles.  In the complex Hermitian case, this can
    be done using exact determinantal formulas for GUE-divisible ensembles
     \cite{johansson2001universality,erdHos2010bulk}.  For general universality classes,  it follows
    from the short-time relaxation of the Dyson Brownian motion
    \cite{MR2810797,MR2919197,landon2019fixed},   identifying local statistics by comparison with the integrable
ensembles;
    \item[(iii)] remove the Gaussian component by a perturbative, density argument.  For
    sufficiently regular entry distributions this can be done by the reverse
    heat flow argument  \cite{erdHos2010bulk}.  More robust comparison methods include
    the four moment theorem of Tao and Vu
    \cite{MR2669449, tao2011random},  the Green function comparison theorem
   \cite{erdHos2012bulk} and the continuity along matrix dynamics \cite{MR3606475}.  
\end{itemize}

 This comparison approach through dynamics applies beyond Wigner matrices,  to prove
random matrix universality for random graphs,   band matrices,  $\beta$-ensembles,  non-centered  models with variance profiles,  convolution models appearing in free probability among others, 
see e.g.  \cite{YauICM}.  
A different comparison mechanism,  based on approximate transport maps,  provides alternative proofs of universality for $\beta$-ensembles \cite{MR3390602,bekerman2015transport} and also applies to perturbative multi-matrix models  \cite{MR3646880}.\\

 The purpose of the present paper is to develop a different point of view.
Rather than proving universality by exact algebraic formulas for correlation functions,  or by reducing a model to an exactly solvable
reference ensemble, we seek an intrinsic characterization of the limiting point
processes themselves,  through a BBGKY-type hierarchy.  Our characterization is expressed through \emph{loop
equations}.  Loop equations, also known as Dyson--Schwinger equations, master
equations, or Ward identities, originate from { linear responses,} invariance principles or
integration by parts identities.  They appeared in the theoretical physics
literature in the context of gauge theory and matrix models, for instance in
the work of 't Hooft, Makeenko--Migdal and Migdal  
\cite{t1974planar, makeenko1979exact,migdal1983loop}.  They were later introduced into the
mathematical analysis of random matrices by Johansson
\cite{johansson1998fluctuations}, who used them to derive macroscopic central
limit theorems for $\beta$-ensembles.  Loop-equation methods have since become
a standard tool in the study of global fluctuations and large-$N$ expansions of
log-gases and matrix models; see, for example,
\cite{shcherbina2013fluctuations,borot2013asymptotic,borot2015large,MR2346575}
and the monograph \cite{guionnet2019asymptotics}. 

For a $\beta$-ensemble with $N$ particles and  potential $V$, the loop equation hierarchy is a
system of identities for resolvent observables, or equivalently for cumulants of Stieltjes transforms of the eigenvalue
process.  
This hierarchy encodes the interplay between the logarithmic repulsion, the
confining potential $V$, and the finite-size correction terms.

Although the finite-$N$ loop equations are model-dependent, their microscopic
limits are universal.   When one zooms into the bulk at the scale of the
mean eigenvalue spacing, or into a regular spectral edge at the $N^{-2/3}$
scale, the rescaled loop equations converge to universal limiting hierarchies.
We call these limiting identities the \emph{bulk loop equations}
\eqref{e:loopeq2-bulk} and the \emph{edge loop equations}
\eqref{e:loopeq2-edge}.  Unlike the finite-$N$ loop equations, these
microscopic equations no longer contain the original potential.  They are
universal equations for the limiting local point processes.

Our main results show that the universal microscopic loop equations do not
merely hold for the $\beta$-random matrix limits, but 
characterize them. More precisely, for every rational $\beta>0$, we prove that
the $\mathrm{Sine}_{\beta}$ point process is the unique solution of the bulk
loop equation hierarchy, and that the $\mathrm{Airy}_{\beta}$ point process is
the unique solution of the edge loop equation hierarchy. These processes arise, respectively, as the universal bulk and
edge limits of Gaussian $\beta$-ensembles  \cite{valko2009continuum,ramirez2011beta}.

This result is analogous in spirit to Stein's method.  The standard Gaussian
distribution is characterized by the integration by parts identity
\begin{equation*}
    \mathbb E[X f(X)] = \mathbb E[f'(X)]
\end{equation*}
for a suitable class of test functions $f$.  In our setting, the role of
Stein's identity is played by the loop equation hierarchy.  
These identities are again derived by integration by parts, or equivalently by infinitesimal changes of variables, 
but now in the setting of the one-dimensional logarithmic gas and for a specific family of test functions, namely moments of Stieltjes transforms.\\

The key point is the passage from identities to characterization, which turns the approximate loop equations into a convergence
criterion for universal local statistics of random matrices and related models.
In the bulk, verifying the approximate loop equations yields convergence to the
$\mathrm{Sine}_{\beta}$ point process. At the edge,  the analogous criterion  yields convergence to the \(\mathrm{Airy}_{\beta}\) point process,
and hence to the Tracy--Widom\(_\beta\) law for the top eigenvalue. These
equations can often be verified by integration by parts, or by a discrete
analogue thereof, with the local law as the main input.

To demonstrate the scope of this approach beyond invariant ensembles, we apply
our characterization to random \(d\)-regular graphs. We prove bulk universality
for random \(d\)-regular graphs on \(N\) vertices with degree \(d\) growing
polylogarithmically in \(N\), and we show that the same characterization
principle significantly simplifies the proof of edge universality for
fixed-degree random \(d\)-regular graphs in \cite{huang2024ramanujan}.

\subsection{Main results}
Throughout the paper, we fix $\beta>0$. The $\mathrm{Sine}_{\beta}$ point
process is the universal bulk scaling limit of the Gaussian $\beta$-ensemble,
characterized by Valk{\'o} and Vir{\'a}g through the Brownian carousel
construction \cite{valko2009continuum}. We use the normalization in which its
intensity is $1/\pi$, or equivalently, the local mean spacing is $\pi$.

The $\mathrm{Airy}_{\beta}$ point process is the universal soft-edge scaling
limit of the Gaussian $\beta$-ensemble, characterized by Ram{\'i}rez, Rider,
and Vir{\'a}g through the spectrum of the stochastic Airy operator
\cite{ramirez2011beta}. We use the reflected convention in which the Airy
points are ordered decreasingly and extend to $-\infty$, matching the usual
ordering of the negative zeros of the Airy function. In this normalization,
the one-point density has the standard square-root left-tail asymptotics with
leading constant $1/\pi$.

We regard both $\mathrm{Sine}_{\beta}$ and $\mathrm{Airy}_{\beta}$ as locally
finite random point configurations on $\bR$. We write a configuration as
$\bmx=\{x_i\}_{i\in\bI}$, where $\bI=\bZ$ in the bulk case and
$\bI=\bN$ in the edge case. In the bulk case the particles are indexed
in increasing order, while in the edge case they are indexed in decreasing
order, so that $x_1\geq x_2\geq \cdots$ and $x_i\to -\infty$. Finally, we denote by
$\bC_+$ and $\bC_-$ the open upper and lower half-planes, respectively.

\begin{definition}
A holomorphic function \(s:\bC_+\to \bC_+\cup \bR\) is called a \emph{Nevanlinna function}. Such functions are also called \emph{Herglotz} or \emph{Pick} functions in the
literature.
Any Nevanlinna function \(s\) admits the integral representation
\begin{align}\label{e:Nevanlinna_representation}
s(w)=b+cw+\int_\bR\left( \frac{1}{x-w}-\frac{x}{1+x^2}\right)\rd \mu(x),
\end{align}
where \(b,c\in \bR\), \(c\ge 0\), and \(\mu\) is a Borel measure satisfying
\[
\int_\bR \frac{\rd \mu(x)}{1+x^2}<\infty.
\]
\end{definition}

Nevanlinna functions form a subclass of
$
\mathcal O(\bC_+):=\{f:\bC_+\to \bC \text{ holomorphic}\}$.
Whenever appropriate, we extend \(f\in \mathcal O(\bC_+)\) to the lower half-plane by Schwarz reflection,
$
f(\overline w)=\overline{f(w)}$,
and interpret boundary values on \(\bR\) as non-tangential limits.

\begin{definition}[Configuration]
A \emph{configuration} on \(\bR\) is a Radon measure of the form
\[
\mu=\sum_{x\in P}\delta_x,
\]
where \(P\) is a finite or countable multiset in \(\bR\), such that
\[
\mu(K)<\infty \qquad \text{for every compact set } K\subset \bR.
\]
Equivalently, \(\mu\) is a purely atomic Radon measure whose atoms have positive integer masses and such that every bounded set contains only finitely many atoms.
\end{definition}

\begin{definition}\label{def:s}
We say that a Nevanlinna function \(s\) is \emph{particle-generated} if the measure \(\mu\) in the representation \eqref{e:Nevanlinna_representation} is a configuration.
In this case, \(s\) extends to a meromorphic function on \(\bC\), satisfies
$
s(\overline w)=\overline{s(w)}$,
and its poles are precisely the atoms of \(\mu\), each with positive integer residue.\footnote{More precisely, if \(x\in \bR\) appears in \(P\) with multiplicity \(m\), then the residue of \(s\) at \(x\) is \(m\).}
\end{definition}

In particular, when \(P\) is finite, the Stieltjes transform
\[
\sum_{x\in P}\frac{1}{x-w}
\]
is a particle-generated Nevanlinna function, corresponding to the choice \(c=0\) and
$
b=\int_{\mathbb R}{x}/({1+x^2})\,\rd\mu(x).
$

When \(P\) is infinite,   well-balanced,  in the bulk case one could define the Stieltjes transform as a principal value.  We work with the broader class of 
particle-generated Nevanlinna functions instead,  for three reasons. First,  it allows a more unified treatment of loop equations in the bulk and at the edge.  
Second,  Nevanlinna functions are stable under locally uniform limits.
Third,  the Stieltjes transform is 
a basic observable in the random matrix setting which may not be natural in other contexts.
Our characterization criterion applies to general holomorphic observables \(s:\bC_+\to \bC_+\cup \bR\).

\begin{assumption}[Bulk loop equations]\label{a:bulk}
We assume that \(s(w)\) is a particle-generated Nevanlinna function, in ${\rm L}^{q}$ for any $q>0$.
For every integer \(p\ge 1\) and every
\(w_1,\dots,w_p\in\bC_+\cup\bC_-\), \(s\) satisfies the loop equation
\begin{align}\label{e:loopeq2-bulk}
\bE\left[
\left(
\frac{2-\beta}{\beta}\,\del_{w_1}s(w_1)
+s(w_1)^2+1
\right)
\prod_{j=2}^p s(w_j)
\right]
+\frac{2}{\beta}\,
\bE\left[
\sum_{j=2}^p
\del_{w_j}
\frac{s(w_1)-s(w_j)}{w_1-w_j}
\prod_{\substack{\ell=2\\ \ell\neq j}}^p s(w_\ell)
\right]
=0.
\end{align}
When \(p=1\), the empty product is interpreted as \(1\), and the second
term is absent. When \(w_j=w_1\), the corresponding difference quotient
is interpreted by its holomorphic extension to the diagonal; in particular,
\[
\left.
\del_{w_j}
\frac{s(w_1)-s(w_j)}{w_1-w_j}
\right|_{w_j=w_1}
=
\frac12\,\del_{w_1}^2s(w_1).
\]
\end{assumption}

\begin{theorem}\label{t:bulk-characterization}
Assume that \(\beta>0\) is rational. Under \Cref{a:bulk}, the poles of \(s\) define a point process on \(\bR\), and this point process has the same law as the $\mathrm{Sine}_{\beta}$ point process.
\end{theorem}

\begin{remark}[Non-uniqueness for $\beta=0$]
In the degenerate limit $\beta\to 0$, the bulk loop-equation hierarchy
\eqref{e:loopeq2-bulk} loses the interaction term and reduces to
\begin{align}
\bE\left[
\del_{w_1}s(w_1)
\prod_{j=2}^p s(w_j)
\right] 
+
\bE\left[
\sum_{j=2}^p 
\del_{w_j}
\frac{s(w_1)-s(w_j)}{w_1-w_j}
\prod_{\substack{\ell=2\\ \ell\neq j}}^p s(w_\ell)
\right]
=0 .
\end{align}
The resulting hierarchy is highly non-unique. In fact, it is satisfied by
any homogeneous Poisson point process of intensity $\lambda>0$. Moreover, since the
identities are linear in the law of the point process, arbitrary convex
mixtures of such solutions are again solutions.
\end{remark}

\begin{remark}[Uniqueness for general hierarchies]
Constraints on the correlations of particle systems at thermal equilibrium, often referred to as sum rules, 
are a basic tool for the study of charged fluids,  see e.g. \cite{Martin}.  
Various families of such constraints, or hierarchies thereof,  such as Dobrushin-Lanford-Ruelle, 
Kirkwood–Salzburg,  Kubo-Martin-Schwinger,  and Bogoliubov–Born–Green–Kirkwood–Yvon (BBGKY),  have been 
instrumental to describe particle systems on mesoscopic and macroscopic scales, often by truncations of the hierarchy.  

How full hierarchies may characterize the microscopic statistics is less understood.  In the case of regular,  short-range interaction, 
existence and uniqueness of solutions holds at low density \cite{Morrey, GallavottiVarboven},  and for integrable interaction \cite{GenoveseSimonella}.
\Cref{App:BBGKY} proves that the bulk loop equations from  \Cref{a:bulk} are equivalent to the BBGKY hierarchy at equilibrium,  for the logarithmic interaction.
 \Cref{t:bulk-characterization} therefore provides one example of sum rules characterizing a Gibbs state in the case of singular,  long-range interactions.
\end{remark}

\begin{remark}[The distribution of $s$ on the real line]
Herglotz--Pick (Nevanlinna) functions also play an important role in \cite{AizWar}, where it is proved that Stieltjes transforms of a wide class configurations,  evaluated at random points on the real line,  have a Cauchy distribution.
\end{remark}

The above theorem allows to prove the following convergence criterion in the bulk.

\begin{theorem}\label{t:bulk-approx-convergence}
Assume that \(\beta>0\) is rational. Let \(s_N\) be a sequence of random particle-generated Nevanlinna functions, and let \(\mu_N\) denote the associated configurations. Assume that there exists a sequence of events \(\Omega_N\) such that
\[
\bP(\Omega_N)=1-\oo_N(1)
\]
and $s_N(w)\bm1({\Omega_N})$ is in ${\rm L}^{q}$ for any $q>0$.
Assume moreover that for every \(p\ge 1\) and every compact set \(K\subset \bC_+\cup \bC_-\),  the following holds uniformly for all \(w_1,w_2,\dots,w_p\in K\):
\begin{align}\begin{split}\label{e:approx-loopeq2-bulk}
&
\Bigg|
\bE\left[
\bm1({\Omega_N})
\left(
\frac{2-\beta}{\beta}\del_{w_1} s_N(w_1)+s_N(w_1)^2+1
\right)
\prod_{j=2}^p s_N(w_j)
\right] \\
&
+\frac{2}{\beta}\bE\left[
\bm1({\Omega_N})
\sum_{j=2}^p
\del_{w_j}
\frac{s_N(w_1)-s_N(w_j)}{w_1-w_j}
\prod_{\substack{\ell=2\\ \ell\neq j}}^ps_N(w_\ell)
\right]
\Bigg|
=\oo_N(1)\cdot\left(1+\bE\left[\bm1(\Omega_N)\cR_N^{p+1}\right]\right),
\end{split}\end{align}
where $\cR_N=\max\{|s_N(w_1)|, |s_N(w_2)|,\cdots, |s_N(w_p)|\}$.
Then \(\mu_N\) converges in distribution, with respect to the vague topology, to the $\mathrm{Sine}_{\beta}$ point process.
\end{theorem}

We now turn to analogous statements at the edge.  Note that the only difference in the loop equations below, when compared to Assumption \ref{a:bulk}, is the constant $+1$ transformed into $-w_1$, 
which is eventually responsible for breaking translation invariance.

\begin{assumption}[Edge loop equations]\label{a:edge}
We assume that \(s(w)\) is a particle-generated Nevanlinna function, 
in ${\rm L}^{q}$ for any $q>0$. 
For every integer \(p\ge 1\) and every 
\(w_1,w_2,\dots,w_p\in \bC_+\cup \bC_-\), \(s\) satisfies the loop equations
\begin{align}\label{e:loopeq2-edge}
\bE\left[
\left(
\frac{2-\beta}{\beta}\,\del_{w_1}s(w_1)
+s(w_1)^2-w_1
\right)
\prod_{j=2}^p s(w_j)
\right] 
+\frac{2}{\beta}\,
\bE\left[
\sum_{j=2}^p 
\del_{w_j}
\frac{s(w_1)-s(w_j)}{w_1-w_j}
\prod_{\substack{\ell=2\\ \ell\neq j}}^p s(w_\ell)
\right]
=0 .
\end{align}
When \(p=1\), the empty product is interpreted as \(1\), and the second term is absent.
\end{assumption}

\begin{theorem}\label{t:edge-characterization}
Assume that \(\beta>0\) is rational. Under \Cref{a:edge}, the poles of \(s\) define a point process on \(\bR\), and this point process has the same law as the \(\mathrm{Airy}_{\beta}\) point process.
\end{theorem}

\begin{theorem}\label{t:edge-approx-convergence}
Assume that \(\beta>0\) is rational. Let \(s_N\) be a sequence of random particle-generated Nevanlinna functions, and let \(\mu_N\) denote the associated configurations. Assume that there exists a sequence of events \(\Omega_N\) such that
\[
\bP(\Omega_N)=1-\oo_N(1)
\]
and $s_N(w)\bm1({\Omega_N})$ is in ${\rm L}^{q}$ for any $q>0$.
Assume moreover that for every \(p\ge 1\) and every compact set \(K\subset \bC_+\cup \bC_-\),  the following holds uniformly for all \(w_1,w_2,\dots,w_p\in K\):
\begin{align}\begin{split}\label{e:approx-loopeq2-edge}
&
\Bigg|
\bE\left[
\bm1({\Omega_N})
\left(
\frac{2-\beta}{\beta}\del_{w_1} s_N(w_1)+s_N(w_1)^2-w_1
\right)
\prod_{j=2}^p s_N(w_j)
\right]\\
&
+\frac{2}{\beta}\bE\left[
\bm1({\Omega_N})
\sum_{j=2}^p
\del_{w_j}
\frac{s_N(w_1)-s_N(w_j)}{w_1-w_j}
\prod_{\substack{\ell=2\\ \ell\neq j}}^ps_N(w_\ell)
\right]
\Bigg|
=\oo_N(1)\cdot\left(1+\bE\left[\bm1(\Omega_N)\cR_N^{p+1}\right]\right),
\end{split}\end{align}
where $\cR_N=\max\{|s_N(w_1)|, |s_N(w_2)|,\cdots, |s_N(w_p)|\}$.
Then \(\mu_N\) converges in distribution, with respect to the vague topology, to the \(\mathrm{Airy}_{\beta}\)  point process.
\end{theorem}

In the preceding theorems, we assume that \(\beta>0\) is rational. This
assumption enters through the proof, rather than through the expected
characterization itself. More precisely, our argument uses the loop equations
to identify certain exponential observables, which in turn determine the
correlation functions of the point process. At present, the construction of
these exponential observables requires \(\beta\) to be rational.

This rationality phenomenon is consistent with the integrable-structure
literature surrounding \(\beta\)-ensembles. Exact formulas for rational
\(\beta\) have appeared in
\cite{ha1994exact,serban1996single,okounkov1997n,lesage1995dynamical};
these formulas are often naturally organized in terms of finitely many
quasiparticle and quasihole variables. We believe, however, that the
uniqueness statements themselves should remain valid for every \(\beta>0\).

\begin{conjecture}
For any \(\beta>0\), not necessarily rational, the following statements hold. Under \Cref{a:bulk}, the poles of \(s\) define a point process on \(\bR\), and this point process has the same law as the $\mathrm{Sine}_{\beta}$ point process. Under \Cref{a:edge}, the poles of \(s\) define a point process on \(\bR\), and this point process has the same law as the \(\mathrm{Airy}_{\beta}\) point process.
\end{conjecture}

 Finally, ${\rm Sine}_{\beta}$ has also been studied from the perspective of Gibbs measures: it was shown to satisfy the Dobrushin--Lanford--Ruelle equations \cite{MR4178183}. Our approach instead starts from the BBGKY/loop equation hierarchy and proves that it uniquely determines the process.

Another approach to uniqueness of the Sine$_\beta$ process is based on an
infinite-volume variational principle. Lebl\'e and Serfaty
\cite{leble2017large} established a large deviation principle for microscopic
empirical fields of log gases, whose rate function is an infinite-volume free
energy consisting of a specific relative entropy term and a renormalized
logarithmic energy term. Erbar, Huesmann, and Lebl\'e \cite{erbar2021one}
subsequently proved that, in one dimension, this free energy has a unique
minimizer, namely the Sine$_\beta$ point process.

The relation between this variational result and the loop-equation approach is
not direct. In particular, it is not known whether solving the full hierarchy
of loop equations forces a point process to minimize the infinite-volume
log-gas free energy. Moreover, the variational problem in
\cite{erbar2021one} is formulated for stationary point processes,  it  is
intrinsically a bulk theory after local averaging. It therefore does not cover non-stationary edge
limits, such as the Airy$_\beta$ process.

\subsection{Proof ideas}
We explain the proof in the bulk case; the proof of the edge theorem is parallel,
with the bulk loop equations replaced by the edge loop equations and
\(\mathrm{Sine}_\beta\) replaced by \(\mathrm{Airy}_\beta\). The argument has three main steps.\\

\noindent
\emph{Step 1: the loop equations imply concentration of the Stieltjes transform.}
The first role of the loop-equation hierarchy is quantitative: it forces the
microscopic Stieltjes transform to concentrate around its deterministic values,
\(\ri\) in the upper half-plane and \(-\ri\) in the lower half-plane, with
sub-Gaussian tails. This phenomenon is
already known in the setting of general \(\beta\)-ensembles, where the loop
equations have been shown to imply an optimal local law, together with rigidity
and sharp fluctuation estimates \cite{bourgade2022optimal}. Thus, in our
argument, the local law is not an additional input but a consequence of the loop
equations. We use this concentration estimate as the asymptotic input at
infinity for the auxiliary observables introduced in the next step.\\

\noindent
\emph{Step 2: the nonlinear loop hierarchy can be linearized.}
The central observation is that the loop equation hierarchy admits a linear
reformulation after passing to the observables
\begin{equation}\label{e:introF}
F(t_1,\dots,t_n;s_1,\dots,s_m)
:=
\bE\left[
\operatorname{P.V.}
\prod_{j\in \bZ}
\frac{\prod_{r=1}^n (t_r-x_j)}
{\prod_{a=1}^m (s_a-x_j)^{\beta/2}}
\right],
\end{equation}
where the expectation is taken with respect to the random point process and
\(t_1,\dots,t_n,s_1,\dots,s_m\in \bC_+\cup \bC_-\). These observables are
exponential linear statistics of the Nevanlinna function.

Under the half-plane balance condition \eqref{e:balance-halfplanes}, which in
particular forces \(\beta>0\) to be rational, the function \(F\) satisfies the
following system of linear differential equations:
\begin{align}\label{e:intro_eq1}
\left(
\partial_{t_i}^2+1
+\frac{2}{\beta}\sum_{j\neq i}
\frac{\partial_{t_i}-\partial_{t_j}}{t_i-t_j}
-\sum_{a=1}^m
\frac{\partial_{t_i}+(2/\beta)\partial_{s_a}}{t_i-s_a}
\right)F&=0,
\qquad 1\le i\le n,
\\[1mm]
\left(
\partial_{s_a}^2+\frac{\beta^2}{4}
+\frac{\beta}{2}\sum_{b\neq a}
\frac{\partial_{s_a}-\partial_{s_b}}{s_a-s_b}
-\sum_{i=1}^n
\frac{\partial_{s_a}+(\beta/2)\partial_{t_i}}{s_a-t_i}
\right)F&=0,
\qquad 1\le a\le m.
\label{e:intro_eq2}
\end{align}

The differential system \eqref{e:intro_eq1}--\eqref{e:intro_eq2} is closely
related to deformed Calogero--Moser--Sutherland (CMS) systems with two species
of variables. These are integrable systems of particles with long-range pair
interactions, where the interaction strengths depend on whether the two
particles belong to the same species or to different species. Such systems
arise naturally from the representation theory of Lie superalgebras, which are
variants of Lie algebras with two types of generators. Deformed CMS systems were introduced and studied, in the case \(m=1\), by
Chalykh--Feigin--Veselov \cite{chalykh1998new}, and were later systematically
developed by Sergeev and Veselov in the framework of generalized root systems
and Lie superalgebras
\cite{sergeev2001superanalogs,sergeev2004deformed,sergeev2005generalised,
sergeev2015dunkl}. In the trigonometric case, the integrability of the model is
encoded by a family of mutually commuting differential operators, often called
commuting quantum integrals. Their joint eigenfunctions are the super-Jack
polynomials, introduced in \cite{kerov1998boundary}. Recent work
\cite{atai2019orthogonality,hallnas2024new} further develops the orthogonality
theory of super-Jack polynomials and its relation to rational and
rational-harmonic deformed CMS systems. In our setting, the same two-species structure appears in a non-symmetric
holonomic form, \eqref{e:intro_eq1}--\eqref{e:intro_eq2}. Rather than working
only with the commuting Hamiltonians of the integrable system, we obtain one
second-order differential equation for each external variable. This form is
naturally adapted to the loop-equation hierarchy.

Closely related holonomic systems also appear in Selberg-type integral formulas
for \(\beta\)-ensembles. In particular, Desrosiers and Liu
\cite{desrosiers2015selberg} studied a Selberg-type integral with \(n+m\)
external variables and showed that it satisfies a holonomic system of \(n+m\)
non-symmetric linear PDEs related to deformed Calogero--Sutherland systems.
They further identified a distinguished normalized solution in terms of a
hypergeometric function built from super-Jack polynomials, and applied the
result to ratios of characteristic polynomials in classical
\(\beta\)-ensembles. Our differential system
\eqref{e:intro_eq1}--\eqref{e:intro_eq2} can be viewed as the bulk scaling limit
of this type of holonomic PDE system. In contrast, the limiting system by
itself does not determine a unique local solution: as we show below, its local
solution space has dimension \(2^{n+m}\), corresponding to the \(2^{n+m}\)
possible asymptotic branches.

In the special case \(\beta=2\), the observable \eqref{e:introF} becomes an
average of ratios of characteristic polynomials. This connects the system to
the classical theory of ratios of characteristic-polynomial averages in unitary ensembles,
including the works  \cite{akemann2003universal, borodin2006averages, baik2003products,conrey2005howe}.

The coupled equations \eqref{e:intro_eq1}--\eqref{e:intro_eq2} have a simple
duality.  In the \(\bmt\)-equations, the coefficient of the interaction
between two \(t\)-variables is \(2/\beta\).  In the \(\bms\)-equations, the
corresponding coefficient between two \(s\)-variables is \(\beta/2\).  These
two coefficients are reciprocals. More precisely, if we exchange the two sets of variables and their sizes,
$
   \bmt \leftrightarrow \bms,
   n \leftrightarrow m,
$
replace
\[
   \beta \longmapsto \frac{4}{\beta},
\]
and make the corresponding rescaling of the variables by \(\beta/2\), then
the two systems of differential equations are transformed into each other, up
to an overall nonzero factor.  Since the equations are set equal to zero, this
factor does not change the solution set.  In this sense, the
\(\bmt\)-variables and the \(\bms\)-variables play dual roles.

This is the same type of \(\beta \leftrightarrow 4/\beta\) duality that
appears in \(\beta\)-ensembles.  A common example comes from averages of
characteristic polynomials.  In many such formulas, one can exchange the
number of eigenvalues with the number of inserted characteristic polynomials,
provided one also replaces \(\beta\) by \(4/\beta\).  The equations above can
be viewed as a differential-equation form of this same symmetry.  We refer to
\cite{desrosiers2009duality,forrester2025dualities} for this duality and for
a recent survey of related dualities in random matrix theory.\\

\noindent
\emph{Step 3: solve the deformed CMS system and identify the physical branch.}
Once the observables have been shown to satisfy
\eqref{e:intro_eq1}--\eqref{e:intro_eq2}, the remaining task is to solve a
linear holonomic system. Concretely, this system consists of \(n+m\)
second-order linear partial differential equations in \(n+m\) variables. Using
standard \(D\)-module arguments, we show that the corresponding differential
operators form a Gröbner basis. We refer the reader to \cite{saito2013grobner} for further background.  Hence the quotient module admits a basis given
by square-free derivatives, and the local solution space has dimension
\(2^{n+m}\). Equivalently, the system can be prolonged to a flat first-order
system, whose solutions are determined by parallel transport.

The \(2^{n+m}\) asymptotic branches can moreover be described explicitly. For
each sign pattern
\[
(\epsilon_1,\ldots,\epsilon_n,\eta_1,\ldots,\eta_m)
\in \{\pm1\}^n \times \{\pm1\}^m ,
\]
there is a solution with leading exponential behavior
\begin{equation}
\exp\left(
\ri\sum_{i=1}^n \epsilon_i t_i
+\ri\frac{\beta}{2}\sum_{a=1}^m \eta_a s_a
\right).
\end{equation}
Thus the deformed CMS system has exactly \(2^{n+m}\) linearly independent
formal branches at infinity.

Once \(F\) has been uniquely identified, the correlation functions of
the point process can be recovered from it; these correlation functions, in
turn, identify the limiting process. In the bulk case, this limiting process is
\(\mathrm{Sine}_\beta\).

\subsection*{Outline of the paper.}

In \Cref{s:application}, we develop the convergence framework used throughout
the paper and prove \Cref{t:bulk-approx-convergence,t:edge-approx-convergence}, assuming \Cref{t:bulk-characterization,t:edge-characterization}.
As an application, we apply our characterization results to Wigner matrices in \Cref{s:wigner} and to random
\(d\)-regular graphs in \Cref{s:d-regular-graph}.

In \Cref{s:sine_airy}, we show that the \(\mathrm{Sine}_\beta\) point process
satisfies the bulk loop equation, while the \(\mathrm{Airy}_\beta\) point
process satisfies the edge loop equation. These results are obtained by taking
the bulk and edge scaling limits of the corresponding loop equations for
\(\beta\)-ensembles; the detailed proofs are deferred to \Cref{s:beta_ensemble}.
In \Cref{s:bulk_concentration}, we show that the bulk loop equation implies
concentration of the Stieltjes transform. The analogous proof in the edge case
is deferred to \Cref{s:asymp_edge}.

In \Cref{s:exp_ob}, we introduce the bulk and edge exponential observables.
Using the concentration estimates as input, we derive their asymptotic behaviors
as the spectral parameters tend to infinity. In \Cref{s:linearPDE}, we prove
that these observables satisfy systems of linear differential equations related
to deformed CMS operators.

In \Cref{s:characterization}, we prove our main characterization results,
assuming the classification of solutions to the deformed CMS systems. This
classification is carried out in \Cref{s:bulk_solution,s:edge_solution}, where
we solve the bulk and edge deformed CMS systems, respectively. We show that each
system has exactly \(2^{n+m}\) linearly independent solutions, corresponding to
the \(2^{n+m}\) possible leading exponential behaviors at infinity. For
\(\beta=2\), these asymptotic solutions become exact; this is verified in
\Cref{s:b2solution} in the bulk and in \Cref{s:edge-b2solution} at the
edge.   Finally, we collect some basic estimates on Airy functions in
\Cref{a:Airy}.

\subsection*{Notation.}

We write
$
\bC_+:=\{z\in\bC:\Im z>0\}$ and 
$\bC_-:=\{z\in\bC:\Im z<0\}$ 
for the open upper and lower half-planes. We denote by \(\bQ_{>0}\) the set of
positive rational numbers and by \(\bR_{<0}\) the set of negative real numbers.
For a positive integer, we write $\qq{N}:=\{1,2,\ldots,N\}$.
For two numbers \(a,b\), we use
\[
a\wedge b:=\min\{a,b\},\qquad a\vee b:=\max\{a,b\}.
\]

For a random variable \(X\) and \(\ell>0\), we denote
\[
\|X\|_{L^\ell}:=\bE[|X|^\ell]^{1/\ell}.
\]
For a matrix \(X\), we write \(\|X\|\) for its spectral norm. For a vector
\(\bmv\), we write \(\|\bmv\|_2\) for its Euclidean norm. We denote the identity
matrix by \(\bI\), and write \(\bI_n\) when we wish to emphasize that it is the
\(n\times n\) identity matrix.

For two quantities $X_N$ and $Y_N$ depending on $N$, 
we write that $X_N = \OO(Y_N )$ or $X_N\lesssim Y_N$ if there exists some universal constant such
that $|X_N| \leq C Y_N$ . We write $X_N = \oo(Y_N )$, or $X_N \ll Y_N$ if the ratio $|X_N|/Y_N\rightarrow 0$ as $N$ goes to infinity. We write
$X_N\asymp Y_N$ if there exists a universal constant $C>0$ such that $ Y_N/C \leq |X_N| \leq  C Y_N$. 

In Definition \ref{def:s},  the Nevanlinna functions $s$ correspond to configuration with interparticle distance $\asymp 1$, in particular it differs from  the usual Stieltjes transform of a configuration of $N$ particles $\lambda_i$'s on a compact set, denoted in this paper $m_N(z)=(1/N)\sum_{i=1}^N {1}/(\lambda_i-z)$.

Throughout the paper, fractional powers such as \(\sqrt w\), \(w^{3/2}\),
\(w^{-1/4}\), and \(w^{\beta/2}\) are taken with respect to the principal
branch of the logarithm. That is, for
\(w\in \mathbb C\setminus \bR_{<0}\),
\begin{align}\label{e:principle_branch}
w^a :
      = |w|^a e^{\ri a\arg w},
\qquad \arg w\in(-\pi,\pi).
\end{align}

 \subsection*{Acknowledgements.}
The research of P.B.  was supported by the NSF Grant DMS-2348202.
The research of J.H. was supported in part by NSF Grant DMS-2246664, DMS-2331096  and a Sloan Research Fellowship.
 We thank Alice Guionnet, Michel Pain, Xin Sun, Shengjing Xu, Horng-Tzer Yau, Ofer Zeitouni and Lingfu Zhang for enlightening discussions on loop equations. P.B. is particularly thankful to Alice Guionnet for early discussions about the central question addressed in this article, namely, the  characterization  of local statistics as the unique solution of  loop equations, which began at  the IAS program {\it Non-equilibrium Dynamics and Random Matrices} in 2013-2014. We thank Amol Aggarwal and Alexei Borodin for useful references related to the algebraic approach developed in this paper.

\section{Applications}
\label{s:application}
In this section we prove \Cref{t:bulk-approx-convergence} and
\Cref{t:edge-approx-convergence}. These results state that, for rational
\(\beta>0\), any sequence of random point processes on \(\bR\) satisfying the
approximate bulk or edge loop equations with vanishing error converges,
respectively, to the $\mathrm{Sine}_{\beta}$ point process or to the
$\mathrm{Airy}_{\beta}$ point process.

To illustrate the scope of our results for random matrix models, we apply our characterization to Wigner matrices and to adjacency matrices of random \(d\)-regular graphs. First, we give a short proof of bulk universality for Wigner matrices. We then prove bulk universality for random \(d\)-regular graphs on \(N\) vertices in the growing-degree regime \(d\gg(\log N)^{24}\). Finally, we explain how the proof of edge universality for fixed-degree random \(d\)-regular graphs in \cite{huang2024ramanujan} can be significantly simplified using our characterization.

\subsection{Properties of Nevanlinna functions}

We start with some estimates on the particle-generated Nevanlinna function \(s\).
Recall the integral representation from \eqref{e:Nevanlinna_representation}.
Writing
\[
s(w)=b+cw+\sum_{x\in P}\left(\frac{1}{x-w}-\frac{x}{1+x^2}\right),
\qquad c\ge 0,
\]
we have, for \(w=E+\ri\eta\),
\begin{align}\label{e:ims}
\Im [s(w)]=c\eta+\sum_{x\in P} \frac{\eta}{|x-w|^2}.
\end{align}
Hence
\begin{align}\label{e:sbound}
|s'(w)|
\le c+\sum_{x\in P} \frac{1}{|x-w|^2}
\le \frac{\Im [s(w)]}{\eta}.
\end{align}
Also, for \(z,w\in \bC\setminus\bR\),
\begin{align}
\frac{s(z)-s(w)}{z-w}
=
c+\sum_{x\in P} \frac{1}{(x-z)(x-w)},\qquad \partial_w\left(\frac{s(z)-s(w)}{z-w}\right)
=
\sum_{x\in P} \frac{1}{(x-z)(x-w)^2}.
\end{align}
Finally,
\begin{align}\begin{split}\label{e:fest}
&\left|\frac{s(z)-s(w)}{z-w}\right|
\leq c+\frac{1}{2}\sum_{x\in P}\frac{1}{|z-x|^2}+\frac{1}{|w-x|^2}\leq \frac{\Im [s(z)]}{2 \Im [z]}+\frac{\Im [s(w)]}{2 \Im [w]},\\
&\left|
\partial_w\left(\frac{s(z)-s(w)}{z-w}\right)
\right|
\le
\sum_{x\in P} \frac{1}{\Im [z]}\frac{1}{|x-w|^2}
\le
\frac{\Im [s(w)]}{\Im [z]\, \Im [w]}.
\end{split}\end{align}

We regard Nevanlinna functions as elements of
\[
\mathcal O(\bC_+):=\{f:\bC_+\to \bC \text{ holomorphic}\}.
\]
We endow \(\mathcal O(\bC_+)\) with the topology of locally uniform convergence: a sequence
\(\{f_n\}_{n\ge1}\subset \mathcal O(\bC_+)\) converges to
\(f\in \mathcal O(\bC_+)\) if and only if \(f_n\to f\) uniformly on every
compact subset of \(\bC_+\). With this topology, \(\mathcal O(\bC_+)\) is a
Fr\'echet space. Throughout this section, convergence of Nevanlinna functions
is understood in this sense.

The following lemma shows that locally uniform convergence of Nevanlinna
functions implies vague convergence  of the associated Borel measures. We postpone the proof of
the lemma to the end of this section.
\begin{lemma}\label{l:boundary_measure}
Let \(s\) and \(\{s_N\}_{N\ge1}\) be Nevanlinna functions, and let \(\mu\) and \(\{\mu_N\}_{N\ge1}\) be the corresponding Borel measures in the representation \eqref{e:Nevanlinna_representation}. If \(s_N\) converges to \(s\) uniformly on compact subsets of \(\bC_+\), then \(\mu_N\) converges to \(\mu\) in the vague topology.

Moreover, if each \(\mu_N\) is a configuration, then \(\mu\) is also a configuration. Equivalently, if each \(s_N\) is particle-generated, then so is \(s\).
\end{lemma}

\begin{proof}[Proof of \Cref{l:boundary_measure}]
From the representation \eqref{e:Nevanlinna_representation} of Nevanlinna functions, we have
\begin{align}
s(w)=b+cw+\int_{\bR}\left(\frac{1}{\la-w}-\frac{\la}{1+\la^2}\right)\rd \mu(\la),\quad
s_N(w)=b_N+c_N w+\int_{\bR}\left(\frac{1}{\la-w}-\frac{\la}{1+\la^2}\right)\rd \mu_N(\la),
\end{align}
where \(b,b_N,c,c_N\in\bR\), \(c,c_N\ge0\).
By taking imaginary parts, and using that \(s_N(z)\) converges to \(s(z)\), we have
\begin{align}\label{e:ims0}
\left(c_N +\int \frac{\rd\mu_N(\la)}{(\la-x)^2+y^2}\right) y=\Im[s_N(x+\ri y)]\rightarrow \Im[s(x+\ri y)]=\left(c +\int \frac{\rd\mu(\la)}{(\la-x)^2+y^2}\right) y.
\end{align}
By taking \(x+\ri y=\ri\), we conclude that
\begin{align}\label{e:uniformbb}
\sup_{N\geq 1}\left(c_N+\int_\bR \frac{\rd\mu_N(\la)}{1+\la^2}\right)<\infty,\quad c+\int_\bR \frac{\rd\mu(\la)}{1+\la^2}<\infty.
\end{align}

We now prove vague convergence. Let \(f\in C_c^\infty(\bR)\), and choose \(M\ge1\) such that
$
\supp f\subset[-M,M]$.
Fix \(y\in(0,1]\). We write
\begin{align*}
\int_\bR f(x)\,\Im[ s_N(x+\ri y)]\,\rd x
-\int_\bR f(x)\,\Im[ s(x+\ri y)]\,\rd x =
\int_\bR f(x)\,\Im\bigl[s_N(x+\ri y)-s(x+\ri y)\bigr]\,\rd x,
\end{align*}
which tends to \(0\) as \(N\to\infty\) by local uniform convergence.
Thus it suffices to show that
\begin{align}\label{e:poisson-approx}
\int_\bR f(x)\,\Im[ s_N(x+\ri y)]\,\rd x
\longrightarrow
\pi \int_\bR f(\la)\,\rd\mu_N(\la)
\qquad\text{as }y\downarrow0,
\end{align}
with an error bound uniform in \(N\), and similarly for \(s,\mu\). Once this is established, we obtain
\[
\int_\bR f\,\rd\mu_N \longrightarrow \int_\bR f\,\rd\mu,
\]
hence \(\mu_N\to\mu\) vaguely.

The statement \eqref{e:poisson-approx} follows from showing the following estimate and its analogue obtained by replacing \(\Im[s_N(x+\ri y)]\) and \(\mu_N\) with \(\Im[s(x+\ri y)]\) and \(\mu\):
\begin{align}\begin{split}\label{e:shift_diff}
&\phantom{{}={}}\left|\int_\bR \Im[s_N(x+\ri y)] f(x)\rd x-\pi\int_\bR  f(x)\rd\mu_N(x)\right|\\
&\leq C\bigl(M y\|f\|_\infty+M^2 y\ln (M/y)\|f'\|_\infty\bigr)\int_\bR \frac{\rd \mu_N(\la)}{1+\la^2}
+2c_N M \|f\|_\infty y.
\end{split}\end{align}
In fact, by \eqref{e:uniformbb}, the right-hand side of \eqref{e:shift_diff} tends to \(0\) as \(y\to 0+\), uniformly in \(N\).

Using \eqref{e:ims0}, we can rewrite the left-hand side of \eqref{e:shift_diff} as 
\begin{align}\label{e:diyi}
\int_\bR \Im[s_N(x+\ri y)] f(x)\rd x-\pi\int_\bR  f(x)\rd\mu_N(x)
=\int\left(\int_\bR \frac{y f(x) \rd x}{(x-\la)^2+y^2}-\pi f(\la)\right)  \rd \mu_N(\la) +c_N y\int_\bR f(x)\rd x,
\end{align}
where the last term is bounded by \(2c_NM\|f\|_\infty y\).

To bound the first term on the right-hand side of \eqref{e:diyi}, we consider two cases. For \(|\la|\geq 2M\), we have 
\begin{align}\begin{split}\label{e:dier}
\int_\bR \frac{y f(x) \rd x}{(x-\la)^2+y^2}
\leq \int_{-M}^M \frac{y \|f\|_\infty \rd x}{(x-\la)^2+y^2}
\leq \frac{10 My\|f\|_\infty }{1+\la^2}.
\end{split}\end{align}
For \(|\la|\leq 2M\),
\begin{align}\begin{split}\label{e:disan}
&\phantom{{}={}}\int_\bR \frac{y f(x) \rd x}{(x-\la)^2+y^2}-\pi f(\la)=
\int_{-4M}^{4M} \frac{y \bigl(f(x)-f(\la)\bigr) \rd x}{(x-\la)^2+y^2}
-\int_{\bR\setminus[-4M,4M]} \frac{y f(\la)\rd x}{(x-\la)^2+y^2}\\
&\lesssim 
y\|f'\|_\infty\int_{-4M}^{4M}\frac{|x-\la|\rd x}{(x-\la)^2+y^2}
+\|f\|_\infty\int_{x\in \bR\setminus[-4M, 4M]}\frac{y}{(x-\la)^2+y^2}\rd x\\
&\lesssim
y\ln(M/y)\|f'\|_\infty+\frac{y\|f\|_\infty}{M}\leq \frac{8M^2}{1+\la^2}\left(y\ln(M/y)\|f'\|_\infty+\frac{y\|f\|_\infty}{M}\right).
\end{split}\end{align}
By plugging \eqref{e:dier} and \eqref{e:disan} into \eqref{e:diyi}, we get
\begin{align}
\left|\int\left(\int_\bR \frac{y f(x) \rd x}{(x-\la)^2+y^2}-\pi f(\la)\right)  \rd \mu_N(\la)\right|
\leq C\bigl(M y\|f\|_\infty+M^2 y\ln (M/y)\|f'\|_\infty\bigr)\int_\bR \frac{\rd \mu_N(\la)}{1+\la^2},
\end{align}
and the claim \eqref{e:shift_diff} follows.

It remains to show that \(\mu\) is also a configuration. From \eqref{e:uniformbb} we know that \(\mu\) is locally finite. Now let \(I=(a,b)\) be a bounded interval such that
\(\mu(\{a\})=\mu(\{b\})=0\). By vague convergence,
\[
\mu(I)=\lim_{N\to\infty}\mu_N(I).
\]
Since each \(\mu_N\) is a configuration, \(\mu_N(I)\in \bZ_{\ge 0}\).  So its limit \(\mu(I)\) must also belong to \(\bZ_{\ge0}\).

Thus \(\mu\) assigns a nonnegative integer mass to every bounded interval whose endpoints are not atoms of \(\mu\). It follows that \(\mu\) is purely atomic, and each atom has positive integer mass. Since \(\mu\) is locally finite, it has only finitely many atoms in each bounded interval. Therefore \(\mu\) is a configuration.
\end{proof}

\subsection{Proof of \Cref{t:bulk-approx-convergence} and \Cref{t:edge-approx-convergence}}

\begin{proof}[Proof of \Cref{t:bulk-approx-convergence}]
The proof consists of four steps.

\noindent\emph{Step 1. Pointwise moment bounds.}
Let \(K\subset \bC_+\) be compact and set
\begin{align}\label{e:etaK}
\eta_K:=\inf_{z\in K}\Im z>0.
\end{align}
Fix an integer \(q\ge 2\). Taking \(p=2q-1\), and choosing $w_1=w$, and
\(w_2,w_3,\dots,w_p\) so that \(q-2\) of them are equal to \(w\) and the
remaining \(q\) are equal to \(\bar w\), in the bulk loop equation
\eqref{e:approx-loopeq2-bulk}, we get, since
\(s_N(\bar w)=\overline{s_N(w)}\),
\[
s_N(w)^2
\prod_{j=2}^{2q-1}s_N(w_j)
=
s_N(w)^2 s_N(w)^{q-2}s_N(\bar w)^q
=
|s_N(w)|^{2q}.
\]
Moreover,
$
\cR_N=\max\{|s_N(w_1)|,\dots,|s_N(w_p)|\}=|s_N(w)|$.
Thus, after isolating the leading quadratic term and taking absolute
values, \eqref{e:approx-loopeq2-bulk} gives
\begin{align}
\begin{split}\label{e:mmbound}
&\bE\left[
\bm1({\Omega_N})
|s_N(w)|^{2q}
\right] 
\leq C
\bE\left[
\bm1({\Omega_N})
\bigl(|\del_{w} s_N(w)|+1\bigr) |s_N(w)|^{2q-2}
\right] 
\\
&\quad
+C\bE\left[
\bm1({\Omega_N})
\sum_{j=1}^{2q-2}\left|
\del_{w_j}
\frac{s_N(w)-s_N(w_j)}{w-w_j}\right|
|s_N(w)|^{2q-3}
\right]
+\oo_N(1)\left(1+\bE\left[\bm1(\Omega_N)|s_N(w)|^{2q}\right]\right),
\end{split}
\end{align}
where $C$ depends only on $\beta$ and the \(\oo_N(1)\) error is uniform for \(w\in K\). 

For a particle-generated Nevanlinna function \(s_N\), the estimates
\eqref{e:sbound} and \eqref{e:fest} imply that, uniformly for
\(w\in K\) and \(z\in\{w,\bar w\}\),
\begin{align}\label{e:derivative-ui-bound}
|\del_w s_N(w)|
\leq \frac{|s_N(w)|}{\Im w}
\leq C_K|s_N(w)|,
\qquad
\left|
\del_z\frac{s_N(w)-s_N(z)}{w-z}
\right|
\leq C_K|s_N(w)|.
\end{align}
Plugging \eqref{e:derivative-ui-bound} into \eqref{e:mmbound}, we obtain
\begin{align}
\begin{split}\label{e:mmbound-reduced}
\bE\left[
\bm1({\Omega_N})
|s_N(w)|^{2q}
\right] 
\leq C_K
\bE\left[
\bm1({\Omega_N})
\bigl(|s_N(w)|^{2q-1}+q|s_N(w)|^{2q-2}\bigr)
\right]
+\oo_N(1)\left(1+\bE\left[\bm1(\Omega_N)|s_N(w)|^{2q}\right]\right).
\end{split}
\end{align}
We now apply Young's inequality in the forms
\[
C_K b^{2q-1}\le \frac14 b^{2q}+C_{q,K},
\qquad
C_K q b^{2q-2}\le \frac14 b^{2q}+C_{q,K}.
\]
Using these inequalities in \eqref{e:mmbound-reduced} and absorbing the
\(\oo_N(1)\)-term for all sufficiently large \(N\), we conclude that
there exists a constant \(C_{q,K}>0\) such that
\begin{equation}\label{e:pointwise-momentbb}
\sup_{z\in K}
\bE\bigl[\bm1(\Omega_N)|s_N(z)|^{2q}\bigr]
\leq C_{q,K}
\end{equation}
for all sufficiently large \(N\).

\noindent\emph{Step 2. Compact-uniform moment bounds.}
We next upgrade the pointwise estimate \eqref{e:pointwise-momentbb} to a
compact-uniform estimate. Let \(\delta>0\), and choose a finite
\(\delta\)-net \(\mathcal N_\delta\subset K\). Thus, for every
\(z\in K\), there exists \(w\in\mathcal N_\delta\) with
$
|z-w|\leq \delta$.
For such \(z\) and \(w\), define the line segment
\[
\gamma(t):=w+t(z-w),
\qquad 0\leq t\leq 1.
\]
Then
$
\Im \gamma(t)\geq \eta_K$ (recall from \eqref{e:etaK}).
Using the derivative estimate \eqref{e:derivative-ui-bound}
\[
|s_N'(\gamma(t))|
\leq \frac{|s_N(\gamma(t))|}{\Im \gamma(t)}
\leq \frac{|s_N(\gamma(t))|}{\eta_K},
\]
we get, for \(0\leq t\leq 1\),
\[
|s_N(\gamma(t))|
\leq
|s_N(w)|
+
|z-w|
\int_0^t
|s_N'(\gamma(r))|\,\rd r
\leq
|s_N(w)|
+
\frac{\delta}{\eta_K}
\int_0^t
|s_N(\gamma(r))|\,\rd r .
\]
By Gronwall's inequality,
\[
|s_N(z)|
\leq
\exp\left(\frac{\delta}{\eta_K}\right)
|s_N(w)|.
\]
Consequently,
\begin{equation}\label{e:finite-net-control}
\sup_{z\in K}|s_N(z)|^{2q}
\leq
\exp\left(\frac{2q\delta}{\eta_K}\right)
\sum_{w\in\mathcal N_\delta}
|s_N(w)|^{2q}.
\end{equation}
Multiplying by \(\bm1(\Omega_N)\), taking expectations, and using
\eqref{e:pointwise-momentbb}, we obtain
\begin{equation}\label{e:sup-momentbb}
\bE\left[
\bm1(\Omega_N)
\sup_{z\in K}|s_N(z)|^{2q}
\right]
\leq
C_{q,K}.
\end{equation}
By the symmetry
\(s_N(\bar z)=\overline{s_N(z)}\), the same estimate also holds for
compact sets \(K\subset \bC\setminus\bR\).

\noindent\emph{Step 3. Tightness and subsequential limits.}
By Markov's inequality, \eqref{e:sup-momentbb} implies
\[
\bP\left(
\Omega_N
\cap
\left\{
\sup_{z\in K}|s_N(z)|>M
\right\}
\right)
\leq
\frac{C_{q,K}}{M^{2q}}.
\]
Given \(\varepsilon>0\), choosing \(M\) sufficiently large gives
\begin{equation}\label{e:tight-on-Omega}
\limsup_{N\to\infty}
\bP\left(
\Omega_N
\cap
\left\{
\sup_{z\in K}|s_N(z)|>M
\right\}
\right)
\leq
\varepsilon.
\end{equation}
Since \(\bP(\Omega_N^\complement)=\oo_N(1)\), it follows that
\begin{equation}\label{e:tight-without-Omega}
\limsup_{N\to\infty}
\bP\left(
\sup_{z\in K}|s_N(z)|>M
\right)
\leq
\varepsilon.
\end{equation}
Thus the sequence \(\{s_N\}\) is tight in the topology of locally
uniform convergence on compact subsets of \(\bC\setminus\bR\).

Now take an arbitrary subsequence \(N_1<N_2<N_3<\cdots\). By tightness, after passing to a further subsequence, still denoted by \(N_j\), we may assume that
$
s_{N_j}\Longrightarrow s
$
in the topology of locally uniform convergence on compact subsets of \(\bC_+\). We also keep track of the indicators \(\bm1(\Omega_{N_j})\). Since
$
\bP(\Omega_{N_j})=1-\oo_{N_j}(1),
$
the pair \((s_{N_j},\bm1(\Omega_{N_j}))\) converges subsequentially in distribution to \((s,1)\).

By the Skorokhod representation theorem, we may realize \(s_{N_j}\), \(s\), and \(\bm1(\Omega_{N_j})\) on the same probability space so that
$
s_{N_j}\longrightarrow s
$
almost surely uniformly on compact subsets of \(\bC_+\), and
$
\bm1(\Omega_{N_j})\longrightarrow 1
$
almost surely.

Since each \(s_{N_j}\) is a Nevanlinna function, the locally uniform limit \(s\) is again a Nevanlinna function. Since each \(s_{N_j}\) is particle-generated, \Cref{l:boundary_measure} implies that \(s\) is also particle-generated. If \(\mu\) denotes the Borel measure corresponding to \(s\) in the representation \eqref{e:Nevanlinna_representation}, then, along this subsequence,
$
\mu_{N_j}\longrightarrow \mu
$
almost surely in the vague topology.

\noindent\emph{Step 4. Passing to the exact loop equations.}
It remains to pass to the limit in the approximate loop equations. Fix \(p\ge1\) and \(w_1,\dots,w_p\in \bC\setminus\bR\). We first record the needed uniform integrability. 
Combining \eqref{e:derivative-ui-bound}, \eqref{e:pointwise-momentbb}, and H\"older's inequality, we obtain that, for every fixed \(p\) and every fixed \(w_1,\dots,w_p\in \bC\setminus\bR\), the random variables
\begin{align}\label{e:uniform_integrable}
\bm1(\Omega_N)
\left(\frac{2-\beta}{\beta}\del_{w_1} s_N(w_1)+s_N(w_1)^2+1\right)
\prod_{j=2}^p s_N(w_j),\quad 
\bm1(\Omega_N)
\sum_{j=2}^p
\del_{w_j}\frac{s_N(w_1)-s_N(w_j)}{w_1-w_j}
\prod_{\substack{\ell=2\\ \ell\neq j}}^ps_N(w_\ell)
\end{align}
are uniformly integrable.

On the Skorokhod probability space, the almost sure locally uniform convergence \(s_{N_j}\to s\) implies, by Cauchy's integral formula, that the derivatives also converge locally uniformly. Therefore the integrands in \eqref{e:approx-loopeq2-bulk}, multiplied by \(\bm1(\Omega_{N_j})\), converge almost surely to the corresponding expressions with \(s\) in place of \(s_N\), since \(\bm1(\Omega_{N_j})\to1\) almost surely. By uniform integrability \eqref{e:uniform_integrable}, we may pass to the limit in expectation. Thus, using the approximate loop equation \eqref{e:approx-loopeq2-bulk},  this gives
\begin{align}
\bE\left[\left(\frac{2-\beta}{\beta}\del_{w_1} s(w_1)+s(w_1)^2+1\right)\prod_{j=2}^p s(w_j)\right]
+\frac{2}{\beta}\bE\left[\sum_{j=2}^p
\del_{w_j}\frac{s(w_1)-s(w_j)}{w_1-w_j}
\prod_{\substack{\ell=2\\ \ell\neq j}}^ps(w_\ell)\right]
=0.
\end{align}
Thus \(s\) satisfies the exact bulk loop equations.

By the characterization \Cref{t:bulk-characterization} for the exact bulk loop equations, the configuration \(\mu\) has the same law as the \(\mathrm{Sine}_{\beta}\) point process. We have shown that every subsequence of \(\{\mu_N\}_{N\ge1}\) has a further subsequence converging in distribution to the \(\mathrm{Sine}_{\beta}\)  point process. Therefore the full sequence \(\mu_N\) converges in distribution to the \(\mathrm{Sine}_{\beta}\) point process.
\end{proof}

\begin{proof}[Proof of \Cref{t:edge-approx-convergence}]
Using \Cref{t:edge-characterization} as input, the proof of
\Cref{t:edge-approx-convergence} is identical to the proof of
\Cref{t:bulk-approx-convergence}. We therefore omit the details.
\end{proof}

\subsection{Wigner matrices}
\label{s:wigner}

In this section, we illustrate how our characterization can be applied to 
Wigner matrices. Bulk universality for Wigner matrices was previously
established in several forms: averaged-energy universality and gap universality
were proved in
\cite{tao2011random,MR2810797,erdHos2010bulk,erdHos2012rigidity,MR2919197,erdHos2012bulk,erdHos2015gap},
while fixed-energy universality was proved later in
\cite{bourgade2016fixed,landon2019fixed}.
We show that the local semicircle law, together with a cumulant expansion, implies the microscopic approximate loop equation. Consequently, by \Cref{t:bulk-approx-convergence}, this yields a short proof of bulk universality at fixed  energy.   

\begin{definition}[Real Wigner matrix]
\label{d:wigner}
Let \(H=(h_{ij})_{i,j=1}^N\) be an \(N\times N\) real symmetric random matrix.
We assume that the upper-triangular entries
\(\{h_{ij}:1\le i\le j\le N\}\) are independent, and that
\(h_{ji}=h_{ij}\) for \(i\le j\). We assume
\[
        \bE[ h_{ij}]=0,
        \qquad
        \bE[(\sqrt N h_{ij})^2]=1,
        \qquad 1\le i,j\le N.
\]
Moreover, for every \(q\in\mathbb N\), there exists a constant \(C_q>0\),
independent of \(N,i,j\), such that
$
        \bE [|\sqrt N h_{ij}|^q]\le C_q.
$
\end{definition}

\begin{proposition}[Microscopic approximate loop equation for Wigner matrices]
\label{p:wigner-microscopic-loop}
Consider real Wigner matrices \(H\) as in Definition \ref{d:wigner}, with eigenvalues
\(\lambda_1\ge \lambda_2\ge \cdots\ge \lambda_N\). Fix \(E\in(-2,2)\), and set
\begin{align}
\label{e:defrhoE}
        \varrho_E:=\varrho_{\rm sc}(E):=\frac{\sqrt{4-E^2}}{2\pi}.
\end{align}
For \(w\in\bC\setminus\bR\), define the rescaled Stieltjes transform
\begin{align}
\label{e:defsN}
        s_N(w)
        :=
        \sum_{i=1}^N
        \frac{1}{N\pi\varrho_E(\lambda_i-E)-w}
        +
        \frac{E}{2\pi\varrho_E}.
\end{align}
Then, for every \(p\ge1\) and every compact set
\(K\subset\bC\setminus\bR\), uniformly for
\(w,w_2,\dots,w_p\in K\), we have
\begin{align}
\label{e:rrg-microscopic-approx-loop}
&
\bE\left[
\left(
\partial_w s_N(w)+s_N(w)^2+1
\right)
\prod_{j=2}^p s_N(w_j)
\right]
\nonumber\\
&\qquad
+
2\bE\left[
\sum_{j=2}^p
\partial_{w_j}
\frac{s_N(w)-s_N(w_j)}{w-w_j}
\prod_{\substack{i=2\\ i\neq j}}^p s_N(w_i)
\right]
=
\oo_N(1)
\left(
1+
\bE\left[\cR_N^{p+1}\right]
\right),
\end{align}
where
\[
        \cR_N
        :=
        \max\left\{|s_N(w)|,|s_N(w_2)|,\dots,|s_N(w_p)|\right\}.
\]
 As a consequence,
by \Cref{t:bulk-approx-convergence} with \(\beta=1\), the rescaled bulk
eigenvalue statistics converge to the Sine\(_1\) point process.
\end{proposition}

\subsubsection{Preparatory estimates}
We recall the following local semicircle law for Wigner matrices from \cite[Theorem 2.3]{erdos2013local}.
\begin{theorem}[{\cite[Theorem 2.3]{erdos2013local}}]
\label{t:local-semicircle-law}
Consider a real Wigner matrix as in \Cref{d:wigner}. Fix \(\gamma>0\). Then,
for any small \(\varepsilon>0\) and large \(D>0\), the following holds with
probability at least \(1-N^{-D}\), provided \(N\) is sufficiently large.
Uniformly for spectral parameters \(z=u+\ri\eta\) satisfying
\(|u|\le 10\) and \(N^{-1+\gamma}\le \eta\le 10\),
\begin{align}
\label{e:local-law-entry}
        \max_{i,j}
        \left|
        G_{ij}(z)-\delta_{ij}m_{\rm sc}(z)
        \right|
        \le
        N^\varepsilon
        \left(
        \sqrt{\frac{\Im [m_{\rm sc}(z)]}{N\eta}}
        +
        \frac{1}{N\eta}
        \right).
\end{align}
\end{theorem}

Fix, once and for all, a small exponent
\[
        0<\gamma<\frac1{12}.
\]
Choosing \(\varepsilon<\gamma/2\) in \Cref{t:local-semicircle-law}, using
complex conjugation to pass from the upper half-plane to the lower half-plane,
and using the moment assumption in \Cref{d:wigner}, we obtain the following
good event. For every fixed \(D>0\), after increasing the moment order in the
entry estimate if necessary,
\[
        \bP(\Omega_N^\complement)\le N^{-D},
\]
where
\begin{align}
\label{e:defWOmegaN}
\Omega_N
:=
\left\{
\begin{array}{l}
H:\displaystyle\max_{i,j}|h_{ij}|\le N^{-1/2+\gamma},
\\[0.4em]
\displaystyle
\max_{i,j}|G_{ij}(z)|\le 2
\quad
\text{for all }z=u+\ri\eta,
\ |u|\le 10,
\ N^{-1+\gamma}\le |\eta|\le 10
\end{array}
\right\}.
\end{align}

\begin{lemma}
\label{l:Gijest}
On the event \(\Omega_N\), for every \(z=u+\ri\eta\in\bC\setminus\bR\) with
\(|u|\le 10\),
\begin{equation}
\label{eq:microscopic-entry-control}
        \max_{i,j}|G_{ij}(z)|
        \le
        \Lambda(z),
        \qquad
        \Lambda(z):=
        2\cdot \max\left\{1,\frac{N^\gamma}{N|\eta|}\right\}.
\end{equation}
In particular, if \(z=E+w/(N\pi\varrho_E)\) with \(w\) in a compact subset
\(K\subset\bC\setminus\bR\), then
\[
        \Lambda(z)\le C_KN^\gamma.
\]
\end{lemma}

\begin{proof}
It is enough to prove the estimate in the upper half-plane; the lower
half-plane follows by complex conjugation. Fix \(|u|\le10\) and define
\[
        M(s):=\max_{a,b}|G_{ab}(u+\ri s)|,
        \qquad s>0.
\]
If \(N^{-1+\gamma}\le s\le10\), then the bound
\begin{align}
\label{e:Mtrivial}
        M(s)\le2
\end{align}
holds directly from the definition of \(\Omega_N\). If \(s>10\), then the
trivial resolvent bound gives \(M(s)\le s^{-1}\le1\).

For \(s>0\), the identity \(\partial_sG(u+\ri s)=\ri G(u+\ri s)^2\) and the
Ward identity imply
\begin{align}
\label{e:derbb}
\begin{split}
        |\partial_s G_{ab}(u+\ri s)|
        &\le
        \sum_{k=1}^N |G_{ak}(u+\ri s)|\,|G_{kb}(u+\ri s)|
        \\
        &\le
        \left(\sum_k |G_{ak}(u+\ri s)|^2\right)^{1/2}
        \left(\sum_k |G_{kb}(u+\ri s)|^2\right)^{1/2}
        \\
        &=
        \frac{
        \sqrt{\Im [G_{aa}(u+\ri s)]\,\Im[ G_{bb}(u+\ri s)]}
        }{s}
        \le
        \frac{M(s)}{s}.
\end{split}
\end{align}
Now suppose \(0<s\le N^{-1+\gamma}\). For any indices \(a,b\),
\[
        |G_{ab}(u+\ri s)|
        \le
        |G_{ab}(u+\ri N^{-1+\gamma})|
        +
        \int_s^{N^{-1+\gamma}}
        |\partial_rG_{ab}(u+\ri r)|\,\rd r.
\]
Using \eqref{e:derbb}, we get
\[
        M(s)
        \le
        2+
        \int_s^{N^{-1+\gamma}}\frac{M(r)}{r}\,\rd r.
\]
By Gronwall's inequality,
\begin{align}
\label{e:M_Gronwall}
        M(s)
        \le
        2\exp\left(\int_s^{N^{-1+\gamma}}\frac{\rd r}{r}\right)
        =
        \frac{2N^\gamma}{Ns}.
\end{align}
Combining \eqref{e:Mtrivial} and \eqref{e:M_Gronwall} proves
\eqref{eq:microscopic-entry-control}.
\end{proof}

We recall the following cumulant expansion from \cite[Lemma 2.4]{he2018isotropic}.
\begin{lemma}[{\cite[Lemma 2.4]{he2018isotropic}}]
\label{l:cumulant-expansion}
Let \(h\) be a centered real random variable with
\[
        \bE h=0,
        \qquad
        \bE h^2=\frac1N,
\]
and finite moments of all orders. Let \(f:\mathbb R\to\mathbb C\) be smooth.
Then
\[
        \bE[hf(h)]
        =
        \frac1N\bE[f'(h)]
        +
        \mathcal R,
\]
where
\[
        |\mathcal R|
        \le
        \bE\left[
        \left(N^{-1}|h|+|h|^3\right)
        \sup_{|x|\le |h|}|f''(x)|
        \right].
\]
\end{lemma}

For \(i\le j\) and a real parameter \(x\), define
\begin{align}
\label{e:deltaij}
        (\Delta_{ij})_{ab}
        :=
        \delta_{ia}\delta_{jb}
        +
        \bm1(i\neq j)\delta_{ib}\delta_{ja},
        \qquad 1\le a,b\le N,
\end{align}
and
\begin{align}
\label{e:changeone}
        H^{(ij,x)}
        :=
        H+(x-h_{ij})\Delta_{ij},
        \quad
        G^{(ij,x)}(z)
        :=(H^{(ij,x)}-z)^{-1},
        \quad
        m_N^{(ij,x)}(z)
        :={1\over N}\Tr G^{(ij,x)}(z).
\end{align}
For a differentiable function \(F\) of the matrix entries, set
\[
        \partial_{ij}F(H)
        :=
        \left.
        \frac{\rd}{\rd x}F\bigl(H^{(ij,x)}\bigr)
        \right|_{x=h_{ij}},
        \qquad i\le j,
\]
and use the convention \(\partial_{ji}:=\partial_{ij}\) for \(i\leq j\).

\begin{lemma}[Single-entry resolvent stability]
\label{lem:Wigner-resolvent-expansion}
Assume that, for a spectral parameter \(z\in\bC\setminus\bR\),
\[
        \max_{a,b}|G_{ab}(z)|\le\Lambda,
        \qquad
        \max_{i,j}|h_{ij}|\le \Lambda N^{-1/2},
        \qquad
        \Lambda\le CN^\gamma,
\]
where \(\gamma>0\) is sufficiently small. Then, uniformly in all indices
\(a,b,i,j\in\qq{N}\),
\begin{align}
\label{e:Gbb}
        \sup_{|x|\le \Lambda N^{-1/2}}
        \left|
        G^{(ij,x)}_{ab}(z)-G_{ab}(z)
        \right|
        \le
        C N^{-1/2}\Lambda^3,
\end{align}
and
\begin{align}
\label{e:mbb}
        \sup_{|x|\le \Lambda N^{-1/2}}
        \left|
        m_N^{(ij,x)}(z)-m_N(z)
        \right|
        \le
        C N^{-1/2}\Lambda^3.
\end{align}
\end{lemma}

\begin{proof}
 We use the finite-rank form of the resolvent identity. For \(i\neq j\), write
\[
        \Delta_{ij}=UV^\top,
        \qquad
        U=(e_i,e_j),
        \qquad
        V=(e_j,e_i).
\]
For \(i=j\), the same argument applies with \(U=V=e_i\). By the Woodbury
formula,
\[
        G^{(ij,x)}
        =
        G-
        (x-h_{ij})
        GU
        \bigl(\bI+(x-h_{ij})V^\top GU\bigr)^{-1}
        V^\top G.
\]
Every entry of \(GU\), \(V^\top G\), and \(V^\top GU\) is a linear combination
of at most two entries of \(G\). Hence
\[
        \max_{\alpha,\beta}
        |(V^\top GU)_{\alpha\beta}|
        \le C\Lambda.
\]
Since by our assumption,$ |x|+|h_{ij}|\le 2\Lambda N^{-1/2}$, we have
\[
        |x-h_{ij}|\Lambda
        \le
        2\Lambda^2N^{-1/2}
        \le
        2N^{-1/2+2\gamma},
\]
the matrix \(\bI+(x-h_{ij})V^\top GU\) is invertible for \(N\) large,
provided \(\gamma<1/4\), and its inverse has norm at most \(C\). Therefore,
for any \(a,b\),
\[
        |(G^{(ij,x)}-G)_{ab}|
  \le
        C|x-h_{ij}|
        \max_\alpha |(GU)_{a\alpha}|
        \max_\beta |(V^\top G)_{\beta b}|
\le
        C|x-h_{ij}|\Lambda^2
        \le
        C N^{-1/2}\Lambda^3.
\]
This proves \eqref{e:Gbb}, and \eqref{e:mbb} follows by taking the normalized
trace.
\end{proof}

\subsubsection{Proof of \Cref{p:wigner-microscopic-loop}}
\Cref{p:wigner-microscopic-loop} follows from cumulant expansion, and the following cumulant error estimate. We postpone its proof to the end of this section.
\begin{lemma}[Cumulant error estimate]
\label{l:errorE}
Fix \(p\ge1\) and a compact set \(K\subset\bC\setminus\bR\). Let
\(w,w_2,w_3,\dots,w_p\in K\), and set
\begin{align}
\label{e:wigner-za}
      z:=E+\frac{w}{N\pi \varrho_E},\quad   z_a:=E+\frac{w_a}{N\pi\varrho_E},
        \qquad 2\le a\le p.
\end{align}
Let
\[
        M_p:=\prod_{a=2}^p m_N(z_a),
        \qquad
        R_N:=\max\{ |m_N(z)|, |m_N(z_2)|,\cdots, |m_N(z_p)|\}.
\]
For \(i,j\in\qq{N}\), define
\[
        f_{ij}(x)
        :=
        G_{ij}^{(ij,x)}(z)
        \prod_{a=2}^p m_N^{(ij,x)}(z_a).
\]
Then
\begin{align}
\label{e:errorbb}
\frac1N
\sum_{i,j=1}^N
\bE\left[
\left(N^{-1}|h_{ij}|+|h_{ij}|^3\right)
\sup_{|x|\le |h_{ij}|}
|\partial_x^2f_{ij}(x)|
\right]
+
\frac1{N^2}
\sum_{i=1}^N
\left|
\bE\left[\partial_{ii}\left(G_{ii}(z)M_p\right)\right]
\right|
=
\oo_N(1)
\left(1+\bE[R_N^{p+1}]\right)
\end{align}
uniformly for \(w, w_2,w_3,\dots,w_p\in K\). 
\end{lemma}

\begin{proof}[Proof of \Cref{p:wigner-microscopic-loop}]
Fix \(w,w_2,w_3,\cdots,w_p\in K\), and define \(z, z_2,z_3 \cdots,z_p \) as
in \eqref{e:wigner-za}. Also set
\[
        M_p:=\prod_{a=2}^p m_N(z_a),
        \qquad
        M_p^{(a)}:=
        \prod_{\substack{b=2\\ b\neq a}}^p m_N(z_b),
        \qquad
        R_N:=\max\{ |m_N(z)|, |m_N(z_2)|,\cdots, |m_N(z_p)|\}.
\]

From \((H-z)G(z)=\bI\), taking the normalized trace gives
\[
        1+zm_N(z)
        =
        \frac1N\sum_{i,j=1}^N h_{ij}G_{ji}(z).
\]
Therefore
\begin{equation}
\label{eq:wigner-loop-start-mN}
\bE\left[
\left(1+zm_N(z)+m_N^2(z)\right)M_p
\right]
=
\frac1N
\sum_{i,j=1}^N
\bE\left[
        h_{ij}G_{ji}(z)M_p
\right]
+
\bE\left[m_N^2(z)M_p\right].
\end{equation}

We apply \Cref{l:cumulant-expansion} to the independent upper-triangular
entries. Rewriting the result as an ordered double sum, and adding one harmless
diagonal copy in order to use uniform derivative identities, gives
\begin{equation}
\label{eq:wigner-cumulant-step-mN}
\frac1N
\sum_{i,j=1}^N
\bE\left[
        h_{ij}G_{ij}(z)M_p
\right]
=
\frac1{N^2}
\sum_{i,j=1}^N
(1+\delta_{ij})
\bE\left[
        \partial_{ij}\left(G_{ij}(z)M_p\right)
\right]
+
\mathcal E_N,
\end{equation}
where the additional diagonal contribution is included in \(\mathcal E_N\), so it is given by \eqref{e:errorbb}.
By \Cref{l:errorE},
\begin{align}
\label{e:EN}
        |\mathcal E_N|
        \le
        \oo_N(1)
        \left(1+\bE[R_N^{p+1}]\right),
\end{align}
uniformly for \(w,w_2, w_3,\cdots,w_p\in K\).

We now compute the main term in \eqref{eq:wigner-cumulant-step-mN}. The
resolvent derivative identity gives
\[
        \partial_{ij}G_{kl}(z)
        =
        -G_{ki}(z)G_{jl}(z)
        -
        \bm1(i\neq j)G_{kj}(z)G_{il}(z).
\]
Consequently,
\[
        (1+\delta_{ij})\partial_{ij}G_{ij}(z)
        =
        -G_{ii}(z)G_{jj}(z)-G_{ij}^2(z),
\]
and, for every \(2\leq a\leq p\),
\[
        (1+\delta_{ij})\partial_{ij}m_N(z_a)
        =
        -\frac2NG^2_{ij}(z_a).
\]
Therefore
\begin{align}
\label{e:main-term-derivative}
\begin{split}
(1+\delta_{ij})
\partial_{ij}\left(G_{ij}(z)M_p\right)
=
\left(
-G_{ii}(z)G_{jj}(z)-G_{ij}^2(z)
\right)M_p
-
\frac{2G_{ij}(z)}{N}
\sum_{a=2}^p
G^2_{ij}(z_a)M_p^{(a)}.
\end{split}
\end{align}
Summing over \(i,j\), we obtain
\begin{align}
\label{e:main-term-summed}
\begin{split}
&
\frac1{N^2}
\sum_{i,j=1}^N
(1+\delta_{ij})
\partial_{ij}\left(G_{ij}(z)M_p\right)
=
-m_N^2(z)M_p
-
\frac1{N^2}\Tr G^2(z)\,M_p
-
\frac2{N^3}
\sum_{a=2}^p
\Tr(G(z)G^2(z_a))M_p^{(a)}.
\end{split}
\end{align}
Using
\[
        \partial_{z}m_N(z)
        =
        \frac1N\Tr G^2(z),\quad 
        \partial_{z_a}
        \frac{m_N(z)-m_N(z_a)}{z-z_a}
        =
        \frac1N\Tr(G(z)G^2(z_a)),
        \qquad 2\leq a\leq p,
\]
we get
\begin{align}
\begin{split}
\frac1{N^2}
\sum_{i,j=1}^N
\bE\left[(1+\delta_{ij})
\partial_{ij}\left(G_{ij}(z)M_p\right)
\right]
&=
-\bE[m_N^2(z)M_p]
-
\frac1N
\bE\left[(\partial_{z}m_N(z))M_p\right]
\\
&
-
\frac2{N^2}
\sum_{a=2}^p
\bE\left[
        \partial_{z_a}
        \frac{m_N(z)-m_N(z_a)}{z-z_a}
        M_p^{(a)}
\right].
\end{split}
\end{align}
Substituting this into \eqref{eq:wigner-loop-start-mN} and
\eqref{eq:wigner-cumulant-step-mN}, the term \(-\bE[m_N^2(z)M_p]\) cancels the
term \(+\bE[m_N^2(z)M_p]\). Hence
\begin{align}\begin{split}\label{e:loop1}
&
\bE\left[
\left(
1+z m_N(z)+m_N(z)^2
\right)M_p
\right]
=
-\frac1N\bE\left[
\left(
\partial_zm_N(z)
\right)M_p
\right]\\
&\qquad\quad
-\frac{2}{N^2}
\sum_{a=2}^p
\bE\left[
\partial_{z_a}
\frac{m_N(z)-m_N(z_a)}{z-z_a}
M_p^{(a)}
\right]
+
\oo_N(1)
\left(
1+
\bE\left[
R_N^{p+1}
\right]
\right).
\end{split}\end{align}

Starting from the loop equation in the form \eqref{e:loop1}, by linearity of the loop equation, we may replace
$
	M_p
$
by
$
	\prod_{a=2}^p
	(
		m_N(z_a)+E/2),
$ and get
\begin{align}\begin{split}\label{e:loop2}
&
\bE\left[
\left(
1+zm_N(z)+m_N(z)^2
+\frac1N\partial_{z}m_N(z)
\right)
\prod_{a=2}^p \left(m_N(z_a)+\frac{E}{2}\right)
\right]\\
&\qquad
+
\frac{2}{N^2}
\sum_{a=2}^p
\bE\left[
\partial_{z_a}
\frac{m_N(z)-m_N(z_a)}{z-z_a}
\prod_{\substack{b=2\\ b\neq a}}^p\left(m_N(z_b)+\frac{E}{2}\right)
\right]
=
\oo_N(1)
\left(
1+\bE\left[R_N^{p+1}\right]
\right).
\end{split}\end{align}

We now pass to microscopic variables. Recall from \eqref{e:defsN} and
\eqref{e:wigner-za} that
\[
m_N(z)=\pi \varrho_E s_N(w)-\frac E2,
\qquad
m_N(z_a)=\pi \varrho_E s_N(w_a)-\frac E2,
\quad 2\le a\le p,
\]
and the chain rule gives
\[
\frac1N\partial_z m_N(z)
=
(\pi\varrho_E)^2 \partial_w s_N(w).
\]
Moreover, by \eqref{e:defrhoE},
\[
(\pi \varrho_E)^2=1-\frac{E^2}{4}.
\]
Therefore
\begin{align}\begin{split}
\label{e:microscopic-leading-transform}
&1+z m_N(z)+m_N(z)^2+\frac1N\partial_zm_N(z)
 \\
&=
1+
\left(E+\frac{w}{N\pi\varrho_E}\right)
\left(\pi\varrho_E s_N(w)-\frac E2\right)
+
\left(\pi\varrho_E s_N(w)-\frac E2\right)^2
+
(\pi\varrho_E)^2\partial_w s_N(w)
 \\
&=
(\pi\varrho_E)^2
\left(
1+s_N(w)^2+\partial_w s_N(w)
\right)
+
\frac{w}{N}
\left(
s_N(w)-\frac{E}{2\pi\varrho_E}
\right).
\end{split}\end{align}
Similarly, for \(a\ge2\),
\begin{align}
\label{e:microscopic-difference-transform}
\partial_{z_a}
\frac{m_N(z)-m_N(z_a)}{z-z_a}
=
N^2(\pi \varrho_E)^3
\partial_{w_a}
\frac{s_N(w)-s_N(w_a)}{w-w_a}.
\end{align}

Substituting \eqref{e:microscopic-leading-transform} and
\eqref{e:microscopic-difference-transform} into \eqref{e:loop2}, using
\(R_N\lesssim 1+\cR_N\), and dividing by
\((\pi\varrho_E)^{p+1}\), we get
\begin{align}
\begin{split}
\label{e:final_loop}
&
\bE\left[
\left(
\partial_w s_N(w)
+
s_N(w)^2
+
1
\right)
\prod_{a=2}^p s_N(w_a)
\right]
+
2
\bE\left[
\sum_{a=2}^p
\left(
\partial_{w_a}
\frac{s_N(w)-s_N(w_a)}{w-w_a}
\right)
\prod_{\substack{b=2\\ b\neq a}}^p
s_N(w_b)
\right]
\\
&=
\oo_N(1)
\left(
1+\bE\left[\cR_N^{p+1}\right]
\right)
-
\frac{w}{N(\pi\varrho_E)^2}
\bE\left[
\left(
s_N(w)-\frac{E}{2\pi\varrho_E}
\right)
\prod_{a=2}^p s_N(w_a)
\right].
\end{split}
\end{align}

It remains to absorb the last term. Since \(K\subset \bC\setminus\bR\) is
compact, \(w\) is uniformly bounded on \(K\). Moreover,
\[
\left|
\left(
s_N(w)-\frac{E}{2\pi\varrho_E}
\right)
\prod_{a=2}^p s_N(w_a)
\right|
\le
C\left(
1+\cR_N^{p-1}+\cR_N^p
\right)
\le
C'\left(
1+\cR_N^{p+1}
\right).
\]
Thus
\[
\frac{|w|}{N(\pi\varrho_E)^2}
\bE\left[
\left|
\left(
s_N(w)-\frac{E}{2\pi\varrho_E}
\right)
\prod_{a=2}^p s_N(w_a)
\right|
\right]
=
\oo_N(1)
\left(
1+\bE\left[\cR_N^{p+1}\right]
\right).
\]
The final term on the right-hand side of \eqref{e:final_loop} can therefore
be absorbed into the error, yielding \eqref{e:rrg-microscopic-approx-loop}.
\end{proof}

\begin{proof}[Proof of \Cref{l:errorE}]
Since \(K\subset\bC\setminus\bR\) is compact, there are constants
\(0<c_K<C_K<\infty\) such that
\[
        c_KN^{-1}
        \le
        |\Im[z]|, |\Im[z_a]|
        \le
        C_KN^{-1},
        \qquad 2\le a\le p.
\]
On \(\Omega_N\), \Cref{l:Gijest} and \Cref{lem:Wigner-resolvent-expansion}
imply that, for \(|x|\le |h_{ij}|\leq N^{-1/2+\gamma}\) and all $z$
\[
        \max_{k,l}|G_{kl}^{(ij,x)}(z)|
        \le
        C_KN^\gamma
        =:
        \Lambda.
\]
The resolvent derivative identities are
\begin{align}\begin{split}
\label{e:Gder}
        \partial_xG^{(ij,x)}(z)
        &=-G^{(ij,x)}(z)\Delta_{ij}G^{(ij,x)}(z),
        \\
        \partial_x^2G^{(ij,x)}(z)
        &=
        2G^{(ij,x)}(z)\Delta_{ij}G^{(ij,x)}(z)
        \Delta_{ij}G^{(ij,x)}(z).
\end{split}\end{align}
Therefore
\[
        |\partial_x^rG_{kl}^{(ij,x)}(z)|
        \le
        C\Lambda^{r+1},
        \qquad r=0,1,2.
\]
For the normalized trace we use the Ward identity. Since
\(N|\Im z_a|\asymp_K1\),
\[
        |\partial_x m_N^{(ij,x)}(z_a)|
        \le C_K\Lambda,
        \qquad
        |\partial_x^2m_N^{(ij,x)}(z_a)|
        \le C_K\Lambda^2.
\]
Consequently, by the Leibniz rule,
\begin{align}
\label{e:ON}
        \bm1(\Omega_N)
        \sup_{|x|\le |h_{ij}|}
        |\partial_x^2f_{ij}(x)|
        \le
        C_{p,K}\Lambda^3(1+R_N)^{p-1}.
\end{align}
On \(\Omega_N\), we also have
\[
        N^{-1}|h_{ij}|+|h_{ij}|^3
        \le
        C N^{-3/2+3\gamma}.
\]
Thus
\begin{align}
\label{e:cc_error1}
\begin{split}
&
\frac1N
\sum_{i,j}
\bE\left[
\bm1(\Omega_N)
\left(N^{-1}|h_{ij}|+|h_{ij}|^3\right)
\sup_{|x|\le |h_{ij}|}
|\partial_x^2f_{ij}(x)|
\right]
\\
&\qquad
\le
C_{p,K}N^{-1/2+6\gamma}
\bE\left[(1+R_N)^{p-1}\right]
=
\oo_N(1)
\bE\left[(1+R_N)^{p-1}\right],
\end{split}
\end{align}
because \(\gamma<1/12\).

On \(\Omega_N^\complement\), we use the trivial resolvent bound
\[
        \|G^{(ij,x)}(z_a)\|
        \le
        |\Im z_a|^{-1}
        \le
        C_KN.
\]
Hence
\[
        \sup_{|x|\le |h_{ij}|}
        |\partial_x^2f_{ij}(x)|
        \le
        C_{p,K}N^{p+2}.
\]
By H\"older's inequality, the moment assumptions on the entries, and the bound
\(\bP(\Omega_N^\complement)\le N^{-D}\) with \(D\) chosen sufficiently large,
we get
\begin{align}
\label{e:C_error2}
&
\frac1N
\sum_{i,j}
\bE\left[
\bm1(\Omega_N^\complement)
\left(N^{-1}|h_{ij}|+|h_{ij}|^3\right)
\sup_{|x|\le |h_{ij}|}
|\partial_x^2f_{ij}(x)|
\right]
=
\oo_N(1).
\end{align}

It remains to estimate the diagonal correction. On \(\Omega_N\), the preceding
derivative bounds give
\[
        |\partial_{ii}(G_{ii}(z)M_p)|
        \le
        C_{p,K}\Lambda^2(1+R_N)^{p-1}.
\]
Therefore
\begin{align}
\label{e:diag_error}
        \frac1{N^2}
        \sum_i
        \bE\left[
        \bm1(\Omega_N)
        |\partial_{ii}(G_{ii}(z)M_p)|
        \right]
        \le
        C_{p,K}N^{-1+2\gamma}
        \bE\left[(1+R_N)^{p-1}\right].
\end{align}
The contribution of \(\Omega_N^\complement\) is again \(\oo_N(1)\) by the
trivial resolvent bound and H\"older's inequality. Combining
\eqref{e:cc_error1}, \eqref{e:C_error2}, and \eqref{e:diag_error}, and using
\((1+R_N)^{p-1}\le C_p(1+R_N^{p+1})\), proves \eqref{e:errorbb}.
\end{proof}

\begin{remark}[The role of Quantum Unique Ergodicity]
For symmetric random matrices with independent entries and a general expectation or variance profile,
the approximate loop equations can also be established provided one more a priori estimate is available,
namely Quantum Unique Ergodicity (QUE) \cite{MR3606475}.

To illustrate this, consider centered random matrices and denote
$s_{ij}=\bE[|h_{ij}|^2]$, normalized so that $\sum_j s_{ij}=1$ for all $i$.
Then the main term on the right-hand side of \eqref{eq:wigner-cumulant-step-mN} becomes
\[
\frac1{N}
\sum_{i,j=1}^N
s_{ij}
\bE\left[
        \partial_{ij}\left(G_{ij}(z)M_p\right)
\right].
\]
Summing over $j$ the $G_{ii}G_{jj}$ term in the analogue of \eqref{e:main-term-derivative}
yields the following multiple of $M_p$, where $v_\alpha$ denotes the normalized eigenvector
associated with the eigenvalue $\lambda_\alpha$:
\[
G_{ii}\sum_{j=1}^N s_{ij}G_{jj}
=
G_{ii}\sum_\alpha \sum_{j=1}^N
\frac{s_{ij}v_\alpha(j)^2}{\lambda_\alpha-z}  \notag\\
=
G_{ii}m_N
+\sum_\alpha
\frac{1}{\lambda_\alpha-z}
\sum_j
s_{ij}\Big(v_\alpha(j)^2-\frac1N\Big).
\]
QUE implies that the latter centered sum exhibits cancellations, namely it is
${\rm O}(N^{-1-\varepsilon})$ for some $\varepsilon>0$, with overwhelming probability.
Applying this estimate to all the terms arising from the summation in
\eqref{e:main-term-derivative}, one obtains the approximate loop equations.

For models with spatial dependence, including random band matrices,
QUE is known to be essential in the third step of the three-step strategy;
see, for example, \cite{YauICM} and the references therein.
QUE itself follows from multi-point resolvent estimates on mesoscopic scales
\cite{MR4334253}.
\end{remark}

\subsection{Random \(d\)-regular graphs}
\label{s:d-regular-graph}

In this section, we apply our characterization to random \(d\)-regular graphs.
We first prove bulk universality for random \(d\)-regular graphs on \(N\)
vertices in the growing-degree regime
$
d\gg(\log N)^{24}
$ in \Cref{s:d-reg-bulk}.
We then explain how the proof of edge universality for fixed-degree random
\(d\)-regular graphs in \cite{huang2024ramanujan} can be simplified using our
new characterization in \Cref{s:d-reg-edge}.

\subsubsection{Bulk loop equation}\label{s:d-reg-bulk}
Let \(A\) be the adjacency matrix of the uniform simple \(d\)-regular graph on
\(\llbracket N\rrbracket\), and introduce the normalized adjacency matrix
\begin{align}\label{e:defH}
  H:=\frac{A}{\sqrt{d-1}}.
  \end{align}

We give
a short proof of bulk universality for random \(d\)-regular graphs on \(N\)
vertices in the regime
$
d\gg (\log N)^{24}.
$
This extends the previous bulk universality results
\cite{bauerschmidt2017bulk,landon2019fixed}, which treated the range
$
N^{\fc}\le d\le N^{2/3-\fc},
$
to degrees growing as slowly as a polylogarithm of \(N\). By \Cref{t:bulk-approx-convergence}, bulk universality follows from the
following approximate local bulk loop equation. Our main focus is on the sparse regime,  for simplicity we assume that $d\leq N^{1/2}$.

\begin{proposition}[Microscopic approximate loop equation for random \(d\)-regular graphs]
\label{prop:rrg-approx-loop}
Consider the uniform simple random \(d\)-regular graph in the regime
\[
(\log N)^{24}\ll d\le N^{1/2}.
\]
We denote the eigenvalues of the rescaled adjacency matrix $H$ (recall from \eqref{e:defH}) by $\lambda_1\geq \lambda_2\geq \cdots\geq \lambda_N$.

Fix \(E\in(-2,2)\), and set
\begin{align}
\varrho_E:=\varrho_{\rm sc}(E):=\frac{\sqrt{4-E^2}}{2\pi},
\end{align}
For \(w\in\bC\setminus\bR\), define the rescaled Stieltjes transform
\begin{align}
s_N(w)
:=\sum_{i=1}^N \frac{1}{N\pi \varrho_E(\lambda_i-E)-w}
+\frac{E}{2\pi \varrho_E}
\end{align}
Then, for every \(p\ge1\) and every compact set
\(K\subset\bC\setminus\bR\), uniformly for
\(w, w_2, w_3,\dots,w_p\in K\), we have
\begin{align}\begin{split}
\label{e:rrg-microscopic-approx-loopbis}
&
\bE\left[
\left(
\partial_{w}s_N(w)
+
s_N(w)^2
+
1
\right)
\prod_{j=2}^p s_N(w_j)
\right]\\
&\qquad
+
2\bE\left[
\sum_{j=2}^p
\partial_{w_j}
\frac{s_N(w)-s_N(w_j)}{w-w_j}
\prod_{\substack{i=2\\ i\neq j}}^p
s_N(w_i)
\right]
=
\oo_N(1)
\left(
1+
\bE\left[\cR_N^{p+1}
\right]
\right),
\end{split}\end{align}
where
$
\cR_N
:=
\max\{ |s_N(w)|, |s_N(w_2)|,\dots,
\left|s_N(w_p)\right|\}$.
As a consequence, by \Cref{t:bulk-approx-convergence} with $\beta=1$, the rescaled bulk eigenvalue statistics converge to the Sine$_1$ point process.
\end{proposition}

The exponent in the condition \(d\gg (\log N)^{24}\) is not optimal, and we do
not attempt to optimize it. In fact, bulk universality is expected to hold for
every fixed degree \(d\geq 3\) \cite{jakobson1999eigenvalue}.

\begin{conjecture}
For every fixed integer \(d\geq 3\), the local statistics of the bulk eigenvalues
of random \(d\)-regular graphs on \(N\) vertices converge, as \(N\to\infty\), to
those of the {\rm GOE}.
\end{conjecture}

\subsubsection{Local law}
In this section, we recall the local law for random \(d\)-regular graphs with
polylogarithmic degree growth from \cite{bauerschmidt2017local}.
We recall the normalized adjacency matrix $H$ from \eqref{e:defH}, and define the Green's function and Stieltjes transform
\begin{align}\label{e:defG} 
        G(z):=(H-z)^{-1},
        \qquad
        m_N(z):=\frac1N \Tr G(z).
\end{align}
In the rest of the proof, we will repeatedly use the following relations 
\begin{equation}  \label{eq:rrg-resolvent-identities}
          \partial_z G(z)= G(z)^2,\quad
          G(z)\mathbf 1
        =
        \frac{1}{d/\sqrt{d-1}-z}\mathbf 1,
        \quad
        AG(z)
        =
        \sqrt{d-1}\bigl(\bI+zG(z)\bigr),
        \end{equation}
where $\mathbf 1$ denotes the trivial all-ones eigenvector of $H$, and the Ward identities
\begin{align}\label{e:Ward}
        \sum_{j=1}^N |G_{ij}(z)|^2
        =
        \sum_{j=1}^N |G_{ji}(z)|^2
        =
        \frac{\Im [G_{ii}(z)]}{\Im[ z]},\quad
        \frac1N\sum_{i,j=1}^N |G_{ij}(z)|^2
        =
        \frac{\Im[ m_N(z)]}{\Im [z]}.
\end{align}

For \(z\in \bC_+\), let
\[
        \varrho_{\rm sc}(x):=\frac{1}{2\pi}\sqrt{(4-x^2)_+},
        \qquad
        \msc(z)
        :=
        \int_{\mathbb R}\frac{\varrho_{\rm sc}(x)}{x-z}\,\rd x
        =
        \frac{-z+\sqrt{z^2-4}}{2},
\]
where the square root is chosen so that \(\sqrt{z^2-4}\sim z\) as
\(|z|\to\infty\). 

Define
\[
        D:=d\wedge \frac{N^2}{d^3},
        \qquad
        \Phi(z):=\frac{1}{\sqrt{N\eta}}+\frac{1}{\sqrt D},
\]
and, for \(r\in[0,1]\),
\[
        F_z(r)
        :=
        \left[
                \left(
                        1+\frac{1}{\sqrt{|z^2-4|}}
                \right)r
        \right]
        \wedge \sqrt r .
\]

We recall the following local semicircle law for the uniform random regular graph
from \cite[Theorem 1.1]{bauerschmidt2017local}.

\begin{theorem}[ {\cite[Theorem 1.1]{bauerschmidt2017local}}]
        \label{t:local-semicircle-law-rrg-uniform}
        Let \(G(z)\) be the Green's function \eqref{e:defG} of the uniform
        simple random \(d\)-regular graph. Let \(\xi=\xi_N\) satisfy
$
        \xi\log \xi \gg (\log N)^2$,
and assume that
$
        \xi^2\ll d\ll \left(N/\xi\right)^{2/3}$.
Then, with probability at least \(1-\exp(-\xi\log\xi)\),
\[
        \max_i |G_{ii}(z)-\msc(z)|
        =
        \OO\bigl(F_z(\xi\Phi(z))\bigr),
        \qquad
        \max_{i\neq j}|G_{ij}(z)|
        =
        \OO\bigl(\xi\Phi(z)\bigr),
\]
simultaneously for all \(z=E+\ri\eta\in\mathbb C_+\) satisfying
$
        \eta\gg{\xi^2}/{N}$.
\end{theorem}

 Choosing
\[
        \xi:=\frac{(\log N)^2}{\sqrt{\log\log N}},
\]
the assumptions of the local law \Cref{t:local-semicircle-law-rrg-uniform} are satisfied whenever
\[
        d\ge (\log N)^4,
        \qquad
        d\le \frac{N^{2/3}}{(\log N)^{4/3}},
        \qquad
        \eta\ge \frac{(\log N)^4}{N}.
\]
Therefore, with probability at least \(1-\exp(-\xi\log\xi)\geq 1-e^{-(\log N)^2}\),
\[
        \max_i |G_{ii}(z)-\msc(z)|
        =
        \OO\bigl(F_z(\xi\Phi(z))\bigr),
        \qquad
        \max_{i\neq j}|G_{ij}(z)|
        =
        \OO\bigl(\xi\Phi(z)\bigr),
\]
simultaneously for all \(z=E+\ri\eta\in\mathbb C_+\) with
$
        \eta\ge {(\log N)^4}/{N}$.
In particular,
\[
        \max_{i,j}|G_{ij}(z)|
        \le 2
\]
throughout this spectral domain.

We introduce the following event \(\Omega_N\) of \(d\)-regular graphs on \(N\)
vertices:
\begin{align}\label{e:defOmegaN}
\Omega_N
:=
\left\{
\mathcal G:
\max_{i,j}|G_{ij}(z)|\le 2
\text{ for all }z=E+\ri\eta\in\bC_+,\ \eta\ge \frac{(\log N)^4}{N}
\right\}.
\end{align}

\begin{lemma}
On the event \(\Omega_N\), for any \(z=E+\ri\eta\in\bC_+\),
\begin{equation}
        \max_{i,j}|G_{ij}(z)|
        \le
        \Lambda(z),
        \qquad
        \Lambda(z):=2\cdot
        \max\left\{1, \frac{(\log N)^4}{N\eta}\right\},
        \qquad
        \eta>0.
        \label{eq:microscopic-entry-controlbis}
\end{equation}
Thus, if \(z=E+w/(\pi \varrho_E N)\) with \(w\) in a compact subset \(K\) of
\(\mathbb C_+\), then
\[
        \Lambda(z)\le C_K(\log N)^4 .
\]
\end{lemma}

\begin{proof}
The proof is exactly the same as that of \Cref{l:Gijest} after replacing $N^\gamma$ by $(\log N)^4$, so we omit it.
\end{proof}

\subsubsection{Switching calculus}

We define the signed adjacency matrices
\begin{equation}\label{e:defswitching}
    (\Delta_{ij})_{ab}
    :=
    \delta_{ia}\delta_{jb}
    +
    \delta_{ib}\delta_{ja},
    \qquad
    \xi_{ij}^{kl}
    :=
    \Delta_{ij}
    +
    \Delta_{kl}
    -
    \Delta_{ik}
    -
    \Delta_{jl}.
\end{equation}
Thus \(\Delta_{ij}\) is the adjacency matrix of the edge \(ij\), while
\(\xi_{ij}^{kl}\) is the signed matrix associated with the switching of edges.
For any function \(F\) of the adjacency matrix \(A\), define the discrete
difference operator and the continuous differential operator by
\begin{equation}
(D_{ij}^{kl}F)(A)
:=
F(A+\xi_{ij}^{kl})-F(A),
\qquad
(\partial_{ij}^{kl}F)(A)
:=
\sqrt{d-1}\left.\frac{\rd}{\rd t}F(A+t\xi_{ij}^{kl})\right|_{t=0}.
\end{equation}

We recall the following discrete integration-by-parts formula from
\cite[Corollary 3.2]{bauerschmidt2020edge}.

\begin{lemma}[{\cite[Corollary 3.2]{bauerschmidt2020edge}}]
    \label{c:ibp-switching}
Fix indices \(i,j\in\llbracket N\rrbracket\). Let
\(F=F_{ij}\) be a random variable, possibly depending on \(i,j\). Define
\[
    \mathcal C_{ij}(F,A)
    :=
    |F(A)|
    +
    \max_{k,l}
    \left|
        F\bigl(A+\xi_{ij}^{kl}\bigr)
    \right|.
\]
Then for $i\neq j$ we have
\begin{align}
    \label{e:ibp-switching}
    \bE\left[A_{ij}F(A)\right]
    &=
    \frac{d}{N}\,\bE\left[F(A)\right]
    +
    \frac{1}{Nd}
    \sum_{k,l}
    \bE\left[
        A_{ik}A_{jl}
        \left(
            F\bigl(A+\xi_{ij}^{kl}\bigr)-F(A)
        \right)
    \right]
    +
    \OO\left(
        \frac{d}{N}
        \bE\left[
            A_{ij}\mathcal C_{ij}(F,A)
        \right]
    \right).
\end{align}
\end{lemma}

The next lemma replaces the discrete switching derivative by its linearized
continuous version.

\begin{lemma}[Resolvent expansion under one switching]
\label{lem:rrg-switching-resolvent-expansion}
Assume
\[
        \max_{a,b}|G_{ab}(z)|\le \Lambda,
        \qquad
        \frac{\Lambda}{\sqrt d}\ll 1 .
\]
Then, uniformly in all indices, $a,b,i,j,k,l\in \qq{N}$, we have
\begin{equation}
        D_{ij}^{kl}G_{ab}(z)
        =
        \frac1{\sqrt{d-1}}\,
        \partial_{ij}^{kl}G_{ab}(z)
        +
        \OO\!\left(\frac{\Lambda^3}{d}\right),
        \label{eq:DG-linearization}
\end{equation}
and similarly
\begin{equation}
        D_{ij}^{kl}m_N(z)
        =
        \frac1{\sqrt{d-1}}\,
        \partial_{ij}^{kl}m_N(z)
        +
        \OO\!\left(\frac{\Lambda^3}{d}\right).
        \label{eq:Dm-linearization}
\end{equation}
Moreover, the derivatives are given by
\begin{align}\begin{split}\label{e:DG}
\partial_{ij}^{kl}G(z)&=-G(z)\xi_{ij}^{kl}G(z)
=
-G_{ji}(z)^2
-G_{jj}(z)G_{ii}(z)
-G_{jk}(z)G_{li}(z)
-G_{jl}(z)G_{ki}(z)
\\
&
+G_{ji}(z)G_{ki}(z)
+G_{jl}(z)G_{ji}(z)
+G_{jk}(z)G_{ii}(z)
+G_{jj}(z)G_{li}(z),
\end{split}\end{align}
and
\begin{align}\begin{split}\label{e:Dm}
\partial_{ij}^{kl}m_N(w)
=
-\frac1N\Tr\left(G(w)\xi_{ij}^{kl}G(w)\right)
=
-\frac2N
\left[
(G(w)^2)_{ij}
+
(G(w)^2)_{kl}
-
(G(w)^2)_{ik}
-
(G(w)^2)_{jl}
\right].
\end{split}\end{align}
\end{lemma}

\begin{proof}
We recall from \eqref{e:defswitching} that $ \xi_{ij}^{kl}
    =
    \Delta_{ij}
    +
    \Delta_{kl}
    -
    \Delta_{ik}
    -
    \Delta_{jl}$. Write
\[
        \xi_{ij}^{kl}=UV^\top,
        \qquad
        U=(e_i-e_l,\ e_j-e_k),
        \qquad
        V=(e_j-e_k,\ e_i-e_l).
\]
By the Woodbury formula,
\begin{align}\label{e:woodbury}
        G(A+\xi_{ij}^{kl})
        =
        G
        -
        \frac1{\sqrt{d-1}}\,
        GU
        \left(
               \bI_2+\frac1{\sqrt{d-1}}V^\top GU
        \right)^{-1}
        V^\top G .
\end{align}
Every entry of \(V^\top GU\), \(GU\), and \(V^\top G\) is a linear combination
of at most four entries of \(G\). Hence, using \(\Lambda/\sqrt d\ll 1\), the
matrix
\[
       \bI_2+\frac1{\sqrt{d-1}}V^\top GU
\]
is invertible, and the Neumann series gives
\[
        \left\|
        \left(
               \bI_2+\frac1{\sqrt{d-1}}V^\top GU
        \right)^{-1}
        - \bI_2
        \right\|
        \le
        C\frac{\Lambda}{\sqrt d}.
\]
Subtracting the linear term
\[
-\frac{1}{\sqrt{d-1}}GUV^\top G=-\frac1{\sqrt{d-1}}G\xi_{ij}^{kl}G
\]
the expression \eqref{e:woodbury} gives
\[
        \left|
        D_{ij}^{kl}G_{ab}
        +
        \frac1{\sqrt{d-1}}(G\xi_{ij}^{kl}G)_{ab}
        \right|
        \le
        C\frac{\Lambda^3}{d}.
\]
Since
\[
\partial_{ij}^{kl}G(z)=-G(z)\xi_{ij}^{kl}G(z),
\]
this proves \eqref{eq:DG-linearization}. Taking the normalized trace gives the
claim for \(m_N\). Finally, \eqref{e:DG} and \eqref{e:Dm} follow from \eqref{eq:rrg-resolvent-identities} by straightforward computations.
\end{proof}

\subsubsection{Proof of \Cref{prop:rrg-approx-loop}}

In this section we prove \Cref{prop:rrg-approx-loop}, using the local law \Cref{t:local-semicircle-law-rrg-uniform} as input, the bulk loop equation follows from the discrete integration by part formula \Cref{c:ibp-switching}.

For
\(w, w_2, w_3,\dots,w_p\in K\), set
\begin{align}\label{e:zwrelation}
z=E+\frac{w}{N\pi \varrho_E}, \quad z_a=E+\frac{w_a}{N\pi \varrho_E}, \quad 2\leq a\leq p,
\end{align}
and 
\[
M_p:=\prod_{a=2}^p m_N(z_a),
\qquad
M_p^{(a)}:=\prod_{\substack{b=2\\ b\neq a}}^p m_N(z_b),
\qquad
R_N:=\max\{|m_N(z)|,|m_N(z_2)|,\dots,|m_N(z_p)|\}.
\]

Taking the trace in the identity
\[
HG(z)=\bI+zG(z)
\]
gives
\begin{equation}
\label{e:trace-resolvent-identity-rrg}
1+z m_N(z)
=
\frac1N\sum_{i,j}H_{ij}G_{ji}(z).
\end{equation}
Multiplying \eqref{e:trace-resolvent-identity-rrg} by \(M_p\) and adding
\(m_N(z)^2M_p\), we obtain
\begin{align}
\label{e:start-higher-loop}
\bE\left[
\left(
1+z m_N(z)+m_N(z)^2
\right)M_p
\right]
=
\bE\left[m_N(z)^2M_p\right]
+
\frac1{N\sqrt{d-1}}
\sum_{i,j}
\bE\left[
A_{ij}G_{ji}(z)M_p
\right].
\end{align}

We apply the switching integration-by-parts formula \Cref{c:ibp-switching} to
the observable
\[
F_{ij}(A):=G_{ji}(z)M_p.
\]
This gives
\begin{align}\begin{split}
\label{e:ibp-switching-copy}
&\phantom{{}={}}\frac1{N\sqrt{d-1}}
\sum_{i,j}
\bE\left[A_{ij}F_{ij}(A)\right]
=
\frac{d}{N^2\sqrt{d-1}}
\sum_{i\neq j}
\bE\left[F_{ij}(A)\right]\\
&+
\frac1{N^2d\sqrt{d-1}}
\sum_{i\neq j,k,l}
\bE\left[
A_{ik}A_{jl}D_{ij}^{kl}F_{ij}(A)
\right]
+
\OO\left(
\frac{\sqrt d}{N^2}
\sum_{i\neq j}
\bE\left[
A_{ij}\mathcal C_{ij}(F_{ij},A)
\right]
\right).
\end{split}\end{align}
For the first two terms on the right-hand side of \eqref{e:ibp-switching-copy}, the summations are over $i\neq j$, we can replace them by the corresponding summations with all $i,j$, together with an extra error
\begin{align}
\frac{d}{N^2\sqrt{d-1}}
\sum_{i}
\bE\left[F_{ii}(A)\right]+\frac1{N^2d\sqrt{d-1}}
\sum_{i,k,l}
\bE\left[
A_{ik}A_{il}D_{ii}^{kl}F_{ii}(A)
\right]=\OO\left(\frac{\sqrt{d}}{N^2}\sum_i \bE[\cC_{ii}(F_{ii},A)]\right)
\end{align}
Thus we can rewrite \eqref{e:ibp-switching-copy} as
\begin{align}\begin{split}
\label{e:ibp-switching-copy-clean}
&\phantom{{}={}}\frac1{N\sqrt{d-1}}
\sum_{i,j}
\bE\left[A_{ij}F_{ij}(A)\right]
=
\frac{d}{N^2\sqrt{d-1}}
\sum_{i,j}
\bE\left[F_{ij}(A)\right]\\
&+
\frac1{N^2d\sqrt{d-1}}
\sum_{i,j,k,l}
\bE\left[
A_{ik}A_{jl}D_{ij}^{kl}F_{ij}(A)
\right]
+
\OO\left(
\frac{\sqrt d}{N^2}
\left(\sum_{i\neq j}
\bE[
A_{ij}\mathcal C_{ij}(F_{ij},A)]+\sum_i \bE[\cC_{ii}(F_{ii},A)]\right)
\right).
\end{split}\end{align}

The first term on the right-hand side of \eqref{e:ibp-switching-copy-clean} is negligible. Indeed, by
\eqref{eq:rrg-resolvent-identities},
\[
\sum_iG_{ji}(z)
=
\frac{1}{d/\sqrt{d-1}-z}
=
\OO(d^{-1/2}),
\qquad
\sum_{i,j}G_{ji}(z)
=
\OO\left(\frac{N}{\sqrt d}\right).
\]
Therefore
\begin{align}
\frac{d}{N^2\sqrt{d-1}}
\sum_{i,j}
\bE\left[F_{ij}(A)\right]
=
\OO\left(
\frac1N
\bE\left[
 |M_p|
\right]
\right)
=
\oo_N(1)
\left(
1+\bE[R_N^{p+1}]
\right).
\label{eq:first-ibp-negligible}
\end{align}

We next control the integration-by-parts error in \eqref{e:ibp-switching-copy-clean}. Recall the event \(\Omega_N\)
from \eqref{e:defOmegaN} and \(\Lambda(z)\) from
\eqref{eq:microscopic-entry-controlbis}. For
\(z=E+w/(\pi\varrho_E N)\) with \(w\in K\), we have
\begin{align}\label{e:z_parameter}
\Lambda(z)\le C_K(\log N)^4,
\qquad
\frac1{N\Im [z]}\le C_K,
\qquad
(\log N)^{24}\ll d\leq N^{1/2}.
\end{align}
On \(\Omega_N\), by \Cref{lem:rrg-switching-resolvent-expansion} and
\eqref{eq:microscopic-entry-controlbis},
\[
|G_{ji}(z)|
+
\max_{k,l}|G_{ji}(z;A+\xi_{ij}^{kl})|
\le C_K\Lambda(z),\quad |m_N(z_a)|
+
\max_{k,l}|m_N(z_a;A+\xi_{ij}^{kl})|
\le C_K\bigl(1+|m_N(z_a)|\bigr),
\]
 for \(a=2,\dots,p\).
Thus
\begin{align}
\label{e:Cerror1}
\mathbf 1({\Omega_N})\mathcal C_{ij}(F,A)
\le
C_K\Lambda(z)
\prod_{a=2}^p\bigl(1+|m_N(z_a)|\bigr)\leq C_K\Lambda(z)
\prod_{a=2}^p(1+R_N)^{p-1}.
\end{align}

On the complement \(\Omega_N^\complement\), we use the trivial resolvent bound. Since
\(w,w_2,\dots,w_p\in K\subset \mathbb C\setminus\mathbb R\), all imaginary
parts satisfy
$
|\Im[ z_a]|\asymp_K N^{-1}$.
Hence
\begin{align}\label{e:trivial}
|G_{uv}(z_a)|\le \frac{C_K}{|\Im [z_a]|}\le C_K N,
\qquad
|m_N(z_a)|\le C_KN,
\end{align}
and the same bound holds for the switched matrix \(A+\xi_{ij}^{kl}\). Therefore
\[
\mathbf 1({\Omega_N^\complement})\mathcal C_{ij}(F,A)
\le
\mathbf 1({\Omega_N^\complement}) (C_K N)^p.
\]
Taking expectations and using
$
\mathbb P(\Omega_N^\complement)\le e^{-(\log N)^2}$,
we obtain
\begin{align}\label{e:Cerror2}
\bE\left[\mathbf 1({\Omega_N^\complement})\mathcal C_{ij}(F,A)\right]
\le
(C_KN)^p e^{-(\log N)^2}
=
\oo_N(1),
\end{align}
since \(p\) is fixed.

Combining the estimates \eqref{e:Cerror1} and \eqref{e:Cerror2} on \(\Omega_N\) and \(\Omega_N^\complement\), and using
\(\sum_{i,j}A_{ij}=Nd\), we get
\begin{align}
\frac{\sqrt d}{N^2}
\sum_{i,j}
\bE\left[
A_{ij}\mathcal C_{ij}(F,A)
\right]
&\le
C_K\frac{d^{3/2}\Lambda(z)}{N}
\bE\left[
(1+R_N)^{p-1}
\right]
+\oo_N(1)=
\oo_N(1)
\left(
1+\bE\left[R_N^{p+1}\right]
\right),
\label{eq:ibp-error-final}
\end{align}
where we used \(d\le N^{1/2}\), \(\Lambda(z)\le C_K(\log N)^4\), and
$
(1+R_N)^{p-1}\le C_p(1+R_N^{p+1})$.

It remains to expand the discrete derivative in the main switching term. We
write
\begin{align}
D_{ij}^{kl}F_{ij}(A)
&=
\frac{1}{\sqrt{d-1}}
\left[
\left(\partial_{ij}^{kl}G_{ji}(z)\right)M_p
+
G_{ji}(z)
\sum_{a=2}^p
\left(\partial_{ij}^{kl}m_N(z_a)\right)M_p^{(a)}
\right]
+
\mathcal E_{ij}^{kl}.
\label{eq:linearized-product-clean}
\end{align}
All nonlinear terms in the discrete product rule and all Taylor remainders are
included in \(\mathcal E_{ij}^{kl}\).

We shall use the following deterministic estimates for the linearized switching
terms in \eqref{eq:linearized-product-clean}. The first estimate controls the
Taylor and product-rule errors, while the last two identities identify the
contributions of the derivatives falling on \(G_{ji}(z)\) and on the external
Stieltjes transforms. We postpone the proof to \Cref{s:proof_main_error}.

\begin{lemma}\label{e:error_main}
Adopt the notation and assumptions of \Cref{prop:rrg-approx-loop}, and let
\(\Lambda=\Lambda(z)\) be as in \eqref{eq:microscopic-entry-controlbis}. For $z=E+w/(N\pi\varrho_E)$, $z_a=E+w_a/(N\pi \varrho_E)$, and $w, w_2, \cdots, w_p\in K\subset \bC\setminus \bR$, we have
\begin{align}
\label{e:replace_error}
\frac{\bm1(\Omega_N)}{N^2 d\sqrt{d-1}}
\sum_{i,j,k,l}A_{ik}A_{jl}|\mathcal E_{ij}^{kl}|
&\leq
\frac{C_K \Lambda^3}{\sqrt{d}}(1+R_N)^{p-1}.
\end{align}
Moreover, for \(z\in\bC_+\),
\begin{align}
\label{eq:weighted-switching-sum}
\frac{1}{N^2d(d-1)}
\sum_{i,j,k,l}
A_{ik}A_{jl}\partial_{ij}^{kl}G_{ji}(z)
&=
-m_N(z)^2
-\frac1N\partial_z m_N(z)
+
\OO\left(
\frac{1}{d}
\left(
1+|m_N(z)|^2+\frac{|m_N(z)|}{N\Im [z]}
\right)
\right).
\end{align}
Finally, for $z,w\in \bC_+$, we have
\begin{align}
\label{eq:weighted-external-derivative-sum}
&\frac1{N^2d(d-1)}
\sum_{i,j,k,l}
A_{ik}A_{jl}G_{ji}(z)\,
\partial_{ij}^{kl}m_N(w)
\nonumber\\
&\qquad
=
-\frac{2d}{N^2(d-1)}
\partial_w
\frac{m_N(z)-m_N(w)}{z-w}
+
\OO\left(
\frac{|m_N(z)|+|m_N(w)|}
{N^2d\,\Im[ z]\,\Im [w]}
\right).
\end{align}
If \(w=z\), the right-hand side is understood by taking the limit
\(w\to z\).
\end{lemma}

\begin{proof}[Proof of \Cref{prop:rrg-approx-loop}]
Splitting the error term according to \(\Omega_N\) and \(\Omega_N^\complement\), using
\eqref{e:replace_error} and \eqref{e:z_parameter} such that $\Lambda(z)\leq C_K(\log N)^4\ll d^{1/6}$ on \(\Omega_N\), and using the trivial resolvent bound \eqref{e:trivial}
on \(\Omega_N^\complement\), we have 
\[
\frac1{N^2d\sqrt{d-1}}
\sum_{i,j,k,l}
\bE\left[
A_{ik}A_{jl}\mathcal E_{ij}^{kl}
\right]
=
\oo_N(1)
\left(
1+
\bE\left[
R_N^{p+1}
\right]
\right).
\]

Substituting \eqref{eq:linearized-product-clean} into
\eqref{e:ibp-switching-copy-clean}, the main switching term becomes
\begin{align}
&\frac1{N^2d(d-1)}
\sum_{i,j,k,l}
\bE\left[
A_{ik}A_{jl}
\left(\partial_{ij}^{kl}G_{ji}(z)\right)M_p
\right]
\nonumber\\
&\quad
+
\frac1{N^2d(d-1)}
\sum_{a=2}^p
\sum_{i,j,k,l}
\bE\left[
A_{ik}A_{jl}
G_{ji}(z)
\left(\partial_{ij}^{kl}m_N(z_a)\right)M_p^{(a)}
\right]
+
\oo_N(1)
\left(
1+
\bE\left[
R_N^{p+1}
\right]
\right).
\label{eq:switching-main-expanded-clean}
\end{align}

We now apply the two main estimates in \Cref{e:error_main}. The estimate \eqref{eq:weighted-switching-sum} gives
\begin{align}
&\frac1{N^2d(d-1)}
\sum_{i,j,k,l}
A_{ik}A_{jl}
\partial_{ij}^{kl}G_{ji}(z)
=
-m_N(z)^2
-\frac1N\partial_zm_N(z)
+
\oo_N(1)\bigl(1+R_N^2\bigr),
\label{eq:main-G-derivative-clean}
\end{align}
after using \eqref{e:z_parameter}.
Multiplying by \(M_p\) and taking expectations, the error contributes
$
\oo_N(1)
(
1+\bE[R_N^{p+1}]
)$.

Similarly, for each \(a=2,\dots,p\), the estimate \eqref{eq:weighted-external-derivative-sum} gives
\begin{align}
&\frac1{N^2d(d-1)}
\sum_{i,j,k,l}
A_{ik}A_{jl}
G_{ji}(z)
\partial_{ij}^{kl}m_N(z_a)
=
-\frac{2}{N^2}
\partial_{z_a}
\frac{m_N(z)-m_N(z_a)}{z-z_a}
+
\oo_N(1)\bigl(1+R_N\bigr),
\label{eq:main-external-derivative-clean}
\end{align}
where the factor \(d/(d-1)=1+\OO(d^{-1})\) has been absorbed into the error.
Again, after multiplying by \(M_p^{(a)}\) and taking expectations, the error is
$
\oo_N(1)
(
1+\bE[R_N^{p+1}]
)$.

Combining
\eqref{e:start-higher-loop},
\eqref{eq:first-ibp-negligible},
\eqref{eq:ibp-error-final},
\eqref{eq:switching-main-expanded-clean},
\eqref{eq:main-G-derivative-clean}, and
\eqref{eq:main-external-derivative-clean}, we get
\begin{align}\begin{split}\label{e:loop_d}
&
\bE\left[
\left(
1+z m_N(z)+m_N(z)^2
\right)M_p
\right]
=
\bE\left[
\left(
-\frac1N\partial_zm_N(z)
\right)M_p
\right]\\
&\qquad\quad
-\frac{2}{N^2}
\sum_{a=2}^p
\bE\left[
\partial_{z_a}
\frac{m_N(z)-m_N(z_a)}{z-z_a}
M_p^{(a)}
\right]
+
\oo_N(1)
\left(
1+
\bE\left[
R_N^{p+1}
\right]
\right).
\end{split}\end{align}

It remains only to pass from \eqref{e:loop_d} to microscopic variables. This is identical to the final step in the proof of \Cref{p:wigner-microscopic-loop}, and hence yields \eqref{e:rrg-microscopic-approx-loopbis}. We omit the details.
\end{proof}

\subsubsection{Proof of \Cref{e:error_main}}
\label{s:proof_main_error}
\begin{proof}[Proof of \Cref{e:error_main}]
We prove the three estimates separately.

\medskip
\noindent
\emph{Step 1: The replacement error.}

On the event \(\Omega_N\), we recall from
\Cref{lem:rrg-switching-resolvent-expansion} with $\Lambda=\Lambda(z)\leq C_K(\log N)^4\ll \sqrt{d}$ as defined in \eqref{eq:microscopic-entry-controlbis}, that
\begin{align}\begin{split}
D_{ij}^{kl}G_{ji}(z)
&=
(d-1)^{-1/2}\partial_{ij}^{kl}G_{ji}(z)
+\OO\left(\frac{\Lambda^3}{d}\right),
\label{eq:switching-error-G}
\\
D_{ij}^{kl}m_N(z_a)
&=
(d-1)^{-1/2}\partial_{ij}^{kl}m_N(z_a)
+\OO\left(\frac{\Lambda^3}{d}\right).
\end{split}\end{align}
Moreover using the explicit formulas \eqref{e:DG} and \eqref{e:Dm},
\[
|\partial_{ij}^{kl}G_{ji}(z)|\lesssim \Lambda^2,
\qquad
|\partial_{ij}^{kl}m_N(z_a)|\lesssim \frac{\Lambda}{N\Im[z_a]}\leq C_K \Lambda.
\]

By definition,
\[
\mathcal E_{ij}^{kl}
=
D_{ij}^{kl}\bigl(G_{ji}(z)M_p\bigr)
-
\frac1{\sqrt{d-1}}
\partial_{ij}^{kl}\bigl(G_{ji}(z)M_p\bigr).
\]
Using the discrete product rule and
\eqref{eq:switching-error-G}, we get
\begin{align}
|\mathcal E_{ij}^{kl}|
&\lesssim
|G_{ji}(z)|
\left|
D_{ij}^{kl}M_p
-
\frac1{\sqrt{d-1}}\partial_{ij}^{kl}M_p
\right|
+
\frac1{\sqrt d}
|\partial_{ij}^{kl}G_{ji}(z)|
\,|D_{ij}^{kl}M_p|
+
\frac{\Lambda^3}{d}
\prod_{a=2}^p
\left(|m_N(z_a)|+\frac{C_K\Lambda}{\sqrt d}\right).
\label{eq:Eijk-bound-start}
\end{align}
Again by the product rule and
\eqref{eq:switching-error-G},
\begin{align}\begin{split}
\left|
D_{ij}^{kl}M_p
-
\frac1{\sqrt{d-1}}\partial_{ij}^{kl}M_p
\right|
\lesssim
\frac{\Lambda^3}{d}(1+R_N)^{p-2},\quad
|D_{ij}^{kl}M_p|
\lesssim
\frac{\Lambda}{\sqrt d}(1+R_N)^{p-2}.
\label{eq:Mp-D-bound}
\end{split}\end{align}
Substituting \eqref{eq:Mp-D-bound} into
\eqref{eq:Eijk-bound-start} gives
\begin{equation}
\label{eq:Eijk-bound-final}
|\mathcal E_{ij}^{kl}|
\lesssim
|G_{ji}(z)|\frac{\Lambda^3}{d}(1+R_N)^{p-2}
+
\frac{\Lambda^3}{d}(1+R_N)^{p-1}.
\end{equation}

We now sum this estimate. By the Ward identity \eqref{e:Ward},
\[
\sum_{i,j}|G_{ij}(z)|
\le
N\left(\sum_{i,j}|G_{ij}(z)|^2\right)^{1/2}
=
N\left(\frac{N\Im [m_N(z)]}{\eta}\right)^{1/2}
\le
N\left(\frac{N|m_N(z)|}{\eta}\right)^{1/2}.
\]
Since
\[
\sum_{k,l}A_{ik}A_{jl}=d^2,
\]
the first term in \eqref{eq:Eijk-bound-final} gives
\begin{align*}
&\frac1{N^2d\sqrt{d-1}}
\sum_{i,j,k,l}
A_{ik}A_{jl}
|G_{ji}(z)|
\frac{\Lambda^3}{d}(1+R_N)^{p-2}
\lesssim
\frac{\Lambda^3}{\sqrt d}
\frac{1}{\sqrt{N\eta}}
(1+R_N)^{p-1}.
\end{align*}
The second term in \eqref{eq:Eijk-bound-final} gives
\[
\frac1{N^2d\sqrt{d-1}}
\sum_{i,j,k,l}
A_{ik}A_{jl}
\frac{\Lambda^3}{d}(1+R_N)^{p-1}
\lesssim
\frac{\Lambda^3}{\sqrt d}(1+R_N)^{p-1}.
\]
Therefore
\begin{align}
\frac{\bm1(\Omega_N)}{N^2 d\sqrt{d-1}}
\sum_{i,j,k,l}A_{ik}A_{jl}|\mathcal E_{ij}^{kl}|
&\lesssim
\frac{\Lambda^3}{\sqrt d}
\left(
\frac{1}{\sqrt{N\eta}}(1+R_N)^{p-1}
+
(1+R_N)^{p-1}\right)\lesssim \frac{C_K \Lambda^3}{\sqrt{d}}(1+R_N)^{p-1}
.
\end{align}

\medskip
\noindent
\emph{Step 2: The weighted sum with \(\partial_{ij}^{kl}G_{ji}(z)\).}

We will use the following identities from \eqref{eq:rrg-resolvent-identities} repeatedly in this proof
\begin{align}\label{e:idd}
AG(z)=\sqrt{d-1}\bigl(\bI+zG(z)\bigr),\quad 
\sum_iG_{ij}(z)
=
\sum_jG_{ij}(z)
=
\frac{1}{d/\sqrt{d-1}-z}.
\end{align}
We also recall from \eqref{e:DG}
\begin{align*}
\partial_{ij}^{kl}G_{ji}(z)
&=
-G_{ji}(z)^2
-G_{jj}(z)G_{ii}(z)
-G_{jk}(z)G_{li}(z)
-G_{jl}(z)G_{ki}(z)
\\
&\qquad
+G_{ji}(z)G_{ki}(z)
+G_{jl}(z)G_{ji}(z)
+G_{jk}(z)G_{ii}(z)
+G_{jj}(z)G_{li}(z)
.
\end{align*}
We compute the weighted sum term by term.

The first two terms are
\begin{align*}
&-\sum_{i,j,k,l}A_{ik}A_{jl}G_{ji}^2
=
-d^2\Tr G^2
=
-d^2N\,\partial_zm_N,\\
&-\sum_{i,j,k,l}A_{ik}A_{jl}G_{jj}G_{ii}
=
-d^2(\Tr G)^2
=
-d^2N^2m_N^2.
\end{align*}

For the third term, we use \eqref{e:idd}
\begin{align*}
&-\sum_{i,j,k,l}A_{ik}A_{jl}G_{jk}G_{li}
=
-\sum_{i,j}
\left(\sum_k A_{ik}G_{jk}\right)
\left(\sum_l A_{jl}G_{li}\right)
=
-\sum_{i,j}(AG)_{ij}(AG)_{ji}
\\
&=
-(d-1)
\left(
N+2z\Tr G+z^2\Tr G^2
\right)=
-(d-1)N
\left(
1+2zm_N+z^2\partial_zm_N
\right).
\end{align*}

For the fourth term,
\[
-\sum_{i,j,k,l}A_{ik}A_{jl}G_{jl}G_{ki}
=
-\left(\Tr AG\right)^2
=
-(d-1)N^2(1+zm_N)^2.
\]

For the fifth and sixth terms, using the row-sum identity \eqref{e:idd},
\begin{align*}
&\sum_{i,j,k,l}A_{ik}A_{jl}G_{ji}G_{ki}
=
d\sqrt{d-1}\,N
\frac{1+zm_N}{d/\sqrt{d-1}-z},\\
&\sum_{i,j,k,l}A_{ik}A_{jl}G_{jl}G_{ji}
=
d\sqrt{d-1}\,N
\frac{1+zm_N}{d/\sqrt{d-1}-z}.
\end{align*}

For the seventh and eighth terms,
\begin{align*}
&\sum_{i,j,k,l}A_{ik}A_{jl}G_{jk}G_{ii}
=
d^2N
\frac{m_N}{d/\sqrt{d-1}-z},\\
&\sum_{i,j,k,l}A_{ik}A_{jl}G_{jj}G_{li}
=
d^2N
\frac{m_N}{d/\sqrt{d-1}-z}.
\end{align*}

Adding the eight contributions, we obtain
\begin{align*}
&\sum_{i,j,k,l}A_{ik}A_{jl}\partial_{ij}^{kl}G_{ji}(z)
=
-d^2N^2m_N^2
-d^2N\partial_zm_N
-(d-1)N
\left(
1+2zm_N+z^2\partial_zm_N
\right)
\\
&\qquad
-(d-1)N^2(1+zm_N)^2
+
2d\sqrt{d-1}\,N
\frac{1+zm_N}{d/\sqrt{d-1}-z}
+
2d^2N
\frac{m_N}{d/\sqrt{d-1}-z}.
\end{align*}
Dividing by \(N^2d(d-1)\), and using
$
|\partial_zm_N(z)|
\le
{|m_N(z)|}/{\Im [z]}$,
we get
\begin{align*}
\frac{1}{N^2d(d-1)}
\sum_{i,j,k,l}
A_{ik}A_{jl}\partial_{ij}^{kl}G_{ji}(z)
&=
-m_N(z)^2
-\frac1N\partial_zm_N(z)
+
\OO\left(
\frac{
1+|m_N(z)|^2+\frac{|m_N(z)|}{N\Im [z]}
}{d}
\right).
\end{align*}

\medskip
\noindent
\emph{Step 3: The weighted sum with \(\partial_{ij}^{kl}m_N(w)\).}

We recall from \eqref{e:Dm}
\begin{align}
\partial_{ij}^{kl}m_N(w)
=
-\frac2N
\left[
(G(w)^2)_{ij}
+
(G(w)^2)_{kl}
-
(G(w)^2)_{ik}
-
(G(w)^2)_{jl}
\right].
\label{eq:switching-derivative-mw}
\end{align}
Substituting \eqref{eq:switching-derivative-mw} into \eqref{eq:weighted-external-derivative-sum}, we get
\begin{align}
&\frac1{N^2d(d-1)}
\sum_{i,j,k,l}
A_{ik}A_{jl}G_{ji}(z)\,
\partial_{ij}^{kl}m_N(w)
\nonumber\\
&=
-\frac2{d(d-1)N^3}
\sum_{i,j,k,l}
A_{ik}A_{jl}G_{ji}(z)
\left[
(G(w)^2)_{ij}
+
(G(w)^2)_{kl}
-
(G(w)^2)_{ik}
-
(G(w)^2)_{jl}
\right].
\label{eq:weighted-external-start}
\end{align}

We compute the four sums separately. First,
\[
\sum_{i,j,k,l}
A_{ik}A_{jl}G_{ji}(z)(G(w)^2)_{ij}
=
d^2\Tr\bigl(G(z)G(w)^2\bigr).
\]
Second,
\begin{align*}
\sum_{i,j,k,l}
A_{ik}A_{jl}G_{ji}(z)(G(w)^2)_{kl}
&=
\Tr\bigl(G(z)A G(w)^2 A\bigr)
\\
&=
(d-1)
\left[
\Tr G(z)
+
2w\Tr(G(z)G(w))
+
w^2\Tr(G(z)G(w)^2)
\right].
\end{align*}
Third,
\[
\sum_{i,j,k,l}
A_{ik}A_{jl}G_{ji}(z)(G(w)^2)_{ik}
=
d\sqrt{d-1}\,
\frac{\Tr G(w)+w\Tr G(w)^2}
{d/\sqrt{d-1}-z}.
\]
The fourth term is the same:
\[
\sum_{i,j,k,l}
A_{ik}A_{jl}G_{ji}(z)(G(w)^2)_{jl}
=
d\sqrt{d-1}\,
\frac{\Tr G(w)+w\Tr G(w)^2}
{d/\sqrt{d-1}-z}.
\]
Therefore
\begin{align}\begin{split}
&\frac1{N^2d(d-1)}
\sum_{i,j,k,l}
A_{ik}A_{jl}G_{ji}(z)\,
\partial_{ij}^{kl}m_N(w)
=
-\frac2{d(d-1)N^3}
\Bigg[
d^2\Tr\bigl(G(z)G(w)^2\bigr)
\\
&
+
(d-1)
\left(
\Tr G(z)
+
2w\Tr(G(z)G(w))
+
w^2\Tr(G(z)G(w)^2)
\right)
-
2d\sqrt{d-1}
\frac{\Tr G(w)+w\Tr G(w)^2}
{d/\sqrt{d-1}-z}
\Bigg].
\label{eq:weighted-external-trace-form}
\end{split}\end{align}

Finally, by the resolvent identity,
\[
\frac1N\Tr G(w)^2=\partial_wm_N(w),\quad \frac1N\Tr(G(z)G(w))
=
\frac{m_N(z)-m_N(w)}{z-w},
\]
and differentiating in \(w\) gives
\[
\frac1N\Tr(G(z)G(w)^2)
=
\partial_w
\frac{m_N(z)-m_N(w)}{z-w}.
\]
Substituting these identities into
\eqref{eq:weighted-external-trace-form}, we obtain
\begin{align}
&\frac1{N^2d(d-1)}
\sum_{i,j,k,l}
A_{ik}A_{jl}G_{ji}(z)\,
\partial_{ij}^{kl}m_N(w)
=
-\frac{2d}{N^2(d-1)}
\partial_w
\frac{m_N(z)-m_N(w)}{z-w}
\\
&
-\frac{2}{dN^2}
\left[
m_N(z)
+
2w\frac{m_N(z)-m_N(w)}{z-w}
+
w^2
\partial_w
\frac{m_N(z)-m_N(w)}{z-w}
\right]
+
\frac{4(m_N(w)+w\partial_wm_N(w))}
{\sqrt{d-1}N^2(d/\sqrt{d-1}-z)}.
\label{eq:weighted-external-expanded}
\end{align}
We recall from \eqref{e:sbound} and \eqref{e:fest}
\begin{align*}
&|\partial_wm_N(w)|
\lesssim
\frac{|m_N(w)|}{\Im [w]},\quad \left|
\frac{m_N(z)-m_N(w)}{z-w}
\right|
\lesssim
\frac{|m_N(z)|}{\Im [z]}
+
\frac{|m_N(w)|}{\Im[ w]},\\
\quad 
&\left|
\partial_w
\frac{m_N(z)-m_N(w)}{z-w}
\right|
\lesssim
\frac{|m_N(z)|+|m_N(w)|}{\Im [z]\,\Im [w]},
\end{align*}
the last  line of \eqref{eq:weighted-external-expanded} are bounded by
\[
\OO\left(
\frac{|m_N(z)|+|m_N(w)|}
{N^2d\,\Im[ z]\,\Im [w]}
\right).
\]
Thus
\begin{align*}
\frac1{N^2d(d-1)}
\sum_{i,j,k,l}
A_{ik}A_{jl}G_{ji}(z)\,
\partial_{ij}^{kl}m_N(w)=
-\frac{2d}{N^2(d-1)}
\partial_w
\frac{m_N(z)-m_N(w)}{z-w}
+
\OO\left(
\frac{|m_N(z)|+|m_N(w)|}
{N^2d\,\Im[ z]\,\Im[ w]}
\right).
\end{align*}
This completes the proof of \Cref{e:error_main}.
\end{proof}
\subsubsection{Edge loop equations for random \(d\)-regular graphs}
\label{s:d-reg-edge}

In this section, we explain how the proof of edge universality for random
\(d\)-regular graphs with fixed degree in \cite{huang2024ramanujan} can be
significantly simplified using the new characterization. 
In that work, edge universality is proved by first
establishing the edge loop equations for the Gaussian divisible ensemble
$
H+\sqrt t\,\mathrm{GOE}$,
where \(H\) is the normalized adjacency matrix of a random \(d\)-regular graph
and \(t>0\) is small. The rest of the argument then follows the standard
three-step scheme for random matrix universality
\cite{erdHos2017dynamical}: a local law, universality for a short Gaussian
divisible flow, and a comparison step back to the original ensemble.

For random \(d\)-regular graphs, especially when \(d\) is fixed or small, the
usual Green function comparison methods for Wigner-type matrices do not apply
directly. The microscopic loop equations provide an alternative comparison
mechanism: they show that the edge statistics are unchanged along the Gaussian
divisible flow \(H+\sqrt t\,\mathrm{GOE}\), and hence allow one to transfer edge
universality from the Gaussian divisible ensemble back to \(H\).

Our new characterization, \Cref{t:edge-approx-convergence}, simplifies this
argument. Instead of deriving and analyzing the edge loop equations for
the entire Gaussian divisible family \(H+\sqrt t\,\mathrm{GOE}\), it is enough
to verify the approximate edge loop equations directly for the original
ensemble \(H\). In other words, the characterization turns the comparison
problem into a direct criterion: once the edge loop equations for \(H\) are
established with sufficiently small errors, edge universality follows.

We recall the Kesten--McKay distribution
\begin{align}\label{e:KMdistribution}
\varrho_d(x)
:=
\mathbf 1_{x\in[-2,2]}
\left(
1+\frac1{d-1}-\frac{x^2}{d}
\right)^{-1}
\frac{\sqrt{4-x^2}}{2\pi}.
\end{align}
Near the spectral edges \(\pm 2\), the Kesten--McKay distribution has square-root
behavior:
\begin{align}\label{e:edge_behavior}
x\to \pm 2,
\qquad
\varrho_d(x)
=
\frac{\cA\sqrt{2\mp x}}{\pi}
+
\OO(|2\mp x|),
\qquad
\cA:=\frac{d(d-1)}{(d-2)^2}.
\end{align}
We denote by \(\md(z)\) the Stieltjes transform of \(\varrho_d\):
\begin{align*}
    \md(z)
    =
    \int_\bR \frac{\varrho_d(x)\,\rd x}{x-z},
    \qquad
    z\in \bC^+.
\end{align*}

\begin{definition}
Fix \(d\ge 3\) and a sufficiently small constant \(0<\fc<1\). Let
$
\fR:=({\fc}/{4})\log_{d-1}N$.
We define \(\Omega_N\) to be the event on which the following two conditions
hold:
\begin{enumerate}
    \item The number of vertices whose radius-\(\fR\) neighborhood is not a tree
    is at most \(N^\fc\).

    \item The radius-\(\fR\) neighborhood of every vertex has excess at most
    one. Here the excess means the number of independent cycles.
\end{enumerate}
\end{definition}

The event \(\Omega_N\) is typical. More precisely, by
\cite[Proposition 2.1]{huang2024spectrum}, applied with \(\omega_d=1\), we have
\[
\mathbb P(\Omega_N)
\ge
1-\OO\bigl(N^{-(1-\fc)}\bigr).
\]

For random \(d\)-regular graphs, the required approximate loop equations are
given in the following proposition. These equations can be viewed as the direct
analogue, for the original random regular graph ensemble, of the edge loop
equations proved in \cite[Theorem 3.8]{huang2024ramanujan} for the Gaussian
divisible family
$
H+\sqrt t\,\mathrm{GOE}$.
For an expository discussion of the first loop equation, corresponding to
\(p=1\), we refer to \cite{huang2026lecture}.

\begin{proposition}\label{p:loop}
Fix a degree \(d\ge 3\), and set the right spectral edge \(E=2\). Fix a small
\(\fc>0\). Introduce the microscopic window
\begin{align}\label{e:micro_window}
\mathbf M
:=
\left\{
w\in\bC:
N^{-2/3-\fc}\le |\Im[ w]|\le N^{-2/3+\fc},
\quad
-N^{-2/3+\fc}\le \Re[w]\le N^{-2/3+\fc}
\right\}.
\end{align}
Let $
\Delta(z):=m_N(z)-m_d(z)$.
Fix \(p\ge1\) and \(w,w_2,\dots,w_p\in\mathbf M\), and set
\[
z:=E+w,
\qquad
z_a:=E+w_a,
\qquad 2\le a\le p.
\]
Then, for \(N\) sufficiently large,
\begin{align}
\begin{split}\label{e:micro_loop}
&\bE\left[
\bm1(\Omega_N)\left(
\Delta(z)^2
+
2\cA\sqrt{z-E}\,\Delta(z)
+
\frac{1}{N}\partial_z m_N(z)
\right)
\prod_{a=2}^p \Delta(z_a)
\right]
\\
&\quad
+
\frac{2}{N^2}
\sum_{a=2}^p
\bE\left[
\bm1(\Omega_N)\partial_{z_a}
\left(
\frac{m_N(z)-m_N(z_a)}{z-z_a}
\right)
\prod_{\substack{b=2\\ b\neq a}}^p
\Delta(z_b)
\right]
=
\OO\left(N^{-(p+1)/3-10\fc}\right).
\end{split}
\end{align}
Here, when \(p=1\), the empty product is interpreted as \(1\), and the sum over
\(a=2,\dots,p\) is absent.
\end{proposition}

The equations \eqref{e:micro_loop} can be viewed as a centered version of the
edge loop equations \eqref{e:loopeq2-edge}. Namely, under the edge scaling,
\[
\frac{N^{1/3}}{\cA^{2/3}}\,
\Delta\left(E+\frac{w}{(\cA N)^{2/3}}\right)
\]
converges to \(s(w)-\sqrt w\) in the edge loop equations
\eqref{e:loopeq2-edge}. The full set of edge loop equations then follows by
taking suitable linear combinations.

\section{Bulk and edge loop equations}

The purpose of this section is twofold. First, we identify the loop equations
satisfied by the limiting point processes in the bulk and edge regimes. Namely,
the \(\mathrm{Sine}_\beta\) point process satisfies the bulk loop equation, while
the \(\mathrm{Airy}_\beta\) point process satisfies the edge loop equation.
Second, we show that these loop equations yield concentration estimates for the
corresponding normalized Stieltjes transforms. More precisely, after the
appropriate normalization, these Stieltjes transforms concentrate around the
deterministic quantities \(\pm \ri\) in the bulk and \(\sqrt z\) at the edge,
with sub-Gaussian tails. All results in this section hold for every
\(\beta>0\).

\subsection{{\rm Sine}$_\beta$ and {\rm Airy}$_\beta$ point process}
\label{s:sine_airy}
For the $\mathrm{Sine}_{\beta}$ point
process, we use the normalization in which its
intensity is $1/\pi$, or equivalently, the local mean spacing is $\pi$. For the $\mathrm{Airy}_{\beta}$ point process, we use the convention in which the Airy
points are ordered decreasingly and extend to $-\infty$. In this normalization,
the one-point density has the standard square-root left-tail asymptotics with
leading constant $1/\pi$. The following theorems state that
the \(\mathrm{Sine}_\beta\) point process satisfies the bulk loop equation, and
the \(\mathrm{Airy}_\beta\) point process satisfies the edge loop equation.

\begin{theorem}\label{t:bulk_loop}
Fix \(\beta>0\), and let \((x_j)_{j\in\mathbb Z}\) denote the 
\(\mathrm{Sine}_\beta\) point process with intensity \(1/\pi\), indexed so that
\begin{align}
\cdots \leq x_{-2}\leq x_{-1}\leq x_0\leq x_1\leq x_2\leq \cdots .
\end{align}
Then its principal-value Stieltjes transform
\begin{align}
s(w)=\operatorname{P.V.}\sum_{j\in\mathbb Z}\frac{1}{x_j-w}
\end{align}
satisfies the bulk loop equation \eqref{e:loopeq2-bulk}.
\end{theorem}

\begin{theorem}\label{t:edge_loop}
Fix \(\beta>0\), and let \((x_j)_{j\geq 1}\) denote the 
\(\mathrm{Airy}_\beta\) point process, ordered decreasingly so that
\begin{align}
x_1\geq x_2\geq x_3\geq \cdots, 
\qquad x_j\to -\infty \quad \text{as } j\to\infty .
\end{align}
Then its normalized Stieltjes transform
\begin{align}
s(w)=\sum_{j=1}^\infty
\left(\frac{1}{x_j-w} -\frac{1}{\fa_j}\right)
-\frac{\Ai'(0)}{\Ai(0)} .
\end{align}
satisfies the edge loop equation \eqref{e:loopeq2-edge}.
\end{theorem}

Both \Cref{t:bulk_loop} and \Cref{t:edge_loop} are proved by taking the bulk and
edge scaling limits, respectively, of the Gaussian \(\beta\)-ensemble. Their
proofs are given in \Cref{s:beta_loop}.

\subsection{Local Law}
In this section, we collect  some consequences for point processes satisfying the
bulk/edge loop equations.  \Cref{p:bulk_concentration} states that if a Nevanlinna function $s$
associated to a point
process satisfies the bulk loop equation \eqref{e:loopeq2-bulk}, then \(s(w)\)  has a sub-Gaussian tail and
is the Stieltjes transform of its empirical density.
\Cref{p:edge_concentration} states that if a point process satisfies the edge
loop equation \eqref{e:loopeq2-edge}, then \(s(w)\) has a sub-Gaussian tail and is the \emph{normalized} Stieltjes
transform of its empirical density.

The proofs follow a method introduced in \cite{bourgade2022optimal}, which was recently used in \cite{bourgadeSodinZeitouni} to compute the excess of the ${\rm Sine}_{\beta}$
  point process. The local law below is formulated for general Nevanlinna functions, whereas \cite{bourgadeSodinZeitouni} works with the principal-value Stieltjes transform,  and obtains concentration estimates with explicit dependence on the inverse temperature $\beta$.

\label{s:bulk_concentration}
\begin{proposition}\label{p:bulk_concentration}
Under Assumption \ref{a:bulk}, for each fixed \(\beta>0\), there exists a constant
\(C=C(\beta)>0\) such that, for every \(w=E+\ri\eta\in \bC_+\) with
\(\eta>0\) and every integer \(q\ge 1\),
\begin{equation}\label{eqn:LocalLaw}
\mathbb E\bigl[|s(w)-\ri|^{2q}\bigr]
\le \frac{(Cq)^{q}}{\eta^{2q}} .
\end{equation}
Moreover, the poles of \(s(w)\) can be ordered as
\[
\cdots \le x_{-2}\le x_{-1}\le x_0\le x_1\le x_2\le \cdots .
\]
There exists a random variable \(M>0\) such that, almost surely,
\begin{align}\label{e:concentration1}
|x_j-\pi j|\le M\bigl(\log(2+|j|)\bigr)^2 .
\end{align}
Furthermore, for any \(w\in \bC_+\cup \bC_-\), almost surely,
\begin{align}\label{e:szrep}
s(w)={\rm P.V.}\sum_{j=-\infty}^{\infty}\frac{1}{x_j-w}.
\end{align}
\end{proposition}

We next recall the Airy function, denoted by~$\Ai$, which solves the Airy equation: $\Ai''(w)-w\Ai(w)=0$, with $\Ai(w)\to 0$ as $w\to \infty$ along $\bR_+$. 
All the zeros of $\Ai$ are on the real line, and are all negative, and we denote them as $0>\fa_1>\fa_2>\fa_3>\cdots$. We refer to \Cref{a:Airy} for further background on Airy functions.

\begin{proposition}\label{p:edge_concentration}
Under \Cref{a:edge}, for each fixed \(\beta>0\), there exists a constant
\(C=C(\beta)>0\) such that the following holds. Let
\(w=E+\ri\eta\in\bC_+\), with \(\eta>0\), and let \(q\ge 1\) be an integer, then
\begin{equation}\label{eqn:edge}
\mathbb{E}\left[\left|s(w)-\sqrt w\right|^{2q}\right]
\le
\frac{(Cq)^q}{\eta^{2q}} \frac{|w|^q}{(\Im[\sqrt{w}])^{2q}}.
\end{equation}
If we further assume that \(E\ge \eta\), then
\begin{equation}\label{eqn:edge2}
\mathbb{E}\left[\left|\Im\left[s(w)-\sqrt w\right]\right|^{2q}\right]
\le
\frac{(Cq)^q}{\eta^{q}E^q}
+
\frac{(Cq)^{2q}}{\eta^{4q}E^q}.
\end{equation}
Moreover, the poles of \(s(w)\) can be ordered as
\[
x_1\ge x_2\ge \cdots .
\]
There exists a random variable \(M>0\) such that, almost surely,
\begin{align}\label{e:concentration2}
\left|x_j+\left(\frac{3\pi j}{2}\right)^{2/3}\right|
\le
M\frac{\bigl(\log(2+j)\bigr)^2}{j^{1/3}} .
\end{align}
Furthermore, for any \(w\in \bC_+\cup \bC_-\), almost surely,
\begin{align}\label{e:normalized}
s(w)=\sum_{i=1}^\infty
\left(\frac{1}{x_i-w} -\frac{1}{\fa_i}\right)
-\frac{\Ai'(0)}{\Ai(0)} .
\end{align}
\end{proposition}

 \Cref{s:bulk_local_law}
is devoted to the proof of \Cref{p:bulk_concentration}. The proof of
\Cref{p:edge_concentration} is analogous and is postponed to
\Cref{s:asymp_edge}.

\subsection{Bulk and edge loop equations for $\beta$-ensemble}
\label{s:beta_loop}
In this section, we show that, under the one-cut condition, the
\(\beta\)-ensemble with a general potential satisfies the approximate bulk and
edge loop equations, namely \eqref{e:approx-loopeq2-bulk} and
\eqref{e:approx-loopeq2-edge}. Both results follow from straightforward
computations, with the local laws for \(\beta\)-ensembles up to microscopic scale
\cite{bourgade2022optimal} as an input.
We defer the proofs to
\Cref{s:beta_ensemble}. We then prove \Cref{t:bulk_loop} and \Cref{t:edge_loop} by taking the bulk and
edge scaling limits of these approximate loop equations.

Consider the $\beta$-ensemble
\begin{equation}
	\label{eq:beta_ensembles_density}
	\rd \mu_N(\lambda_1,\dots,\lambda_N)
	=
	\frac{1}{Z_N}
	\prod_{1\leq i<j\leq N}
	\abs{\lambda_i-\lambda_j}^{\beta}
	e^{
		-\frac{\beta N}{2}\sum_{k=1}^N V(\lambda_k)
	}
	\rd \lambda_1\cdots \rd \lambda_N .
\end{equation}
\begin{assumption}\label{a:Vasump}
We assume the following.
\begin{enumerate}[label=(A\arabic*)]
\item \label{assumption:analytic}
The potential \(V\) is analytic on \(\bC\) and real-valued on $\bR$.

\item \label{assumption:V'_at_infinity}
There exist constants \(M_0,C,c>0\) such that
\[
	V'(x)\geq c,
	\qquad
	\sup_{y\in[M_0,x]}\frac{V'(y)}{y}\leq C V(x),
	\qquad x\geq M_0,
\]
and the analogous estimates hold as \(x\to -\infty\), equivalently for
\(\widetilde V(x):=V(-x)\).

\item \label{assumption:off-criticality}
The equilibrium measure \(\mu_V\) is one-cut regular. More precisely, its density
\(\varrho_V\) is positive on a single interval \([A,B]\) and has square-root behavior
at both endpoints.

\item \label{assumption:large_dev}
The effective potential
\[
	x\longmapsto
	\frac{V(x)}{2}
	-
	\int \log\abs{x-t}\,\rd\mu_V(t)
\]
achieves its minimum only on \([A,B]\).
\end{enumerate}
\end{assumption}
Write
\begin{align}\label{e:defr}
	\varrho_V(t)
	=
	\frac{1}{\pi}r(t)\sqrt{(t-A)(B-t)},
	\qquad t\in[A,B],
\end{align}
and denote by \(m_V\) the Stieltjes transform of \(\mu_V\),
\[
	m_V(z):=\int_A^B \frac{\rd\mu_V(t)}{t-z}.
\]
We also write
\begin{align}\label{e:defmB}
	m_B:=m_V(B),
\end{align}
where \(m_V(B)\) denotes the continuous extension of \(m_V\) to the right edge.

\begin{proposition}[Approximate bulk loop equations for \(\beta\)-ensembles]
	\label{p:bulk-loop-beta}
Assume that the potential \(V\) satisfies \Cref{a:Vasump}. Fix a bulk energy
\(E\in(A,B)\), and write
$
	\varrho_E:=\varrho_V(E)$.
For \(w\in\bC\setminus\bR\), define
\begin{align}\label{e:defsN0}
	s_N(w)
	=
	\sum_{k=1}^N
	\frac{1}{
		N\pi\varrho_E(\lambda_k-E)-w
	}
	+
	\frac{V'(E)}{2\pi\varrho_E}.
\end{align}
Thus \(s_N\) is a particle-generated Nevanlinna function associated with the
configuration
\[
	\sum_{k=1}^N
	\delta_{N\pi \varrho_E(\lambda_k-E)}.
\]
Then, for every \(p\ge1\) and every compact set
\(K\subset\bC\setminus\bR\), uniformly for
\(w_1,w_2\dots,w_p\in K\),
\begin{align}
	\label{e:beta-loopeq2-bulk}
	\bE\left[
		\left(
			\frac{2-\beta}{\beta}\partial_{w_1} s_N(w_1)
			+s_N(w_1)^2+1
		\right)
		\prod_{j=2}^p s_N(w_j)
	\right]
	+
	\frac{2}{\beta}
	\bE\left[
		\sum_{j=2}^p
		\partial_{w_j}
		\frac{s_N(w_1)-s_N(w_j)}{w_1-w_j}
		\prod_{\substack{\ell=2\\ \ell\neq j}}^ps_N(w_\ell)
	\right]
	=
	\oo_N(1).
\end{align}
\end{proposition}

\begin{proposition}[Approximate edge loop equations for \(\beta\)-ensembles]
	\label{p:edge-loop-beta}
Assume that \(V\) satisfies \Cref{a:Vasump}. Let
\begin{align}\label{e:defgamma}
	\gamma:=\bigl(r(B)^2(B-A)\bigr)^{-1/3}.
\end{align}
For \(w\in\bC\setminus\bR\), define the edge-scaled empirical Stieltjes transform
\begin{align}\label{e:defsN3}
	s_N(w)
	=
	\sum_{k=1}^N
	\frac{1}{
		\gamma^{-1}N^{2/3}(\lambda_k-B)-w
	}
	-
	\gamma N^{1/3}m_B.
\end{align}
Thus \(s_N\) is a particle-generated Nevanlinna function associated with the
configuration
\[
	\sum_{k=1}^N
	\delta_{\gamma^{-1}N^{2/3}(\lambda_k-B)}.
\]
Then, for every \(p\geq1\) and every compact set
\(K\subset\bC\setminus\bR\), uniformly for
\(w_1,\dots,w_p\in K\),
\begin{align}
	\label{e:edge-loop-from-beta}
	\bE\left[
		\left(
			\frac{2-\beta}{\beta}\partial_{w_1} s_N(w_1)
			+s_N(w_1)^2
			-w_1
		\right)
		\prod_{j=2}^p s_N(w_j)
	\right]
	+
	\frac{2}{\beta}
	\bE\left[
		\sum_{j=2}^p
		\partial_{w_j}
		\frac{s_N(w_1)-s_N(w_j)}{w_1-w_j}
		\prod_{\substack{\ell=2\\ \ell\neq j}}^ps_N(w_\ell)
	\right]
	=
	\oo_N(1).
\end{align}
\end{proposition}

\begin{remark}
Propositions \ref{p:bulk-loop-beta}  and \ref{p:edge-loop-beta} together with  \Cref{t:bulk-approx-convergence,t:edge-approx-convergence} imply universality for
$\beta$-ensembles,  in the bulk and at the edge.  Previous proofs proceed by comparison with the Gaussian $\beta$-ensemble \cite{MR3253704,MR3192527,landon2019fixed,bekerman2015transport,MR3390602}.
\end{remark}

\begin{proof}[Proof of \Cref{t:bulk_loop}]
We apply \Cref{p:bulk-loop-beta} to the Gaussian \(\beta\)-ensemble,
that is, to \(V(x)=x^2\), at the bulk point \(E=0\). Then \(V'(0)=0\).
Writing \(\varrho_0:=\varrho_V(0)\), the renormalized Stieltjes transform
in \eqref{e:defsN0} becomes
\[
s_N(w)
=
\sum_{k=1}^N
\frac{1}{N\pi\varrho_0\lambda_k-w}.
\]
The corresponding microscopic point process
\[
\mu_N:=\sum_{k=1}^N \delta_{N\pi\varrho_0\lambda_k}
\]
converges in distribution to the \(\mathrm{Sine}_\beta\) point process
with intensity \(1/\pi\), as constructed in \cite{valko2009continuum}.

By the same tightness and uniform-integrability argument as in
\Cref{t:bulk-approx-convergence}, after passing to a subsequence we may
assume that
$
s_N \Longrightarrow s$ 
in the topology of locally uniform convergence on compact subsets of
\(\bC\setminus\bR\), and the limit $s$ satisfies the exact bulk loop equation
\eqref{e:loopeq2-bulk}. Moreover, the associated limiting particle
configuration \(\mu\) has the \(\mathrm{Sine}_\beta\) law.

Ordering the atoms of \(\mu\) as
\[
\cdots \le x_{-2}\le x_{-1}\le x_0\le x_1\le x_2\le\cdots,
\]
the representation \eqref{e:szrep} gives
\[
s(w)
=
\mathrm{P.V.}\int \frac{\rd\mu(x)}{x-w}
=
\mathrm{P.V.}\sum_{j=-\infty}^{\infty}
\frac{1}{x_j-w}.
\]
Therefore the Stieltjes transform of the \(\mathrm{Sine}_\beta\) point
process satisfies the bulk loop equation \eqref{e:loopeq2-bulk}.
\end{proof}

\begin{proof}[Proof of \Cref{t:edge_loop}]
The result follows by the same argument as in the proof of \Cref{t:bulk_loop}, 
now applying \Cref{p:edge-loop-beta} to the Gaussian \(\beta\)-ensemble and using 
the fact that its edge scaling limit is the \(\mathrm{Airy}_\beta\) point process 
\cite{ramirez2011beta}. We therefore omit the details.
\end{proof}

\subsection{Proof of bulk local law}
\label{s:bulk_local_law}

\subsubsection{Proof of \Cref{eqn:LocalLaw}}
The high-moment estimate \eqref{eqn:LocalLaw} follows essentially from the
analysis of loop equations for \(\beta\)-ensembles in
\cite[Theorem 1.1]{bourgade2022optimal}. The proof is based on a moment bootstrap from the bulk loop equation.  We first
derive deterministic bounds on \(s'(w)\) and on the difference quotient
\(\partial_w((s(z)-s(w))/(z-w))\) in terms of \(\Im[s(w)]\).  Applying
the loop equation to suitable polynomial test functions then yields a
self-consistent estimate for the high moments of
\(Q(w):=s(w)^2+1\).  Since \(Q(w)=(s(w)-\ri)(s(w)+\ri)\) and \(s\) is a
Nevanlinna function, moments of \(Q(w)\) control moments of \(s(w)-\ri\).
A Young inequality argument closes the estimate and gives the desired
sub-Gaussian-type bound for \(s(w)-\ri\).

\begin{proof}[Proof of \eqref{eqn:LocalLaw}]
Let \(w=E+\ri\eta\), with
\(\eta>0\). Define
\[
Q(w):=s(w)^2+1,
\qquad
f(z,w):=\partial_w\left(\frac{s(z)-s(w)}{z-w}\right),
\qquad
f(w,w):=\frac12 s''(w).
\]
For a monomial \(F(u,v)=u^a v^b\), the loop equation
\eqref{e:loopeq2-bulk}, applied with $w_1=w$,  \(a\) copies of \(w\) and \(b\) copies of
\(\bar w\), together with \(s(\bar w)=\overline{s(w)}\), gives
\begin{align*}
\mathbb E\!\left[Q(w)\,F(s(w),s(\bar w))\right]
&=
\left(1-\frac{2}{\beta}\right)
\mathbb E\!\left[s'(w)\,F(s(w),s(\bar w))\right] \\
&\quad
-\frac{2}{\beta}\,
\mathbb E\!\left[
f(w,w)\,\partial_1F(s(w),s(\bar w))
+f(w,\bar w)\,\partial_2F(s(w),s(\bar w))
\right].
\end{align*}
By linearity, the same identity holds for every polynomial \(F\).

Now take
\[
F(u,v)=(u^2+1)^{q-1}(v^2+1)^q.
\]
Then
\[
\partial_1F(u,v)
=
2(q-1)u(u^2+1)^{q-2}(v^2+1)^q,
\qquad
\partial_2F(u,v)
=
2qv(u^2+1)^{q-1}(v^2+1)^{q-1}.
\]
Therefore,
\begin{align}
\mathbb E\bigl[|Q(w)|^{2q}\bigr]
\le
C\,\mathbb E\!\left[|s'(w)|\,|Q(w)|^{2q-1}\right]
+
Cq\,\mathbb E\!\left[
\bigl(|f(w,w)|+|f(w,\bar w)|\bigr)
|s(w)|\,|Q(w)|^{2q-2}
\right].
\label{eq:Xq-pre}
\end{align}

Substituting the bounds \eqref{e:sbound} and \eqref{e:fest} into
\eqref{eq:Xq-pre} yields
\begin{align*}
\mathbb E\bigl[|Q(w)|^{2q}\bigr]
&\le
C\,\mathbb E\!\left[
\frac{\Im [s(w)]}{\eta}\,|Q(w)|^{2q-1}
\right]
+
Cq\,\mathbb E\!\left[
\frac{\Im [s(w)]\,|s(w)|}{\eta^2}\,|Q(w)|^{2q-2}
\right].
\end{align*}

We now apply Young's inequality in the forms
\begin{align}\label{e:Young}
Cab^{2q-1}\le \frac14 b^{2q}+(C'a)^{2q},
\qquad
Cqa b^{2q-2}\le \frac14 b^{2q}+(C'qa)^q.
\end{align}
After absorbing the resulting \(\frac12\mathbb E[|Q(w)|^{2q}]\) contribution
into the left-hand side, we obtain
\[
\mathbb E\bigl[|Q(w)|^{2q}\bigr]
\le
\left(\frac{Cq}{\eta^2}\right)^q
\left(
\mathbb E\bigl[(\Im [s(w)])^{2q}\bigr]
+
\mathbb E\bigl[(\Im [s(w)])^q|s(w)|^q\bigr]
\right).
\]
In particular, since $|\Im[s(w)]|\leq |s(w)|$,  we have
\begin{align}\label{e:EQbb}
\mathbb E\bigl[|Q(w)|^{2q}\bigr]
\le
\left(\frac{Cq}{\eta^2}\right)^q
\mathbb E\bigl[(\Im [s(w)])^q|s(w)|^q\bigr].
\end{align}

Since \(s\) is a Nevanlinna function, \(\Im [s(w)]\ge 0\). Moreover,
\[
Q(w)=(s(w)-\ri)(s(w)+\ri),
\]
and for \(\Im [s(w)]\ge 0\) we have
\[
|s(w)-\ri|\le |s(w)+\ri|,
\qquad
|s(w)+\ri|\ge 1.
\]
Hence
\begin{equation}\label{eq:sQcompare}
|s(w)-\ri|^2\le |Q(w)|,
\qquad
|s(w)-\ri|\le |Q(w)|.
\end{equation}

As a consequence of \eqref{eq:sQcompare}, we have
\[
\Im[s(w)]^q|s(w)|^{q}\leq |s(w)|^{2q}\le (1+|s(w)-\ri|)^{2q}\le C^q(1+|Q(w)|^{q}).
\]
By plugging the above estimate into \eqref{e:EQbb}, we conclude
\begin{align}\label{e:Qzbound}
\mathbb E\bigl[|Q(w)|^{2q}\bigr]
\le
\left(\frac{Cq}{\eta^2}\right)^q
\left(
1+\mathbb E\bigl[|Q(w)|^{2q}\bigr]^{1/2}
\right).
\end{align}

By \eqref{eq:sQcompare},
\begin{align}\label{e:Qzbound2}
\mathbb E\bigl[|s(w)-\ri|^{2q}\bigr]
\le
\mathbb E\bigl[|Q(w)|^q\bigr]
\le
\mathbb E\bigl[|Q(w)|^{2q}\bigr]^{1/2}, \quad
\mathbb E\bigl[|s(w)-\ri|^{2q}\bigr]
\le
\mathbb E\bigl[|Q(w)|^{2q}\bigr]
.
\end{align}

There are two cases. If
\(\mathbb E[|Q(w)|^{2q}]\le 1\), then \eqref{e:Qzbound} and
\eqref{e:Qzbound2} imply
\[
\mathbb E\bigl[|s(w)-\ri|^{2q}\bigr]
\le
\mathbb E\bigl[|Q(w)|^{2q}\bigr]
\le
2\left(\frac{Cq}{\eta^2}\right)^q.
\]
If \(\mathbb E[|Q(w)|^{2q}]\ge 1\), then dividing \eqref{e:Qzbound} by
\(\mathbb E[|Q(w)|^{2q}]^{1/2}\) gives
\[
\mathbb E\bigl[|s(w)-\ri|^{2q}\bigr]
\le
\mathbb E\bigl[|Q(w)|^{2q}\bigr]^{1/2}
\le
2\left(\frac{Cq}{\eta^2}\right)^q.
\]
Hence \eqref{eqn:LocalLaw} holds in all cases.
\end{proof}

The sub-Gaussian tail estimate \eqref{eqn:LocalLaw}, together with a union
bound, implies the following almost-sure uniform bound.

\begin{lemma}[Almost-sure uniform bound on dyadic rectangles]
\label{prop:as-uniform-rectangles}
Adopt the assumptions of \Cref{p:bulk_concentration}. For \(n\ge 1\), define
\[
\mathcal R_n
:=
\Bigl\{w=E+\ri\eta:\ |E|\le 2^n,\ 2^{-n}\leq\eta\le 2^n\Bigr\},
\qquad
\mathcal R'_n
:=
\Bigl\{w=E+\ri\eta:\ |E|\le 2^n,\ 0<\eta\le 2^n\Bigr\}.
\]
Then there exists an almost surely finite random variable \(M\) such that,
almost surely,
\begin{equation}\label{eq:as-uniform-rectangles}
\sup_{w=E+\ri\eta\in \mathcal R_n}\eta\,|s(w)-\ri|
\le M n,
\qquad n\ge 1,
\end{equation}
and
\begin{equation}\label{eq:im-as-uniform-rectangles}
\sup_{w=E+\ri\eta\in \mathcal R'_n}
\eta\,|\Im[s(w)-\ri]|
\le M n,
\qquad n\ge 1.
\end{equation}
\end{lemma}

\begin{proof}
By \eqref{eqn:LocalLaw} and Markov's inequality,
\[
\mathbb P\!\left(\eta |s(E+\ri\eta)-\ri|\ge t\right)
\le \inf_{q\ge 1}\frac{(Cq)^q}{t^{2q}} .
\]
Optimizing in \(q\) gives constants \(c,C>0\) such that
\begin{equation}\label{eq:dyadic-subgaussian}
\mathbb P\!\left(\eta |s(E+\ri\eta)-\ri|\ge t\right)
\le C e^{-c t^2},
\qquad E\in \bR,\ \eta>0,\ t\ge 1 .
\end{equation}

Fix \(n\ge 1\), and define the lattice net
\[
\mathcal N_n:=\bigl\{2^{-n}(a+\ri b):\ a\in \mathbb Z,\ |a|\le 2^{2n},\ b\in \mathbb Z,\ 1\le b\le 2^{2n}\bigr\}.
\]
Its cardinality satisfies
\[
|\mathcal N_n|\leq C 2^{4n}
\]
By \eqref{eq:dyadic-subgaussian} and a union bound,
\[
\mathbb P\Bigl(
\exists\, w=E+\ri\eta\in \mathcal N_n:
\eta |s(w)-\ri|\ge n
\Bigr)
\le C2^{4n}e^{-cn^2}.
\]

The right-hand side is summable in \(n\). Hence, by Borel--Cantelli, almost
surely there exists \(n_0(\omega)\) such that, for all \(n\ge n_0(\omega)\),
\begin{equation}\label{eq:dyadic-net-good}
\eta |s(w)-\ri|\le n
\qquad
\text{for every } w=E+\ri\eta\in \mathcal N_n .
\end{equation}

We next pass from the net to the whole rectangle $\cR_n$. From \eqref{e:sbound} and
\(\Im[s(w)]\le |s(w)-\ri|+1\), we have
\begin{equation}\label{eq:dyadic-derivative-bound}
|s'(w)|\le \frac{|s(w)-\ri|+1}{\Im[ w]}.
\end{equation}
Let \(w_0=E_0+\ri\eta_0\in\mathcal N_n\), and let
\(w=E+\ri\eta\in {\mathcal R}_n\) satisfy
\[
|\Re [w]-\Re [w_0]|\le 2^{-n-1},
\qquad
|\Im [w]-\Im [w_0]|\le 2^{-n-1}.
\]
Since \(\eta_0\ge 2^{-n}\), the straight-line path from \(w_0\) to \(w\)
stays in the region \(\Im [\zeta]\ge \eta_0/2\). Applying
\eqref{eq:dyadic-derivative-bound} along this path and using Gronwall's
inequality gives
\begin{equation}\label{eq:local-stability-short}
|s(w)-s(w_0)|
\le
C\frac{2^{-n}}{\eta_0}\bigl(|s(w_0)-\ri|+1\bigr).
\end{equation}
Moreover, \(\eta/\eta_0\le C\). Therefore, using
\eqref{eq:dyadic-net-good},
\[
\eta |s(w)-\ri|
\le
C\eta_0|s(w_0)-\ri|
+
C2^{-n}\bigl(|s(w_0)-\ri|+1\bigr).
\]
Since \(2^{-n}\le \eta_0\), \(2^{-n}\le 1\le n\), and
\(\eta_0|s(w_0)-\ri|\le n\), we get
\[
\eta |s(w)-\ri|\le Cn .
\]
Thus, almost surely,
\[
\sup_{w\in {\mathcal R}_n}\eta |s(w)-\ri|
\le Cn,
\qquad n\ge n_0(\omega).
\]
Increasing the random constant to handle the finitely many indices
\(n<n_0(\omega)\), we obtain \eqref{eq:as-uniform-rectangles}.

It remains to control the imaginary part for \(0<\eta\le 2^{-n}\). By
\eqref{e:ims}, for each fixed \(E\in\bR\), the function
\[
\eta\longmapsto \eta\,\Im [s(E+\ri\eta)]
=
c\eta^2+\sum_{x\in P}
\frac{\eta^2}{(x-E)^2+\eta^2}
\]
is nondecreasing. Hence, with \(\eta_0:=2^{-n}\),
\[
\eta\,\Im [s(E+\ri\eta)]
\le
\eta_0\,\Im [s(E+\ri\eta_0)].
\]
Using \eqref{eq:as-uniform-rectangles} at \(E+\ri\eta_0\), we get
\[
\eta\,|\Im[s(E+\ri\eta)-\ri]|
\le
\eta\,\Im[ s(E+\ri\eta)]+\eta
\le
\eta_0\,\Im[ s(E+\ri\eta_0)]+1
\le
Mn+2 .
\]
After increasing \(M\), this gives \eqref{eq:im-as-uniform-rectangles}.
\end{proof}

\subsubsection{Helffer--Sj{\"o}strand formula}
In this section, we recall the Helffer--Sj{\"o}strand formula,
\Cref{lem:HS}, which allows us to control the difference between two measures
through their associated Nevanlinna functions. We will use \Cref{lem:HS} to
derive particle concentration estimates in \Cref{l:numberbound}.

Let \(f\in C_c^\infty(\bR;\bR)\), and let
\(\chi\in C_c^\infty(\bR;\bR)\) be an even cutoff function such that
\(\chi\equiv 1\) in a neighborhood of \(0\). Define the almost-analytic
extension
\begin{align}\label{e:ftilde0}
\widetilde f(x+\ri y):=(f(x)+\ri y f'(x))\chi(y),
\qquad w=x+\ri y .
\end{align}
Then, for every \(\lambda\in \bR\),
\[
f(\lambda)=\frac{1}{\pi}\int_{\bR^2}
\frac{\partial_{\bar w}\widetilde f(x+\ri y)}
{\lambda-x-\ri y}\,\rd x\,\rd y,
\qquad
\partial_{\bar w}:=\frac12(\partial_x+\ri\partial_y).
\]
A direct computation gives
\begin{equation}\label{e:dbar-ftilde}
\partial_{\bar w}\widetilde f(x+\ri y)
=
\frac{\ri}{2}\,y\chi(y)f''(x)
+
\frac{\ri}{2}\,\bigl(f(x)+\ri y f'(x)\bigr)\chi'(y).
\end{equation}

The scalar formula immediately implies the following standard consequence.

\begin{lemma}\label{lem:HSf}
Let \(s\) be a Nevanlinna function associated with the
Borel measure $\mu$. Then, for every test function \(f\in C_c^\infty(\bR;\bR)\),
\[
\int_{\bR} f(\lambda)\,\rd\mu(\lambda)
=
\frac{2}{\pi}
\int_{x+\ri y\in \bC_+}
\Re\!\left[
\partial_{\bar w}\widetilde f(x+\ri y)\,
s(x+\ri y)
\right]
\,\rd x\,\rd y .
\]
\end{lemma}

\begin{lemma}\label{lem:HS}
Let \(s\) and \(s_\varrho\) be two Nevanlinna functions associated with the
Borel measures
\[
\sum_{x\in P}\delta_x,
\qquad
\frac{1}{\pi}\varrho(x)\,\rd x,
\]
respectively. Then, for any \(\eta>0\) such that \(\chi(\eta)=1\),
\begin{align}
\left|
\sum_{x\in P} f(x)
-\frac{1}{\pi}\int_{\bR} f(x)\varrho(x)\,\rd x
\right|
\le
C\bigl(\mathrm{I}+\mathrm{II}+\mathrm{III}+\mathrm{IV}\bigr),
\label{e:HS}
\end{align}
where
\begin{align*}
\mathrm{I}
&:=
\iint_{y\ge 0}
\bigl(|f(x)|+y|f'(x)|\bigr)\,|\chi'(y)|\,
|s(x+\ri y)-s_\varrho(x+\ri y)|
\,\rd x\,\rd y,\\[1mm]
\mathrm{II}
&:=
\iint_{0\le y\le \eta}
|f''(x)|\,y|\chi(y)|\,
\bigl|\Im\bigl[s(x+\ri y)-s_\varrho(x+\ri y)\bigr]\bigr|
\,\rd x\,\rd y,\\[1mm]
\mathrm{III}
&:=
\iint_{y\ge \eta}
|f'(x)|\,\bigl|\partial_y\bigl(y\chi(y)\bigr)\bigr|\,
|s(x+\ri y)-s_\varrho(x+\ri y)|
\,\rd x\,\rd y,\\[1mm]
\mathrm{IV}
&:=
\eta\int_{\bR}
|f'(x)|\,|s(x+\ri \eta)-s_\varrho(x+\ri \eta)|
\,\rd x .
\end{align*}
Here \(C\) depends only on the cutoff function \(\chi\).
\end{lemma}

These identities and estimates are classical; see, for example,
\cite[Section 11.2]{erdHos2017dynamical}. We omit the proofs.

\subsubsection{Proof of \Cref{e:concentration1}}

The concentration of the particle locations follows from the following counting
estimate. We prove it using the Helffer--Sj{\"o}strand formula,
\Cref{lem:HS}, together with the uniform bound on the Nevanlinna function from
\Cref{prop:as-uniform-rectangles}.

\begin{lemma}[Bulk counting estimate]\label{l:numberbound}
Adopt the assumptions of \Cref{p:bulk_concentration}. Then there exists an
almost surely finite random variable \(M\) such that, almost surely, for every
\(a\ge 2\),
\begin{align}\label{e:interval_bb}
\left|\#(P\cap[0,a])-\frac{a}{\pi}\right|
\le M(\log a)^2, \quad
\left|\#(P\cap[-a,0])-\frac{a}{\pi}\right|
\le M(\log a)^2 .
\end{align}
\end{lemma}

\begin{proof}[Proof of \eqref{e:concentration1}]
Fix an outcome \(\omega\) on which \Cref{l:numberbound} holds. Enumerate the
poles of \(s\) increasingly as
\[
\cdots \le x_{-1}\le x_0\le x_1\le \cdots ,
\]
with only a finite ambiguity near the origin. Then there exists an almost
surely finite random variable \(M\) such that
\begin{equation}\label{e:estimatexj}
|x_j-\pi j|\le M\bigl(\log(2+|j|)\bigr)^2,
\qquad j\in \mathbb Z .
\end{equation}
Indeed, for \(j\ge 1\), \Cref{l:numberbound} gives
\[
j=\#(P\cap(0,x_j])+\OO(1)
=
\frac{x_j}{\pi}
+
\OO\bigl((\log x_j)^2\bigr),
\]
and hence
\[
|x_j-\pi j|\le M\bigl(\log(2+j)\bigr)^2 .
\]
The argument for \(j\le -1\) is identical, using the estimate on
\([-a,0]\). Enlarging \(M\) handles the finitely many indices near the origin.
\end{proof}

\begin{proof}[Proof of \Cref{l:numberbound}]
We prove the estimate on \([0,a]\); the estimate on \([-a,0]\) is identical.

Fix an outcome \(\omega\) on which \Cref{prop:as-uniform-rectangles} holds, and
let \(M=M(\omega)\) be the corresponding random constant. Given \(a\ge 2\),
choose \(n\ge 1\) such that
\[
2^{n-1}\le a+1\le 2^n .
\]
Then \(n\asymp \log a\).

Choose smooth functions \(f_+,f_-\in C_c^\infty(\mathbb R)\) such that
\[
0\le f_+\le 1,\qquad
f_+\equiv 1 \text{ on } [0,a],\qquad
\supp f_+\subset [-1,a+1],
\]
and
\[
0\le f_-\le 1,\qquad
f_-\equiv 1 \text{ on } [1,a-1],\qquad
\supp f_-\subset [0,a],
\]
with
\[
|f_\pm'|+|f_\pm''|\le C .
\]

We claim that, for either \(f=f_+\) or \(f=f_-\),
\begin{equation}\label{e:HS-sharp-bound}
\left|
\sum_{x\in P}f(x)
-\frac{1}{\pi}\int_{\mathbb R} f(x)\,\rd x
\right|
\le C M n^2 .
\end{equation}
To prove this claim, choose the cutoff in the almost-analytic extension \eqref{e:ftilde0} so that
\[
\chi\equiv 1 \text{ on } [-a,a],
\qquad
\supp \chi\subset [-a-1,a+1],
\qquad
|\chi'|\le C .
\]
We apply \Cref{lem:HS} with \(s_\varrho(w)\equiv \ri\), corresponding to
\(\varrho\equiv 1\), and with \(\eta=1\). Since
\[
\supp f\subset [-1,a+1]\subset [-2^n,2^n],
\qquad
\supp \chi\subset [-a-1,a+1]\subset [-2^n,2^n],
\]
the estimates from \Cref{prop:as-uniform-rectangles} apply throughout the
region of integration.

We now bound the four terms in \Cref{lem:HS}. For the term \(\mathrm{II}\), we
use
\[
y\,|\Im[s(x+\ri y)-\ri]|\le M n,
\qquad 0<y\le 1,
\]
and the fact that \(f''\) is supported in two intervals of \(\OO(1)\)-length.
Thus
\[
\mathrm{II}\le CMn .
\]
For \(\mathrm{IV}\), using \(|s(x+\ri)-\ri|\le Mn\) and the fact that \(f'\) is
supported in two intervals of \(\OO(1)\)-length gives
\[
\mathrm{IV}\le CMn .
\]
For \(\mathrm{III}\), note that \(\partial_y(y\chi(y))=1\) for
\(1\le y\le a\), while the transition region has \(\OO(1)\)-length. Since
\[
|s(x+\ri y)-\ri|\le \frac{Mn}{y},
\qquad 1\le y\le a+1,
\]
and \(f'\) is supported in a set of \(\OO(1)\)-measure, we obtain
\[
\mathrm{III}\le CMn\int_1^{a+1}\frac{\rd y}{y}\le CMn^2 .
\]
Finally, \(\mathrm{I}\) is supported where \(\chi'\neq 0\), hence where
\(y\in[a,a+1]\). Using again \(|s(x+\ri y)-\ri|\le Mn/y\), together with
\(|\supp f|\le Ca\) and \(|\supp f'|\le C\), yields
\[
\mathrm{I}\le CMn .
\]
Combining these bounds gives \eqref{e:HS-sharp-bound}.

We now apply \eqref{e:HS-sharp-bound} to \(f_+\). Since
\(f_+\ge \mathbf 1_{[0,a]}\),
\[
\#(P\cap[0,a])
\le
\sum_{x\in P} f_+(x)
\le
\frac{1}{\pi}\int_{\mathbb R} f_+(x)\,\rd x
+
CMn^2
\le
\frac{a}{\pi}+C+CMn^2 .
\]
Similarly, applying \eqref{e:HS-sharp-bound} to \(f_-\) and using
\(f_-\le \mathbf 1_{[0,a]}\), we get
\[
\#(P\cap[0,a])
\ge
\sum_{x\in P} f_-(x)
\ge
\frac{1}{\pi}\int_{\mathbb R} f_-(x)\,\rd x
-
CMn^2
\ge
\frac{a}{\pi}-C-CMn^2 .
\]
Combining the upper and lower bounds, and using \(n\asymp \log a\), gives
\[
\left|\#(P\cap[0,a])-\frac{a}{\pi}\right|
\le M(\omega)(\log a)^2 ,
\]
after increasing \(M(\omega)\). This proves the first estimate in \eqref{e:interval_bb}.
\end{proof}

\subsubsection{Proof of \Cref{e:szrep}}
\label{s:decomposition}
The proof compares \(s\) with an explicitly constructed principal-value sum \eqref{e:defs-tilde-s} having the same poles and the same normalization at infinity; the Nevanlinna representation then shows that their imaginary parts agree, so their difference is a real constant, which must vanish by the normalization.
\begin{proof}[Proof of \eqref{e:szrep}]
Define
\begin{equation}\label{e:defs-tilde-s}
\widetilde s(w)
:=
-\cot w
+
\sum_{j\in\bZ}
\left(
\frac1{x_j-w}
-
\frac1{\pi j-w}
\right),
\qquad w\in \bC\setminus\bR .
\end{equation}
By \eqref{e:concentration1}, the series in \eqref{e:defs-tilde-s} converges
absolutely and locally uniformly on \(\bC\setminus\bR\). Indeed, for every
compact \(K\subset \bC\setminus\bR\),
\begin{align}\label{e:xpdiff}
\left|
\frac1{x_j-w}
-
\frac1{\pi j-w}
\right|
=
\frac{|x_j-\pi j|}{|x_j-w|\,|\pi j-w|}
\le
C_K\frac{(\log(2+|j|))^2}{1+j^2},
\qquad w\in K .
\end{align}
Hence \(\widetilde s\) is holomorphic on \(\bC\setminus\bR\). Using the
partial fraction expansion
\[
-\cot w=\mathrm{P.V.}\sum_{j\in\bZ}\frac1{\pi j-w},
\]
we obtain
\[
\widetilde s(w)
=
\mathrm{P.V.}\sum_{j\in\bZ}\frac1{x_j-w}.
\]

We next check the normalization at infinity. Since
\[
-\cot(\ri y)\to \ri,
\qquad y\to+\infty,
\]
it remains to show that the summation term in \eqref{e:defs-tilde-s} tends to
zero at \(w=\ri y\). Split the sum into \(|j|\le y\) and \(|j|>y\). For
\(|j|\le y\), \eqref{e:concentration1} implies
\[
|x_j-\ri y|+|\pi j-\ri y|\gtrsim y,
\]
and therefore
\[
\sum_{|j|\le y}
\left|
\frac1{x_j-\ri y}
-
\frac1{\pi j-\ri y}
\right|
\le
\sum_{|j|\le y}
\frac{C(\log(2+|j|))^2}{y^2}
\le
\frac{C(\log y)^2}{y}.
\]
For \(|j|>y\), \eqref{e:concentration1} gives
\[
|x_j-\ri y|+|\pi j-\ri y|\gtrsim |j|,
\]
and hence
\[
\sum_{|j|>y}
\left|
\frac1{x_j-\ri y}
-
\frac1{\pi j-\ri y}
\right|
\le
\sum_{|j|>y}
\frac{C(\log(2+|j|))^2}{j^2}
\le
\frac{C(\log y)^2}{y}.
\]
Thus the summation term in \eqref{e:defs-tilde-s} tends to zero, and therefore
\[
\widetilde s(\ri y)\longrightarrow \ri,
\qquad y\to+\infty.
\]

On the other hand, \Cref{prop:as-uniform-rectangles} implies
\[
|s(\ri y)-\ri|
\le
\frac{M\log(2+y)}{y}
\longrightarrow 0,
\qquad y\to+\infty.
\]
Thus
\[
s(\ri y)\longrightarrow \ri .
\]

For \(w=E+\ri\eta\in\bC_+\), the definition of \(\widetilde s\) gives
\[
\Im [\widetilde s(w)]
=
\sum_{j\in\bZ}
\frac{\eta}{(x_j-E)^2+\eta^2}.
\]
Since \(s\) is a Nevanlinna function whose representing measure is
\(\sum_{j\in\bZ}\delta_{x_j}\), its imaginary part has the form
\[
\Im[s(w)]
=
c\eta
+
\sum_{j\in\bZ}
\frac{\eta}{(x_j-E)^2+\eta^2}
\]
for some \(c\ge 0\). The convergence \(s(\ri y)\to\ri\) as \(y\to+\infty\)
forces \(c=0\). Therefore,
\[
\Im[ s(w)]=\Im [\widetilde s(w)],
\qquad w\in\bC_+.
\]

It follows that
\[
g(w):=s(w)-\widetilde s(w)
\]
is holomorphic on \(\bC_+\) and has identically vanishing imaginary part.
Hence \(g\) is a real constant on \(\bC_+\). Since both \(s(\ri y)\) and
\(\widetilde s(\ri y)\) converge to \(\ri\) as \(y\to+\infty\), this constant
is zero. Therefore
\[
s(w)
=
\widetilde s(w)
=
\mathrm{P.V.}\sum_{j\in\bZ}\frac1{x_j-w},
\qquad w\in\bC_+.
\]
By Schwarz reflection, the same identity holds for all
\(w\in\bC\setminus\bR\). This proves \eqref{e:szrep}.
\end{proof}

\section{Exponential observables}
\label{s:exp_ob}
In this section, we introduce certain observables, which are exponential linear statistics of the Nevanlinna function \(s\). Equivalently, when the summation representation of
\(s\) is available, namely \eqref{e:szrep} and \eqref{e:normalized},  they can be viewed as products of ratios of
linear factors over the random point configuration. These
observables are the  analogues of ratios of characteristic polynomials.
Using the concentration estimates \Cref{p:bulk_concentration} and \Cref{p:edge_concentration} as input, we derive the asymptotic behavior of these observables at infinity.

\subsection{Bulk exponential observables}

\begin{proposition}[Bulk exponential observables]\label{prop:pv-product-smallL}
Adopt \Cref{a:bulk}, and assume
\[
	\frac{\beta}{2}=\frac{p}{q}\in \bQ_{>0}.
\]
Fix \(k\ge 1\), signs \(\tau_\alpha\in\{\pm1\}\), and points
\[
	t_r^{(\alpha)},\,s_a^{(\alpha)}\in \bC_{\tau_\alpha},
	\qquad
	1\le r\le p,\quad 1\le a\le q,\quad 1\le \alpha\le k .
\]
For each triple \((\alpha,r,a)\), choose a piecewise \(C^1\) path
\[
	\gamma_{r,a}^{(\alpha)}\subset \bC_{\tau_\alpha}
\]
from \(s_a^{(\alpha)}\) to \(t_r^{(\alpha)}\), and define
\begin{equation}\label{e:def-G-product}
	G
	:=
	-\frac{1}{q}
	\sum_{\alpha=1}^k
	\sum_{r=1}^p
	\sum_{a=1}^q
	\int_{\gamma_{r,a}^{(\alpha)}} s(u)\,\rd u .
\end{equation}
Since \(s\) is holomorphic on each half-plane, \(G\) is independent of the
choice of paths.

From the representation \eqref{e:szrep}, the exponential \(e^G\) coincides
with the principal-value product
\[
	e^G
	=
	\operatorname{P.V.}
	\prod_{j\in \bZ}
	\prod_{\alpha=1}^k
	\frac{
		\prod_{r=1}^p (t_r^{(\alpha)}-x_j)
	}{
		\prod_{a=1}^q (s_a^{(\alpha)}-x_j)^{\beta/2}
	},
\]
where the powers \((s^{(\alpha)}_a-x_j)^{\beta/2}\) are defined using the principal
branch on \(\bC\setminus\bR_{<0}\), as specified in
\eqref{e:principle_branch}.
Moreover, this random variable is integrable. We set
\begin{equation}\label{e:defF0}
	F:=\bE[e^G].
\end{equation}
Define
\begin{align}
	M
	&:=
	-\ri
	\sum_{\alpha=1}^k \tau_\alpha
	\left(
		\sum_{r=1}^p t_r^{(\alpha)}
		-\frac{\beta}{2}\sum_{a=1}^q s_a^{(\alpha)}
	\right),
	\label{e:def-M-product}
	\\
	L
	&:=
	\sum_{\alpha=1}^k
	\sum_{r=1}^p
	\sum_{a=1}^q
	\log\!\left(
		1+
		\frac{
			|t_r^{(\alpha)}-s_a^{(\alpha)}|
		}{
			\sqrt{
				|\Im[t_r^{(\alpha)}]|\,
				|\Im[s_a^{(\alpha)}]|
			}
		}
	\right).
	\label{e:def-L-product}
\end{align}
Then the following estimates hold.

\begin{enumerate}
\item[\textup{(i)}]
There exists a constant \(C>0\), depending only on \(p,q,k\), such that
\begin{equation}\label{e:F-upper-general}
	|F|\leq e^{\Re[M]+C(L+L^2)}.
\end{equation}

\item[\textup{(ii)}]
There exists a deterministic constant \(L_0>0\), depending only on \(p,q,k\),
such that, whenever \(L\le L_0\),
\begin{equation}\label{e:logF-smallL-general}
	F=e^{M+\OO(L)}
\end{equation}
where the implicit constant depends only on \(p,q,k\).
\end{enumerate}\end{proposition}

\begin{corollary}\label{c:absolute_bound}
Adopt \Cref{a:bulk}, and assume
\[
	\frac{\beta}{2}=\frac{p}{q}\in \bQ_{>0},
	\qquad
	n=kp,
	\qquad
	m=kq .
\]
Fix
\[
	t_r,s_a\in \bC_+\cup \bC_-,
	\qquad
	1\le r\le n,\quad 1\le a\le m .
\]
Define
\begin{align}
	M
	&:=
	\sum_{r=1}^n |\Im[t_r]|
	-\frac{\beta}{2}\sum_{a=1}^m |\Im[s_a]|,
	\label{e:def-M-abs}
	\\
	L
	&:=
	\sum_{r=1}^n
	\sum_{a=1}^m
	\log\!\left(
		1+
		\frac{
			|(\Re[t_r]+\ri|\Im[t_r]|)-(\Re[s_a]+\ri|\Im[s_a]|)|
		}{
			\sqrt{|\Im[t_r]|\,|\Im[s_a]|}
		}
	\right).
	\label{e:def-L-abs}
\end{align}
Then there exists a constant \(C>0\), depending only on \(p,q,k\), such that,
\begin{align}\label{e:absolutebound}
	\bE\left[
		\left|
		\operatorname{P.V.}
		\prod_{j\in \bZ}
		\frac{
			\prod_{r=1}^n (t_r-x_j)
		}{
			\prod_{a=1}^m (s_a-x_j)^{\beta/2}
		}
		\right|
	\right]
	\le
	\bE\left[
		\left|
		\operatorname{P.V.}
		\prod_{j\in \bZ}
		\frac{
			\prod_{r=1}^n (t_r-x_j)
		}{
			\prod_{a=1}^m (s_a-x_j)^{\beta/2}
		}
		\right|^2
	\right]^{1/2}
	\le
	e^{M+C(L+L^2)}.
\end{align}
Moreover, the exponential observable
\begin{equation}\label{e:defF1}
F(t_1,\dots,t_n;s_1,\dots,s_m)
:=
\bE\left[
	\operatorname{P.V.}
	\prod_{j\in \bZ}
	\frac{
		\prod_{r=1}^n (t_r-x_j)
	}{
		\prod_{a=1}^m (s_a-x_j)^{\beta/2}
	}
\right]
\end{equation}
is holomorphic on
$
\bC^n\times(\bC_+\cup\bC_-)^m$.
\end{corollary}

\begin{proof}[Proof of \Cref{prop:pv-product-smallL}]

We first recall the concentration estimates in both half-planes. By \eqref{eqn:LocalLaw} and
Schwarz reflection, for every \(z\in \bC_\tau\), \(\tau\in\{\pm1\}\), and every
integer \(\ell\ge 1\),
\begin{equation}\label{e:s-local-law-both-halfplanes}
\bE\left[|s(z)-\ri\tau|^{\ell}\right]\leq \bE\left[|s(z)-\ri\tau|^{2\ell}\right]^{1/2}
\le \frac{(C\sqrt{\ell})^\ell}{|\Im [z]|^{\ell}}\Longrightarrow \|s(z)-\ri \tau\|_{L^\ell}\leq \frac{C\ell}{|\Im[z]|}
\end{equation}

For \(w,z\in \bC_\tau\), define
\[
Y_\tau(w,z):=
-\int_\gamma \bigl(s(u)-\ri\tau\bigr)\,\rd u,
\]
where \(\gamma\subset \bC_\tau\) is any piecewise \(C^1\) path from \(z\) to \(w\).
Again, \(Y_\tau(w,z)\) is path-independent because \(s\) is holomorphic on
\(\bC_\tau\).

By Minkowski's inequality and \eqref{e:s-local-law-both-halfplanes},
\[
\|Y_\tau(w,z)\|_{L^\ell}
\le \int_\gamma \|s(u)-\ri\tau\|_{L^\ell}\,|\rd u|
\le C\sqrt \ell \int_\gamma \frac{|\rd u|}{|\Im [u]|}.
\]
Choose \(\gamma\) to be the hyperbolic geodesic (see e.g. \cite[Section 7]{beardon2012geometry}) from \(z\) to \(w\) in the half-plane
\(\bC_\tau\). 
Then the last integral is the hyperbolic distance, so
\begin{equation}\label{e:Y-hyp-bound}
\|Y_\tau(w,z)\|_{L^\ell}
\le C\sqrt \ell \cdot 2{\rm arsinh}\left(\frac{|w-z|}{2\sqrt{\Im[w]\Im[z]}}\right)
\leq
2C\sqrt \ell\,
\log\!\left(
1+\frac{|w-z|}{\sqrt{|\Im [w]|\,|\Im [z]|}}
\right).
\end{equation}

Now decompose the integral in \eqref{e:def-G-product} as
\[
-\int_{\gamma_{r,a}^{(\alpha)}} s(u)\,\rd u
=
-\ri\tau_\alpha\bigl(t_r^{(\alpha)}-s_a^{(\alpha)}\bigr)
+
Y_{\tau_\alpha}\!\bigl(t_r^{(\alpha)},s_a^{(\alpha)}\bigr).
\]
Summing over \((\alpha,r,a)\) and using \(p/q=\beta/2\), we obtain
\[
G=
-\ri\sum_{\alpha=1}^k \tau_\alpha
\left(
\sum_{r=1}^p t_r^{(\alpha)}
-\frac{\beta}{2}\sum_{a=1}^q s_a^{(\alpha)}
\right)
+\Xi=M+\Xi,
\]
where
\begin{equation}\label{e:def-Xi-product}
\Xi:=
\frac{1}{q}
\sum_{\alpha=1}^k\sum_{r=1}^p\sum_{a=1}^q
Y_{\tau_\alpha}\!\bigl(t_r^{(\alpha)},s_a^{(\alpha)}\bigr).
\end{equation}

Applying Minkowski again and using \eqref{e:Y-hyp-bound}, we get
\begin{align}\label{e:Z-Lm-bound}
\|\Xi\|_{L^\ell}
\le
\frac{1}{q}
\sum_{\alpha=1}^k\sum_{r=1}^p\sum_{a=1}^q
\left\|Y_{\tau_\alpha}\!\bigl(t_r^{(\alpha)},s_a^{(\alpha)}\bigr)\right\|_{L^\ell}
\le C\sqrt \ell\,L.
\end{align}
 In particular,
\begin{align}\label{e:themean}
|\bE[\Xi]|\le \|\Xi\|_{L^1}\le CL.
\end{align}

To prove \eqref{e:F-upper-general}, note that \(\Re[ \Xi]\) is a real
sub-Gaussian random variable with scale \(CL\) and mean bounded by $CL$, by \eqref{e:Z-Lm-bound} and \eqref{e:themean}. Hence
\[
\bE[e^{\Re[ \Xi]}]\le e^{C(L+L^2)}.
\]
Therefore
\[
|F|
\le e^{\Re [M]}\,\bE[e^{\Re [\Xi]}]
\le e^{\Re [M] + C(L+L^2)},
\]
which proves \eqref{e:F-upper-general}.

We now turn to the small-\(L\) regime. By Taylor's expansion
\[
\bE[e^\Xi]-1
=
\sum_{\ell=1}^\infty \frac{\bE[\Xi^\ell]}{\ell!}.
\]
Using \eqref{e:Z-Lm-bound},
\[
\left|\frac{\bE[\Xi^\ell]}{\ell!}\right|
\le \frac{\|\Xi\|_{L^\ell}^\ell}{\ell!}
\le \frac{(C\sqrt \ell\,L)^\ell}{\ell!}.
\]
By Stirling's bound \(\ell!\ge (\ell/e)^\ell\),
\[
\frac{(C\sqrt \ell\,L)^\ell}{\ell!}
\le \left(\frac{CeL}{\sqrt \ell}\right)^\ell
\le (C_1L)^\ell.
\]
Hence there exists \(L_0>0\) such that for all \(L\le L_0\),
\[
\left|\bE[e^\Xi]-1\right|
\le \sum_{\ell=1}^\infty (C_1L)^\ell
\le C_2L.
\]
Consequently,
\begin{equation}\label{e:F-factor-smallL}
|e^{-M}F -1|\le C_2L,
\end{equation}
and \eqref{e:logF-smallL-general} follows.
\end{proof}

\begin{proof}[Proof of \Cref{c:absolute_bound}]
The first inequality in \eqref{e:absolutebound} is the Cauchy--Schwarz inequality.

For the second inequality, note first that both sides of \eqref{e:absolutebound}
are unchanged if any individual \(t_r\) or \(s_a\) is replaced by its complex
conjugate. Indeed, on the left this is immediate since each \(x_j\in \bR\), so
the modulus of the principal-value product depends only on
\(\Re [t_r], |\Im [t_r]|, \Re [s_a], |\Im [s_a]|\). On the right, the quantities \(M\)
and \(L\) are defined precisely in terms of these data. Therefore, without loss
of generality, we may assume that
$
t_1,\dots,t_n,\ s_1,\dots,s_m\in \bC_+$.

Set
\[
X:=
\operatorname{P.V.}\prod_{j\in \bZ}
\frac{\prod_{r=1}^n (t_r-x_j)}
{\prod_{a=1}^m (s_a-x_j)^{\beta/2}}.
\]
Since \(x_j\in \bR\), we then have
\begin{equation}\label{e:modsq-product}
|X|^2
=
\operatorname{P.V.}\prod_{j\in \bZ}
\frac{\prod_{r=1}^n (t_r-x_j)(\overline{t_r}-x_j)}
{\prod_{a=1}^m (s_a-x_j)^{\beta/2}(\overline{s_a}-x_j)^{\beta/2}}.
\end{equation}

Next choose arbitrary partitions
\[
\{1,\dots,n\}=I_1\sqcup\cdots\sqcup I_k,
\qquad
\{1,\dots,m\}=A_1\sqcup\cdots\sqcup A_k,
\]
with \(|I_\alpha|=p\) and \(|A_\alpha|=q\) for every \(1\le \alpha\le k\).
Using these partitions, define \(2k\) blocks by
\begin{align*}
&t_r^{(\alpha)}:=t_{I_\alpha(r)},
\qquad
s_a^{(\alpha)}:=s_{A_\alpha(a)},
\qquad
\tau_\alpha:=+1,
\qquad 1\le \alpha\le k,\quad 1\leq r\leq p, \quad 1\leq a\leq q,\\
&t_r^{(k+\alpha)}:=\overline{t_{I_\alpha(r)}},
\qquad
s_a^{(k+\alpha)}:=\overline{s_{A_\alpha(a)}},
\qquad
\tau_{k+\alpha}:=-1,
\qquad 1\le \alpha\le k,\quad 1\leq r\leq p, \quad 1\leq a\leq q,
\end{align*}
where $I_\al(r)$ is the $r$-th element of $I_\al$, and $A_\al$ is the $a$-th element of $A_\al$.
Then \eqref{e:modsq-product} is exactly of the form covered by
\Cref{prop:pv-product-smallL}. Hence \eqref{e:F-upper-general} yields
\begin{equation}\label{e:modsq-bound}
\bE[|X|^2]
\le
\exp\{\Re [M_\ast] + C(L_\ast+L_\ast^2)\},
\end{equation}
where
\[
M_\ast
:=
-\ri\sum_{\alpha=1}^{2k}\tau_\alpha
\left(
\sum_{r=1}^p t_r^{(\alpha)}
-\frac{\beta}{2}\sum_{a=1}^q s_a^{(\alpha)}
\right),
\]
and
\[
L_\ast
:=
\sum_{\alpha=1}^{2k}\sum_{r=1}^p\sum_{a=1}^q
\log\!\left(
1+\frac{|t_r^{(\alpha)}-s_a^{(\alpha)}|}
{\sqrt{|\Im [t_r^{(\alpha)}]|\,|\Im [s_a^{(\alpha)}]|}}
\right).
\]

Recall $M$ and $L$ from \eqref{e:def-M-abs} and \eqref{e:def-L-abs}. Since all \(t_r,s_a\in \bC_+\), we have
\[
\Re [M_\ast]
=
2\sum_{r=1}^n \Im [t_r]
-\beta\sum_{a=1}^m \Im [s_a]
=
2M.
\]
Moreover,
\begin{align*}
L_\ast
&=
2\sum_{\alpha=1}^{k}\sum_{r=1}^p\sum_{a=1}^q
\log\!\left(
1+\frac{|t_r^{(\alpha)}-s_a^{(\alpha)}|}
{\sqrt{|\Im [t_r^{(\alpha)}]|\,|\Im [s_a^{(\alpha)}]|}}
\right) =
2\sum_{r=1}^n\sum_{a=1}^m
\log\!\left(
1+\frac{|t_r-s_a|}
{\sqrt{\Im [t_r]\Im [s_a]}}
\right) 
=
2L.
\end{align*}

Substituting these relations into \eqref{e:modsq-bound}, taking square root, and enlarging the constant
\(C\) if necessary, we obtain
\[
\bE[|X|^2]^{1/2}
\le
e^{M+C(L+L^2)},
\]
which is exactly the second inequality in \eqref{e:absolutebound}.

In the following we prove that \(F\) as in \eqref{e:defF1}  is holomorphic on
\[
\bC^n\times(\bC_+\cup\bC_-)^m.
\]
Since \((\bC_+\cup\bC_-)^m\) is a disjoint union of connected components, it
suffices to fix a sign vector
\[
\bm\eta=(\eta_1,\dots,\eta_m)\in\{\pm1\}^m
\]
and prove holomorphicity on
\[
\Omega_{\bm\eta}:=\bC^n\times \prod_{a=1}^m \bC_{\eta_a}.
\]

Fix \(s_1,\dots,s_m\in \prod_{a=1}^m\bC_{\eta_a}\). For each realization of
the point process, the principal-value product
\begin{equation}\label{e:Gexp}
\mathcal G(t_1,\dots,t_n;s_1,\dots,s_m)
:=
\operatorname{P.V.}
\prod_{j\in \bZ}
\frac{\prod_{r=1}^n(t_r-x_j)}
{\prod_{a=1}^m(s_a-x_j)^{\beta/2}}
\end{equation}
is entire in \(t_1,\dots,t_n\). Moreover, after fixing the branches of
\((s_a-x_j)^{\beta/2}\) on the half-plane \(\bC_{\eta_a}\), it is holomorphic
in each \(s_a\in\bC_{\eta_a}\).

We first prove a local \(L^1\) bound. By Cauchy--Schwarz,
\begin{equation}\label{e:ubb}
\bE\bigl[|\mathcal G(t_1,\dots,t_n;s_1,\dots,s_m)|\bigr]
\le
\bE\bigl[|\mathcal G(t_1,\dots,t_n;s_1,\dots,s_m)|^2\bigr]^{1/2}.
\end{equation}
Let \(K\subset \Omega_{\bm\eta}\) be compact. Choose \(R>0\) such that
\[
R\ge
\max_{1\le r\le n}\sup_{(t,s)\in K}|\Im[t_r]|+1.
\]
For \((t_1, \cdots, t_n,s_1,\cdots, s_m)\in K\), define
\[
\widehat t_r:=\Re[t_r]+\ri R,
\qquad 1\le r\le n.
\]
Since \(x_j\in\bR\), we have
$
|t_r-x_j|\le |\widehat t_r-x_j|,
$ for $j\in\bZ$.
Therefore, first for finite symmetric truncations of the principal-value
product and then by passing to the limit,
\[
|\mathcal G(t_1,\dots,t_n;s_1,\dots,s_m)|
\le
|\mathcal G(\widehat t_1,\dots,\widehat t_n;s_1,\dots,s_m)|.
\]
The variables \(\widehat t_1,\dots,\widehat t_n\) lie in the upper half-plane
and range over a compact subset of \((\bC_+)^n\), while
\((s_1,\dots,s_m)\) ranges over a compact subset of
\(\bC_{\eta_1}\times \cdots\times \bC_{\eta_m}\). Hence the absolute bound \eqref{e:absolutebound}
implies
\[
\sup_{(t,s)\in K}
\bE\bigl[
|\mathcal G(\widehat t_1,\dots,\widehat t_n;s_1,\dots,s_m)|^2
\bigr]^{1/2}
<\infty.
\]
Consequently,
\begin{equation}\label{e:local-L1-bound-G}
\sup_{(t,s)\in K}
\bE\bigl[
|\mathcal G(t_1,\dots,t_n;s_1,\dots,s_m)|
\bigr]
<\infty.
\end{equation}

We now prove separate holomorphicity. Fix all variables except \(t_r\), and let
\(\omega_r\) be any closed contour in the \(t_r\)-plane. The contour together
with the fixed values of the remaining variables is contained in some compact
subset of \(\Omega_{\bm\eta}\). By \eqref{e:local-L1-bound-G}, Fubini's theorem
allows us to interchange expectation and contour integration:
\[
\int_{\omega_r}
F(t_1,\dots,t_n;s_1,\dots,s_m)\,\rd t_r
=
\bE\left[
\int_{\omega_r}
\mathcal G(t_1,\dots,t_n;s_1,\dots,s_m)\,\rd t_r
\right].
\]
For each realization, \(t_r\mapsto \mathcal G(t_1,\dots,t_n;s_1,\dots,s_m)\)
is entire, so the inner contour integral vanishes. Hence
\[
\int_{\omega_r}
F(t_1,\dots,t_n;s_1,\dots,s_m)\,\rd t_r=0.
\]
By Morera's theorem, \(F\) is holomorphic in each \(t_r\).

The same argument applies to each \(s_a\), with contours lying inside the fixed
half-plane \(\bC_{\eta_a}\). Indeed, for each realization,
\(s_a\mapsto \mathcal G(t_1,\dots,t_n;s_1,\dots,s_m)\) is holomorphic on
\(\bC_{\eta_a}\), and the local bound \eqref{e:local-L1-bound-G} justifies the
same interchange of expectation and contour integration. Thus \(F\) is
separately holomorphic in all variables on \(\Omega_{\bm\eta}\).

Finally, \eqref{e:local-L1-bound-G} shows that \(F\) is locally bounded on
\(\Omega_{\bm\eta}\). By the Hartogs--Osgood theorem (see \cite[Theorem 1.2.5]{krantz2001function}), separate holomorphicity
together with local boundedness implies joint holomorphicity. Therefore \(F\)
is jointly holomorphic on \(\Omega_{\bm\eta}\). Since \(\bm\eta\) was arbitrary,
\(F\) is holomorphic on
$
\bC^n\times(\bC_+\cup\bC_-)^m$.
\end{proof}

\subsection{Edge exponential observables}
\begin{proposition}[Edge exponential observables]\label{p:edge_est}
Adopt \Cref{a:edge}, and assume
\[
	\frac{\beta}{2}=\frac{p}{q}\in \bQ_{>0}.
\]
Fix \(k\ge 1\), signs \(\tau_\alpha\in\{\pm1\}\), and points
\begin{equation}\label{e:domain-edge-exp}
	t_r^{(\alpha)},\,s_a^{(\alpha)}\in \bC_{\tau_\alpha},
	\qquad
	1\le r\le p,\quad 1\le a\le q,\quad 1\le \alpha\le k .
\end{equation}
For each triple \((\alpha,r,a)\), choose a piecewise \(C^1\) path
\[
	\gamma_{r,a}^{(\alpha)}\subset \bC_{\tau_\alpha}
\]
from \(s_a^{(\alpha)}\) to \(t_r^{(\alpha)}\), and define
\begin{equation}\label{e:def-G-product-edge}
	G
	:=
	-\frac{1}{q}
	\sum_{\alpha=1}^k
	\sum_{r=1}^p
	\sum_{a=1}^q
	\int_{\gamma_{r,a}^{(\alpha)}} s(u)\,\rd u .
\end{equation}
Since \(s\) is holomorphic on each half-plane, \(G\) is independent of the
choice of paths.
From the representation \eqref{e:normalized}, \(e^G\)
admits the renormalized product representation
\begin{equation}
	e^G
	=
	\prod_{\alpha=1}^k
	e^{
		c_0
		\left(
			\sum_{r=1}^p t_r^{(\alpha)}
			-\frac{\beta}{2}\sum_{a=1}^q s_a^{(\alpha)}
		\right)
	}
	\times
	\prod_{j\ge 1}
	\prod_{\alpha=1}^k
	\frac{
		\prod_{r=1}^p
		(t_r^{(\alpha)}-x_j)
		e^{t_r^{(\alpha)}/\fa_j}
	}{
		\prod_{a=1}^q
		\left(
			(s_a^{(\alpha)}-x_j)
			e^{s_a^{(\alpha)}/\fa_j}
		\right)^{\beta/2}
	},
\end{equation}
where
$c_0:=\Ai'(0)/\Ai(0)$, and  the powers \((s^{(\alpha)}_a-x_j)^{\beta/2}\) are defined using the principal
branch on \(\bC\setminus\bR_{<0}\), as specified in
\eqref{e:principle_branch}.
Moreover, this random variable is integrable. We set
\begin{equation}\label{e:defF0-edge}
	F:=\bE[e^G].
\end{equation}
Then \(F\) is holomorphic on the domain specified in
\eqref{e:domain-edge-exp}.

Define
\begin{align}
	M
	&:=
	-\frac{2}{3}
	\sum_{\alpha=1}^k
	\left(
		\sum_{r=1}^p
		\left(t_r^{(\alpha)}\right)^{3/2}
		-\frac{\beta}{2}
		\sum_{a=1}^q
		\left(s_a^{(\alpha)}\right)^{3/2}
	\right),
	\label{e:def-M-product-edge}
	\\
	L
	&:=
	\sum_{\alpha=1}^k
	\sum_{r=1}^p
	\sum_{a=1}^q
	\log\!\left(
		1+
		\frac{
			|t_r^{(\alpha)}-s_a^{(\alpha)}|
		}{
			\sqrt{
				|\Im[t_r^{(\alpha)}]|\,
				|\Im[s_a^{(\alpha)}]|
			}
		}
	\right).
	\label{e:def-L-product-edge}
\end{align}
Here the powers \(w^{3/2}\) are taken with the principal branch.

If, in addition,
\begin{equation}\label{e:firstqq}
	|\arg t_r^{(\alpha)}|,\,
	|\arg s_a^{(\alpha)}|
	\in
	(\pi/4,\pi),
	\qquad
	1\le r\le p,\quad 1\le a\le q,\quad 1\le \alpha\le k,
\end{equation}
then there exists a deterministic constant \(L_0>0\), depending only on
\(p,q,k\), such that, whenever \(L\le L_0\),
\begin{equation}\label{e:logF-smallL-edge}
	F
	=
	\exp\{M+\OO(L)\},
\end{equation}
where the implicit constant depends only on \(p,q,k\).
\end{proposition}

\begin{proof}
The proof is parallel to that of the bulk estimate, so we only indicate the
changes.

Write
\[
G=M+\Xi,
\]
where $M$ is from \eqref{e:def-M-product-edge} and 
\[
\Xi:=
-\frac{1}{q}
\sum_{\alpha=1}^k\sum_{r=1}^p\sum_{a=1}^q
\int_{\gamma_{r,a}^{(\alpha)}}
\left(s(u)-\sqrt{u}\right)\rd u.
\]

We recall the edge concentration estimate \eqref{eqn:edge}: for any
\(w=E+\ri\eta\) with \(\eta>0\), and for every \(q\ge 1\),
\begin{equation}\label{e:edge2}
\mathbb{E}\left[\left|s(w)-\sqrt w\right|^{2q}\right]
\le
\frac{(Cq)^q}{\eta^{2q}}
\frac{|w|^q}{(\Im[\sqrt{w}])^{2q}}.
\end{equation}
We also note that
$
s(\overline w)=\overline{s(w)}$ and
$\sqrt{\overline w}=\overline{\sqrt w}$.
Therefore, \(s(w)-\sqrt w\) has a uniform sub-Gaussian tail on every compact
subset of \(\bC\setminus\bR\). In particular, \(e^G\) is integrable. And \(F\), as defined in \eqref{e:defF0-edge}, is holomorphic on the domain specified in
\eqref{e:domain-edge-exp}.

When \(\arg w\in(\pi/4,\pi)\), we have
\[
\Im[\sqrt w]\asymp |w|^{1/2},
\]
and hence \eqref{e:edge2} reduces to
\begin{equation}\label{e:edge3}
\mathbb{E}\left[\left|s(w)-\sqrt w\right|^{2q}\right]
\le
\frac{(Cq)^q}{\eta^{2q}}.
\end{equation}
This is the same form as the bulk concentration estimate \eqref{eqn:LocalLaw}.

Thus, under the assumption \eqref{e:firstqq}, using \eqref{e:edge3} as input, noticing that the region $\arg w\in (\pi/4,\pi)$ is hyperbolically geodesically convex, 
exactly as in the proof of \Cref{prop:pv-product-smallL},  this gives, for every
integer \(\ell\ge 1\),
\[
\|\Xi\|_{L^\ell}\le C\sqrt \ell\,L.
\]
Therefore
\[
F=\bE[e^G]=e^{M}\bE[e^\Xi].
\]
As in the bulk case, the sub-Gaussian bound implies
\[
\bigl|\bE[e^\Xi]-1\bigr|\le C L
\]
provided \(L\le L_0\), for some sufficiently small deterministic constant
\(L_0>0\). Hence
\[
F=e^{M+\OO(L)},
\]
which is exactly \eqref{e:logF-smallL-edge}.
\end{proof}

\section{Linear differential equations}
\label{s:linearPDE}

In this section, we prove that the exponential observables introduced in
\Cref{s:exp_ob} satisfy linear systems of second-order differential equations.
These systems are the bulk and edge deformed CMS equations corresponding to
the two sets of variables \(t\) and \(s\).  They are obtained from the loop
equation hierarchy for the Nevanlinna function.  While this hierarchy
is nonlinear in the Nevanlinna function itself, its action on exponential
observables linearizes, yielding the deformed CMS equations.

\subsection{Bulk deformed Calogero–Moser–Sutherland operators}

\begin{proposition}\label{p:bulk_equation}
Adopt \Cref{a:bulk}, assume that \(\beta/2\in\bQ_{>0}\), and let
\[
n=\frac{\beta m}{2}.
\]
Define the exponential observable
\begin{equation}\label{e:defF}
F(t_1,\dots,t_n;s_1,\dots,s_m)
:=
\bE\left[
\operatorname{P.V.}
\prod_{j\in \bZ}
\frac{\prod_{r=1}^n (t_r-x_j)}
{\prod_{a=1}^m (s_a-x_j)^{\beta/2}}
\right],
\end{equation}
where the expectation is taken with respect to the underlying random point
process, and the powers \((s_a-x_j)^{\beta/2}\) are defined using the principal
branch on \(\bC\setminus\bR_{<0}\), as specified in
\eqref{e:principle_branch}.

Assume that
$
	t_1,\dots,t_n,s_1,\dots,s_m\in \bC_+\cup\bC_-
$,
satisfy the half-plane balance conditions
\begin{align}\label{e:balance-halfplanes}
\sum_{r=1}^n \bm 1(t_r\in \bC_+)
-\frac{\beta}{2}\sum_{a=1}^m \bm 1(s_a\in \bC_+)
=0,\quad 
\sum_{r=1}^n \bm 1(t_r\in \bC_-)
-\frac{\beta}{2}\sum_{a=1}^m \bm 1(s_a\in \bC_-)
=0 .
\end{align}
Then \(F\) satisfies the bulk deformed CMS equations
\begin{align}\label{e:eq1}
\left(
\partial_{t_i}^2+1
+\frac{2}{\beta}\sum_{j\neq i}
\frac{\partial_{t_i}-\partial_{t_j}}{t_i-t_j}
-\sum_{a=1}^m
\frac{\partial_{t_i}+(2/\beta)\partial_{s_a}}{t_i-s_a}
\right)F=0,
\qquad 1\le i\le n,
\end{align}
and
\begin{align}\label{e:eq2}
\left(
\partial_{s_a}^2+\frac{\beta^2}{4}
+\frac{\beta}{2}\sum_{b\neq a}
\frac{\partial_{s_a}-\partial_{s_b}}{s_a-s_b}
-\sum_{i=1}^n
\frac{\partial_{s_a}+(\beta/2)\partial_{t_i}}{s_a-t_i}
\right)F=0,
\qquad 1\le a\le m.
\end{align}

Moreover, by \Cref{c:absolute_bound}, the observable \(F\) defined in
\eqref{e:defF} is holomorphic in \(t_1,\dots,t_n\) on \(\bC^n\), and
holomorphic in each \(s_a\) on its chosen half-plane. The equations \eqref{e:eq1}--\eqref{e:eq2} continue to hold under this analytic
continuation, as meromorphic identities in the variables.
\end{proposition}

\begin{remark}
When \(\beta=2\), the two systems \eqref{e:eq1} and \eqref{e:eq2} coincide and reduce to
\begin{align}\begin{split}\label{e:bulkbeta=2}
\left(
\del_{t_i}^2+1
+\sum_{j\neq i}\frac{\del_{t_i}-\del_{t_j}}{t_i-t_j}
-\sum_{j=1}^m\frac{\del_{t_i}+\del_{s_j}}{t_i-s_j}
\right)F=0,
\qquad 1\le i\le n,\\
\left(
\del_{s_i}^2+1
+\sum_{j\neq i}\frac{\del_{s_i}-\del_{s_j}}{s_i-s_j}
-\sum_{j=1}^m\frac{\del_{s_i}+\del_{t_j}}{s_i-t_j}
\right)F
=0,
\qquad 1\le i\le n.
\end{split}\end{align}
\end{remark}

\begin{proof}[Proof of \Cref{p:bulk_equation}]
We divide the proof into several steps.

\medskip
\noindent
\emph{Step 1. Integrated loop equation.}
For \(u\in \bC_+\cup\bC_-\), define
\[
\varepsilon(u)=
\begin{cases}
+1,& u\in \bC_+,\\
-1,& u\in \bC_-,
\end{cases}
\qquad
S(u):=\int_{\varepsilon(u)\ri}^{\,u}s(\zeta)\,\rd\zeta .
\]
Then \(S\) is holomorphic on each half-plane and satisfies \(S'(u)=s(u)\).

Throughout the proof, all difference quotients are interpreted by their
holomorphic extensions whenever their two arguments coincide. In particular,
\[
\left.
\frac{S'(w_1)-S'(w_j)}{w_1-w_j}
\right|_{w_j=w_1}
=
S''(w_1),
\qquad
\left.
\partial_{w_j}
\frac{S'(w_1)-S'(w_j)}{w_1-w_j}
\right|_{w_j=w_1}
=
\frac12 S'''(w_1).
\]

We start from the bulk loop equation \eqref{e:loopeq2-bulk}. For
\(w_1,\dots,w_p\in\bC_+\cup\bC_-\), it gives
\begin{align}\label{e:loopeq}
&\bE\left[
\left(
\frac{2-\beta}{\beta}S''(w_1)
+(S'(w_1))^2
+1
\right)
\prod_{j=2}^p S'(w_j)
\right]
+\frac2\beta
\bE\left[
\sum_{j=2}^p
\partial_{w_j}
\frac{S'(w_1)-S'(w_j)}{w_1-w_j}
\prod_{\substack{\ell=2\\ \ell\ne j}}^pS'(w_\ell)
\right]
=0.
\end{align}
Because the integrand in the second term is holomorphic in \(w_j\), including
at \(w_j=w_1\) under the convention above, we may integrate
\eqref{e:loopeq} in each variable \(w_j\), \(2\le j\le p\), from
\(\varepsilon(w_j)\ri\) to \(w_j\). We obtain
\begin{align}\begin{split}\label{e:integrated-loop}
&\bE\left[
\left(
\frac{2-\beta}{\beta}S''(w_1)
+(S'(w_1))^2
+1
\right)
\prod_{j=2}^p S(w_j)
\right]
\\
&\qquad
+\frac2\beta
\bE\left[
\sum_{j=2}^p
\left(
\frac{S'(w_1)-S'(w_j)}{w_1-w_j}
-
\frac{S'(w_1)-S'(\varepsilon(w_j)\ri)}
{w_1-\varepsilon(w_j)\ri}
\right)
\prod_{\substack{\ell=2\\ \ell\ne j}}^pS(w_\ell)
\right]
=0.
\end{split}\end{align}

\medskip
\noindent
\emph{Step 2. Exponential test functions.}
Choose
\[
O_\ell(u)=e^{-a_\ell u},
\qquad
a_\ell\in\{1,-\beta/2\},
\qquad
1\le \ell\le p,
\]
and set
\[
\widetilde O_\ell(u):=e^{|a_\ell|u},
\qquad
u\ge0.
\]
Then
\[
|O_\ell(S(w_\ell))|+|O_\ell'(S(w_\ell))|
\le
C\,\widetilde O_\ell(|S(w_\ell)|),
\]
where \(C\) depends only on \(\beta\).

Using \eqref{e:sbound}, we have
\[
\left|
\frac{2-\beta}{\beta}S''(w_1)
+(S'(w_1))^2
+1
\right|
\le
C\left(
\frac{\Im[s(w_1)]}{\Im[w_1]}
+|s(w_1)|^2+1
\right).
\]
Hence
\begin{align}\begin{split}\label{e:loop-int-1}
&\bE\left[
\left|
\frac{2-\beta}{\beta}S''(w_1)
+(S'(w_1))^2
+1
\right|
\prod_{\ell=1}^p \widetilde O_\ell(|S(w_\ell)|)
\right]
\\
&\qquad\le
C\,\bE\left[
\left(
\frac{\Im[s(w_1)]}{\Im[w_1]}
+|s(w_1)|^2+1
\right)
\prod_{\ell=1}^p \widetilde O_\ell(|S(w_\ell)|)
\right]
<\infty.
\end{split}\end{align}
Here the last finiteness follows from \eqref{eqn:LocalLaw}, which gives
sub-Gaussian tails for \(s(w_1)\), together with the corresponding exponential
integrability of the path integrals \(S(w_\ell)\).

Similarly, \eqref{e:sbound} and \eqref{e:fest} give
\begin{align}
\begin{split}\label{e:loop-int-2}
&\bE\left[
\left|
S''(w_1)
-
\frac{S'(w_1)-S'(\varepsilon(w_1)\ri)}
{w_1-\varepsilon(w_1)\ri}
\right|
\prod_{\ell=1}^p \widetilde O_\ell(|S(w_\ell)|)
\right]
<\infty,
\\
&\bE\left[
\left|
\frac{S'(w_1)-S'(w_j)}{w_1-w_j}
-
\frac{S'(w_1)-S'(\varepsilon(w_j)\ri)}
{w_1-\varepsilon(w_j)\ri}
\right|
\prod_{\ell=1}^p \widetilde O_\ell(|S(w_\ell)|)
\right]
<\infty,
\end{split}
\end{align}
for every \(2\le j\le p\). The second estimate remains valid when
\(w_j=w_1\), in which case it reduces to the first estimate.

To pass from \eqref{e:integrated-loop} to general test functions, we first
consider polynomials in \(S(w_1),\dots,S(w_p)\). We apply
\eqref{e:integrated-loop} at higher levels of the loop hierarchy, allowing
additional variables to coincide with any of \(w_1,\dots,w_p\), and then
group the equal variables. This is legitimate because the difference
quotients have been extended holomorphically to the diagonals. A repeated
variable at \(w_1\) produces the factor
\[
S''(w_1)
-
\frac{S'(w_1)-S'(\varepsilon(w_1)\ri)}
{w_1-\varepsilon(w_1)\ri}
\]
multiplying the derivative of the polynomial with respect to its first
argument. Similarly, a repeated variable at \(w_j\), \(j\ge2\), produces
\[
\frac{S'(w_1)-S'(w_j)}{w_1-w_j}
-
\frac{S'(w_1)-S'(\varepsilon(w_j)\ri)}
{w_1-\varepsilon(w_j)\ri}
\]
multiplying the derivative with respect to its \(j\)-th argument.

Thus \eqref{e:integrated-loop} extends to polynomial test functions of
\(S(w_1),\dots,S(w_p)\). Using the power-series expansions of the
exponentials and the integrability estimates
\eqref{e:loop-int-1}--\eqref{e:loop-int-2}, we may then pass to the limit by
dominated convergence. This gives
\begin{align}\begin{split}
0={}&
\bE\left[
\left(
\frac{2-\beta}{\beta}S''(w_1)
+(S'(w_1))^2
+1
\right)
\prod_{\ell=1}^p O_\ell(S(w_\ell))
\right]
\\
&+\frac2\beta
\bE\left[
\left(
S''(w_1)
-
\frac{S'(w_1)-S'(\varepsilon(w_1)\ri)}
{w_1-\varepsilon(w_1)\ri}
\right)
O_1'(S(w_1))
\prod_{\ell=2}^p O_\ell(S(w_\ell))
\right]
\\
&+\frac2\beta
\sum_{j=2}^p
\bE\left[
\left(
\frac{S'(w_1)-S'(w_j)}{w_1-w_j}
-
\frac{S'(w_1)-S'(\varepsilon(w_j)\ri)}
{w_1-\varepsilon(w_j)\ri}
\right)
O_1(S(w_1))O_j'(S(w_j))
\prod_{\ell\ne 1,j}O_\ell(S(w_\ell))
\right].
\end{split}\end{align}
Since \(O_\ell'(u)=-a_\ell O_\ell(u)\), this becomes
\begin{align}\begin{split}\label{e:loopeq3}
0={}&
\bE\left[
\left(
\frac{2-\beta}{\beta}S''(w_1)
+(S'(w_1))^2
+1
\right)
\prod_{\ell=1}^p O_\ell(S(w_\ell))
\right]
\\
&-\frac{2a_1}{\beta}
\bE\left[
\left(
S''(w_1)
-
\frac{S'(w_1)-S'(\varepsilon(w_1)\ri)}
{w_1-\varepsilon(w_1)\ri}
\right)
\prod_{\ell=1}^p O_\ell(S(w_\ell))
\right]
\\
&-\frac2\beta
\sum_{j=2}^p a_j
\bE\left[
\left(
\frac{S'(w_1)-S'(w_j)}{w_1-w_j}
-
\frac{S'(w_1)-S'(\varepsilon(w_j)\ri)}
{w_1-\varepsilon(w_j)\ri}
\right)
\prod_{\ell=1}^p O_\ell(S(w_\ell))
\right].
\end{split}\end{align}

\medskip
\noindent
\emph{Step 3. Cancellation of base-point terms.}
Assume that
\begin{equation}\label{e:separate-balance}
\sum_{\ell:\,\varepsilon(w_\ell)=+1} a_\ell=0,
\qquad
\sum_{\ell:\,\varepsilon(w_\ell)=-1} a_\ell=0.
\end{equation}
Then the terms involving \(S'(\ri)\) and \(S'(-\ri)\) cancel in
\eqref{e:loopeq3}. Hence \eqref{e:loopeq3} reduces to
\begin{align}\label{e:loopeq4}
0={}&
\bE\left[
\left(
\frac{2-\beta-2a_1}{\beta}S''(w_1)
+(S'(w_1))^2
+1
\right)
\prod_{\ell=1}^p O_\ell(S(w_\ell))
\right]
\notag\\
&-\frac2\beta
\sum_{j=2}^p a_j
\bE\left[
\frac{S'(w_1)-S'(w_j)}{w_1-w_j}
\prod_{\ell=1}^p O_\ell(S(w_\ell))
\right].
\end{align}

Since \(a_1\in\{1,-\beta/2\}\), we have
\[
\frac{2-\beta-2a_1}{\beta}
=
-\frac{1}{a_1}.
\]
Therefore
\[
\left(
-\frac{1}{a_1}S''(w_1)
+(S'(w_1))^2
+1
\right)e^{-a_1S(w_1)}
=
\left(
\frac{\partial_{w_1}^2}{a_1^2}+1
\right)e^{-a_1S(w_1)}.
\]
Moreover,
\[
-a_j
\frac{S'(w_1)-S'(w_j)}{w_1-w_j}
\prod_{\ell=1}^p O_\ell(S(w_\ell))
=
\frac{(a_j/a_1)\partial_{w_1}-\partial_{w_j}}{w_1-w_j}
\left(
\prod_{\ell=1}^p O_\ell(S(w_\ell))
\right).
\]
Substituting these identities into \eqref{e:loopeq4}, we obtain the master
identity
\begin{align}\label{e:master}
&\bE\left[
\left(
\frac{\partial_{w_1}^2}{a_1^2}
+1
\right)
\prod_{\ell=1}^p O_\ell(S(w_\ell))
\right]
+\frac2\beta
\sum_{j=2}^p
\bE\left[
\frac{(a_j/a_1)\partial_{w_1}-\partial_{w_j}}{w_1-w_j}
\left(
\prod_{\ell=1}^p O_\ell(S(w_\ell))
\right)
\right]
=0.
\end{align}

\medskip
\noindent
\emph{Step 4. Derivation of \eqref{e:eq1} and \eqref{e:eq2}}
By the representation \eqref{e:szrep}
\[
e^{-a_\ell S(w_\ell)}
=
\prod_{j\in \bZ}
\frac{(w_\ell-x_j)^{a_\ell}}
{(\varepsilon(w_\ell)\ri-x_j)^{a_\ell}},
\]
the balance condition \eqref{e:separate-balance} implies that all base-point
factors cancel. Therefore
\begin{equation}\label{e:observable-product}
\prod_{\ell=1}^p O_\ell(S(w_\ell))
=
\prod_{j\in \bZ}\prod_{\ell=1}^p (w_\ell-x_j)^{a_\ell}.
\end{equation}

To derive \eqref{e:eq1}, take \(w_1=t_i\) and \(a_1=1\). Let the remaining
variables be the other \(t\)-variables with exponent \(1\), and the
\(s\)-variables with exponent \(-\beta/2\). Then
\eqref{e:separate-balance} is exactly \eqref{e:balance-halfplanes}, and
\eqref{e:observable-product} becomes
\[
\prod_{j\in \bZ}
\frac{\prod_{r=1}^n(t_r-x_j)}
{\prod_{a=1}^m(s_a-x_j)^{\beta/2}}.
\]
We obtain
\eqref{e:eq1} from \eqref{e:master}.

Similarly, to derive \eqref{e:eq2}, take \(w_1=s_i\) and
\(a_1=-\beta/2\), with the remaining variables again given by the other
\(t\)- and \(s\)-variables. Multiplying the resulting \eqref{e:master} by
\(\beta^2/4\),  this gives \eqref{e:eq2}.

\medskip
\noindent
\emph{Step 5. Holomorphic continuation in the \(t\)-variables.}

The equations \eqref{e:eq1} and \eqref{e:eq2} hold on the nonempty open domain
where the half-plane balance condition \eqref{e:balance-halfplanes} is
satisfied. Since both sides are meromorphic in \(t_1,\dots,t_n\), and by \Cref{c:absolute_bound}, \(F\) is
entire in these variables, the identities extend by analytic continuation to
all \(t_1,\dots,t_n\in\bC\), away from the collision hyperplanes where the
coefficients have poles. Equivalently, they hold as meromorphic identities in
the \(t\)-variables. This completes the proof.
\end{proof}

\subsection{Edge deformed Calogero–Moser–Sutherland operators}

\begin{proposition}\label{p:edge_equation}
Adopt \Cref{a:edge}, assume that \(\beta/2\in\bQ_{>0}\), and let
\[
n=\frac{\beta m}{2}.
\]
Define the exponential observable
\begin{equation}\label{e:defFedge}
	F(t_1,\dots,t_n;s_1,\dots,s_m)
	=
	e^{
		c_0
		\left(
			\sum_{r=1}^n t_r
			-\frac{\beta}{2}\sum_{a=1}^m s_a
		\right)
	}
	\times
	\bE\left[\prod_{j\ge 1}
	\frac{
		\prod_{r=1}^n
		(t_r-x_j)
		e^{t_r/\fa_j}
	}{
		\prod_{a=1}^m
		\left(
			(s_a-x_j)
			e^{s_a/\fa_j}
		\right)^{\beta/2}
	}\right],
\end{equation}
where $c_0:=\Ai'(0)/\Ai(0)$, and  the powers \((s^{(\alpha)}_a-x_j)^{\beta/2}\) are defined using the principal
branch on \(\bC\setminus\bR_{<0}\), as specified in
\eqref{e:principle_branch}.

Assume that \(t_1,\dots,t_n,s_1,\dots,s_m\in \bC_+\cup \bC_-\) satisfy
\begin{align}\label{e:edge-balance-halfplanes}
\sum_{r=1}^n \bm 1(t_r\in \bC_+)-\frac{\beta}{2}\sum_{a=1}^m \bm 1(s_a\in \bC_+)=0,
\qquad
\sum_{r=1}^n \bm 1(t_r\in \bC_-)-\frac{\beta}{2}\sum_{a=1}^m \bm 1(s_a\in \bC_-)=0.
\end{align}
Then \(F\) satisfies
\begin{align}\label{e:edge_eq1}
\left(
\del_{t_i}^2-t_i
+\frac{2}{\beta}\sum_{j\neq i}\frac{\del_{t_i}-\del_{t_j}}{t_i-t_j}
-\sum_{j=1}^m \frac{\del_{t_i}+(2/\beta)\del_{s_j}}{t_i-s_j}
\right)F=0,
\qquad 1\le i\le n,
\end{align}
and
\begin{align}\label{e:edge_eq2}
\left(
\del_{s_i}^2-\frac{\beta^2}{4}s_i
+\frac{\beta}{2}\sum_{j\neq i}\frac{\del_{s_i}-\del_{s_j}}{s_i-s_j}
-\sum_{j=1}^n \frac{\del_{s_i}+(\beta/2)\del_{t_j}}{s_i-t_j}
\right)F=0,
\qquad 1\le i\le m.
\end{align}
\end{proposition}

\begin{remark}
When \(\beta=2\), the two systems \eqref{e:edge_eq1} and \eqref{e:edge_eq2} coincide, and reduce to
\begin{align}\begin{split}\label{e:edgebeta=2}
\left(
\del_{t_i}^2-t_i
+\sum_{j\neq i}\frac{\del_{t_i}-\del_{t_j}}{t_i-t_j}
-\sum_{j=1}^m\frac{\del_{t_i}+\del_{s_j}}{t_i-s_j}
\right)F
=0,
\qquad 1\le i\le n,\\
\left(
\del_{s_i}^2-s_i
+\sum_{j\neq i}\frac{\del_{s_i}-\del_{s_j}}{s_i-s_j}
-\sum_{j=1}^m\frac{\del_{s_i}+\del_{t_j}}{s_i-t_j}
\right)F
=0,
\qquad 1\le i\le n.
\end{split}\end{align}
\end{remark}

\begin{proof}[Proof of \Cref{p:edge_equation}]
For \(u\in \bC_+\cup\bC_-\), define
\[
\varepsilon(u)=
\begin{cases}
+1,& u\in \bC_+,\\
-1,& u\in \bC_-,
\end{cases}
\qquad
S(u):=\int_{\varepsilon(u)\ri}^{\,u} s(\zeta)\,\rd\zeta .
\]
Then \(S\) is holomorphic on each half-plane and satisfies \(S'(u)=s(u)\).

Assume that
\begin{equation}\label{e:edge-separate-balance}
\sum_{\ell:\,\varepsilon(w_\ell)=+1} a_\ell=0,
\qquad
\sum_{\ell:\,\varepsilon(w_\ell)=-1} a_\ell=0.
\end{equation}
Take
\[
O_\ell(u)=e^{-a_\ell u},
\qquad
a_\ell\in\{1,-\beta/2\}.
\]
Exactly as in the bulk case, using the edge loop equation in place of the bulk
loop equation,  this gives the master identity
\begin{align}\label{e:edge-master}
\bE\left[
\left(\frac{\partial_{w_1}^2}{a_1^2}-w_1\right)
\prod_{\ell=1}^p O_\ell(S(w_\ell))
\right]
+\frac{2}{\beta}\sum_{j=2}^p
\bE\left[
\frac{(a_j/a_1)\partial_{w_1}-\partial_{w_j}}{w_1-w_j}
\left(\prod_{\ell=1}^p O_\ell(S(w_\ell))\right)
\right]
=0.
\end{align}

Next we identify the observable. In the edge case, the natural analogue of the
characteristic product is the Airy-regularized product
\begin{equation}\label{e:defGammaEdge}
\Gamma_{\bmx}(z)
:=
e^{c_0 z}\prod_{j\ge1}(z-x_j)e^{z/\fa_j},
\qquad
c_0:=\frac{\Ai'(0)}{\Ai(0)}.
\end{equation}
We emphasize that \eqref{e:defGammaEdge} should be understood as a formal
regularized product; the ratios appearing in \eqref{e:defFedge}, however, are
well defined.

From the representation \eqref{e:normalized}, we formally have
\[
\partial_u\log \Gamma_{\bmx}(u)=-s(u)=-S'(u).
\]
Therefore,
\[
e^{-S(u)}
=
\frac{\Gamma_{\bmx}(u)}
{\Gamma_{\bmx}(\varepsilon(u)\ri)}.
\]
Hence
\[
\prod_{\ell=1}^p O_\ell(S(w_\ell))
=
\prod_{\ell=1}^p
\left(
\frac{\Gamma_{\bmx}(w_\ell)}
{\Gamma_{\bmx}(\varepsilon(w_\ell)\ri)}
\right)^{a_\ell}.
\]
Under the balance condition \eqref{e:edge-separate-balance}, the base-point
factors cancel, and so
\begin{equation}\label{e:edge-observable-product}
\prod_{\ell=1}^p O_\ell(S(w_\ell))
=
\prod_{\ell=1}^p \Gamma_{\bmx}(w_\ell)^{a_\ell}.
\end{equation}

To derive \eqref{e:edge_eq1}, take \(w_1=t_i\) and \(a_1=1\). Let the remaining
variables be the other \(t\)-variables with exponent \(1\), and the
\(s\)-variables with exponent \(-\beta/2\). Then
\eqref{e:edge-separate-balance} is exactly \eqref{e:edge-balance-halfplanes},
and \eqref{e:edge-observable-product} becomes
\[
\exp\left\{
	c_0
	\left(
		\sum_{r=1}^n t_r
		-\frac{\beta}{2}\sum_{a=1}^m s_a
	\right)
\right\}
\prod_{j\ge 1}
\frac{
	\prod_{r=1}^n
	(t_r-x_j)e^{t_r/\fa_j}
}{
	\prod_{a=1}^m
	\left(
		(s_a-x_j)e^{s_a/\fa_j}
	\right)^{\beta/2}
}.
\]
We obtain \eqref{e:edge_eq1} from \eqref{e:edge-master}.

Similarly, to derive \eqref{e:edge_eq2}, take \(w_1=s_i\) and
\(a_1=-\beta/2\), with the remaining variables again given by the other
\(t\)- and \(s\)-variables. Multiplying the resulting identity by
\(\beta^2/4\),  this gives \eqref{e:edge_eq2}.
\end{proof}

\section{Unique characterization}
\label{s:characterization}
In this section, we prove our main results \Cref{t:bulk-characterization} and \Cref{t:edge-characterization}.
\subsection{Solutions of the bulk deformed CMS operators}

In this section we discuss solutions of the bulk deformed CMS system
\eqref{e:eq1}--\eqref{e:eq2}. These are second-order linear differential
equations in the \(n\) variables \(t_1,\dots,t_n\) and the \(m\) variables
\(s_1,\dots,s_m\). We construct sectorial solutions indexed by sign patterns
\[
(\bm\epsilon,\bm\eta)
=
(\epsilon_1,\dots,\epsilon_n,\eta_1,\dots,\eta_m)
\in \{\pm1\}^n\times \{\pm1\}^m.
\]
The solution associated with \((\bm\epsilon,\bm\eta)\) is characterized by the
leading exponential behavior
\[
\exp\!\left(
\ri\sum_{i=1}^n \epsilon_i t_i
+
\ri\frac{\beta}{2}\sum_{a=1}^m \eta_a s_a
\right).
\]

We first consider the special case \(\beta=2\), where the bulk deformed CMS
operators are given explicitly in \eqref{e:bulkbeta=2}. In this case, one can
write down the solutions in closed form. The proof is a direct verification and
is postponed to \Cref{s:b2solution}.

\begin{theorem}[Explicit solutions for \(\beta=2\)]
\label{thm:beta2-explicit}
For each sign pattern
\[
(\bm\epsilon,\bm\eta)
=
(\epsilon_1,\dots,\epsilon_n,\eta_1,\dots,\eta_m)
\in \{\pm1\}^n\times\{\pm1\}^m,
\]
write
\[
(\bmt,\bms)=(t_1,\dots,t_n,s_1,\dots,s_m).
\]
Define
\begin{align}\begin{split}
\label{eq:beta2-explicit-F}
F_{\bm\epsilon,\bm\eta}(\bmt,\bms)
&=
\exp\!\left(
\ri\sum_{i=1}^n \epsilon_i t_i
+
\ri\sum_{a=1}^m \eta_a s_a
\right)
\prod_{1\le i<j\le n}
(t_i-t_j)^{(\epsilon_i\epsilon_j-1)/2}
\\
&\qquad\times
\prod_{1\le a<b\le m}
(s_a-s_b)^{(\eta_a\eta_b-1)/2}
\prod_{i=1}^n
\prod_{a=1}^m
(t_i-s_a)^{(1+\epsilon_i\eta_a)/2}.
\end{split}\end{align}
Then \(F_{\bm\epsilon,\bm\eta}\) solves the \(\beta=2\) system
\eqref{e:bulkbeta=2}. Moreover, the \(2^{n+m}\) solutions obtained in this way
are linearly independent.
\end{theorem}

We now return to general \(\beta>0\). Fix a sign pattern
\[
(\bm\epsilon,\bm\eta)
=
(\epsilon_1,\dots,\epsilon_n,\eta_1,\dots,\eta_m)
\in \{\pm1\}^n\times\{\pm1\}^m.
\]
For \(1\le i<j\le n\), \(1\le a<b\le m\), and
\(1\le i\le n\), \(1\le a\le m\), define
\begin{equation}
\label{eq:alpha-gamma-rho-clean}
\alpha_{ij}:=\frac{\epsilon_i\epsilon_j-1}{\beta},
\qquad
\gamma_{ab}:=\frac{\beta}{4}(\eta_a\eta_b-1),
\qquad
\rho_{ia}:=\frac{1+\epsilon_i\eta_a}{2}.
\end{equation}
Equivalently,
\begin{align}\label{eq:alpha-gamma-rho2}
\alpha_{ij}=
\begin{cases}
0,& \epsilon_i=\epsilon_j,\\[1mm]
-2/\beta,& \epsilon_i\neq \epsilon_j,
\end{cases}
\qquad
\gamma_{ab}=
\begin{cases}
0,& \eta_a=\eta_b,\\[1mm]
-\beta/2,& \eta_a\neq \eta_b,
\end{cases}
\qquad
\rho_{ia}=
\begin{cases}
1,& \epsilon_i=\eta_a,\\[1mm]
0,& \epsilon_i\neq \eta_a.
\end{cases}
\end{align}

Let
\begin{equation}\label{e:defX}
\mathfrak X
:=
\Bigl\{
(\bmt,\bms)\in \bC^{n+m}:
t_i\neq t_j,\ s_a\neq s_b,\ t_i\neq s_a
\text{ for all } i\neq j,\ a\neq b
\Bigr\}.
\end{equation}
Fix a simply connected domain \(\mathfrak C\subset \mathfrak X\). On
\(\mathfrak C\), choose branches of
\[
\log(t_i-t_j),\qquad
\log(s_a-s_b),\qquad
\log(t_i-s_a).
\]
We define the branch factor
\begin{align}\begin{split}
\label{eq:general-branch-factor-clean}
\Phi_{\bm\epsilon,\bm\eta}(\bmt,\bms)
&:=
\exp\!\left(
\ri\sum_{i=1}^n \epsilon_i t_i
+
\ri\frac{\beta}{2}\sum_{a=1}^m \eta_a s_a
\right)
\prod_{1\le i<j\le n}(t_i-t_j)^{\alpha_{ij}}
\\
&\qquad\times
\prod_{1\le a<b\le m}(s_a-s_b)^{\gamma_{ab}}
\prod_{i=1}^n
\prod_{a=1}^m
(t_i-s_a)^{\rho_{ia}}.
\end{split}\end{align}
We remark that when $\beta=2$, the branch factor $\Phi_{\bm\epsilon,\bm\eta}(\bmt,\bms)=F_{\bm\epsilon,\bm\eta}(\bmt,\bms)$
solves the \(\beta=2\) system
\eqref{e:bulkbeta=2}.

We also introduce the separation scale
\begin{align}\label{e:defRbulk}
R(\bmt,\bms)
:=
\min\Bigl(
\min_{1\le i<j\le n}|t_i-t_j|,
\min_{1\le i\le n,\ 1\le a\le m}|t_i-s_a|,
\min_{1\le a<b\le m}|s_a-s_b|
\Bigr).
\end{align}

The following theorem shows that, for general \(\beta>0\), the same phenomenon
persists: each sign pattern gives rise to a solution of the bulk deformed CMS
system \eqref{e:eq1}--\eqref{e:eq2}, whose leading asymptotic factor is the branch factor
\eqref{eq:general-branch-factor-clean}. We postpone its proof to \Cref{s:bulk_solution}.

\begin{theorem}[Bulk local solution space and ordered sectorial solutions]
\label{thm:sectorial-branch-solution}
The following two statements hold.

\begin{enumerate}
\item[\textup{(i)}]
For every point \((\bmt_0,\bms_0)\in\mathfrak X\), the space of germs at
\((\bmt_0,\bms_0)\) of holomorphic solutions of the bulk deformed CMS system
\eqref{e:eq1}--\eqref{e:eq2} has dimension \(2^{n+m}\).

\item[\textup{(ii)}]
There exists a sufficiently large constant \(D=D(\beta,n,m)\) such that
the following holds. Write
\[
(z_1,\dots,z_{n+m}):=(t_1,\dots,t_n,s_1,\dots,s_m),
\]
and consider the unbounded ordered domain
\begin{align}\label{e:bulkC}
\fC:=\left\{(\bmt,\bms)\in\mathfrak X:
\Re[z_{\ell+1}]-\Re[z_\ell]>D,\quad 1\leq \ell<n+m\right\}.
\end{align}
On \(\fC\), choose branches of the logarithms entering
\(\Phi_{\bm\epsilon,\bm\eta}\). For each sign pattern
\((\bm\epsilon,\bm\eta)\in\{\pm1\}^n\times\{\pm1\}^m\), there exists a
holomorphic solution \(F_{\bm\epsilon,\bm\eta}\) of
\eqref{e:eq1}--\eqref{e:eq2} on \(\fC\) of the form
\begin{equation}\label{e:bulk_defF}
F_{\bm\epsilon,\bm\eta}(\bmt,\bms)
=
\Phi_{\bm\epsilon,\bm\eta}(\bmt,\bms)
W_{\bm\epsilon,\bm\eta}(\bmt,\bms).
\end{equation}
The correction factor satisfies
\begin{equation}\label{eq:W-basic-est}
W_{\bm\epsilon,\bm\eta}(\bmt,\bms)
=
1+\OO\!\bigl(R(\bmt,\bms)^{-1}\bigr),
\end{equation}
and, for subsets \(I\subseteq\{1,\dots,n\}\) and
\(A\subseteq\{1,\dots,m\}\),
\begin{equation}\label{eq:jet-basic-est}
\prod_{i\in I}\partial_{t_i}
\prod_{a\in A}\partial_{s_a}
W_{\bm\epsilon,\bm\eta}(\bmt,\bms)
=
\OO\!\bigl(R(\bmt,\bms)^{-|I|-|A|-1}\bigr).
\end{equation}
The estimates are uniform on \(\fC\), and the implicit constants depend only
on \(\beta,n,m\). Moreover, the \(2^{n+m}\) solutions obtained in this way
are linearly independent. The same statement holds if the inequalities defining the ordered domain are
imposed in any other fixed order of the \(t\)- and \(s\)-coordinates.
\end{enumerate}
\end{theorem}

\subsection{Proof of \Cref{t:bulk-characterization}}
Our main result, \Cref{t:bulk-characterization}, follows from the proposition
below. The proposition states that the classification of the bulk deformed CMS
system \eqref{e:eq1}--\eqref{e:eq2} in
\Cref{thm:sectorial-branch-solution}, together with the asymptotics of the
exponential observables from \Cref{prop:pv-product-smallL} and
\Cref{c:absolute_bound}, uniquely determines the exponential observable.

In the remainder of this section, we first prove
\Cref{t:bulk-characterization}, using \Cref{p:bulk_connection_formula} as an
input. We then prove \Cref{p:bulk_connection_formula} in \Cref{s:proof_connect}.

\begin{proposition}\label{p:bulk_connection_formula}
Adopt \Cref{a:bulk}, and assume
\[
\frac{\beta}{2}=\frac pq\in\bQ_{>0},
\qquad n=kp,
\qquad m=kq.
\]
Define
\[
F(t_1,\dots,t_n;s_1,\dots,s_m)
:=
\bE\left[
\operatorname{P.V.}\prod_{j\in\bZ}
\frac{\prod_{i=1}^n(t_i-x_j)}
{\prod_{a=1}^m(s_a-x_j)^{\beta/2}}
\right].
\]
Fix \(\bm\eta=(\eta_1,\dots,\eta_m)\in\{\pm1\}^m\) such that
\[
\#\{a:\eta_a=+1\}=k_+q,
\qquad
\#\{a:\eta_a=-1\}=k_-q,
\qquad
k_++k_-=k.
\]
Let \(D\) be as in \Cref{thm:sectorial-branch-solution}, write
\((z_1,\dots,z_{n+m})=(t_1,\dots,t_n,s_1,\dots,s_m)\), and consider the
unbounded domain
\[
\fC:=\left\{(\bmt,\bms)\in\mathfrak X:
\Re[z_{\ell+1}]-\Re[z_\ell]>D,\ 1\leq\ell<n+m,
\quad s_a\in\bC_{\eta_a},\ 1\leq a\leq m\right\}.
\]
Recall the sectorial solutions from \eqref{e:bulk_defF}, then, on \(\fC\),
\[
F
=
\sum_{\bm\epsilon\in\mathcal E_{\bm\eta}}
F_{\bm\epsilon,\bm\eta},
\qquad
\mathcal E_{\bm\eta}
:=
\left\{\bm\epsilon\in\{\pm1\}^n:
\sum_{i=1}^n\epsilon_i+\frac\beta2\sum_{a=1}^m\eta_a=0\right\}.
\]
Equivalently,
\[
\#\{i:\epsilon_i=+1\}=k_-p,
\qquad
\#\{i:\epsilon_i=-1\}=k_+p.
\]
\end{proposition}

\begin{proof}[Proof of \Cref{t:bulk-characterization}]
By \Cref{t:bulk_loop}, we already know that the
\(\mathrm{Sine}_{\beta}\) point process satisfies the microscopic bulk loop
equations \eqref{e:loopeq2-bulk}. It remains to prove uniqueness: any random
point process satisfying \eqref{e:loopeq2-bulk} and the assumptions of the
\Cref{t:bulk-characterization} must have the same law.

By the representation \eqref{e:szrep},
\[
s(w)=\operatorname{P.V.}\sum_{j=-\infty}^{\infty}\frac{1}{x_j-w}.
\]
Thus the random point process
$
\{\cdots \le x_{-2}\le x_{-1}\le x_0\le x_1\le x_2\le \cdots\}
$
is characterized by the resolvent correlation functions
\begin{equation}\label{e:k-point-correlation}
\bE\bigl[s(w_1)s(w_2)\cdots s(w_k)\bigr],
\qquad
w_1,\dots,w_k\in \bC_+\cup\bC_- .
\end{equation}

We next show that the correlation functions in
\eqref{e:k-point-correlation} can be recovered from the exponential
observables. Consequently, once these observables are uniquely determined, the
correlation functions are uniquely determined as well.
Let
\(\tau_\alpha\in\{\pm1\}\) be such that
\[
w_\alpha\in\bC_{\tau_\alpha},
\qquad 1\le \alpha\le k.
\]
Consider
\[
	F
	=
	\bE\left[
	\operatorname{P.V.}
	\prod_{j\in \bZ}
	\prod_{\alpha=1}^k
	\frac{
		\prod_{i=1}^p (t_i^{(\alpha)}-x_j)
	}{
		\prod_{a=1}^q (s_a^{(\alpha)}-x_j)^{\beta/2}
	}
	\right],
\]
where
\[
	t_i^{(\alpha)},\,s_a^{(\alpha)}\in \bC_{\tau_\alpha},
	\qquad
	1\le i\le p,\quad 1\le a\le q,\quad 1\le \alpha\le k .
\]
Since \(\beta/2=p/q\), the observable equals \(1\) after the diagonal
specialization
\begin{equation}\label{e:wspeical}
w_\alpha
=
t_1^{(\alpha)}
=\cdots=
t_p^{(\alpha)}
=
s_1^{(\alpha)}
=\cdots=
s_q^{(\alpha)},
\qquad 1\le \alpha\le k.
\end{equation}
Differentiating before this specialization gives
\[
\partial_{t_1^{(\alpha)}}
\log\left[
\operatorname{P.V.}
\prod_{j\in\bZ}
\frac{
	\prod_{i=1}^p (t_i^{(\alpha)}-x_j)
}{
	\prod_{a=1}^q (s_a^{(\alpha)}-x_j)^{\beta/2}
}
\right]
=
\operatorname{P.V.}\sum_{j\in\bZ}\frac{1}{t_1^{(\alpha)}-x_j}
=
-s(t_1^{(\alpha)}).
\]
Therefore, after applying
\(\partial_{t_1^{(1)}}\cdots\partial_{t_1^{(k)}}\) and then imposing
\eqref{e:wspeical}, all undifferentiated product factors become \(1\), and we
obtain
\[
\bE\left[s(w_1)\cdots s(w_k)\right]
=
(-1)^k
\left.
\partial_{t_1^{(1)}}\cdots\partial_{t_1^{(k)}}F
\right|_{w_\alpha
=
t_1^{(\alpha)}
=\cdots=
t_p^{(\alpha)}
=
s_1^{(\alpha)}
=\cdots=
s_q^{(\alpha)},1\le \alpha\le k.
}.
\]

By \Cref{p:bulk_connection_formula}, the observable \(F\) is uniquely
determined on the nonempty ordered domain \(\fC\) appearing there.
 Moreover, by \Cref{c:absolute_bound},
\(F\) is holomorphic in the numerator variables
\(t_i^{(\alpha)}\in\bC\) and holomorphic in the denominator variables
\(s_a^{(\alpha)}\in\bC_{\tau_\alpha}\). Hence, by the identity theorem for
holomorphic functions, this uniquely determines \(F\) on the whole domain
\[
(t_i^{(\alpha)})\in\bC,
\qquad
(s_a^{(\alpha)})\in\bC_{\tau_\alpha},\quad 1\leq i\leq p, \quad 1\leq a\leq q, \quad 1\leq \al\leq k.
\]
In particular, the derivatives at the specialization \eqref{e:wspeical} are
unique. Thus all correlation functions \eqref{e:k-point-correlation} are
uniquely determined. This proves
\Cref{t:bulk-characterization}.
\end{proof}

\subsection{Proof of \Cref{p:bulk_connection_formula}}\label{s:proof_connect}
\begin{proof}[Proof of \Cref{p:bulk_connection_formula}]
Choose real numbers
\[
\xi_1<\cdots<\xi_n<\zeta_1<\cdots<\zeta_m
\]
whose successive gaps are larger than \(2D\). We use these same real parts
throughout the proof. Since the constants in
\eqref{eq:W-basic-est}--\eqref{eq:jet-basic-est} are universal, we may
increase \(D\) so that
\[
|W_{\bm\epsilon,\bm\eta'}-1|\leq\frac14
\]
on the ordered domain, uniformly over all sign patterns. In particular, all
correction factors are bounded above and away from zero.

The observable \(F\) is a holomorphic solution of
\eqref{e:eq1}--\eqref{e:eq2} on \(\fC\). By
\Cref{thm:sectorial-branch-solution}, the sectorial solutions form a basis,
and hence
\begin{equation}\label{e:F-branch-expansion-local}
F
=
\sum_{(\bm\epsilon,\bm\eta')\in\{\pm1\}^{n+m}}
 c_{\bm\epsilon,\bm\eta'}F_{\bm\epsilon,\bm\eta'},
\qquad
c_{\bm\epsilon,\bm\eta'}\in\bC.
\end{equation}
All configurations below remain in the same unbounded domain \(\fC\), so the
solutions and the coefficients in this expansion do not change during the
three comparisons.

\medskip
\noindent
\emph{Step 1: Forcing \(\bm\eta'=\bm\eta\).}

For a large parameter \(\Lambda>0\), set
\[
t_r:=\xi_r+\ri\Lambda^r,
\quad 1\leq r\leq n,
\qquad
s_a:=\zeta_a+\eta_a\ri\Lambda^{n+a},
\quad 1\leq a\leq m.
\]
For all sufficiently large \(\Lambda\), these variables are pairwise
distinct and
$
R(\bmt,\bms)\gtrsim \Lambda.
$

By \Cref{c:absolute_bound}, applied to the configuration
\((\bmt,\bms)\) above, we have
\begin{equation}\label{e:F-decay-sM}
|F(\bmt,\bms)|
\leq
\exp\!\left(
\sum_{r=1}^n\Lambda^r
-\frac{\beta}{2}\sum_{a=1}^m\Lambda^{n+a}
+\OO\bigl((\log\Lambda)^2\bigr)
\right).
\end{equation}

On the other hand, by the explicit formula for
\(\Phi_{\bm\epsilon,\bm\eta'}\) and the estimate
\[
W_{\bm\epsilon,\bm\eta'}(\bmt,\bms)
=
1+\OO\bigl(R(\bmt,\bms)^{-1}\bigr),
\]
each branch satisfies
\begin{equation}\label{e:branch-sM-asymp}
\begin{split}
F_{\bm\epsilon,\bm\eta'}(\bmt,\bms)
&=
\exp\!\left(
-\sum_{r=1}^n\epsilon_r\Lambda^r
-\frac{\beta}{2}
 \sum_{a=1}^m\eta_a\eta'_a\Lambda^{n+a}
\right)
\Lambda^{\nu_{\bm\epsilon,\bm\eta'}}
C_{\bm\epsilon,\bm\eta'}
\left(1+\OO(\Lambda^{-1})\right),
\end{split}
\end{equation}
for some exponent
\(\nu_{\bm\epsilon,\bm\eta'}\in\bR\) and some nonzero coefficient
\(C_{\bm\epsilon,\bm\eta'}\) depending on
\((\bm\xi,\bm\zeta)\).

Suppose that
$
c_{\bm\epsilon,\bm\eta'}\neq0
$
for some \(\bm\eta'\neq\bm\eta\), and let \(a_0\) be the largest index
such that
$
\eta'_{a_0}\neq\eta_{a_0}.
$
Then
\begin{align}
-\frac{\beta}{2}
\sum_{a=1}^m
\eta_a\eta'_a\Lambda^{n+a}
&=
-\frac{\beta}{2}
\sum_{a=1}^m\Lambda^{n+a}
+
\beta
\sum_{\{a:\eta'_a\neq\eta_a\}}
\Lambda^{n+a}
=
-\frac{\beta}{2}
\sum_{a=1}^m\Lambda^{n+a}
+
\beta\Lambda^{n+a_0}
+
\OO(\Lambda^{n+a_0-1}).
\end{align}
The additional term
\(\beta\Lambda^{n+a_0}\) dominates all contributions involving the
\(t\)-variables, which are of order at most \(\Lambda^n\), as well as
the error \(\OO((\log\Lambda)^2)\) in
\eqref{e:F-decay-sM}.

Moreover, because all variables have been placed on separated scales,
distinct sign patterns
\((\bm\epsilon,\bm\eta')\) give distinct exponential weights in
\(\Lambda\). Thus, among all terms with nonzero coefficients and
\(\bm\eta'\neq\bm\eta\), there is a unique lexicographically dominant
exponential weight, which cannot be cancelled by the remaining terms.
This contradicts the upper bound \eqref{e:F-decay-sM}. Hence
\[
c_{\bm\epsilon,\bm\eta'}=0
\qquad
\text{whenever }\bm\eta'\neq\bm\eta.
\]

Therefore,
\begin{equation}\label{e:F-epsilon-only}
F
=
\sum_{\bm\epsilon\in\{\pm1\}^n}
c_{\bm\epsilon}
F_{\bm\epsilon,\bm\eta},
\qquad
c_{\bm\epsilon}
:=
c_{\bm\epsilon,\bm\eta}.
\end{equation}

\smallskip
\noindent
\emph{Step 2: only balanced sign patterns survive.}

Let
\[
u:=\#\{a:\eta_a=+1\}=k_+q,
\qquad
m-u=\#\{a:\eta_a=-1\}=k_-q.
\]
Fix \(\bm\epsilon=(\epsilon_1, \epsilon_2,\cdots, \epsilon_n)\in\{\pm1\}^n\), and let
\[
v:=\#\{i:\epsilon_i=-1\}.
\]
Take large \(\Lambda>0\) and set
\begin{align}\label{e:c_t_s}
t_r:=\xi_r-\ri \epsilon_r \Lambda,\quad 1\leq r\leq n,
\quad 
s_a:=\zeta_a+\ri \eta_a \Lambda, \quad 1\leq a\leq m.
\end{align}
Then by \Cref{c:absolute_bound}, with
\[
M=\sum_{r=1}^n \Lambda-\frac{\beta}{2}\sum_{a=1}^m \Lambda=0,
\quad
L=\sum_{r=1}^n\sum_{a=1}^m
\log \left(1+\frac{|\xi_r-\zeta_a|}{\Lambda}\right)=\OO(1),
\]
we conclude that
\begin{align}\label{e:fub1}
|F(\bmt,\bms)|\leq e^{M+C(L+L^2)}\lesssim 1,
\end{align}
provided we take $\Lambda$ large enough.

With the choice of parameters in \eqref{e:c_t_s}, \eqref{eq:general-branch-factor-clean} gives
\begin{align*}
|\Phi_{\bm\epsilon,\bm\eta}(\bmt,\bms)|
&= \exp\left(\Lambda {\sum_{i=1}^n \epsilon^2_i -\Lambda\frac{\beta}{2}\sum_{a=1}^m\eta^2_a}\right)\prod_{1\le i<j\le n}|t_i-t_j|^{\alpha_{ij}}
\prod_{1\le a<b\le m}|s_a-s_b|^{\gamma_{ab}}
\prod_{i=1}^n\prod_{a=1}^m |t_i-s_a|^{\rho_{ia}}\\
&=\prod_{1\le i<j\le n}|t_i-t_j|^{\alpha_{ij}}
\prod_{1\le a<b\le m}|s_a-s_b|^{\gamma_{ab}}
\prod_{i=1}^n\prod_{a=1}^m |t_i-s_a|^{\rho_{ia}}.
\end{align*}
We recall that $v=\#\{i: \epsilon_i=-1\}$ and $u=\#\{a: \eta_a=+1 \}$, and recall $\al_{ij}$, $\gamma_{ab}$ and $\rho_{ia}$ from \eqref{eq:alpha-gamma-rho2}.
By a direct count of the algebraic powers of \(\Lambda\), we can further simplify the above expression as
\begin{align}\label{e:fexp1}
\Phi_{\bm\epsilon,\bm\eta}(\bmt,\bms)=\Lambda^{\kappa}
C_{\bm\epsilon,\bm\eta}(\bm\xi,\bm\zeta)(1+\OO(\Lambda^{-1}),\quad
F_{\bm\epsilon,\bm\eta}(\bmt,\bms)=\Phi_{\bm\epsilon,\bm\eta}(\bmt,\bms)(1+\OO(R(\bmt,\bms)^{-1})),
\end{align}
where \(|C_{\bm\epsilon,\bm\eta}(\bm\xi,\bm\zeta)|\asymp 1\), and
\begin{align*}\begin{split}
\kappa
&=
-\frac{2}{\beta}v(n-v)-\frac{\beta}{2}u(m-u)+v(m-u)+u(n-v)\\
&=\frac{q}{p}v^2+\frac{p}{q}u^2-2uv
=\frac{q}{p}\left(v-\frac{pu}{q}\right)^2
\ge 0.
\end{split}\end{align*}
Thus the exponent  \(\kappa=0\)  if and only if
\[
v=\frac{pu}{q}=pk_+=\frac{\beta u}{2}.
\]
Since
\[
\sum_{i=1}^n \epsilon_i=n-2v,
\qquad
\sum_{a=1}^m \eta_a=2u-m,
\]
this is equivalent to
\begin{equation}\label{e:balance-eps-eta}
\sum_{i=1}^n \epsilon_i+\frac{\beta}{2}\sum_{a=1}^m \eta_a=0.
\end{equation}

If \eqref{e:balance-eps-eta} fails, then \(\kappa>0\), so the branch
\(F_{\bm\epsilon,\bm\eta}(\bmt,\bms)\) in \eqref{e:fexp1} grows like a positive power of \(\Lambda\). For
generic choices of \((\bm\xi,\bm\zeta)\), the leading coefficients
\(C_{\bm\epsilon,\bm\eta}(\bm\xi,\bm\zeta)\) corresponding to distinct \(\bm\epsilon\) are
linearly independent, since they come from distinct monomials in the explicit
factor \(\Phi_{\bm\epsilon,\bm\eta}\). Because \(F(\bmt,\bms)\) in \eqref{e:fub1} remains bounded, this
forces
\[
c_{\bm\epsilon}=0
\qquad\text{unless}\qquad
\sum_{i=1}^n \epsilon_i+\frac{\beta}{2}\sum_{a=1}^m \eta_a=0.
\]
Hence
\begin{equation}\label{e:F-admissible}
F=\sum_{\bm\epsilon\in \mathcal E_{\bm\eta}} c_{\bm\epsilon}\,F_{\bm\epsilon,\bm\eta},
\qquad
\mathcal E_{\bm\eta}
:=
\left\{
\bm\epsilon\in\{\pm1\}^n:
\sum_{i=1}^n \epsilon_i+\frac{\beta}{2}\sum_{a=1}^m \eta_a=0
\right\}.
\end{equation}

\smallskip
\noindent
\emph{Step 3: for every admissible \(\bm\epsilon\), one has \(c_{\bm\epsilon}=1\).}
Fix \(\bm\epsilon\in\mathcal E_{\bm\eta}\). Then there exist partitions
\[
\{1,\dots,n\}=I_1\sqcup\cdots\sqcup I_k,
\qquad
\{1,\dots,m\}=A_1\sqcup\cdots\sqcup A_k,
\]
with \(|I_\alpha|=p\) and \(|A_\alpha|=q\) for every
\(1\le \alpha\le k\), such that for each block there is a sign
\(\tau_\alpha\in\{\pm1\}\) satisfying
\[
\epsilon_{I_\alpha(r)}=-\tau_\alpha,
\qquad
\eta_{A_\alpha(a)}=\tau_\alpha,
\qquad
1\le r\le p,\quad 1\le a\le q.
\]
Here \(I_\alpha(r)\) denotes the \(r\)-th element of \(I_\alpha\), and
\(A_\alpha(a)\) denotes the \(a\)-th element of \(A_\alpha\).

After relabelling the variables according to these blocks, write
\[
\epsilon_r^{(\alpha)}:=\epsilon_{I_\alpha(r)},
\qquad
t_r^{(\alpha)}:=t_{I_\alpha(r)},
\qquad
\eta_a^{(\alpha)}:=\eta_{A_\alpha(a)},
\qquad
s_a^{(\alpha)}:=s_{A_\alpha(a)}.
\]
Thus
\begin{align}\label{e:blocksign}
\epsilon_r^{(\alpha)}=-\tau_\alpha,
\quad
\eta_a^{(\alpha)}=\tau_\alpha,\quad 1\le r\le p,\quad 1\le a\le q,\quad 1\leq \al \leq k.
\end{align}

We then specialize the variables in each block by setting
\begin{align}\label{e:parameter}
t_r^{(\alpha)}
:=
\Lambda\xi_{I_\alpha(r)}
+
\tau_\alpha\ri(\Lambda^4+\Lambda),
\qquad
s_a^{(\alpha)}
:=
\Lambda\zeta_{A_\alpha(a)}
+
\tau_\alpha\ri(\Lambda^4-\Lambda).
\end{align}
For all sufficiently large \(\Lambda\), this configuration lies in the same
domain \(\fC\), and \(R(\bmt,\bms)\asymp\Lambda\). Moreover, the variables
in each block lie in \(\bC_{\tau_\alpha}\).

Then \eqref{e:logF-smallL-general}, together with
\begin{align}\begin{split}\label{e:M_exp}
M
&=
-\ri \sum_{\alpha=1}^k \tau_\alpha
\left(
\sum_{r=1}^p t_r^{(\alpha)}
-\frac{\beta}{2}\sum_{a=1}^q s_a^{(\alpha)}
\right)
\\
&=
-\ri \sum_{\alpha=1}^k\tau_\alpha
\left[
\Lambda
\left(
\sum_{r=1}^p \xi_{I_\alpha(r)}
-\frac{\beta}{2}\sum_{a=1}^q \zeta_{A_\alpha(a)}
\right)
+
2p\tau_\alpha\ri\Lambda
\right]
\\
&=
2pk\Lambda
-\ri\Lambda\sum_{\alpha=1}^k\tau_\alpha
\left(
\sum_{r=1}^p \xi_{I_\alpha(r)}
-\frac{\beta}{2}\sum_{a=1}^q \zeta_{A_\alpha(a)}
\right),
\end{split}\end{align}
and
\[
L
=
\sum_{\alpha=1}^k \sum_{r=1}^p\sum_{a=1}^q
\log\!\left(
1+\frac{|t_r^{(\alpha)}-s_a^{(\alpha)}|}
{\sqrt{|\Im[ t_r^{(\alpha)}]|\,|\Im [s_a^{(\alpha)}]|}}
\right)
=
\OO\!\left(\frac{\Lambda}{\Lambda^4}\right)
=
\OO(\Lambda^{-3}),
\]
yields
\begin{equation}\label{e:F-main-asymp}
F(\bmt,\bms)
=
e^{2pk\Lambda}
\exp\!\left(
-\ri\Lambda\sum_{\alpha=1}^k\tau_\alpha
\left(
\sum_{r=1}^p \xi_{I_\alpha(r)}
-\frac{\beta}{2}\sum_{a=1}^q \zeta_{A_\alpha(a)}
\right)
\right)
\bigl(1+\OO(\Lambda^{-1})\bigr).
\end{equation}

For the distinguished branch \(F_{\bm\epsilon,\bm\eta}\), the formula
\eqref{eq:general-branch-factor-clean} gives
\begin{equation}\label{e:F11}
F_{\bm\epsilon,\bm\eta}
=
\Phi_{\bm\epsilon,\bm\eta}W_{\bm\epsilon,\bm\eta}
=
\Phi_{\bm\epsilon,\bm\eta}\bigl(1+\OO(\Lambda^{-1})\bigr),
\end{equation}
because \(R(\bmt,\bms)\asymp \Lambda\) and
\(W_{\bm\epsilon,\bm\eta}=1+\OO(R(\bmt,\bms)^{-1})\).
By plugging in \eqref{e:blocksign} and \eqref{e:parameter}, the exponential
part of \(\Phi_{\bm\epsilon,\bm\eta}\) is
\begin{align*}
&\exp\!\left(
\ri\sum_{\alpha=1}^k\sum_{r=1}^p
\epsilon_r^{(\alpha)}t_r^{(\alpha)}
+
\ri\frac{\beta}{2}
\sum_{\alpha=1}^k\sum_{a=1}^q
\eta_a^{(\alpha)}s_a^{(\alpha)}
\right)
=
\exp\!\left[
-\ri\sum_{\alpha=1}^k\tau_\alpha
\left(
\sum_{r=1}^p t_r^{(\alpha)}
-\frac{\beta}{2}\sum_{a=1}^q s_a^{(\alpha)}
\right)
\right].
\end{align*}
By the computation of \(M\) in \eqref{e:M_exp}, this equals
\begin{equation}\label{e:F22}
\exp\!\left[
2pk\Lambda
-\ri\Lambda\sum_{\alpha=1}^k\tau_\alpha
\left(
\sum_{r=1}^p \xi_{I_\alpha(r)}
-\frac{\beta}{2}\sum_{a=1}^q \zeta_{A_\alpha(a)}
\right)
\right].
\end{equation}

Next we show that the algebraic part of
\(\Phi_{\bm\epsilon,\bm\eta}\) contributes only a
\begin{equation}\label{e:F33}
1+\OO(\Lambda^{-1})
\end{equation}
factor. Within a fixed block, or for two different blocks
\(\alpha\neq\beta\) such that \(\tau_\alpha=\tau_\beta\), all \(t\)-variables
have the same \(\epsilon\)-sign, all \(s\)-variables have the same
\(\eta\)-sign, and the \(\epsilon\)-signs and \(\eta\)-signs are opposite.
Therefore
\[
\alpha_{ij}=0,\qquad \gamma_{ab}=0,\qquad \rho_{ia}=0
\]
for all such pairs of variables. Thus these pairs give no algebraic
contribution.

For two different blocks \(\alpha\neq\beta\) with
\(\tau_\alpha\neq\tau_\beta\), all cross-block differences are of order
\(\Lambda^4\), while their lower-order corrections are of size
\(\OO(\Lambda)\). Hence, after factoring out the corresponding powers of
\(\Lambda^4\), their product is a fixed nonzero constant times
\(1+\OO(\Lambda^{-3})\). Moreover, the total power of \(\Lambda^4\)
cancels. Indeed, the cross-block exponents are
\[
-\frac{2}{\beta}\cdot p^2,
\qquad
-\frac{\beta}{2}\cdot q^2,
\qquad
1\cdot 2pq,
\]
and their sum is
\[
-\frac{2p^2}{\beta}-\frac{\beta q^2}{2}+2pq
=
-pq-pq+2pq
=
0,
\]
where we used \(\beta/2=p/q\). Thus the algebraic part is a fixed nonzero
constant times \(1+\OO(\Lambda^{-3})\). By the multiplicative normalization
of \(\Phi_{\bm\epsilon,\bm\eta}\) fixed above, this constant is equal to
\(1\). Therefore the algebraic part is \(1+\OO(\Lambda^{-1})\).

Combining \eqref{e:F11}, the exponential part \eqref{e:F22}, and the
algebraic estimate \eqref{e:F33}, we obtain
\begin{equation}\label{e:Feps-main-asymp}
F_{\bm\epsilon,\bm\eta}(\bmt,\bms)
=
e^{2pk\Lambda}
\exp\!\left(
-\ri\Lambda\sum_{\alpha=1}^k\tau_\alpha
\left(
\sum_{r=1}^p \xi_{I_\alpha(r)}
-\frac{\beta}{2}\sum_{a=1}^q \zeta_{A_\alpha(a)}
\right)
\right)
\bigl(1+\OO(\Lambda^{-1})\bigr).
\end{equation}

Now let \(\bm\epsilon'\in\mathcal E_{\bm\eta}\) with
\(\bm\epsilon'\neq\bm\epsilon\). In the following, we estimate
\(F_{\bm\epsilon',\bm\eta}(\bmt,\bms)\). For each block \(\alpha\), define the
number of mismatches
\[
m_\alpha(\bm\epsilon')
:=
\#\left\{
1\le r\le p:
\epsilon'_{I_\alpha(r)}\neq -\tau_\alpha
\right\}.
\]
Since \(\bm\epsilon'\neq\bm\epsilon\), we have
\[
m(\bm\epsilon'):=\sum_{\alpha=1}^k m_\alpha(\bm\epsilon')\ge 1.
\]

We compare the exponential part of
\(\Phi_{\bm\epsilon',\bm\eta}\) with that of the distinguished branch. Using
\eqref{e:parameter}, the modulus of the exponential contribution from the
\(\alpha\)-th block is
\begin{align*}
&\left|
\exp\left(
\ri\sum_{r=1}^p \epsilon'_{I_\alpha(r)}t_r^{(\alpha)}
+
\ri\frac{\beta}{2}\sum_{a=1}^q \eta_a^{(\alpha)}s_a^{(\alpha)}
\right)
\right|
\\
&\qquad=
\exp\left(
-\tau_\alpha(\Lambda^4+\Lambda)
\sum_{r=1}^p \epsilon'_{I_\alpha(r)}
-\frac{\beta}{2}q(\Lambda^4-\Lambda)
\right)
\\
&\qquad=
\exp\left(
(p-2m_\alpha(\bm\epsilon'))(\Lambda^4+\Lambda)
-p(\Lambda^4-\Lambda)
\right)
\\
&\qquad=
\exp\left(
2p\Lambda
-
2m_\alpha(\bm\epsilon')(\Lambda^4+\Lambda)
\right),
\end{align*}
where we used \(\beta q/2=p\). Multiplying over
\(\alpha=1,\dots,k\), this gives
\begin{align}\label{e:exppart}
\left|
\textnormal{exponential part of }\Phi_{\bm\epsilon',\bm\eta}
\right|
=
e^{2pk\Lambda}
e^{-2m(\bm\epsilon')(\Lambda^4+\Lambda)}
\le
e^{2pk\Lambda}e^{-2(\Lambda^4+\Lambda)} .
\end{align}

It remains to bound the algebraic part. Between different blocks, the
separations are at most of order \(\Lambda^4\), while all nonzero
separations are bounded below by a constant multiple of \(\Lambda\). Hence
all algebraic factors in \(\Phi_{\bm\epsilon',\bm\eta}\) are bounded by
\(C\Lambda^4\), and there are at most \((n+m)^2\) such factors. Therefore
\begin{align}\label{e:algpart}
\left|
\textnormal{algebraic part of }\Phi_{\bm\epsilon',\bm\eta}
\right|
\le
(C\Lambda)^{4(n+m)^2}.
\end{align}
Since \(W_{\bm\epsilon',\bm\eta}=1+\OO(\Lambda^{-1})\), combining
\eqref{e:exppart} and \eqref{e:algpart} gives
\begin{equation}\label{e:wrong-branch-negligible}
|F_{\bm\epsilon',\bm\eta}(\bmt,\bms)|
\le
e^{2pk\Lambda}
e^{-2(\Lambda^4+\Lambda)}
(C\Lambda)^{4(n+m)^2}
=
\oo(e^{2pk\Lambda}),
\end{equation}
provided \(\Lambda\) is sufficiently large.

Substituting the special configuration above into \eqref{e:F-admissible}, and
using \eqref{e:F-main-asymp}, \eqref{e:Feps-main-asymp}, and
\eqref{e:wrong-branch-negligible}, we obtain
\[
1+\oo(1)=c_{\bm\epsilon}(1+\oo(1))+\oo(1).
\]
Hence \(c_{\bm\epsilon}=1\). Since
\(\bm\epsilon\in \mathcal E_{\bm\eta}\) was arbitrary, we conclude
that
\[
c_{\bm\epsilon}=1
\qquad\text{for every }\bm\epsilon\in \mathcal E_{\bm\eta}.
\]

Combining this with \eqref{e:F-admissible}, we have proved on \(\fC\) that
\[
F=\sum_{\bm\epsilon\in \mathcal E_{\bm\eta}}
F_{\bm\epsilon,\bm\eta}.
\]
\end{proof}

\subsection{Solutions of the edge deformed CMS operators}
In this section, we study solutions of the edge deformed CMS system
\eqref{e:edge_eq1}--\eqref{e:edge_eq2}. As in the bulk case, we construct
sectorial solutions indexed by sign patterns
\[
(\bm\epsilon,\bm\eta)
=
(\epsilon_1,\dots,\epsilon_n,\eta_1,\dots,\eta_m)
\in \{\pm1\}^n\times\{\pm1\}^m.
\]

For \(z\in \bC_\pm\), set
\begin{align}\label{e:defphi}
\phi_-(z):=\Ai(z),
\qquad
\phi_+(z):=
\begin{cases}
\ri\omega \Ai(\omega z),
 & z\in \bC_+,\\[2mm]
-\ri\omega^2 \Ai(\omega^2 z) & z\in \bC_-.
\end{cases}
\end{align}
Both functions solve the Airy equation
\[
\phi''(z)-z\phi(z)=0.
\]
Moreover, by \eqref{e:Ai_asymp}, \eqref{e:Bi-minus-iAi-upper}, and
\eqref{e:Bi-plus-iAi-lower}, they satisfy the asymptotics
\begin{align}\label{e:phi_asymptotics}
\phi_\pm(z)
\sim
\frac{1}{2\sqrt{\pi}} z^{-1/4}
\exp\!\left(\pm\frac23 z^{3/2}\right),
\qquad
z\to\infty,\quad |\arg z|< \pi-\delta .
\end{align}
We refer to \Cref{a:Airy} for further background on Airy functions. We remark that $\phi_+$ in \eqref{e:defphi} is constructed such that the asymptotics \eqref{e:phi_asymptotics} hold.

For the \(s\)-variables, we introduce the rescaled Airy functions
\begin{align}\label{e:defpsi}
\psi_\pm(s)
:=
\phi_\pm\!\left(\left(\frac{\beta}{2}\right)^{2/3}s\right).
\end{align}
Then \(\psi_\pm\) solve
\[
\psi_\pm''(s)-\frac{\beta^2}{4}s\,\psi_\pm(s)=0.
\]
The solution associated with a sign pattern
\((\bm\epsilon,\bm\eta)\) is characterized by the leading Airy behavior
\[
\left(\prod_{i=1}^n\phi_{\epsilon_i}(t_i)\right)
\left(\prod_{a=1}^m\psi_{\eta_a}(s_a)\right).
\]

We first consider the special case \(\beta=2\), where the edge deformed CMS
operators are given explicitly in \eqref{e:edgebeta=2}. In this case, one can
write down the solutions in closed form. The proof is a direct verification and
is postponed to \Cref{s:edge-b2solution}.

\begin{theorem}[Explicit solutions for \(\beta=2\)]
\label{thm:edge-beta2-explicit}
For each sign pattern
\[
\bm\sigma
=
(\epsilon_1,\dots,\epsilon_n,\eta_1,\dots,\eta_m)
\in \{\pm1\}^{n+m},
\]
write
\[
\bmx=(x_1,\dots,x_{n+m})
:=
(\bmt,\bms)=(t_1,\dots,t_n,s_1,\dots,s_m).
\]
Define
\[
D_{\bm\sigma}(\bmx)
:=
\det\!\left[
\partial_x^{\,j-1}\phi_{\sigma_r}(x_r)
\right]_{1\le r,j\le n+m}.
\]
Let
\[
\Delta(\bmt):=\prod_{1\le i<j\le n}(t_i-t_j),
\qquad
\Delta(\bms):=\prod_{1\le a<b\le m}(s_a-s_b),
\]
and set
\begin{equation}
\label{eq:F-sigma-edge}
F_{\bm\sigma}(\bmt,\bms)
:=
\frac{D_{\bm\sigma}(\bmx)}{\Delta(\bmt)\Delta(\bms)}.
\end{equation}
Then \(F_{\bm\sigma}\) solves the \(\beta=2\) edge system
\eqref{e:edgebeta=2}. Moreover, the \(2^{n+m}\) solutions obtained in this way
are linearly independent.
\end{theorem}

We now return to general \(\beta>0\). Fix a sign pattern
\[
(\bm\epsilon,\bm\eta)
=
(\epsilon_1,\dots,\epsilon_n,\eta_1,\dots,\eta_m)
\in \{\pm1\}^n\times\{\pm1\}^m.
\]
Define the configuration space
\begin{equation}\label{e:defedgeX}
\sfX
:=
\Bigl\{
(\bmt,\bms)\in(\bC_+\cup \bC_-)^{n+m}:
t_i,s_a\neq 0,\ 
t_i\neq t_j,\ s_a\neq s_b,\ t_i\neq s_a,\
|\arg t_i|, |\arg s_a|<\pi
\Bigr\}.
\end{equation}
Fix a simply connected domain
\(\sfC\subset\sfX\). Since
\(|\arg t_i|,|\arg s_a|< \pi-\delta\), the principal square-root branches
$
\sqrt{t_i},\sqrt{s_a}$
are well defined on \(\sfC\). Then, for all choices of signs, the functions
\[
\pm\sqrt{t_i}\pm\sqrt{t_j},\qquad
\pm\sqrt{s_a}\pm\sqrt{s_b},\qquad
\pm\sqrt{t_i}\pm\sqrt{s_a}
\]
are holomorphic and nonvanishing on \(\sfC\). Hence, we may also choose branches of
\[
\log(\pm\sqrt{t_i}\pm\sqrt{t_j}),\qquad
\log(\pm\sqrt{s_a}\pm\sqrt{s_b}),\qquad
\log(\pm\sqrt{t_i}\pm\sqrt{s_a}).
\]
We then define the branch factor
\begin{align}\begin{split}
\label{eq:edge-PhiF-clean}
\Phi_{\bm\epsilon,\bm\eta}(\bmt,\bms)
&:=
\left(\prod_{i=1}^n\phi_{\epsilon_i}(t_i)\right)
\left(\prod_{a=1}^m\psi_{\eta_a}(s_a)\right)
\prod_{1\le i<j\le n}
(\epsilon_i\sqrt{t_i}+\epsilon_j\sqrt{t_j})^{-2/\beta}
\\
&\qquad\times
\prod_{1\le a<b\le m}
(\eta_a\sqrt{s_a}+\eta_b\sqrt{s_b})^{-\beta/2}
\prod_{i=1}^n
\prod_{a=1}^m
(\epsilon_i\sqrt{t_i}-\eta_a\sqrt{s_a}).
\end{split}\end{align}

We also introduce the separation scale
\begin{align}\label{e:defsfR}
\sfR(\bmt,\bms)
:=
\min\Bigl(
\min_{1\le i\le n}|t_i|,
\min_{1\le a\le m}|s_a|,
\min_{1\le i<j\le n}|t_i-t_j|,
\min_{1\le i\le n,\ 1\le a\le m}|t_i-s_a|,
\min_{1\le a<b\le m}|s_a-s_b|
\Bigr).
\end{align}

The following theorem shows that, for general \(\beta>0\), the same phenomenon
persists: each sign pattern gives rise to a solution of the edge deformed CMS
system \eqref{e:edge_eq1}--\eqref{e:edge_eq2}, whose leading asymptotic factor
is the branch factor \eqref{eq:edge-PhiF-clean}. We postpone the proof to
\Cref{s:edge_solution}.

\begin{theorem}[Edge local solution space and sectorial solutions]
\label{thm:edge-sectorial-solution}
Fix \(0<\delta<\pi/2\). The following two statements hold.

\begin{enumerate}
\item[\textup{(i)}]
For every point \((\bmt_0,\bms_0)\in\sfX\), the space of germs at
\((\bmt_0,\bms_0)\) of holomorphic solutions of the edge deformed CMS system
\eqref{e:edge_eq1}--\eqref{e:edge_eq2} has dimension \(2^{n+m}\).

\item[\textup{(ii)}]
There exists a sufficiently large constant \(D=D(\beta,n,m)\) such that
the following holds. Write
\[
(z_1,\dots,z_{n+m}):=(t_1,\dots,t_n,s_1,\dots,s_m).
\]
For any $\bm\tau\in \{\pm\}^{n+m}$, consider the unbounded ordered domain
\begin{align}\begin{split}\label{e:edgeC}
\sfC=\sfC(\bm\tau,D)
:=\Bigl\{\bmz\in\bC^{n+m}:\;&
\left|\arg z_j-\frac{\pi}{3}\tau_j\right|<\frac{\pi}{12},
\quad 1\le j\le n+m,\\
&\Re [z_{N}]>D,\qquad
\Re [z_{r}]-\Re [z_{r+1}]>D,
\quad 1\le r<n+m
\Bigr\}.\end{split}\end{align}
On \(\sfC\), choose branches of the logarithms entering
\(\Phi_{\bm\epsilon,\bm\eta}\). For each sign pattern
\((\bm\epsilon,\bm\eta)\in\{\pm1\}^n\times\{\pm1\}^m\), there exists a
holomorphic solution \(F_{\bm\epsilon,\bm\eta}\) of
\eqref{e:edge_eq1}--\eqref{e:edge_eq2} on \(\sfC\) of the form
\begin{equation}\label{e:edge_defF}
F_{\bm\epsilon,\bm\eta}
=
\Phi_{\bm\epsilon,\bm\eta}
W_{\bm\epsilon,\bm\eta}.
\end{equation}
The correction factor satisfies
\begin{equation}\label{eq:edge-W-basic-est}
W_{\bm\epsilon,\bm\eta}(\bmt,\bms)
=
1+\OO\!\left(\sfR(\bmt,\bms)^{-3/2}\right),
\end{equation}
and, for subsets \(I\subseteq\{1,\dots,n\}\) and
\(A\subseteq\{1,\dots,m\}\),
\begin{equation}\label{eq:edge-jet-basic-est}
\partial_{\bmt_I}\partial_{\bms_A}
W_{\bm\epsilon,\bm\eta}(\bmt,\bms)
=
\OO\!\left(
\sfR(\bmt,\bms)^{-|I|-|A|-3/2}
\right).
\end{equation}
The estimates are uniform on \(\sfC\), and their implicit constants depend
only on \(\beta,n,m\). Moreover, the \(2^{n+m}\) solutions obtained in this
way are linearly independent.
\end{enumerate}
\end{theorem}

\subsection{Proof of \Cref{t:edge-characterization}}
Our characterization of Airy point process \Cref{t:edge-characterization}, follows from the proposition
below. The proposition states that the classification of the edge deformed CMS
system \eqref{e:edge_eq1}--\eqref{e:edge_eq2} in
\Cref{thm:edge-sectorial-solution}, together with the asymptotics of the
exponential observables from \Cref{p:edge_est}, uniquely determines the exponential observable.

In the remainder of this section, we first prove
\Cref{t:edge-characterization}, using \Cref{p:edge_formula} as an
input. We then prove \Cref{p:edge_formula} in \Cref{s:edge_proof_connect}.

\begin{proposition}
\label{p:edge_formula}
Adopt \Cref{a:edge}, and assume
\[
\frac{\beta}{2}=\frac{p}{q}\in\bQ_{>0},
\qquad
n=kp,
\qquad
m=kq.
\]
Define
\begin{align}\label{e:defF-edge}
F(t_1,\dots,t_n;s_1,\dots,s_m)
:=
\bE\left[
\frac{\prod_{r=1}^n e^{c_0t_r}}
{\prod_{a=1}^m e^{(\beta/2)c_0s_a}}
\prod_{j\ge 1}
\frac{\prod_{r=1}^n (t_r-x_j)e^{t_r/\fa_j}}
{\prod_{a=1}^m \bigl((s_a-x_j)e^{s_a/\fa_j}\bigr)^{\beta/2}}
\right],
\end{align}
where \(c_0=\Ai'(0)/\Ai(0)\).

Fix sign patterns
\[
\bm\epsilon=(\epsilon_1,\dots,\epsilon_n)\in\{\pm1\}^n,
\qquad
\bm\eta=(\eta_1,\dots,\eta_m)\in\{\pm1\}^m,
\]
satisfying
\begin{align}\label{e:balance-halfplanes-edge}
\#\{r:\epsilon_r=+1\}=k_+p,\quad
\#\{r:\epsilon_r=-1\}=k_-p,\quad
\#\{a:\eta_a=+1\}=k_+q,\quad
\#\{a:\eta_a=-1\}=k_-q,
\end{align}
where \(k_++k_-=k\), and set
\[
\bm\tau
:=
(\epsilon_1,\dots,\epsilon_n,\eta_1,\dots,\eta_m).
\]
For \(D=D(\beta,n,m)>0\) sufficiently large, and  let $\sfC=\sfC(\bm\tau,D)$ be as in \eqref{e:edgeC}.

Recall the sectorial solutions \(F_{\bm\epsilon,\bm\eta}\) from
\eqref{e:edge_defF}. Choose the logarithmic branches entering
\(\Phi_{\bm-,\bm+}\) according to the conventions used in the proof below.
Then, on \(\sfC\),
\begin{equation}\label{e:F-edge-one-branch}
F
=
2^{k(p+q)/2}
\left(\frac{\beta}{2}\right)^{kq/6}
\pi^{k(p+q)/2}
\exp\!\left(-\ri\pi\vartheta_{p,q}(k_+,k_-)\right)
F_{\bm-,\bm+},
\end{equation}
where
\begin{equation}\label{e:edge-phase-vartheta}
\vartheta_{p,q}(k_+,k_-)
:=
\frac{q}{2}
\left[
p\bigl(k_-^2+2k_+k_- -k_+^2\bigr)
+k_- -k_+
\right].
\end{equation}
Here \(\bm+\) and \(\bm-\) denote the all-plus and all-minus sign patterns,
respectively.
\end{proposition}

\begin{proof}[Proof of \Cref{t:edge-characterization}]
By \Cref{t:edge_loop}, we already know that the
\(\mathrm{Airy}_{\beta}\) point process satisfies the microscopic edge loop
equations \eqref{e:loopeq2-edge}. It remains to prove uniqueness: any random
point process satisfying \eqref{e:loopeq2-edge} and the assumptions of
\Cref{t:edge-characterization} must have the same law.

By the representation \eqref{e:normalized},
\[
s(w)
=
\sum_{i=1}^\infty
\left(
\frac{1}{x_i-w}
-
\frac{1}{\fa_i}
\right)
-\frac{\Ai'(0)}{\Ai(0)} .
\]
 Hence to show the uniqueness of the random point process, it is enough to show that the
resolvent correlation functions
\begin{align}\label{e:edge-correlation}
\bE\bigl[s(w_1)s(w_2)\cdots s(w_k)\bigr],
\qquad
w_1,\dots,w_k\in \bC_+\cup\bC_-,
\end{align}
are uniquely determined.

We next show that the correlation functions in \eqref{e:edge-correlation} can
be recovered from the edge exponential observables. Consequently, once these
observables are uniquely determined, the correlation functions are uniquely
determined as well.

Let \(\tau_\alpha\in\{\pm1\}\) be such that
\[\label{e:edge-exp-observable-char}
w_\alpha\in\bC_{\tau_\alpha},
\qquad 1\le \alpha\le k.
\]
Consider the edge exponential observable
\begin{equation} F = \prod_{\alpha=1}^k e^{ c_0 \left( \sum_{r=1}^p t_r^{(\alpha)} -\frac{\beta}{2}\sum_{a=1}^q s_a^{(\alpha)} \right) } \times \bE\left[\prod_{j\ge 1} \prod_{\alpha=1}^k \frac{ \prod_{r=1}^p (t_r^{(\alpha)}-x_j) e^{t_r^{(\alpha)}/\fa_j} }{ \prod_{a=1}^q \left( (s_a^{(\alpha)}-x_j) e^{s_a^{(\alpha)}/\fa_j} \right)^{\beta/2} }\right], \end{equation}
where
\[
t_r^{(\alpha)},\,s_a^{(\alpha)}\in\bC_{\tau_\alpha},
\qquad
1\le r\le p,\quad 1\le a\le q,\quad 1\le \alpha\le k,\quad c_0:=\frac{\Ai'(0)}{\Ai(0)}.
\]

As in the bulk case, differentiating the observable before taking the diagonal
specialization recovers the  resolvent correlations. Therefore,
\begin{align}\label{e:wspeical2}
\bE\bigl[s(w_1)s(w_2)\cdots s(w_k)\bigr]
=
(-1)^k
\left.
\partial_{t_1^{(1)}}\cdots\partial_{t_1^{(k)}}F
\right|_{
w_\alpha
=
t_1^{(\alpha)}
=\cdots=
t_p^{(\alpha)}
=
s_1^{(\alpha)}
=\cdots=
s_q^{(\alpha)},
\ 1\le \alpha\le k } .
\end{align}

By \Cref{p:edge_formula}, the observable \(F\) is uniquely determined on the
nonempty open domain \(\sfC\) introduced there.  Moreover, by
\Cref{p:edge_est}, \(F\) is holomorphic in the variables
\[
t_r^{(\alpha)},\,s_a^{(\alpha)}\in\bC_{\tau_\alpha},\quad 1\leq r\leq p, \quad 1\leq a\leq q, \quad 1\leq \al\leq k.
\]
Hence, by the identity theorem for holomorphic functions, \(F\) is uniquely
determined on this domain. In particular, the derivatives at the specialization
in \eqref{e:wspeical2} are unique. Thus all correlation functions
\eqref{e:edge-correlation} are uniquely determined. This proves \Cref{t:edge-characterization}.
\end{proof}

\subsection{Proof of \Cref{p:edge_formula}}\label{s:edge_proof_connect}
\begin{proof}[Proof of \Cref{p:edge_formula}]
By \Cref{thm:edge-sectorial-solution}, the sectorial solutions form a basis
on \(\sfC\). Hence
\begin{equation}\label{e:F-edge-branch-expansion}
F
=
\sum_{(\bm\epsilon',\bm\eta')\in\{\pm1\}^{n+m}}
c_{\bm\epsilon',\bm\eta'}
F_{\bm\epsilon',\bm\eta'},
\qquad
c_{\bm\epsilon',\bm\eta'}\in\bC.
\end{equation}
We will show that all coefficients vanish except possibly the one corresponding
to the $F_{\bm-,\bm+}$ branch.

\smallskip
\noindent
\emph{Step 1: forcing $\bm\epsilon'=\bm-$ and $\bm\eta'=\bm+$.}

By the balance condition \eqref{e:balance-halfplanes-edge}, there exist
partitions
\[
\{1,\dots,n\}=I_1\sqcup\cdots\sqcup I_k,
\qquad
\{1,\dots,m\}=A_1\sqcup\cdots\sqcup A_k,
\]
with \(|I_\alpha|=p\) and \(|A_\alpha|=q\) for every
\(1\le \alpha\le k\), such that each block has a sign
\(\tau_\alpha\in\{\pm1\}\) satisfying
\[
\epsilon_{I_\alpha(r)}=\tau_\alpha,
\qquad
\eta_{A_\alpha(a)}=\tau_\alpha,
\qquad
1\le r\le p,\quad 1\le a\le q.
\]
Here \(I_\alpha(r)\) denotes the \(r\)-th element of \(I_\alpha\), and
\(A_\alpha(a)\) denotes the \(a\)-th element of \(A_\alpha\).

After relabelling the variables according to these blocks, write
\[
\epsilon_r^{(\alpha)}:=\epsilon_{I_\alpha(r)},
\qquad
t_r^{(\alpha)}:=t_{I_\alpha(r)},
\qquad
\eta_a^{(\alpha)}:=\eta_{A_\alpha(a)},
\qquad
s_a^{(\alpha)}:=s_{A_\alpha(a)}.
\]
Then
\begin{equation}
\epsilon_r^{(\alpha)}=\eta_a^{(\alpha)}=\tau_\alpha,
\qquad
t_r^{(\alpha)},s_a^{(\alpha)}\in\bC_{\tau_\alpha},
\qquad
1\le r\le p,\quad 1\le a\le q,\quad 1\leq \al\leq k .
\end{equation}

Choose positive parameters
\[
\xi_1>\xi_2>\cdots>\xi_n>0,
\qquad
0<\zeta_1<\zeta_2<\cdots<\zeta_m,
\]
outside the finitely many proper real hyperplanes on which two distinct sign
patterns have the same real exponential rate below. For \(\Lambda>0\), set
\begin{align}\label{e:edge-test-config}
t_r^{(\alpha)}
&:=
\Lambda^2
\left(
1-\frac{\ri\tau_\alpha\xi_{I_\alpha(r)}}{\Lambda}
\right)^{2/3}
e^{\ri\tau_\alpha\pi/3},
&
s_a^{(\alpha)}
&:=
\Lambda^2
\left(
1+\frac{\ri\tau_\alpha\zeta_{A_\alpha(a)}}{\Lambda}
\right)^{2/3}
e^{\ri\tau_\alpha\pi/3}.
\end{align}
All powers are principal powers. Uniformly in the indices,
\begin{align}
t_r^{(\alpha)}
&=
\Lambda^2e^{\ri\tau_\alpha\pi/3}
+
\frac{\Lambda}{3}
(\sqrt3-\ri\tau_\alpha)\xi_{I_\alpha(r)}
+
\OO(1),
\nonumber\\
s_a^{(\alpha)}
&=
\Lambda^2e^{\ri\tau_\alpha\pi/3}
+
\frac{\Lambda}{3}
(-\sqrt3+\ri\tau_\alpha)\zeta_{A_\alpha(a)}
+
\OO(1).
\label{eq:edge-test-common-expansion}
\end{align}
In particular,
\[
\Re[t_i]
=
\frac{\Lambda^2}{2}
+
\frac{\sqrt3}{3}\Lambda\xi_i
+
\OO(1),
\qquad
\Re[s_a]
=
\frac{\Lambda^2}{2}
-
\frac{\sqrt3}{3}\Lambda\zeta_a
+
\OO(1).
\]
Consequently,
\[
\Re[t_1]>\cdots>\Re[t_n]>
\Re[s_1]>\cdots>\Re[s_m]>D
\]
for every sufficiently large \(\Lambda\). Moreover,
\[
\arg t_r^{(\alpha)}
\longrightarrow
\frac{\pi}{3}\tau_\alpha,
\qquad
\arg s_a^{(\alpha)}
\longrightarrow
\frac{\pi}{3}\tau_\alpha.
\]
Thus
$
(\bmt(\Lambda),\bms(\Lambda))\in\sfC
$
for every sufficiently large \(\Lambda\). The strict inequalities among the
parameters also imply
\begin{equation}\label{e:R-edge-scale}
\sfR(\bmt(\Lambda),\bms(\Lambda))
\asymp\Lambda.
\end{equation}
Indeed, two variables in opposite half-planes are separated on the scale
\(\Lambda^2\), while two variables in the same half-plane are separated on
the scale \(\Lambda\).

Using the principal branches, we have
\begin{align}\label{e:edge-powers}
\bigl(t_r^{(\alpha)}\bigr)^{3/2}
=
\ri\tau_\alpha\Lambda^3
+
\xi_{I_\alpha(r)}\Lambda^2,
\qquad
\bigl(s_a^{(\alpha)}\bigr)^{3/2}
=
\ri\tau_\alpha\Lambda^3
-
\zeta_{A_\alpha(a)}\Lambda^2.
\end{align}
The quantity \(L\) in \Cref{p:edge_est} satisfies
\[
L
=
\sum_{\alpha=1}^k\sum_{r=1}^p\sum_{a=1}^q
\log\!\left(
1+\frac{|t_r^{(\alpha)}-s_a^{(\alpha)}|}
{\sqrt{|\Im [t_r^{(\alpha)}]|\,|\Im[ s_a^{(\alpha)}]|}}
\right)
=
\OO(\Lambda^{-1}),
\]
Furthermore, using \(p=(\beta/2)q\), the \(\ri\tau_\alpha\Lambda^3\)-terms cancel
inside each block, and therefore
\begin{align}
M
&=
-\frac{2}{3}
\sum_{\alpha=1}^k
\left(
\sum_{r=1}^p \bigl(t_r^{(\alpha)}\bigr)^{3/2}
-\frac{\beta}{2}\sum_{a=1}^q \bigl(s_a^{(\alpha)}\bigr)^{3/2}
\right)
=
-\frac{2}{3}\Lambda^2
\sum_{\alpha=1}^k
\left(
\sum_{r=1}^p\xi_{I_\alpha(r)}
+
\frac{\beta}{2}
\sum_{a=1}^q\zeta_{A_\alpha(a)}
\right).
\end{align}
Hence
\eqref{e:logF-smallL-edge} gives
\begin{equation}\label{e:F-edge-asymp}
F(\bmt,\bms)
=
\exp\!\left[
-\frac{2}{3}\Lambda^2
\sum_{\alpha=1}^k
\left(
\sum_{r=1}^p\xi_{I_\alpha(r)}
+
\frac{\beta}{2}
\sum_{a=1}^q\zeta_{A_\alpha(a)}
\right)
\right]
\bigl(1+\OO(\Lambda^{-1})\bigr).
\end{equation}

Next we estimate the sectorial branches. Recall that
\[
F_{\bm\epsilon',\bm\eta'}=\Phi_{\bm\epsilon',\bm\eta'}W_{\bm\epsilon',\bm\eta'},
\]
where the explicit branch factor is given by \eqref{eq:edge-PhiF-clean}. Namely,
\begin{align}\begin{split}\label{e:factorp}
\Phi_{\bm\epsilon',\bm\eta'}(t,s)
&=
\left(\prod_{r=1}^n\phi_{\epsilon'_r}(t_r)\right)
\left(\prod_{a=1}^m\psi_{\eta'_a}(s_a)\right)
\prod_{1\le i<j\le n}
(\epsilon'_i\sqrt{t_i}+\epsilon'_j\sqrt{t_j})^{-2/\beta}\\
&\qquad
\times
\prod_{1\le a<b\le m}
(\eta'_a\sqrt{s_a}+\eta'_b\sqrt{s_b})^{-\beta/2}
\prod_{i=1}^n\prod_{a=1}^m
(\epsilon'_i\sqrt{t_i}-\eta'_a\sqrt{s_a}).
\end{split}\end{align}

For \(0<|\arg z|<\pi-\delta\), the Airy asymptotics
\eqref{e:phi_asymptotics} imply
\begin{align}
\begin{split}\label{e:phi-psi-edge-asymp}
\phi_+(z)
=
\frac{1}{2\sqrt{\pi}\,z^{1/4}}
e^{\frac{2}{3}z^{3/2}}
\left(1+\OO_\delta(|z|^{-3/2})\right),\quad 
\phi_-(z)
=
\frac{1}{2\sqrt{\pi}\,z^{1/4}}
e^{-\frac{2}{3}z^{3/2}}
\left(1+\OO_\delta(|z|^{-3/2})\right).
\end{split}
\end{align}
Consequently, for the normalized Airy functions \(\psi_\pm\) defined in
\eqref{e:defpsi}, we have
\begin{align}
\begin{split}\label{e:psi-edge-asymp}
\psi_+(z)
&=
\frac{1}{2^{5/6}\beta^{1/6}\sqrt{\pi}\,z^{1/4}}
e^{\frac{\beta}{3}z^{3/2}}
\left(1+\OO_\delta(|z|^{-3/2})\right),
\\
\psi_-(z)
&=
\frac{1}{2^{5/6}\beta^{1/6}\sqrt{\pi}\,z^{1/4}}
e^{-\frac{\beta}{3}z^{3/2}}
\left(1+\OO_\delta(|z|^{-3/2})\right).
\end{split}
\end{align}

For the special configuration \eqref{e:edge-test-config} above, \eqref{e:edge-powers} implies that
\begin{align}\begin{split}\label{e:phi-edge-special}
\phi_+\bigl(t_r^{(\alpha)}\bigr)
&=
e^{
\frac23\ri\tau_\alpha\Lambda^3+\frac23\xi_{I_\alpha(r)}\Lambda^2
+\OO(\log\Lambda)
},\quad
\phi_-\bigl(t_r^{(\alpha)}\bigr)
=
e^{
-\frac23\ri\tau_\alpha\Lambda^3-\frac23\xi_{I_\alpha(r)}\Lambda^2
+\OO(\log\Lambda)
},
\\
\psi_+\bigl(s_a^{(\alpha)}\bigr)
&=
e^{
\frac{\beta}{3}\ri\tau_\alpha\Lambda^3-\frac{\beta}{3}\zeta_{A_\alpha(a)}\Lambda^2
+\OO(\log\Lambda)
},
\quad
\psi_-\bigl(s_a^{(\alpha)}\bigr)
=
e^{
-\frac{\beta}{3}\ri\tau_\alpha\Lambda^3+\frac{\beta}{3}\zeta_{A_\alpha(a)}\Lambda^2
+\OO(\log\Lambda)
}.
\end{split}
\end{align}

Moreover,  every algebraic factor in \eqref{e:factorp} contributes at most
\(\exp(\OO(\log\Lambda))\). Since \(\sfR(\bmt,\bms)\asymp \Lambda\), we also have
\begin{align}\label{e:Wterm}
W_{\bm\epsilon',\bm\eta'}(\bmt,\bms)=1+\OO(\sfR(\bmt,\bms)^{-3/2})=1+\OO(\Lambda^{-3/2}).
\end{align}

Combining these estimates, we obtain
\begin{equation}\label{e:edge-branch-asymp-clean}
F_{\bm\epsilon',\bm\eta'}(\bmt,\bms)
=
\exp\!\left[
\frac{2}{3}\Lambda^2
\sum_{\alpha=1}^k
\left(
\sum_{r=1}^p \epsilon_r^{\prime(\alpha)}\xi_{I_\alpha(r)}
-\frac{\beta}{2}\sum_{a=1}^q \eta_a^{\prime(\alpha)}\zeta_{A_\alpha(a)}
\right)
+\ri \Lambda^3\Theta_{\bm\epsilon',\bm\eta'}
+\OO(\log\Lambda)
\right],
\end{equation}
where 
\begin{align}
\Theta_{\bm\epsilon',\bm\eta'}=\frac{2}{3}\sum_{\alpha=1}^k\tau_\al
\left(
\sum_{r=1}^p \epsilon_r^{\prime(\alpha)}
+\frac{\beta}{2}\sum_{a=1}^q \eta_a^{\prime(\alpha)}
\right)\in \mathbb R.
\end{align}
 In particular, the term
\(\ri\Lambda^3\Theta_{\bm\epsilon',\bm\eta'}\) is purely oscillatory and does not affect
the magnitude of \(F_{\bm\epsilon',\bm\eta'}(\bmt,\bms)\), and we conclude
\begin{align}\label{e:edge-branch-asymp-clean}
\left|
F_{\bm\epsilon',\bm\eta'}(\bmt,\bms)
\right|
=
\exp\!\Biggl[
&\frac{2}{3}\Lambda^2
\sum_{\alpha=1}^k
\left(
\sum_{r=1}^p
\epsilon'_{I_\alpha(r)}\xi_{I_\alpha(r)}
-
\frac{\beta}{2}
\sum_{a=1}^q
\eta'_{A_\alpha(a)}\zeta_{A_\alpha(a)}
\right)
+
\OO(\log\Lambda)
\Biggr].
\end{align}
The real rate of the
\((\bm-,\bm+)\)-branch is
\[
-\frac{2}{3}
\sum_{\alpha=1}^k
\left(
\sum_{r=1}^p\xi_{I_\alpha(r)}
+
\frac{\beta}{2}
\sum_{a=1}^q\zeta_{A_\alpha(a)}
\right),
\]
which is exactly the rate in \eqref{e:F-edge-asymp}. Since all the parameters
are positive, every other sign pattern has a strictly larger real rate. By
the generic choice of the parameters, all these rates are distinct.

If a coefficient of a branch other than
\((\bm-,\bm+)\) in \eqref{e:F-edge-branch-expansion} were nonzero, the
unique branch of maximal real rate among the nonzero terms would dominate
the right-hand side exponentially, contradicting
\eqref{e:F-edge-asymp}. Therefore
\[
c_{\bm\epsilon',\bm\eta'}=0
\qquad
\text{unless}\qquad
(\bm\epsilon',\bm\eta')=(\bm-,\bm+),
\]
and hence
\[
F=c_{\bm-,\bm+}F_{\bm-,\bm+}.
\]

\smallskip
\noindent
\emph{Step 2: determine $c_{\bm-,\bm+}$.}

It remains to determine the coefficient $c_{\bm-,\bm+}$. Using the Airy asymptotics \eqref{e:psi-edge-asymp}, we can rewrite this distinguished branch as
\begin{align}\begin{split}\label{e:edge-minus-plus-factor}
\Phi_{\bm-,\bm+}
={}&
\left(\prod_{r=1}^n\phi_-(t_r)\right)
\left(\prod_{a=1}^m\psi_+(s_a)\right)
\prod_{1\le i<j\le n}
(-\sqrt{t_i}-\sqrt{t_j})^{-2/\beta}
\\
&\times
\prod_{1\le a<b\le m}
(\sqrt{s_a}+\sqrt{s_b})^{-\beta/2}
\prod_{i=1}^n\prod_{a=1}^m
(-\sqrt{t_i}-\sqrt{s_a})\\
&=(1+\OO(\Lambda^{-3})\exp\!\left[
-\frac{2}{3}\Lambda^2
\sum_{\alpha=1}^k
\left(
\sum_{r=1}^p\xi_{I_\alpha(r)}
+
\frac{\beta}{2}
\sum_{a=1}^q\zeta_{A_\alpha(a)}
\right)
\right]\times I
\end{split}\end{align}
where  the algebraic factors $I$ in \eqref{e:edge-minus-plus-factor} is given by
\begin{align}\label{e:algebraic_exp}
I:=\frac{\prod_{1\le i<j\le n}
(-\sqrt{t_i}-\sqrt{t_j})^{-2/\beta}
\prod_{1\le a<b\le m}
(\sqrt{s_a}+\sqrt{s_b})^{-\beta/2} \prod_{i=1}^n\prod_{a=1}^m
(-\sqrt{t_i}-\sqrt{s_a})}{\prod_{i=1}^n 2\sqrt{\pi}t_r^{1/4}\prod_{a=1}^m 2^{5/6}\beta^{1/6}\sqrt{\pi}s_a^{1/4}}
\end{align}
For any \(t_i\) or \(s_a\) lying in \(\bC_\sigma\), \(\sigma\in\{\pm1\}\), the test
configuration \eqref{e:edge-test-config} gives
\[
\sqrt{t_i},\sqrt{s_a}
=
\Lambda e^{\ri\sigma\pi/6}\bigl(1+\OO(\Lambda^{-1})\bigr),
\qquad
t_i^{1/4},s_a^{1/4}
=
\Lambda^{1/2}e^{\ri\sigma\pi/12}\bigl(1+\OO(\Lambda^{-1})\bigr).
\]
Hence, if two variables lie in the same half-plane \(\bC_\sigma\), then
\[
\sqrt{u}+\sqrt{v}
=
2\Lambda e^{\ri\sigma\pi/6}\bigl(1+\OO(\Lambda^{-1})\bigr),\quad 
-\sqrt{u}-\sqrt{v}
=
2\Lambda e^{-\ri5\sigma\pi/6}\bigl(1+\OO(\Lambda^{-1})\bigr),
\]
whereas if they lie in opposite half-planes, then
\[
\sqrt{u}+\sqrt{v}
=
\sqrt{3}\,\Lambda\bigl(1+\OO(\Lambda^{-1})\bigr),\quad -\sqrt{u}-\sqrt{v}
=
\sqrt{3}e^{\ri \pi}\,\Lambda\bigl(1+\OO(\Lambda^{-1})\bigr).
\]
In the above, we chose the logarithmic branches so that, 
same-half-plane negative sums have arguments tending to
\(-5\pi/6\) in the upper half-plane and \(5\pi/6\) in the lower
half-plane, while cross-half-plane negative sums have arguments tending to
\(\pi\),

Using \(\beta=2p/q\), $n=kp$ and $m=kq$, the total power of \(\Lambda\) in the algebraic prefactor $I$ is
\[
-\frac{kp(kp-1)}{\beta}
-\frac{\beta kq(kq-1)}{4}
+k^2pq-\frac{k(p+q)}{2}
=0.
\]
The absolute value of the remaining constant is
\[
2^{-k(p+q)/2}
\left(\frac{\beta}{2}\right)^{-kq/6}
\pi^{-k(p+q)/2}.
\]

We recall the balance sign condition \eqref{e:balance-halfplanes-edge}. Then the total phase of the algebraic prefactor $I$ divided by \(\pi\) is
\begin{align*}
&-\frac{p+q}{12}(k_+-k_-)
+
\frac{5}{3\beta}
\left[
\binom{k_+p}{2}-\binom{k_-p}{2}
\right]
-
\frac{2p^2}{\beta}k_+k_-
-
\frac{\beta}{12}
\left[
\binom{k_+q}{2}-\binom{k_-q}{2}
\right]
-
\frac{5pq}{6}(k_+^2-k_-^2)
+
2pqk_+k_-.
\end{align*}
Using \(\beta/2=p/q\), this equals
\[
\frac{q}{2}
\left[
p\bigl(k_-^2+2k_+k_- -k_+^2\bigr)
+k_--k_+
\right].
\]
Since
\(W_{\bm-,\bm+}=1+\OO(\Lambda^{-3/2})\), we obtain
\begin{align}\label{e:F--branch-final}
F_{\bm-,\bm+}(\bmt,\bms)
={}&
2^{-k(p+q)/2}
\left(\frac{\beta}{2}\right)^{-kq/6}
\pi^{-k(p+q)/2}
\nonumber\\
&\times
\exp\!\left\{
\frac{\ri\pi q}{2}
\left[
p\bigl(k_-^2+2k_+k_- -k_+^2\bigr)
+k_--k_+
\right]
\right\}
\nonumber\\
&\times
\exp\!\left[
-\frac{2}{3}\Lambda^2
\sum_{\alpha=1}^k
\left(
\sum_{r=1}^p\xi_{I_\alpha(r)}
+
\frac{\beta}{2}
\sum_{a=1}^q\zeta_{A_\alpha(a)}
\right)
\right]
\bigl(1+\OO(\Lambda^{-1})\bigr).
\end{align}
Comparing \eqref{e:F--branch-final} with
\eqref{e:F-edge-asymp}, we conclude that
\[
c_{\bm-,\bm+}
=
2^{k(p+q)/2}
\left(\frac{\beta}{2}\right)^{kq/6}
\pi^{k(p+q)/2}
\exp\!\left\{
-\frac{\ri\pi q}{2}
\left[
p\bigl(k_-^2+2k_+k_- -k_+^2\bigr)
+k_--k_+
\right]
\right\}.
\]
This proves \eqref{e:F-edge-one-branch}.
\end{proof}

\section{Solutions of the bulk deformed CMS operators}

\label{s:bulk_solution}
We prove the first statement of \Cref{thm:sectorial-branch-solution}, namely that
the solution space has dimension \(2^{n+m}\), in \Cref{s:bulk_local_rank}. We derive the conjugated equations in \Cref{s:conjugate_bulk}.
We then prove the second statement, namely the construction of solutions with
prescribed leading behavior, in \Cref{s:bulk_construct_solution}.

\subsection{Local holonomic rank}
\label{s:bulk_local_rank}
For $1\le i\le n$ and $1\le a\le m$, define
\begin{align}
\cT_i
&:=
\partial_{t_i}^2+1
+\frac{2}{\beta}\sum_{j\neq i}\frac{\partial_{t_i}-\partial_{t_j}}{t_i-t_j}
-\sum_{a=1}^m\frac{\partial_{t_i}+(2/\beta)\partial_{s_a}}{t_i-s_a},
\label{eq:Ti-def}
\\[1mm]
\cS_a
&:=
\partial_{s_a}^2+\frac{\beta^2}{4}
+\frac{\beta}{2}\sum_{b\neq a}\frac{\partial_{s_a}-\partial_{s_b}}{s_a-s_b}
-\sum_{i=1}^n\frac{\partial_{s_a}+(\beta/2)\partial_{t_i}}{s_a-t_i}.
\label{eq:Sa-def}
\end{align}
Then the bulk deformed CMS system \eqref{e:eq1}--\eqref{e:eq2} can be written as
\begin{align}\label{e:bulk_eq}
\cT_iF=0,\qquad 1\le i\le n,\qquad 
\cS_aF=0,\qquad 1\le a\le m.
\end{align}

We first record the following commutator identities for the operators \(\cT_i\) and \(\cS_a\).
\begin{proposition}[Bulk commutator identities]\label{prop:commutator-identities}
For the operators $\cT_i,\cS_a$ defined in \eqref{eq:Ti-def}--\eqref{eq:Sa-def}, the
following commutator identities hold:
\begin{align}
[\cT_i,\cT_j]
&=
\frac{4}{\beta (t_i-t_j)^2}\,(\cT_j-\cT_i),
\qquad 1\le i\neq j\le n,
\label{eq:comm-tt-proof}
\\[1mm]
[\cS_a,\cS_b]
&=
\frac{\beta}{(s_a-s_b)^2}\,(\cS_b-\cS_a),
\qquad 1\le a\neq b\le m,
\label{eq:comm-ss-proof}
\\[1mm]
[\cT_i,\cS_a]
&=
\frac{4}{\beta (t_i-s_a)^2}\,\cS_a
-\frac{\beta}{(t_i-s_a)^2}\,\cT_i,
\qquad 1\le i\le n,\ 1\le a\le m.
\label{eq:comm-ts-proof}
\end{align}
\end{proposition}
\begin{proof}[Proof of \Cref{prop:commutator-identities}]
Introduce the building blocks
\begin{align*}
\cA_i
&:=\partial_{t_i}^2+1,
&
\cB_a
&:=\partial_{s_a}^2+\frac{\beta^2}{4},
\\[1mm]
\cU_{ij}
&:=\frac{2}{\beta}\frac{1}{t_i-t_j}
  \bigl(\partial_{t_i}-\partial_{t_j}\bigr)
=\cU_{ji},
&
\cW_{ab}
&:=\frac{\beta}{2}\frac{1}{s_a-s_b}
  \bigl(\partial_{s_a}-\partial_{s_b}\bigr)
=\cW_{ba},
\\[1mm]
\cX_{ia}
&:=\frac{1}{t_i-s_a}
  \biggl(\frac{\beta}{2}\partial_{t_i}+\partial_{s_a}\biggr).
\end{align*}
Then the definitions of \(\cT_i\) and \(\cS_a\) become
\begin{align*}
\cT_i
&=
\cA_i+\sum_{j\neq i}\cU_{ij}
-\frac{2}{\beta}\sum_{a=1}^m \cX_{ia},
\\[1mm]
\cS_a
&=
\cB_a+\sum_{b\neq a}\cW_{ab}
+\sum_{i=1}^n \cX_{ia}.
\end{align*}

We shall repeatedly use the following elementary consequences of the Leibniz rule. If
\(f,g\) are scalar functions and \(\cP,\cQ\) are commuting constant-coefficient first-order
differential operators, then
\begin{align}
[\partial_x^2,f\cP]
&=
(\partial_x^2 f)\cP+2(\partial_x f)\partial_x \cP,
\label{eq:basic-comm-1-rewrite}
\\[1mm]
[f\cP,g\cQ]
&=
f(\cP g)\cQ-g(\cQ f)\cP.
\label{eq:basic-comm-2-rewrite}
\end{align}

First fix distinct \(i,j\), and set
\[
\lambda_{ij}:=\frac{4}{\beta(t_i-t_j)^2}.
\]
Applying \eqref{eq:basic-comm-1-rewrite}--\eqref{eq:basic-comm-2-rewrite} gives the
following local identities:
\begin{align*}
[\cA_i+\cU_{ij},\,\cA_j+\cU_{ij}]
&=
\lambda_{ij}(\cA_j-\cA_i),
\\[1mm]
[\cA_i+\cU_{ij}+\cU_{ik},\,\cU_{jk}]
+
[\cU_{ik},\,\cA_j+\cU_{ij}]
&=
\lambda_{ij}(\cU_{jk}-\cU_{ik}),
\qquad k\neq i,j,
\\[1mm]
\Bigl[\cA_i+\cU_{ij}-\frac{2}{\beta}\cX_{ia},\,
-\frac{2}{\beta}\cX_{ja}\Bigr]
+
\Bigl[-\frac{2}{\beta}\cX_{ia},\,\cA_j+\cU_{ij}\Bigr]
&=
\lambda_{ij}
\Bigl(-\frac{2}{\beta}\cX_{ja}
+\frac{2}{\beta}\cX_{ia}\Bigr).
\end{align*}

In the expansion of \([\cT_i,\cT_j]\), all terms involving disjoint sets of variables commute.
Thus the only nonzero contributions are precisely the local pieces above. Hence
\begin{align*}
[\cT_i,\cT_j]
=
\lambda_{ij}(\cA_j-\cA_i)
+
\lambda_{ij}\sum_{k\neq i,j}(\cU_{jk}-\cU_{ik})
+
\lambda_{ij}\sum_{a=1}^m
\Bigl(-\frac{2}{\beta}\cX_{ja}
+\frac{2}{\beta}\cX_{ia}\Bigr)
=
\lambda_{ij}(\cT_j-\cT_i).
\end{align*}
This proves \eqref{eq:comm-tt-proof}.

Next fix distinct \(a,b\), and set
\[
\mu_{ab}:=\frac{\beta}{(s_a-s_b)^2}.
\]
The corresponding local identities are
\begin{align*}
[\cB_a+\cW_{ab},\,\cB_b+\cW_{ab}]
&=
\mu_{ab}(\cB_b-\cB_a),
\\[1mm]
[\cB_a+\cW_{ab}+\cW_{ac},\,\cW_{bc}]
+
[\cW_{ac},\,\cB_b+\cW_{ab}]
&=
\mu_{ab}(\cW_{bc}-\cW_{ac}),
\qquad c\neq a,b,
\\[1mm]
[\cB_a+\cW_{ab}+\cX_{ia},\,\cX_{ib}]
+
[\cX_{ia},\,\cB_b+\cW_{ab}]
&=
\mu_{ab}(\cX_{ib}-\cX_{ia}).
\end{align*}
Again, after expanding \([\cS_a,\cS_b]\), all disjoint-variable commutators vanish. Therefore
\begin{align*}
[\cS_a,\cS_b]
=
\mu_{ab}(\cB_b-\cB_a)
+
\mu_{ab}\sum_{c\neq a,b}(\cW_{bc}-\cW_{ac})
+
\mu_{ab}\sum_{i=1}^n(\cX_{ib}-\cX_{ia})
=
\mu_{ab}(\cS_b-\cS_a).
\end{align*}
This proves \eqref{eq:comm-ss-proof}.

Finally fix \(i,a\), and set
\[
\rho_{ia}:=\frac{4}{\beta(t_i-s_a)^2},
\qquad
\sigma_{ia}:=\frac{\beta}{(t_i-s_a)^2}.
\]
The mixed local identities are
\begin{align*}
\Bigl[\cA_i-\frac{2}{\beta}\cX_{ia},\,\cB_a+\cX_{ia}\Bigr]
&=
\rho_{ia}(\cB_a+\cX_{ia})
-
\sigma_{ia}
\Bigl(\cA_i-\frac{2}{\beta}\cX_{ia}\Bigr),
\\[1mm]
[\cA_i-\tfrac{2}{\beta}\cX_{ia}+\cU_{ij},\,\cX_{ja}]
+
[\cU_{ij},\,\cB_a+\cX_{ia}]
&=
\rho_{ia}\cX_{ja}
-
\sigma_{ia}\cU_{ij},
\qquad j\neq i,
\\[1mm]
\Bigl[\cA_i-\frac{2}{\beta}\cX_{ia}-\frac{2}{\beta}\cX_{ib},\,\cW_{ab}\Bigr]
+
\Bigl[-\frac{2}{\beta}\cX_{ib},\,\cB_a+\cX_{ia}\Bigr]
&=
\rho_{ia}\cW_{ab}
-
\sigma_{ia}
\Bigl(-\frac{2}{\beta}\cX_{ib}\Bigr),
\qquad b\neq a.
\end{align*}
Expanding \([\cT_i,\cS_a]\) and keeping only the nonzero local contributions gives
\begin{align*}
[\cT_i,\cS_a]
&=
\rho_{ia}(\cB_a+\cX_{ia})
-
\sigma_{ia}
\Bigl(\cA_i-\frac{2}{\beta}\cX_{ia}\Bigr)
\\
&\quad
+
\sum_{j\neq i}
\bigl(\rho_{ia}\cX_{ja}-\sigma_{ia}\cU_{ij}\bigr)
+
\sum_{b\neq a}
\biggl(
\rho_{ia}\cW_{ab}
-
\sigma_{ia}
\Bigl(-\frac{2}{\beta}\cX_{ib}\Bigr)
\biggr)
=
\rho_{ia}\cS_a-\sigma_{ia}\cT_i.
\end{align*}
This proves \eqref{eq:comm-ts-proof}.
\end{proof}

Let
\begin{align}\label{e:defK}
K:=\bC(t_1,\dots,t_n,s_1,\dots,s_m),
\end{align}
and let
\begin{align}\label{e:defDK}
D_K
:=
K\langle
\partial_{t_1},\dots,\partial_{t_n},
\partial_{s_1},\dots,\partial_{s_m}
\rangle
\end{align}
be the algebra of differential operators with coefficients in \(K\).

For \(1\le i\le n\) and \(1\le a\le m\), write
\[
\cT_i=\partial_{t_i}^2+\cP_i,
\qquad
\cS_a=\partial_{s_a}^2+\cQ_a,
\]
where \(\cP_i,\cQ_a\in D_K\) have derivative order at most \(1\). Let
\begin{align}\label{e:defI}
\mathcal I
:=
D_K\cT_1+\cdots+D_K\cT_n+D_K\cS_1+\cdots+D_K\cS_m
\subset D_K
\end{align}
be the left ideal generated by \(\cT_1,\dots,\cT_n,\cS_1,\dots,\cS_m\).

For subsets
\[
I\subseteq\{1,\dots,n\},
\qquad
A\subseteq\{1,\dots,m\},
\]
write
\[
\partial_{t_I}:=\prod_{i\in I}\partial_{t_i},
\qquad
\partial_{s_A}:=\prod_{a\in A}\partial_{s_a},
\]
where the factors are arranged in increasing order. We call
\[
\partial_{t_I}\partial_{s_A}F
\]
a \emph{square-free derivative} of \(F\). The word ``square-free'' means that no
variable is differentiated more than once. There are exactly \(2^{n+m}\) such
derivatives.

We shall prove that
\[
G:=\{\cT_1,\dots,\cT_n,\cS_1,\dots,\cS_m\}
\]
is a left Gr{\"o}bner basis of \(\mathcal I\). In the present setting, this means
that every operator in \(D_K\) can be reduced, modulo the generators in \(G\), to
a unique normal form whose derivative monomials are not divisible by any of the
leading monomials
\[
\operatorname{lm}(\cT_i)=\partial_{t_i}^2,
\qquad
\operatorname{lm}(\cS_a)=\partial_{s_a}^2.
\]
Consequently, the reduced monomials are precisely the square-free derivative
monomials, namely those in which no derivative appears with exponent \(2\) or
larger. For further background, see \cite{saito2013grobner}.

\begin{proposition}[Square-free normal forms]\label{prop:Grobner-bulk}
Assume that the commutator identities in
\Cref{prop:commutator-identities} hold. Then the residue classes of the
square-free derivative monomials
\[
\partial_{t_I}\partial_{s_A},
\qquad
I\subseteq\{1,\dots,n\},
\quad
A\subseteq\{1,\dots,m\},
\]
form a \(K\)-basis of \(D_K/\mathcal I\). Equivalently, every class in
\(D_K/\mathcal I\) has a unique representative of the form
\[
\sum_{I,A} c_{I,A}\,\partial_{t_I}\partial_{s_A},
\qquad
c_{I,A}\in K.
\]
In particular,
\[
\dim_K(D_K/\mathcal I)=2^{n+m}.
\]
\end{proposition}

\begin{proof}
The equations \(\cT_i=0\) and \(\cS_a=0\) should be thought of as reduction rules:
\begin{equation}\label{e:reduce}
\partial_{t_i}^2\equiv -\cP_i,
\qquad
\partial_{s_a}^2\equiv -\cQ_a
\qquad
\mod \mathcal I.
\end{equation}
Since \(\cP_i\) and \(\cQ_a\) have derivative order at most \(1\), each reduction
replaces a repeated derivative by lower-order terms. Thus the square-free
derivative monomials are the natural candidates for a basis.

The point is to prove that these reductions are consistent.
 For example, suppose
we want to reduce
$
\partial_{t_i}^2\partial_{t_j}^2.
$
There are two possible first steps: we may first replace
\(\partial_{t_i}^2\) by \(-\cP_i\), or we may first replace
\(\partial_{t_j}^2\) by \(-\cP_j\). This gives
$
-\partial_{t_j}^2\cP_i$
or 
$-\partial_{t_i}^2\cP_j$.
Thus the obstruction to consistency is, up to sign, the so-called \(S\)-pair
\[
\operatorname{Sp}(\cT_i,\cT_j)
:=
\partial_{t_j}^2\cT_i-\partial_{t_i}^2\cT_j=\del_{t_j}^2\cP_i-\del_{t_i}^2\cP_j.
\]

Buchberger's criterion is a systematic way to check exactly this. It says that,
to prove that a set of generators is a left Gr{\"o}bner basis, it is enough to check
the differences coming from all possible pairwise overlaps of leading monomials.
These differences are called \(S\)-pairs.
It is enough to check that the \(S\)-pair is trivial modulo the generators, in
the sense that it can be written as a left linear combination of the generators
whose terms all have leading monomial strictly smaller than the original overlap.
Equivalently, its remainder after division by the generators is zero.

In our case,
\[
\cT_i=\partial_{t_i}^2+\cP_i,
\qquad
\cS_a=\partial_{s_a}^2+\cQ_a,
\]
with \(\cP_i,\cQ_a\) of order at most \(1\). Hence the only possible overlaps are
\[
\partial_{t_i}^2\partial_{t_j}^2,
\qquad
\partial_{s_a}^2\partial_{s_b}^2,
\qquad
\partial_{t_i}^2\partial_{s_a}^2.
\]
Equivalently, we only need to check the three types of pairs
\[
(\cT_i,\cT_j),
\qquad
(\cS_a,\cS_b),
\qquad
(\cT_i,\cS_a).
\]

\medskip
\noindent
\emph{1. The pair \((\cT_i,\cT_j)\), with \(i\neq j\).}

The corresponding \(S\)-pair is
\[
\operatorname{Sp}(\cT_i,\cT_j)
:=
\partial_{t_j}^2\cT_i-\partial_{t_i}^2\cT_j.
\]
Using
\[
\partial_{t_j}^2=\cT_j-\cP_j,
\qquad
\partial_{t_i}^2=\cT_i-\cP_i,
\]
we get
\begin{align*}
\operatorname{Sp}(\cT_i,\cT_j)=
(\cT_j-\cP_j)\cT_i-(\cT_i-\cP_i)\cT_j =
[\cT_j,\cT_i]-\cP_j\cT_i+\cP_i\cT_j.
\end{align*}
By the commutator identity \eqref{eq:comm-tt-proof},
\[
\operatorname{Sp}(\cT_i,\cT_j)
=
\left(\frac{4}{\beta(t_i-t_j)^2}-\cP_j\right)\cT_i
+
\left(\cP_i-\frac{4}{\beta(t_i-t_j)^2}\right)\cT_j.
\]
This is a left linear combination of \(\cT_i\) and \(\cT_j\). Moreover, the terms
involving \(\cP_i\) and \(\cP_j\) have derivative order at most \(3\), while the
overlap
$
\partial_{t_i}^2\partial_{t_j}^2
$
has derivative order \(4\). Hence this \(S\)-pair reduces to \(0\) modulo the
generators.

\medskip
\noindent
\emph{2. The pair \((\cS_a,\cS_b)\), with \(a\neq b\).}

The corresponding \(S\)-pair is
\[
\operatorname{Sp}(\cS_a,\cS_b)
:=
\partial_{s_b}^2\cS_a-\partial_{s_a}^2\cS_b.
\]
Using
\[
\partial_{s_b}^2=\cS_b-\cQ_b,
\qquad
\partial_{s_a}^2=\cS_a-\cQ_a,
\]
we compute
\begin{align*}
\operatorname{Sp}(\cS_a,\cS_b)
&=
(\cS_b-\cQ_b)\cS_a-(\cS_a-\cQ_a)\cS_b =
[\cS_b,\cS_a]-Q_b\cS_a+\cQ_a\cS_b.
\end{align*}
By the commutator identity \eqref{eq:comm-ss-proof},
we have
\[
\operatorname{Sp}(\cS_a,\cS_b)
=
\left(\frac{\beta}{(s_a-s_b)^2}-\cQ_b\right)\cS_a
+
\left(Q_a-\frac{\beta}{(s_a-s_b)^2}\right)\cS_b.
\]
Again this is a left linear combination of the generators. The terms involving
\(\cQ_a\) and \(\cQ_b\) have derivative order at most \(3\), while the overlap
$
\partial_{s_a}^2\partial_{s_b}^2
$
has derivative order \(4\). Hence this \(S\)-pair reduces to \(0\).

\medskip
\noindent
\emph{3. The mixed pair \((\cT_i,\cS_a)\).}

The corresponding \(S\)-pair is
\[
\operatorname{Sp}(\cT_i,\cS_a)
:=
\partial_{s_a}^2\cT_i-\partial_{t_i}^2\cS_a.
\]
Using
\[
\partial_{s_a}^2=\cS_a-\cQ_a,
\qquad
\partial_{t_i}^2=\cT_i-\cP_i,
\]
we get
\begin{align*}
\operatorname{Sp}(\cT_i,\cS_a)
=
(\cS_a-\cQ_a)\cT_i-(\cT_i-\cP_i)\cS_a =
[\cS_a,\cT_i]-\cQ_a\cT_i+\cP_i\cS_a.
\end{align*}
By the mixed commutator identity \eqref{eq:comm-ts-proof},
we have
\[
\operatorname{Sp}(\cT_i,\cS_a)
=
\left(\frac{\beta}{(t_i-s_a)^2}-\cQ_a\right)\cT_i
+
\left(\cP_i-\frac{4}{\beta(t_i-s_a)^2}\right)\cS_a.
\]
This is a left linear combination of \(\cT_i\) and \(\cS_a\). The terms involving
\(\cP_i\) and \(\cQ_a\) have derivative order at most \(3\), while the overlap
$
\partial_{t_i}^2\partial_{s_a}^2
$
has derivative order \(4\). Hence this \(S\)-pair also reduces to \(0\).

\medskip

We have checked all possible overlaps. By Buchberger's criterion,
$
G:=\{\cT_1,\dots,\cT_n,\cS_1,\dots,\cS_m\}
$
is a left Gr{\"o}bner basis of \(\mathcal I\). Therefore the initial ideal is
generated by the leading monomials of these operators:
\[
\bigl(
\partial_{t_1}^2,\dots,\partial_{t_n}^2,
\partial_{s_1}^2,\dots,\partial_{s_m}^2
\bigr).
\]

A derivative monomial is standard if and only if it is not divisible by any of
these leading monomials. Equivalently, no derivative
\(\partial_{t_i}\) or \(\partial_{s_a}\) appears with exponent \(2\) or larger.
Thus the standard monomials are precisely
\[
\partial_{t_I}\partial_{s_A},
\qquad
I\subseteq\{1,\dots,n\},
\quad
A\subseteq\{1,\dots,m\}.
\]
These are exactly the square-free derivative monomials.
The normal-form theorem for Gr{\"o}bner bases now implies that the residue classes
of these standard monomials form a \(K\)-basis of \(D_K/\mathcal I\). Since
there are \(2^{n+m}\) choices of \((I,A)\), we obtain
$
\dim_K(D_K/\mathcal I)=2^{n+m}.
$
\end{proof}

\begin{proof}[Proof of the first statement of \Cref{thm:sectorial-branch-solution}]
We show that the second-order system \eqref{e:eq1}--\eqref{e:eq2} is equivalent,
near every point of \(\mathfrak X\), to a flat first-order system of size
\(2^{n+m}\).

\medskip
\noindent
\emph{Step 1: repeated derivatives can be eliminated.}

We can rewrite the equations \eqref{e:eq1}--\eqref{e:eq2} as follows:
\begin{align}
\partial_{t_i}^2F
&=
-F
-\frac{2}{\beta}\sum_{j\neq i}
\frac{\partial_{t_i}-\partial_{t_j}}{t_i-t_j}F
+\sum_{a=1}^m
\frac{\partial_{t_i}+(2/\beta)\partial_{s_a}}{t_i-s_a}F,
\label{eq:ti-second-reduction}
\\[1mm]
\partial_{s_a}^2F
&=
-\frac{\beta^2}{4}F
-\frac{\beta}{2}\sum_{b\neq a}
\frac{\partial_{s_a}-\partial_{s_b}}{s_a-s_b}F
+\sum_{i=1}^n
\frac{\partial_{s_a}+(\beta/2)\partial_{t_i}}{s_a-t_i}F.
\label{eq:sa-second-reduction}
\end{align}
Thus, whenever a derivative of $F$ contains \(\partial_{t_i}^2\) or
\(\partial_{s_a}^2\), we may replace that repeated derivative by lower-order
derivatives using \eqref{eq:ti-second-reduction} or
\eqref{eq:sa-second-reduction}.

\Cref{prop:Grobner-bulk}, says exactly that all such reduction choices are
compatible. Equivalently, every derivative of a solution can be written uniquely
as a rational linear combination of the \(2^{n+m}\) square-free derivatives
$
\partial_{t_I}\partial_{s_A}F$.

\medskip
\noindent
\emph{Step 2: introduce the finite set of unknowns.}
We recall the notation from \eqref{e:defK}, \eqref{e:defDK}, and
\eqref{e:defI}: \(K\) is the field of rational functions, \(D_K\) is the
algebra of differential operators with coefficients in \(K\), and
\(\mathcal I\) is the left ideal generated by the differential equations.

By \Cref{prop:Grobner-bulk}, the residue classes
\[
e_{I,A}:=
\bigl[\partial_{t_I}\partial_{s_A}\bigr],
\qquad
I\subseteq\{1,\dots,n\},
\quad
A\subseteq\{1,\dots,m\},
\]
form a \(K\)-basis of \(D_K/\mathcal I\). In concrete terms, this means that
every derivative can be reduced uniquely to a \(K\)-linear combination of the
square-free derivatives
$
\partial_{t_I}\partial_{s_A}F$.

\medskip
\noindent
\emph{Step 3: the square-free derivatives satisfy a first-order system.}

Since the classes \(e_{I,A}\) form a basis of \(D_K/\mathcal I\), multiplication by
any coordinate derivative can be expressed in this basis. Thus there are matrices
\(M_{t_i}\) and \(M_{s_a}\), with entries in \(K\), such that
\begin{align}\begin{split}\label{eq:M-def-clean}
\partial_{t_i} e_{I,A}
&=
\sum_{J,B}
(M_{t_i})_{(I,A),(J,B)}\,e_{J,B},
\\
\partial_{s_a} e_{I,A}
&=
\sum_{J,B}
(M_{s_a})_{(I,A),(J,B)}\,e_{J,B}.
\end{split}\end{align}
The entries of these matrices are rational functions on \(\mathfrak X\), because
they are obtained by repeatedly using the reduction rules
\eqref{eq:ti-second-reduction} and \eqref{eq:sa-second-reduction}.

Now let \(F\) be a local holomorphic solution of
\eqref{e:bulk_eq}. Define the vector of square-free derivatives
\[
U(F):=
\bigl(\partial_{t_I}\partial_{s_A}F\bigr)_{I,A}.
\]
Then \eqref{eq:M-def-clean} implies
\begin{align}\begin{split}\label{eq:pfaff}
\partial_{t_i}U(F)
&=
M_{t_i}\,U(F),
\qquad 1\le i\le n,
\\
\partial_{s_a}U(F)
&=
M_{s_a}\,U(F),
\qquad 1\le a\le m.
\end{split}
\end{align}
Thus the original scalar second-order system gives a first-order linear system
for the \(2^{n+m}\) square-free derivatives of \(F\).

\medskip
\noindent
\emph{Step 4: the first-order system is flat.}

The coordinate derivatives commute:
\[
[\partial_{t_i},\partial_{t_j}]=0,
\qquad
[\partial_{s_a},\partial_{s_b}]=0,
\qquad
[\partial_{t_i},\partial_{s_a}]=0.
\]
Therefore their induced actions on \(D_K/\mathcal I\) commute. Written in the basis
\(\{e_{I,A}\}\), this gives
\begin{align}\begin{split}\label{eq:flat}
\partial_{t_i}M_{t_j}-\partial_{t_j}M_{t_i}
&=
[M_{t_i},M_{t_j}],
\\
\partial_{s_a}M_{s_b}-\partial_{s_b}M_{s_a}
&=
[M_{s_a},M_{s_b}],
\\
\partial_{t_i}M_{s_a}-\partial_{s_a}M_{t_i}
&=
[M_{t_i},M_{s_a}],
\end{split}\end{align}
where \([A,B]=AB-BA\). These are exactly the compatibility equations for the
first-order system \eqref{eq:pfaff}. Hence the
system is flat.

\medskip
\noindent
\emph{ Step 5: local existence and dimension of the solution space.}

Fix $(\bmt_0,\bms_0)\in \mathfrak X$, and choose a sufficiently small simply connected neighborhood
$\Omega\subset \mathfrak X$ of $(\bmt_0,\bms_0)$. On $\Omega$, the matrices $M_{t_i},M_{s_a}$ are holomorphic.
Since $\Omega$ is simply connected and the  system \eqref{eq:pfaff} is flat, parallel transport is
independent of the path. Equivalently, for each initial value
\[
c\in \bC^{2^{n+m}},
\]
there is a unique holomorphic solution of \eqref{eq:pfaff} with
$U(p)=c$.

Indeed, along any smooth path $\gamma:[0,1]\to \Omega$, the restriction of $U$ satisfies
the linear ordinary differential equation
\[
\frac{\rd V}{\rd\tau}
=
\sum_{i=1}^n \dot\gamma_i(\tau)\,M_{t_i}(\gamma(\tau))\,V
+
\sum_{a=1}^m \dot\gamma_{n+a}(\tau)\,M_{s_a}(\gamma(\tau))\,V,
\qquad V(0)=c,
\]
and flatness means that the resulting transport depends only on the endpoints.
Therefore the space of local holomorphic solutions of the first-order system \eqref{eq:pfaff} has dimension
$2^{n+m}$.

It remains to check that this first-order system \eqref{eq:pfaff} is equivalent to the original
scalar system \eqref{e:bulk_eq}. We have already seen that every scalar solution \(F\) of \eqref{e:bulk_eq} produces a
solution \(U(F)\) of the first-order system \eqref{eq:pfaff}. Conversely, suppose \(U=(U_{I,A})\)
solves the first-order system \eqref{eq:pfaff}. Define
\[
F:=U_{\emptyset,\emptyset}.
\]
Since
\[
\partial_{t_i}e_{\emptyset,\emptyset}=e_{\{i\},\emptyset},
\qquad
\partial_{s_a}e_{\emptyset,\emptyset}=e_{\emptyset,\{a\}},
\]
the first-order system gives
\[
U_{\{i\},\emptyset}=\partial_{t_i}F,
\qquad
U_{\emptyset,\{a\}}=\partial_{s_a}F.
\]
Repeating the same argument for larger subsets \(I,A\), we obtain
\[
U_{I,A}=\partial_{t_I}\partial_{s_A}F
\]
for every square-free pair \((I,A)\).

Finally, the reduction rules encoded in the matrices $M_{t_i}$ and $M_{s_a}$ are precisely the equations
\(\cT_i=0\) and \(\cS_a=0\). Therefore the identities
\[
\partial_{t_i}^2\equiv -\cP_i,
\qquad
\partial_{s_a}^2\equiv -\cQ_a
\qquad
\mod \mathcal I
\]
imply
\[
\cT_iF=0,
\qquad
\cS_aF=0.
\]
Hence \(F\) is a holomorphic solution of the original system.

Therefore the space of germs at \((\bmt_0,\bms_0)\) of holomorphic solutions of
\eqref{e:bulk_eq} is naturally identified with the space of germs of
solutions of the flat first-order system. This space has dimension \(2^{n+m}\),
as claimed.
\end{proof}

\subsection{The conjugated system}\label{s:conjugate_bulk}
For each sign pattern \((\bm\epsilon,\bm\eta)\), we write
$
F=\Phi_{\bm\epsilon,\bm\eta}W$,
where \(\Phi_{\bm\epsilon,\bm\eta}\) contains the prescribed leading exponential
and algebraic factors. This separates the expected leading asymptotic behavior
from the unknown correction term. The correction \(W\) is expected to satisfy
\(W\to 1\) at infinity. In the following proposition, we conjugate the original
operators by \(\Phi_{\bm\epsilon,\bm\eta}\) and derive the resulting equations
for \(W\).

We recall $\Phi_{\bm\epsilon, \bm\eta}$ from \eqref{eq:general-branch-factor-clean}
\begin{align}\label{eq:branch_factor_copy}
\Phi_{\bm\epsilon,\bm\eta}(\bmt,\bms)
:=
\exp\!\left(
\ri\sum_{i=1}^n \epsilon_i t_i
+\ri\frac{\beta}{2}\sum_{a=1}^m \eta_a s_a
\right)
\prod_{1\le i<j\le n}(t_i-t_j)^{\alpha_{ij}}
\prod_{1\le a<b\le m}(s_a-s_b)^{\gamma_{ab}}
\prod_{i=1}^n\prod_{a=1}^m (t_i-s_a)^{\rho_{ia}},
\end{align}
and define
\begin{align}\begin{split}
\label{eq:Omega-Sigma}
\Omega_i
&:=
\sum_{j\neq i}\frac{\alpha_{ij}}{t_i-t_j}
+\sum_{a=1}^m \frac{\rho_{ia}}{t_i-s_a},
\\[1mm]
\Sigma_a
&:=
\sum_{b\neq a}\frac{\gamma_{ab}}{s_a-s_b}
+\sum_{i=1}^n \frac{\rho_{ia}}{s_a-t_i}.
\end{split}\end{align}

\begin{proposition}[Conjugated operator family]
\label{prop:conjugated-family-clean}
Let \(F=\Phi_{\bm\epsilon,\bm\eta}W\). Then the original system \eqref{e:bulk_eq}
\begin{align}
\cT_i F=0,\qquad 1\le i\le n,\qquad 
\cS_a F=0,\qquad 1\le a\le m,
\end{align}
is equivalent to
\begin{align}
(\cL_i^{(t)}+\mathcal R_i^{(t)})W&=0,
\qquad i=1,\dots,n,
\label{eq:conj-t-eq-clean}
\\[1mm]
(\cL_a^{(s)}+\mathcal R_a^{(s)})W&=0,
\qquad a=1,\dots,m,
\label{eq:conj-s-eq-clean}
\end{align}
where
\begin{align}
\cL_i^{(t)}
&=
\partial_{t_i}^2+2\ri\epsilon_i\,\partial_{t_i},
\label{eq:Lti-clean}
\\[1mm]
\mathcal R_i^{(t)}
&=
2\Omega_i\,\partial_{t_i}
+\frac{2}{\beta}\sum_{j\neq i}\frac{\partial_{t_i}-\partial_{t_j}}{t_i-t_j}
-\sum_{a=1}^m\frac{\partial_{t_i}+(2/\beta)\partial_{s_a}}{t_i-s_a}
+V_i^{(t)},
\label{eq:Rti-clean}
\\[2mm]
\cL_a^{(s)}
&=
\partial_{s_a}^2+\ri\beta\eta_a\,\partial_{s_a},
\label{eq:Lsa-clean}
\\[1mm]
\mathcal R_a^{(s)}
&=
2\Sigma_a\,\partial_{s_a}
+\frac{\beta}{2}\sum_{b\neq a}\frac{\partial_{s_a}-\partial_{s_b}}{s_a-s_b}
-\sum_{i=1}^n\frac{\partial_{s_a}+(\beta/2)\partial_{t_i}}{s_a-t_i}
+V_a^{(s)},
\label{eq:Rsa-clean}
\end{align}
and
\begin{align}
V_i^{(t)}
&=
\partial_{t_i}\Omega_i+\Omega_i^2
+\frac{2}{\beta}\sum_{j\neq i}\frac{\Omega_i-\Omega_j}{t_i-t_j}
-\sum_{a=1}^m\frac{\Omega_i+(2/\beta)\Sigma_a}{t_i-s_a},
\label{eq:Vti-simplified}
\\[1mm]
V_a^{(s)}
&=
\partial_{s_a}\Sigma_a+\Sigma_a^2
+\frac{\beta}{2}\sum_{b\neq a}\frac{\Sigma_a-\Sigma_b}{s_a-s_b}
-\sum_{i=1}^n\frac{\Sigma_a+(\beta/2)\Omega_i}{s_a-t_i}.
\label{eq:Vsa-simplified}
\end{align}
\end{proposition}

\begin{proof}
Let
$
\Phi:=\Phi_{\bm\epsilon,\bm\eta}$
be as in \eqref{eq:branch_factor_copy}. Since \(\Phi\) is nowhere vanishing on
$\mathfrak X$, the equations
\[
\cT_iF=0,
\qquad
\cS_aF=0,
\qquad
F=\Phi W,
\]
are equivalent to
\[
\widetilde{\cT}_i W=0,
\qquad
\widetilde{\cS}_a W=0,
\]
where
\[
\widetilde{\cT}_i:=\Phi^{-1}\cT_i\Phi,
\qquad
\widetilde{\cS}_a:=\Phi^{-1}\cS_a\Phi.
\]

We recall \(\Omega_i\) and \(\Sigma_a\) from \eqref{eq:Omega-Sigma} and set
\begin{align}\label{e:mui}
\mu_i:=\partial_{t_i}\log\Phi
=
\ri\epsilon_i+\Omega_i,
\qquad
\nu_a:=\partial_{s_a}\log\Phi
=
\ri\frac{\beta}{2}\eta_a+\Sigma_a.
\end{align}
Then
\begin{align}\label{e:first}
\Phi^{-1}\partial_{t_i}\Phi
=
\partial_{t_i}+\mu_i,
\qquad
\Phi^{-1}\partial_{s_a}\Phi
=
\partial_{s_a}+\nu_a,
\end{align}
and
\begin{align}
\begin{split}\label{e:second}
\Phi^{-1}\partial_{t_i}^2\Phi
&=
(\partial_{t_i}+\mu_i)^2
=
\partial_{t_i}^2
+2\mu_i\partial_{t_i}
+\partial_{t_i}\mu_i+\mu_i^2,
\\
\Phi^{-1}\partial_{s_a}^2\Phi
&=
(\partial_{s_a}+\nu_a)^2
=
\partial_{s_a}^2
+2\nu_a\partial_{s_a}
+\partial_{s_a}\nu_a+\nu_a^2.
\end{split}
\end{align}

We first compute \(\widetilde{\cT}_i\).  A substitution of \eqref{e:first} and
\eqref{e:second} into \(\cT_i\), recalled in \eqref{eq:Ti-def},  gives
\begin{align}
\widetilde{\cT}_i
=
(\partial_{t_i}+\mu_i)^2+1
+\frac{2}{\beta}\sum_{j\neq i}
\frac{(\partial_{t_i}+\mu_i)-(\partial_{t_j}+\mu_j)}{t_i-t_j}
-\sum_{a=1}^m
\frac{(\partial_{t_i}+\mu_i)+(2/\beta)(\partial_{s_a}+\nu_a)}
{t_i-s_a}.
\label{eq:Ttilde-expanded-clean}
\end{align}
Using \(\mu_i=\ri\epsilon_i+\Omega_i\) from \eqref{e:mui}, the derivative terms in
\eqref{eq:Ttilde-expanded-clean} are
\begin{align}\label{e:derivative_part}
\partial_{t_i}^2
+2\ri\epsilon_i\partial_{t_i}
+2\Omega_i\partial_{t_i}
+\frac{2}{\beta}\sum_{j\neq i}
\frac{\partial_{t_i}-\partial_{t_j}}{t_i-t_j}
-\sum_{a=1}^m
\frac{\partial_{t_i}+(2/\beta)\partial_{s_a}}{t_i-s_a}.
\end{align}
The scalar part is
\begin{align}
\begin{split}\label{eq:Vti-hat-clean}
&\partial_{t_i}\Omega_i
+
(\ri\epsilon_i+\Omega_i)^2
+1
+\frac{2}{\beta}\sum_{j\neq i}
\frac{\ri(\epsilon_i-\epsilon_j)+\Omega_i-\Omega_j}{t_i-t_j}
-\sum_{a=1}^m
\frac{\ri(\epsilon_i+\eta_a)+\Omega_i+(2/\beta)\Sigma_a}{t_i-s_a}
\\
&=
(1-\epsilon_i^2)
+\ri\left(
2\epsilon_i\Omega_i
+\frac{2}{\beta}\sum_{j\neq i}
\frac{\epsilon_i-\epsilon_j}{t_i-t_j}
-\sum_{a=1}^m
\frac{\epsilon_i+\eta_a}{t_i-s_a}
\right)
+V_i^{(t)},
\end{split}
\end{align}
where we have separated the terms involving only \(\Omega_i,\Sigma_a\) from the
terms involving \(\epsilon_i,\eta_a\).

Since \(\epsilon_i^2=1\), the constant term \(1-\epsilon_i^2\) in
\eqref{eq:Vti-hat-clean} vanishes. For the remaining term in parentheses,
substitute the definition of \(\Omega_i\) from \eqref{eq:Omega-Sigma}:
\begin{align*}
2\epsilon_i\Omega_i
+\frac{2}{\beta}\sum_{j\neq i}
\frac{\epsilon_i-\epsilon_j}{t_i-t_j}
-\sum_{a=1}^m
\frac{\epsilon_i+\eta_a}{t_i-s_a}
=
\sum_{j\neq i}
\frac{2\epsilon_i\alpha_{ij}+\frac{2}{\beta}(\epsilon_i-\epsilon_j)}
{t_i-t_j}
+
\sum_{a=1}^m
\frac{2\epsilon_i\rho_{ia}-(\epsilon_i+\eta_a)}
{t_i-s_a}
=0.
\end{align*}
This vanishes by the defining relations \eqref{eq:alpha-gamma-rho-clean} for the
exponents \(\alpha_{ij}\) and \(\rho_{ia}\). Hence the scalar part
\eqref{eq:Vti-hat-clean} is \(V_i^{(t)}\), and \eqref{eq:conj-t-eq-clean}
follows by combining \eqref{eq:Ttilde-expanded-clean} and
\eqref{e:derivative_part}.

The computation for \(\widetilde{\cS}_a\) is analogous. Substituting
\eqref{e:first} and \eqref{e:second} into \(\cS_a\), recalled in
\eqref{eq:Sa-def}, we obtain
\begin{align}
\widetilde{\cS}_a
=
(\partial_{s_a}+\nu_a)^2+\frac{\beta^2}{4}
+\frac{\beta}{2}\sum_{b\neq a}
\frac{(\partial_{s_a}+\nu_a)-(\partial_{s_b}+\nu_b)}{s_a-s_b}
-\sum_{i=1}^n
\frac{(\partial_{s_a}+\nu_a)+(\beta/2)(\partial_{t_i}+\mu_i)}
{s_a-t_i}.
\label{eq:Stilde-expanded-clean}
\end{align}
Using
$
\nu_a=\ri\beta\eta_a/2+\Sigma_a
$
from \eqref{e:mui}, the derivative terms in \eqref{eq:Stilde-expanded-clean} are
\begin{align}\label{e:der_part}
\partial_{s_a}^2
+\ri\beta\eta_a\,\partial_{s_a}
+2\Sigma_a\partial_{s_a}
+\frac{\beta}{2}\sum_{b\neq a}
\frac{\partial_{s_a}-\partial_{s_b}}{s_a-s_b}
-\sum_{i=1}^n
\frac{\partial_{s_a}+(\beta/2)\partial_{t_i}}{s_a-t_i}.
\end{align}
The scalar part is
\begin{align}
\begin{split}\label{eq:Vsa-hat-clean}
&\partial_{s_a}\Sigma_a
+
\left(\ri\frac{\beta}{2}\eta_a+\Sigma_a\right)^2
+\frac{\beta^2}{4}
+\frac{\beta}{2}\sum_{b\neq a}
\frac{\ri\frac{\beta}{2}(\eta_a-\eta_b)+\Sigma_a-\Sigma_b}{s_a-s_b}
-\sum_{i=1}^n
\frac{\ri\frac{\beta}{2}(\eta_a+\epsilon_i)+\Sigma_a+(\beta/2)\Omega_i}
{s_a-t_i}
\\
&=
\frac{\beta^2}{4}(1-\eta_a^2)
+\ri\left(
\beta\eta_a\Sigma_a
+
\frac{\beta^2}{4}\sum_{b\neq a}
\frac{\eta_a-\eta_b}{s_a-s_b}
-
\frac{\beta}{2}\sum_{i=1}^n
\frac{\eta_a+\epsilon_i}{s_a-t_i}
\right)
+V_a^{(s)}.
\end{split}
\end{align}
Since \(\eta_a^2=1\), the constant term vanishes. For the remaining term in
parentheses, substitute the definition of \(\Sigma_a\) from \eqref{eq:Omega-Sigma}:
\begin{align*}
\beta\eta_a\Sigma_a
+
\frac{\beta^2}{4}\sum_{b\neq a}
\frac{\eta_a-\eta_b}{s_a-s_b}
-
\frac{\beta}{2}\sum_{i=1}^n
\frac{\eta_a+\epsilon_i}{s_a-t_i}
=
\sum_{b\neq a}
\frac{\beta\eta_a\gamma_{ab}+\frac{\beta^2}{4}(\eta_a-\eta_b)}
{s_a-s_b}
+
\sum_{i=1}^n
\frac{\beta\eta_a\rho_{ia}-\frac{\beta}{2}(\eta_a+\epsilon_i)}
{s_a-t_i}
=0.
\end{align*}
Again, this vanishes by the defining relations \eqref{eq:alpha-gamma-rho-clean}
for the exponents \(\gamma_{ab}\) and \(\rho_{ia}\). Hence the scalar part
\eqref{eq:Vsa-hat-clean} is \(V_a^{(s)}\), and \eqref{eq:conj-s-eq-clean}
follows by combining \eqref{eq:Stilde-expanded-clean} and \eqref{e:der_part}.
This proves the proposition.
\end{proof}

\subsection{Solutions attached to the branch factors}
\label{s:bulk_construct_solution}

In this section, we prove the second statement of
\Cref{thm:sectorial-branch-solution}. The proof reduces the conjugated system
\eqref{eq:conj-t-eq-clean}--\eqref{eq:conj-s-eq-clean} to a one-dimensional
problem along a carefully chosen shifted ray, constructed in \Cref{p:direction}.
In \Cref{prop:radial-block}, we show that along this ray the scaled
square-free derivatives satisfy a block first-order system in the radial
variable \(\lambda\). The non-empty derivative block has a positive spectral
gap, which allows us to impose the condition that this block decays at infinity
by solving a Volterra integral equation. The Volterra operator is a contraction, and hence gives a unique solution of the
first-order system with the desired asymptotic estimates. Translating these
estimates back from the scaled derivative system gives the desired estimates
for \(W\) and its square-free derivatives.

\begin{proposition}\label{p:direction}
Let \(\fC\) be the domain \eqref{e:bulkC}. Fix
\[
\bm\sigma=(\sigma_1,\dots,\sigma_{n+m})\in\{\pm1\}^{n+m},
\]
and define
\[
\bm\omega=(\omega_1,\dots,\omega_{n+m}),
\qquad
\omega_j=j+\ri\sigma_j,
\qquad
1\leq j\leq n+m.
\]
Then, for every
\[
\bm z=(z_1,\dots,z_{n+m})\in\fC
\]
and every \(r\geq0\), we have $\bm z+r\bm\omega\in\fC$ and 
\begin{align}
R\bigl(\bm z+r\bm\omega\bigr)
&\geq
\frac14\bigl(R(\bm z)+r\bigr),
\label{e:Rlowerbb}
\end{align}
where \(R(\bm z)\) is defined in \eqref{e:defRbulk}.
\end{proposition}

\begin{proof}
For \(1\leq \ell<n+m\), we have
\[
\begin{aligned}
&\Re\left[
\bigl(z_{\ell+1}+r\omega_{\ell+1}\bigr)
-
\bigl(z_\ell+r\omega_\ell\bigr)
\right]=
\Re[z_{\ell+1}-z_\ell]+r
>
D.
\end{aligned}
\]
Thus \(\bm z+r\bm\omega\in\fC\).

Fix \(1\leq i<j\leq n+m\). Since \(j-i\geq1\),
\begin{align*}
&
\left|
\bigl(z_j+r\omega_j\bigr)
-
\bigl(z_i+r\omega_i\bigr)
\right|
\geq
\Re[z_j-z_i]+r(j-i)
\geq r.
\end{align*}
On the other hand,
\begin{align*}
\left|
\bigl(z_j+r\omega_j\bigr)
-
\bigl(z_i+r\omega_i\bigr)
\right|
=
\left|
z_j-z_i+r(j-i)+\ri r(\sigma_j-\sigma_i)
\right|
\geq
\left|z_j-z_i+r(j-i)\right|
-r|\sigma_j-\sigma_i|.
\end{align*}
Since \(\Re[z_j-z_i]>0\) and \(r(j-i)\geq0\), adding \(r(j-i)\) to
the real part cannot decrease the modulus. Moreover,
\(|\sigma_j-\sigma_i|\leq2\). Therefore
\[
\left|
\bigl(z_j+r\omega_j\bigr)
-
\bigl(z_i+r\omega_i\bigr)
\right|
\geq
|z_j-z_i|-2r
\geq
R(\bm z)-2r.
\]
Combining the two lower bounds gives
\begin{align*}
&
\left|
\bigl(z_j+r\omega_j\bigr)
-
\bigl(z_i+r\omega_i\bigr)
\right|
\geq
\max\left\{r,\;R(\bm z)-2r\right\}
\geq
\frac34r+\frac14\bigl(R(\bm z)-2r\bigr)
=
\frac14\bigl(R(\bm z)+r\bigr).
\end{align*}
Taking the minimum over \(1\leq i<j\leq n+m\) proves
\eqref{e:Rlowerbb}.
\end{proof}

\begin{proposition}[Radial block system along shifted rays]
\label{prop:radial-block}
Let \(\fC\) be the domain \eqref{e:bulkC}, fix a sign pattern
\[
(\bm\epsilon,\bm\eta)=\bm\sigma
=(\epsilon_1,\dots,\epsilon_n,\eta_1,\dots,\eta_m)
=(\sigma_1,\dots,\sigma_{n+m})\in\{\pm1\}^{n+m},
\]
and let
\[
\bm z=(\bmt,\bms)\in\fC,
\qquad
R_0\leq R(\bm z)\leq 2R_0.
\]
Let \(\bm\omega\) be as in \Cref{p:direction}, and define
\[
\gamma_{\bm z,\bm\omega}(\lambda)
:=
\bm z+(\lambda-R_0)\bm\omega,
\qquad
\lambda\in[R_0,\infty).
\]

Let \(\mathcal J\) be the set of square-free index pairs
\[
(I,A),
\qquad
I\subseteq\{1,\dots,n\},
\quad
A\subseteq\{1,\dots,m\},
\]
and write
\[
\mathcal J^\ast:=\mathcal J\setminus\{(\varnothing,\varnothing)\}.
\]
Order the components indexed by \(\mathcal J^\ast\) by increasing
\(|I|+|A|\).

Let \(W\) be a holomorphic solution of the conjugated bulk system
\eqref{eq:conj-t-eq-clean}--\eqref{eq:conj-s-eq-clean}. Define
\[
Y_{I,A}(\lambda;\bm z)
:=
\lambda^{|I|+|A|}
\partial_{t_I}\partial_{s_A}W
\bigl(\gamma_{\bm z,\bm\omega}(\lambda)\bigr),
\qquad
(I,A)\in\mathcal J,
\]
and write
\[
Y=(Y_0,Y_\ast),
\qquad
Y_0(\lambda;\bm z)
:=
W\bigl(\gamma_{\bm z,\bm\omega}(\lambda)\bigr),
\]
where \(Y_\ast\) collects the components corresponding to
\((I,A)\in\mathcal J^\ast\).

Then there exist matrix-valued functions
\[
C_{01}(\lambda;\bm z),
\qquad
C_{10}(\lambda;\bm z),
\qquad
C_{11}(\lambda;\bm z),
\]
and a matrix \(\Lambda(\lambda;\bm z)\) on the non-empty derivative block,
whose entries are uniformly bounded for \(\lambda\ge R_0\), such that
\begin{align}
\frac{\rd}{\rd\lambda}Y_0
&=
\frac1\lambda C_{01}(\lambda;\bm z)Y_\ast,
\label{eq:radial-top}
\\[1mm]
\frac{\rd}{\rd\lambda}Y_\ast
&=
\Lambda(\lambda;\bm z)Y_\ast
+
\frac1\lambda C_{10}(\lambda;\bm z)Y_0
+
\frac1\lambda C_{11}(\lambda;\bm z)Y_\ast.
\label{eq:radial-bottom}
\end{align}
The matrix \(\Lambda(\lambda;\bm z)\) is block lower triangular with respect
to the ordering by \(|I|+|A|\). More precisely, for
\((I,A),(J,B)\in\mathcal J^\ast\) with \((I,A)\ne(J,B)\),
\begin{equation}
\Lambda_{(I,A),(J,B)}(\lambda;\bm z)=0
\qquad\text{unless}\qquad
|J|+|B|<|I|+|A|.
\label{eq:lambda-triangular}
\end{equation}
Its diagonal entries are independent of \(\lambda\) and \(\bm z\), and are
given by
\begin{equation}
\Lambda_{(I,A),(I,A)}(\lambda;\bm z)
=
\Lambda_{I,A}(\bm\omega)
:=
-2\ri\sum_{i\in I}\epsilon_i\omega_i
-
\beta\ri\sum_{a\in A}\eta_a\omega_{n+a}.
\label{eq:lambda-IA}
\end{equation}
In particular,
\begin{equation}
\Re\Lambda_{I,A}(\bm\omega)
=
2|I|+\beta|A|,
\qquad
(I,A)\in\mathcal J^\ast.
\label{eq:spectral-gap}
\end{equation}
\end{proposition}

\begin{proof}
By \Cref{p:direction}, along the shifted ray
\(\gamma_{\bm z,\bm\omega}\) we have
\[
R\bigl(\gamma_{\bm z,\bm\omega}(\lambda)\bigr)
\gtrsim\lambda,\quad \lambda\geq R_0.
\]
\emph{Step 1: The equation for \(Y_0\).}
First we notice that
\[
Y_0(\lambda;\bm z)
=
W\bigl(\gamma_{\bm z,\bm\omega}(\lambda)\bigr).
\]
By the chain rule,
\[
\frac{\rd}{\rd\lambda}Y_0
=
\sum_{i=1}^n
\omega_i\,\partial_{t_i}W
\bigl(\gamma_{\bm z,\bm\omega}(\lambda)\bigr)
+
\sum_{a=1}^m
\omega_{n+a}\,\partial_{s_a}W
\bigl(\gamma_{\bm z,\bm\omega}(\lambda)\bigr).
\]
Since
\[
Y_{\{i\},\varnothing}
=
\lambda\partial_{t_i}W
\bigl(\gamma_{\bm z,\bm\omega}(\lambda)\bigr),
\qquad
Y_{\varnothing,\{a\}}
=
\lambda\partial_{s_a}W
\bigl(\gamma_{\bm z,\bm\omega}(\lambda)\bigr),
\]
we obtain
\[
\frac{\rd}{\rd\lambda}Y_0
=
\frac{1}{\lambda}
\left(
\sum_{i=1}^n \omega_i\,Y_{\{i\},\varnothing}
+
\sum_{a=1}^m \omega_{n+a}\,Y_{\varnothing,\{a\}}
\right).
\]
This is \eqref{eq:radial-top}. The entries of \(C_{01}\) are either \(0\) or
one of the components of \(\bm\omega\), hence are uniformly bounded.

\medskip

\emph{Step 2: The equations for the non-empty square-free derivatives.}

Fix a non-empty square-free pair
\[
(I,A)\in\mathcal J^\ast.
\]
By the chain rule,
\begin{align}
\begin{split}\label{eq:radial-chain}
\frac{\rd}{\rd\lambda}Y_{I,A}
=
\frac{|I|+|A|}{\lambda}Y_{I,A}
+
\lambda^{|I|+|A|}
\left(
\sum_{i=1}^n
\omega_i\,
\partial_{t_i}\partial_{t_I}\partial_{s_A}W
+
\sum_{a=1}^m
\omega_{n+a}\,
\partial_{s_a}\partial_{t_I}\partial_{s_A}W
\right)
\bigl(\gamma_{\bm z,\bm\omega}(\lambda)\bigr).
\end{split}
\end{align}

If \(i\notin I\), then
\[
\partial_{t_i}\partial_{t_I}\partial_{s_A}W
=
\partial_{t_{I\cup\{i\}}}\partial_{s_A}W
\]
is still square-free. Therefore
\[
\lambda^{|I|+|A|}
\partial_{t_i}\partial_{t_I}\partial_{s_A}W
\bigl(\gamma_{\bm z,\bm\omega}(\lambda)\bigr)
=
\frac{1}{\lambda}
Y_{I\cup\{i\},A}(\lambda;\bm z).
\]
Similarly, if \(a\notin A\), then
\[
\lambda^{|I|+|A|}
\partial_{s_a}\partial_{t_I}\partial_{s_A}W
\bigl(\gamma_{\bm z,\bm\omega}(\lambda)\bigr)
=
\frac{1}{\lambda}
Y_{I,A\cup\{a\}}(\lambda;\bm z).
\]
These terms contribute only to the \(\lambda^{-1}\)-part of
\eqref{eq:radial-bottom}.

Next suppose \(i\in I\). Write
$
I'=I\setminus\{i\}$.
Then the relevant derivative is
\begin{equation}\label{e:two_der}
\lambda^{|I|+|A|}
\partial_{t_i}^2\partial_{t_{I'}}\partial_{s_A}W.
\end{equation}
We recall the conjugated equations from \Cref{prop:conjugated-family-clean}:
\begin{align}
\partial_{t_i}^2W
&=
-2\ri\epsilon_i\,\partial_{t_i}W
-\mathcal R_i^{(t)}W,
\label{eq:radial-reduce-t}
\\
\partial_{s_a}^2W
&=
-\ri\beta\eta_a\,\partial_{s_a}W
-\mathcal R_a^{(s)}W.
\label{eq:radial-reduce-s}
\end{align}
Using \eqref{eq:radial-reduce-t}, we rewrite \eqref{e:two_der} as
\begin{align}
\begin{split}\label{e:two_der2}
\lambda^{|I|+|A|}
\partial_{t_i}^2\partial_{t_{I'}}\partial_{s_A}W
&=
\lambda^{|I|+|A|}
\partial_{t_{I'}}\partial_{s_A}
\left(
-2\ri\epsilon_i\,\partial_{t_i}
-\mathcal R_i^{(t)}
\right)W
\\
&=
-2\ri\epsilon_i\,
\lambda^{|I|+|A|}
\partial_{t_I}\partial_{s_A}W
-
\lambda^{|I|+|A|}
\partial_{t_{I'}}\partial_{s_A}
\bigl(\mathcal R_i^{(t)}W\bigr).
\end{split}
\end{align}
Multiplying by \(\omega_i\), the leading term in \eqref{e:two_der2} contributes
\begin{equation}\label{e:leading1}
-2\ri\epsilon_i\omega_i\,Y_{I,A}.
\end{equation}

We now estimate the second term on the right-hand side of \eqref{e:two_der2}.
The coefficients
\begin{equation}\label{e:fraction}
\frac{1}{t_i-t_j},
\qquad
\frac{1}{s_a-s_b},
\qquad
\frac{1}{t_i-s_a}
\end{equation}
have degree \(-1\). From \eqref{eq:Omega-Sigma}, the functions
\(\Omega_i\) and \(\Sigma_a\) also have degree \(-1\). Hence, in
\eqref{eq:Rti-clean} and \eqref{eq:Rsa-clean}, the coefficients of first-order
derivatives have degree \(-1\). Finally, the formulas
\eqref{eq:Vti-simplified} and \eqref{eq:Vsa-simplified} show directly that
\(V_i^{(t)}\) and \(V_a^{(s)}\) are sums of terms of degree \(-2\). Therefore, along the
shifted ray, for square-free pairs \((J,B)\),
\begin{align}\label{eq:first-order-coefficient-size}
\partial_{t_J}\partial_{s_B}
\bigl(\text{a first-order coefficient}\bigr)
=
\OO\bigl(\lambda^{-1-|J|-|B|}\bigr),
\quad
\partial_{t_J}\partial_{s_B}V_i^{(t)},
\quad
\partial_{t_J}\partial_{s_B}V_a^{(s)}
=
\OO\bigl(\lambda^{-2-|J|-|B|}\bigr).
\end{align}
Here we used \Cref{p:direction}, which gives
\(R(\gamma_{\bm z,\bm\omega}(\lambda))\ge\lambda/4\).

Consider first a term obtained from a first-order part of
\(\mathcal R_i^{(t)}\). Suppose that, after applying Leibniz' rule, exactly
\(r\) of the derivatives in
\(\partial_{t_{I'}}\partial_{s_A}\) fall on its coefficient. The coefficient
is then \(\OO(\lambda^{-1-r})\), while the differential monomial acting on
\(W\) has total order \(|I|+|A|-r\). If this monomial is square-free, then,
after passing to the scaled variables, its coefficient is
\[
\lambda^{|I|+|A|}
\lambda^{-1-r}
\lambda^{-(|I|+|A|-r)}
=
\lambda^{-1}.
\]
Hence it belongs to the \(\lambda^{-1}C_{11}Y_\ast\) part of
\eqref{eq:radial-bottom}.

If the resulting monomial is not square-free, it contains exactly one
repeated variable. Suppose first that the repeated variable is \(t_j\). We
use
\[
\partial_{t_j}^2W
=
-2\ri\epsilon_j\partial_{t_j}W
-
\mathcal R_j^{(t)}W.
\]
The principal term
\(-2\ri\epsilon_j\partial_{t_j}W\) removes the repetition without supplying
an additional factor \(\lambda^{-1}\). The resulting square-free derivative
has total order \(|I|+|A|-r-1\), and its scaled coefficient is
\[
\lambda^{|I|+|A|}
\lambda^{-1-r}
\lambda^{-(|I|+|A|-r-1)}
=1.
\]
Since
\[
|I|+|A|-r-1<|I|+|A|,
\]
this gives a uniformly bounded coupling from a strictly lower-order
square-free derivative into the equation for \(Y_{I,A}\). The case in which
the repeated variable is \(s_b\) is identical, using
\eqref{eq:radial-reduce-s}.

If instead we use the corresponding remainder
\(\mathcal R_j^{(t)}W\), or \(\mathcal R_b^{(s)}W\) when the repeated variable
is \(s_b\), every first-order coefficient supplies one additional factor
\(\lambda^{-1}\), while the total order of the derivative on \(W\) decreases
by one. Derivatives falling on that new coefficient increase its decay and
decrease the remaining derivative order by the same amount. A scalar
coefficient supplies two powers of \(\lambda^{-1}\), and the resulting
differential monomial is square-free. Thus the sum
\[
\text{decay order of the coefficient}
+
\text{total order of the derivative on }W
\]
remains equal to \(|I|+|A|+1\). The new differential monomial is either
square-free, in which case its scaled coefficient is \(\OO(\lambda^{-1})\),
or it contains again exactly one repeated variable. In the latter case we
repeat the same reduction. Each use of a remainder decreases the total order
of the derivative on \(W\), so the procedure terminates after finitely many
steps.

There are only two possible terminal cases. If no principal term is used at
the last step, the final coefficient decay order and the final square-free
derivative order add up to \(|I|+|A|+1\), and hence the scaled coefficient is
\(\OO(\lambda^{-1})\). If a principal term is used, it removes the unique
repetition and lowers the derivative order by one without changing the
coefficient decay. The final coefficient decay order and the final derivative
order then add up to \(|I|+|A|\); the scaled coefficient is \(\OO(1)\), and
the final square-free derivative has strictly smaller total order than
\(|I|+|A|\).

A term obtained directly from the scalar part \(V_i^{(t)}\) is simpler.
If exactly \(r\) derivatives in
\(\partial_{t_{I'}}\partial_{s_A}\) fall on the scalar coefficient, then the
coefficient is \(\OO(\lambda^{-2-r})\), while the derivative acting on \(W\)
has total order \(|I|+|A|-1-r\). Hence its scaled coefficient is
\[
\lambda^{|I|+|A|}
\lambda^{-2-r}
\lambda^{-(|I|+|A|-1-r)}
=
\lambda^{-1}.
\]
Thus all scalar terms contribute to the \(\lambda^{-1}C_{10}Y_0\) or
\(\lambda^{-1}C_{11}Y_\ast\) part of \eqref{eq:radial-bottom}.

Similarly, if \(a\in A\) and
\[
A'=A\setminus\{a\},
\]
then
\[
\lambda^{|I|+|A|}
\partial_{s_a}^2\partial_{t_I}\partial_{s_{A'}}W
\]
is reduced using \eqref{eq:radial-reduce-s}. Its leading contribution, before
multiplication by \(\omega_{n+a}\), is
\[
-\ri\beta\eta_a\,
\lambda^{|I|+|A|}
\partial_{t_I}\partial_{s_A}W.
\]
After multiplying by \(\omega_{n+a}\), this contributes
\begin{equation}\label{e:leading2}
-\beta\ri\eta_a\omega_{n+a}\,Y_{I,A}.
\end{equation}
Again, all remaining terms come from \(\mathcal R_a^{(s)}\), and the same
conclusion as above holds.

We have therefore proved the following dichotomy for every term produced by
\(\mathcal R_i^{(t)}\) or \(\mathcal R_a^{(s)}\):
\begin{itemize}
\item it is a uniformly bounded multiple of a square-free derivative of
strictly smaller total order; or
\item it carries a factor \(\lambda^{-1}\), after which its coefficient is
uniformly bounded.
\end{itemize}
The first class gives the strict block-lower-triangular part of
\(\Lambda(\lambda;\bm z)\), while the second class gives
\(C_{10}\) and \(C_{11}\). Notice also that an \(\OO(1)\) term cannot involve
\(Y_0\), because a principal reduction always leaves at least one derivative
of \(W\).

Combining the leading contributions \eqref{e:leading1} and
\eqref{e:leading2} for all \(i\in I\) and \(a\in A\), the diagonal coefficient
of \(Y_{I,A}\) is
\[
\Lambda_{I,A}(\bm\omega)
:=
-2\ri\sum_{i\in I}\epsilon_i\omega_i
-\beta\ri\sum_{a\in A}\eta_a\omega_{n+a}.
\]
This gives \eqref{eq:radial-bottom}, \eqref{eq:lambda-triangular}, and
\eqref{eq:lambda-IA}. Uniform boundedness follows from
\eqref{eq:first-order-coefficient-size},
the fact that there are only finitely many square-free derivatives, and the
uniform boundedness of the components of \(\bm\omega\).

\medskip

\emph{Step 3: The real parts of the diagonal entries.}

It remains to compute the real part of the diagonal entries. By the construction
of \(\bm\omega\) in \eqref{p:direction},
\[
\Im[\omega_i]=\epsilon_i,
\qquad 1\le i\le n,
\qquad
\Im[\omega_{n+a}]=\eta_a,
\qquad 1\le a\le m.
\]
Hence
\[
\Re[-2\ri\epsilon_i\omega_i]=2\epsilon_i\Im[\omega_i]=2,
\]
and
\[
\Re[-\beta\ri\eta_a\omega_{n+a}]
=
\beta\eta_a\Im[\omega_{n+a}]
=
\beta.
\]
Therefore
$
\Re[\Lambda_{I,A}(\bm\omega)]
=
2|I|+\beta |A|,
$
which proves \eqref{eq:spectral-gap}.
\end{proof}

\begin{proof}[Proof of the second statement in \Cref{thm:sectorial-branch-solution}]

Fix first \(\bm z=(\bmt,\bms)\in\fC\), set \(R_0:=R(\bm z)\), let \(\bm\omega\) be as in \eqref{p:direction},
 and consider
the shifted ray
\[
\gamma_{\bm z,\bm\omega}(\lambda)
:=
\bm z+(\lambda-R_0)\bm\omega,
\qquad \lambda\in[R_0,\infty).
\]
Along this ray, \Cref{prop:radial-block} gives
\begin{align}
\frac{\rd}{\rd\lambda}Y_0
&=
\frac1\lambda C_{01}(\lambda;\bm z)Y_\ast,
\label{eq:thm-radial-top-clean}
\\
\frac{\rd}{\rd\lambda}Y_\ast
&=
\Lambda(\lambda;\bm z)Y_\ast
+
\frac1\lambda C_{10}(\lambda;\bm z)Y_0
+
\frac1\lambda C_{11}(\lambda;\bm z)Y_\ast.
\label{eq:thm-radial-bottom-clean}
\end{align}
The matrices in these equations are uniformly bounded.

\medskip
\noindent
\emph{Step 1: Exponential decay of the principal backward evolution.}
Set
\[
c_0:=\min\{2,\beta\}>0.
\]
Recall from \eqref{eq:lambda-IA} that the diagonal entries of \(\Lambda(\lambda;\bmz)\) are independent
of \(\lambda\) and are given by
\[
\Lambda_{I,A}(\bm\omega)
=
-2\ri\sum_{i\in I}\epsilon_i\omega_i
-\beta\ri\sum_{a\in A}\eta_a\omega_{n+a}.
\]
By \eqref{eq:spectral-gap},
\[
\Re\Lambda_{I,A}(\bm\omega)
=
2|I|+\beta|A|
\ge c_0,
\qquad
(I,A)\in\mathcal J^\ast.
\]
Consequently, the backward evolution generated by the diagonal matrix
\(\Lambda(\bm\omega)\) satisfies
\begin{equation}\label{eq:diagonal-backward-bound}
\left\|
e^{-(\sigma-\lambda)\Lambda(\bm\omega)}
\right\|
\le
e^{-c_0(\sigma-\lambda)},
\qquad
\sigma\ge\lambda\ge R_0.
\end{equation}

Let \(\mathcal E(\lambda,\sigma;\bmz)\) denote the backward evolution associated
with the full triangular matrix \(\Lambda(\lambda;\bmz)\), namely
\begin{align}
\partial_\lambda\mathcal E(\lambda,\sigma;\bmz)
&=
\Lambda(\lambda;\bmz)\mathcal E(\lambda,\sigma;\bmz),
\qquad
\mathcal E(\sigma,\sigma;\bmz)=I.
\label{eq:backward-evolution-def}
\end{align}
The matrix
\[
\Lambda(\lambda;\bmz)-\Lambda(\bm\omega)
\]
has vanishing diagonal entries. Moreover, by the triangularity statement in
\Cref{prop:radial-block},
\begin{equation}\label{eq:strict-degree-increase}
\bigl(
\Lambda(\lambda;\bmz)-\Lambda(\bm\omega)
\bigr)_{(I,A),(J,B)}
=0
\qquad\text{unless}\qquad
|J|+|B|<|I|+|A|.
\end{equation}
Therefore every product containing \(n+m\) factors of
\(\Lambda(\lambda;\bmz)-\Lambda(\bm\omega)\) vanishes, even when the factors are
evaluated at different values of \(\lambda\) and are separated by diagonal
evolution factors.

Repeated variation of constants around the diagonal evolution therefore gives
the finite expansion
\begin{align}\begin{split}
&\mathcal E(\lambda,\sigma;\bmz)
=
e^{-(\sigma-\lambda)\Lambda(\bm\omega)}+
\sum_{r=1}^{n+m-1}(-1)^r
\int_{\lambda\le\tau_1\le\cdots\le\tau_r\le\sigma}
e^{-(\tau_1-\lambda)\Lambda(\bm\omega)}
\bigl(
\Lambda(\tau_1;\bmz)-\Lambda(\bm\omega)
\bigr)
\\
&\times
e^{-(\tau_2-\tau_1)\Lambda(\bm\omega)}
\bigl(
\Lambda(\tau_2;\bmz)-\Lambda(\bm\omega)
\bigr)
\cdots
e^{-(\tau_r-\tau_{r-1})\Lambda(\bm\omega)}
\bigl(
\Lambda(\tau_r;\bmz)-\Lambda(\bm\omega)
\bigr)
e^{-(\sigma-\tau_r)\Lambda(\bm\omega)}
\,\rd\tau_1\cdots\rd\tau_r.
\label{eq:finite-variation-expansion}
\end{split}\end{align}
Here, when \(r=1\), the intermediate diagonal factors are absent.

After increasing \(M\) if necessary, the uniform boundedness of the
off-diagonal entries gives
\[
\left\|
\Lambda(\lambda;\bmz)-\Lambda(\bm\omega)
\right\|
\le M,
\qquad
\lambda\ge R_0.
\]
For the term in \eqref{eq:finite-variation-expansion} containing \(r\)
off-diagonal factors, \eqref{eq:diagonal-backward-bound} gives
\begin{align*}
&
\left\|
e^{-(\tau_1-\lambda)\Lambda(\bm\omega)}
\right\|
\left\|
e^{-(\tau_2-\tau_1)\Lambda(\bm\omega)}
\right\|
\cdots
\left\|
e^{-(\sigma-\tau_r)\Lambda(\bm\omega)}
\right\|
\le
e^{-c_0(\tau_1-\lambda)}
e^{-c_0(\tau_2-\tau_1)}
\cdots
e^{-c_0(\sigma-\tau_r)}
=
e^{-c_0(\sigma-\lambda)}.
\end{align*}
The ordered integration region
$
\lambda\le\tau_1\le\cdots\le\tau_r\le\sigma
$
has volume
$
{(\sigma-\lambda)^r}/{r!}.
$
Taking norms in \eqref{eq:finite-variation-expansion}, we therefore obtain, after enlarging the constant \(C\),
\begin{align}\label{eq:triangular-backward-bound}
\left\|
\mathcal E(\lambda,\sigma;\bmz)
\right\|
&\le
e^{-c_0(\sigma-\lambda)}
\sum_{r=0}^{n+m-1}
\frac{\bigl(M(\sigma-\lambda)\bigr)^r}{r!}
\leq
C e^{-\frac{c_0}{2}(\sigma-\lambda)},
\qquad
\sigma\ge\lambda\ge R_0.
\end{align}

\medskip
\noindent
\emph{Step 2: Reduction to a Volterra equation.}

We impose the normalization
\[
Y_0(\lambda)\to 1,
\qquad
Y_\ast(\lambda)\to 0,
\qquad \lambda\to\infty.
\]
Integrating \eqref{eq:thm-radial-top-clean} from \(\lambda\) to \(\infty\), we
obtain
\begin{equation}\label{eq:Y0-representation-clean}
Y_0(\lambda)=1+(\mathcal K Y_\ast)(\lambda),
\qquad
(\mathcal K Y_\ast)(\lambda):=
-\int_\lambda^\infty
\frac1\sigma C_{01}(\sigma;\bmz)Y_\ast(\sigma)\,\rd\sigma.
\end{equation}
Substituting this into \eqref{eq:thm-radial-bottom-clean} and using variation of
constants gives
\begin{equation}\label{eq:volterra-fixed-clean}
Y_\ast=\mathcal T(Y_\ast),
\end{equation}
where
\begin{align}
(\mathcal T Y_\ast)(\lambda)
:=
-\int_\lambda^\infty
\mathcal E(\lambda,\sigma;\bmz)
\frac1\sigma
\Bigl[
C_{10}(\sigma;\bmz)\bigl(1+\mathcal K Y_\ast(\sigma)\bigr)
+
C_{11}(\sigma;\bmz)Y_\ast(\sigma)
\Bigr]
\,\rd\sigma.
\label{eq:T-def-clean}
\end{align}

\medskip
\noindent
\emph{Step 3: The fixed-point space and the basic estimates.}

Let \(\mathcal B_{\bmz}\) be the Banach space of continuous functions
\[
Y_\ast:[R_0,\infty)\to\mathbb C^{|\mathcal J^\ast|}
\]
with norm
\[
\|Y_\ast\|_{\mathcal B_{\bmz}}
:=
\sup_{\lambda\ge R_0}\lambda \|Y_\ast(\lambda)\|_2.
\]
If \(Y_\ast\in\mathcal B_{\bmz}\), then by
\eqref{eq:Y0-representation-clean},
\begin{align*}
|(\mathcal K Y_\ast)(\lambda)|
&\le
M\int_\lambda^\infty \frac{\|Y_\ast(\sigma)\|_2}{\sigma}\,\rd\sigma
\le
M\|Y_\ast\|_{\mathcal B_{\bmz}}
\int_\lambda^\infty \sigma^{-2}\,\rd\sigma
=
\frac{M}{\lambda}\|Y_\ast\|_{\mathcal B_{\bmz}}.
\end{align*}
Hence
\begin{equation}\label{eq:K-est-clean}
|(\mathcal K Y_\ast)(\lambda)|
\le
\frac{C}{\lambda}\|Y_\ast\|_{\mathcal B_{\bmz}}.
\end{equation}

With \eqref{eq:triangular-backward-bound}, \eqref{eq:T-def-clean} and \eqref{eq:K-est-clean},  we have
\begin{align*}
\lambda\|(\mathcal T Y_\ast)(\lambda)\|_2
&\le
M\lambda\int_\lambda^\infty
\left\|\mathcal E(\lambda,\sigma;\bmz)\right\|
\frac1\sigma
\Bigl(
1+|(\mathcal K Y_\ast)(\sigma)|+\|Y_\ast(\sigma)\|_2
\Bigr)\,\rd\sigma
\\
&\le
CM\lambda\int_\lambda^\infty
e^{-c_0(\sigma-\lambda)/2}
\Bigl(
\sigma^{-1}
+\|Y_\ast\|_{\mathcal B_{\bmz}}\sigma^{-2}
\Bigr)\,\rd\sigma
\le
C_0 + C_1 R_0^{-1}\|Y_\ast\|_{\mathcal B_{\bmz}}.
\end{align*}
Therefore
\begin{equation}\label{eq:T-map-ball-clean}
\|\mathcal T Y_\ast\|_{\mathcal B_{\bmz}}
\le
C_0 + C_1 R_0^{-1}\|Y_\ast\|_{\mathcal B_{\bmz}}.
\end{equation}

Similarly, for \(Y_\ast,Z_\ast\in\mathcal B_{\bmz}\),
\[
|(\mathcal K Y_\ast)(\lambda)-(\mathcal K Z_\ast)(\lambda)|
\le
\frac{C}{\lambda}\|Y_\ast-Z_\ast\|_{\mathcal B_{\bmz}}.
\]
Hence
\begin{align*}
\lambda\|(\mathcal T Y_\ast-\mathcal T Z_\ast)(\lambda)\|_2
&\le
M\lambda\int_\lambda^\infty
\left\|\mathcal E(\lambda,\sigma;\bmz)\right\|
\frac1\sigma
\Bigl(
|(\mathcal K Y_\ast)(\sigma)-(\mathcal K Z_\ast)(\sigma)|
+
\|Y_\ast(\sigma)-Z_\ast(\sigma)\|_2
\Bigr)\,\rd\sigma
\\
&\le
CM\lambda\int_\lambda^\infty
e^{-c_0(\sigma-\lambda)/2}
\sigma^{-2}\rd\sigma\,
\|Y_\ast-Z_\ast\|_{\mathcal B_{\bmz}}
\le
C_1R_0^{-1}\|Y_\ast-Z_\ast\|_{\mathcal B_{\bmz}}.
\end{align*}
Thus
\begin{equation}\label{eq:T-contraction-clean}
\|\mathcal T Y_\ast-\mathcal T Z_\ast\|_{\mathcal B_{\bmz}}
\le
C_1R_0^{-1}\|Y_\ast-Z_\ast\|_{\mathcal B_{\bmz}}.
\end{equation}

\medskip
\noindent
\emph{Step 4: Existence, asymptotics, and raywise uniqueness.}

If \(R_0\) is sufficiently large, then \(C_1R_0^{-1}<1\), so by \eqref{eq:T-contraction-clean}
\(\mathcal T\) is a contraction on \(\mathcal B_{\bmz}\). Hence
\eqref{eq:volterra-fixed-clean} admits a unique fixed point
\(Y_\ast\in\mathcal B_{\bmz}\). Defining \(Y_0\) by
\eqref{eq:Y0-representation-clean}, we obtain a unique solution
\((Y_0,Y_\ast)\) of
\eqref{eq:thm-radial-top-clean}--\eqref{eq:thm-radial-bottom-clean}
satisfying
\[
Y_0(\lambda)\to 1,
\qquad
Y_\ast(\lambda)\to 0,
\qquad \lambda\to\infty.
\]
Moreover, \eqref{eq:K-est-clean} and \eqref{eq:T-map-ball-clean} imply
\[
Y_0(\lambda)=1+\OO(\lambda^{-1}),
\qquad
Y_\ast(\lambda)=\OO(\lambda^{-1}),
\qquad \lambda\ge R_0.
\]

Now define
\[
W(\bmz):=Y_0(R_0).
\]
By construction,
\[
W(\bmz)=1+\OO(R_0^{-1})=1+\OO(R(\bmz)^{-1}).
\]
For each non-empty square-free pair \((I,A)\in\mathcal J^\ast\), the
corresponding component of \(Y_\ast(R_0)\) is
\[
Y_{I,A}(R_0)
=
R(\bmz)^{|I|+|A|}
\partial_{t_I}\partial_{s_A}W(\bmz).
\]
Thus the bound \(Y_\ast(R_0)=\OO(R_0^{-1})\) gives
\[
\partial_{t_I}\partial_{s_A}W(\bmz)
=
\OO\bigl(R(\bmz)^{-|I|-|A|-1}\bigr).
\]

The same argument applied to the difference of two normalized solutions has
no constant term in \eqref{eq:Y0-representation-clean}. Therefore the
corresponding Volterra map is a strict contraction with zero forcing, and the
difference vanishes. This proves uniqueness along every admissible shifted
ray.

\medskip
\noindent
\emph{Step 5: Holomorphic dependence and gluing on the ordered domain.}

Fix \(\bm z_0\in\fC\). After shrinking to a small neighborhood \(\fU\) of
\(\bm z_0\), we may choose a fixed number \(R_0\) such that
\(R_0\leq R(\bm z)\leq 2R_0\) for all
\(\bm z\in\fU\). Use the same vector \(\bm\omega\) and the rays
\(\bm z+(\lambda-R_0)\bm\omega\), \(\lambda\geq R_0\), throughout
\(\fU\). \Cref{p:direction} gives the required separation
uniformly on these rays. Since \(R_0\) and \(\bm\omega\) are fixed, the
coefficients of the Volterra equation depend holomorphically on \(\bm z\),
and the contraction estimate is uniform. Therefore its fixed point depends
holomorphically on \(\bm z\).

We have so far solved the first-order system only in the radial direction. To
check the remaining coordinate directions, fix one coordinate and compare two
ways of differentiating the constructed square-free jet: differentiation with
respect to the base point, and the corresponding coordinate equation of the
full square-free first-order system. Their difference is the error in that
coordinate equation. By flatness, radial differentiation commutes with every
coordinate equation, so this error satisfies the homogeneous radial system.
The estimates above hold uniformly for nearby base points, and therefore the
error has zero normalization at infinity. Raywise uniqueness forces it to
vanish. Thus the constructed jet is a holomorphic horizontal section of the
full flat first-order system, and its empty component is a holomorphic
solution \(W_{\bm\epsilon,\bm\eta}^{(\fU)}\) of the conjugated system.

The local constructions agree on overlaps. Indeed, all neighborhoods use the
same vector \(\bm\omega\). At a point \(\bm z\) in an overlap, restrict both
horizontal sections to the geometric ray \(\bm z+r\bm\omega\), \(r\geq0\).
After using the same radial parameter, their scaled square-free jets solve
the same radial system and have the same normalization at infinity. The
raywise uniqueness from Step~4 therefore shows that their values at
\(r=0\) agree. Hence the local solutions glue to a holomorphic function
\(W_{\bm\epsilon,\bm\eta}\) on all of \(\fC\), and the
estimates \eqref{eq:W-basic-est}--\eqref{eq:jet-basic-est} hold there with
constants depending only on \(\beta,n,m\). Finally,
\(F_{\bm\epsilon,\bm\eta}:=\Phi_{\bm\epsilon,\bm\eta}
W_{\bm\epsilon,\bm\eta}\) solves the original system.

\medskip
\noindent
\emph{Step 6: Linear independence.}

From the explicit formula for \(\Phi_{\epsilon,\eta}\), we have
\[
\partial_{t_i}\log\Phi_{\bm\epsilon,\bm\eta}
=
\ri\epsilon_i+\OO(D^{-1}),
\qquad
\partial_{s_a}\log\Phi_{\epsilon,\eta}
=
\ri\frac{\beta}{2}\eta_a+\OO(D^{-1}).
\]
For a square-free pair \((I,A)\), the explicit formula for the branch factor
and the estimates for \(W_{\bm\epsilon,\bm\eta}\) give, uniformly on
\(\fC\),
\begin{align}\label{e:local_jet_matrix}
&\Phi_{\bm\epsilon,\bm\eta}^{-1}
\partial_{t_I}\partial_{s_A}
F_{\bm\epsilon,\bm\eta}
=
\ri^{|I|+|A|}
\left(\frac\beta2\right)^{|A|}
\prod_{i\in I}\epsilon_i
\prod_{a\in A}\eta_a
+
\OO(D^{-1}).
\end{align}
After dividing each row by its nonzero prefactor, the leading matrix is the
character table
\[
\left(
\prod_{i\in I}\epsilon_i
\prod_{a\in A}\eta_a
\right)_{(I,A),(\bm\epsilon,\bm\eta)}
\]
of \(\{\pm1\}^{n+m}\).
In particular, its rows satisfy the orthogonality
relations
\[
\sum_{\epsilon,\eta}
\left(
\prod_{i\in I}\epsilon_i
\prod_{a\in A}\eta_a
\right)
\left(
\prod_{j\in J}\epsilon_j
\prod_{b\in B}\eta_b
\right)
=
2^{n+m}\bm1_{\{(I,A)=(J,B)\}},
\]
and therefore the leading matrix is invertible.
Increasing \(D\), if necessary, the matrix in
\eqref{e:local_jet_matrix} is also invertible. Therefore the
\(2^{n+m}\) solutions \(F_{\bm\epsilon,\bm\eta}\) are linearly independent.
\end{proof}

\section{Solutions of the edge deformed CMS operators}
\label{s:edge_solution}

The proof of \Cref{thm:edge-sectorial-solution} is parallel to the bulk case. We prove the first statement of \Cref{thm:edge-sectorial-solution}, namely that
the solution space has dimension \(2^{n+m}\), in \Cref{s:edge_local_rank}. We derive the conjugated equations in \Cref{s:conjugate_edge}.
We then prove the second statement, namely the construction of solutions with
prescribed leading behavior, in \Cref{s:edge_construct_solution}.

\subsection{Local holonomic rank}
\label{s:edge_local_rank}
For $1\le i\le n$ and $1\le a\le m$, define
\begin{align}
\mathsf  T_i
&:=
\partial_{t_i}^2-t_i
+\frac{2}{\beta}\sum_{j\neq i}\frac{\partial_{t_i}-\partial_{t_j}}{t_i-t_j}
-\sum_{a=1}^m\frac{\partial_{t_i}+(2/\beta)\partial_{s_a}}{t_i-s_a},
\label{eq:edgeTi-def}
\\[1mm]
\mathsf  S_a
&:=
\partial_{s_a}^2-\frac{\beta^2}{4}s_a
+\frac{\beta}{2}\sum_{b\neq a}\frac{\partial_{s_a}-\partial_{s_b}}{s_a-s_b}
-\sum_{i=1}^n\frac{\partial_{s_a}+(\beta/2)\partial_{t_i}}{s_a-t_i}.
\label{eq:edgeSa-def}
\end{align}
Then the edge deformed CMS operators \eqref{e:edge_eq1}--\eqref{e:edge_eq2} can be written as
\begin{align}\label{e:edgeeq}
\mathsf  T_iF=0,\qquad 1\le i\le n,
\qquad 
\mathsf  S_aF=0,\qquad 1\le a\le m.
\end{align}

Using the following edge commutator identities as input,  the first statement of \Cref{thm:edge-sectorial-solution} can be proven 
in the same way as the bulk case \Cref{s:bulk_local_rank}. So we omit its proof.
\begin{proposition}[Edge commutator identities]
\label{prop:edge-commutators}
For the operators \(\mathsf T_i,\mathsf S_a\) defined in
\eqref{eq:edgeTi-def}--\eqref{eq:edgeSa-def}, the following commutator
identities hold:
\begin{align}
[\mathsf T_i,\mathsf T_j]
&=
\frac{4}{\beta (t_i-t_j)^2}\,(\mathsf T_j-\mathsf T_i),
\qquad 1\le i\neq j\le n,
\label{eq:edge-comm-tt}
\\[1mm]
[\mathsf S_a,\mathsf S_b]
&=
\frac{\beta}{(s_a-s_b)^2}\,(\mathsf S_b-\mathsf S_a),
\qquad 1\le a\neq b\le m,
\label{eq:edge-comm-ss}
\\[1mm]
[\mathsf T_i,\mathsf S_a]
&=
\frac{4}{\beta (t_i-s_a)^2}\,\mathsf S_a
-\frac{\beta}{(t_i-s_a)^2}\,\mathsf T_i,
\qquad 1\le i\le n,\ 1\le a\le m.
\label{eq:edge-comm-ts}
\end{align}
\end{proposition}

\begin{proof}
This is proved by the same local calculation as for the bulk operators \Cref{prop:commutator-identities}.  One decomposes
\(\mathsf T_i\) and \(\mathsf S_a\) into the basic pairwise pieces
\[
\partial_{t_i}^2-t_i,\qquad
\partial_{s_a}^2-\frac{\beta^2}{4}s_a,\qquad
\frac{2}{\beta}\frac{\partial_{t_i}-\partial_{t_j}}{t_i-t_j},\qquad
\frac{\beta}{2}\frac{\partial_{s_a}-\partial_{s_b}}{s_a-s_b},\qquad
\frac{\partial_{t_i}+(2/\beta)\partial_{s_a}}{t_i-s_a},
\]
checks the corresponding two-body and three-body commutators, and then sums the local
identities.  Compared with the bulk case, the linear potentials \(-t_i\) and
\(-\beta^2 s_a/4\) contribute exactly the scalar terms needed to replace
the bulk operators by the edge operators on the right-hand side. They do not
change the coefficients of the commutator identities.
\end{proof}

\subsection{The conjugated system}\label{s:conjugate_edge}
We recall the (rescaled) Airy functions $\phi_\pm,\psi_\pm$ and their asymptotics from \eqref{e:defphi}, \eqref{e:phi_asymptotics} and \eqref{e:defpsi}. We have
\[
\phi_\pm''(z)=z\,\phi_\pm(z),
\qquad
q_\pm(z):=\frac{\phi_\pm'(z)}{\phi_\pm(z)}
=\pm \sqrt{z}-\frac{1}{4z}+\OO(z^{-5/2}),
\]
and
\[
\psi_\pm''(s)=\frac{\beta^2}{4}s\,\psi_\pm(s),
\qquad
p_\pm(s):=\frac{\psi_\pm'(s)}{\psi_\pm(s)}
=
\pm \frac{\beta}{2}\sqrt{s}-\frac{1}{4s}+\OO(s^{-5/2}).
\]
Now fix sign vectors
\[
\bm\epsilon=(\epsilon_1,\dots,\epsilon_n)\in\{\pm1\}^n,
\qquad
\bm\eta=(\eta_1,\dots,\eta_m)\in\{\pm1\}^m,
\]
and write
\begin{align}\label{e:defqi}
q_i:=q_{\epsilon_i}(t_i),
\qquad
p_a:=p_{\eta_a}(s_a).
\end{align}
Then $q_i$ and $p_a$ satisfy the Riccati
identity
\begin{align}\label{e:Riccati}
q_i'+q_i^2=t_i,\quad p_a'+p_a^2=\frac{\beta^2}{4}s_a
\end{align}
It is convenient to separate off the leading square-root terms:
\begin{equation}
\label{eq:edge-r-v}
r_i:=q_i-\epsilon_i\sqrt{t_i}=-\frac{1}{4t_i}+\OO(t_i^{-5/2}),
\qquad
v_a:=p_a-\frac{\beta}{2}\eta_a\sqrt{s_a}=-\frac{1}{4s_a}+\OO(s_a^{-5/2}).
\end{equation}

We also recall the branch factor $\Phi_{\bm\epsilon,\bm\eta}(\bmt,\bms)$ from \eqref{eq:edge-PhiF-clean}
\begin{align}
\label{eq:edge-PhiF}
\Phi_{\bm\epsilon,\bm\eta}(\bmt,\bms)
&:=
\left(\prod_{i=1}^n\phi_{\epsilon_i}(t_i)\right)
\left(\prod_{a=1}^m\psi_{\eta_a}(s_a)\right)
\prod_{1\le i<j\le n}
(\epsilon_i\sqrt{t_i}+\epsilon_j\sqrt{t_j})^{-2/\beta}
\nonumber\\
&\qquad\times
\prod_{1\le a<b\le m}
(\eta_a\sqrt{s_a}+\eta_b\sqrt{s_b})^{-\beta/2}
\prod_{i=1}^n
\prod_{a=1}^m
(\epsilon_i\sqrt{t_i}-\eta_a\sqrt{s_a}).
\end{align}
We also introduce
\begin{align}\begin{split}
\label{eq:edge-Omega-Sigma-clean}
\Omega_i
&:=
-\frac{\epsilon_i}{\beta\sqrt{t_i}}
\sum_{j\neq i}\frac{1}{\epsilon_i\sqrt{t_i}+\epsilon_j\sqrt{t_j}}
+\frac{\epsilon_i}{2\sqrt{t_i}}
\sum_{a=1}^m\frac{1}{\epsilon_i\sqrt{t_i}-\eta_a\sqrt{s_a}},
\\[1mm]
\Sigma_a
&:=
-\frac{\beta\eta_a}{4\sqrt{s_a}}
\sum_{b\neq a}\frac{1}{\eta_a\sqrt{s_a}+\eta_b\sqrt{s_b}}
-\frac{\eta_a}{2\sqrt{s_a}}
\sum_{i=1}^n\frac{1}{\epsilon_i\sqrt{t_i}-\eta_a\sqrt{s_a}}.
\end{split}\end{align}

\begin{proposition}[Conjugated operator family]
\label{prop:edge-conjugated-family}
Let
$
F=\Phi_{\bm\epsilon,\bm\eta}W$.
Then the original edge system \eqref{e:edgeeq},
\[
\sfT_iF=0,\qquad 1\le i\le n,
\qquad
\sfS_aF=0,\qquad 1\le a\le m,
\]
is equivalent to
\begin{align}
(\sfL_i^{(t)}+\sfR_i^{(t)})W&=0,
\qquad i=1,\dots,n,
\label{eq:edge-conj-t-clean}
\\[1mm]
(\sfL_a^{(s)}+\sfR_a^{(s)})W&=0,
\qquad a=1,\dots,m,
\label{eq:edge-conj-s-clean}
\end{align}
where
\begin{align}
\sfL_i^{(t)}
&=
\partial_{t_i}^2+2q_i\,\partial_{t_i},
\label{eq:edge-Lt-clean}
\\[1mm]
\sfR_i^{(t)}
&=
2\Omega_i\,\partial_{t_i}
+\frac{2}{\beta}\sum_{j\neq i}\frac{\partial_{t_i}-\partial_{t_j}}{t_i-t_j}
-\sum_{a=1}^m\frac{\partial_{t_i}+(2/\beta)\partial_{s_a}}{t_i-s_a}
+V_i^{(t)},
\label{eq:edge-Rt-clean}
\\[2mm]
\sfL_a^{(s)}
&=
\partial_{s_a}^2+2p_a\,\partial_{s_a},
\label{eq:edge-Ls-clean}
\\[1mm]
\sfR_a^{(s)}
&=
2\Sigma_a\,\partial_{s_a}
+\frac{\beta}{2}\sum_{b\neq a}\frac{\partial_{s_a}-\partial_{s_b}}{s_a-s_b}
-\sum_{i=1}^n\frac{\partial_{s_a}+(\beta/2)\partial_{t_i}}{s_a-t_i}
+V_a^{(s)}.
\label{eq:edge-Rs-clean}
\end{align}
Here \(q_i,p_a,r_i,v_a\) are defined in \eqref{e:defqi} and
\eqref{eq:edge-r-v}. The scalar terms are
\begin{align}
V_i^{(t)}
&=
\partial_{t_i}\Omega_i+\Omega_i^2+2r_i\Omega_i
+\frac{2}{\beta}\sum_{j\neq i}
\frac{r_i-r_j+\Omega_i-\Omega_j}{t_i-t_j}
-\sum_{a=1}^m
\frac{r_i+(2/\beta)v_a+\Omega_i+(2/\beta)\Sigma_a}{t_i-s_a},
\label{eq:edge-Vt-simplified}
\\[1mm]
V_a^{(s)}
&=
\partial_{s_a}\Sigma_a+\Sigma_a^2+2v_a\Sigma_a
+\frac{\beta}{2}\sum_{b\neq a}
\frac{v_a-v_b+\Sigma_a-\Sigma_b}{s_a-s_b}
-\sum_{i=1}^n
\frac{v_a+(\beta/2)r_i+\Sigma_a+(\beta/2)\Omega_i}{s_a-t_i}.
\label{eq:edge-Vs-simplified}
\end{align}
\end{proposition}
\begin{proof}
Let
$
\Phi:=\Phi_{\bm\epsilon,\bm\eta}$
be as in \eqref{eq:edge-PhiF}. Since \(\Phi\) is nowhere vanishing on $\sfX$, the equations
\[
\sfT_iF=0,
\qquad
\sfS_aF=0,
\qquad
F=\Phi W,
\]
are equivalent to
\[
\widetilde{\sfT}_iW=0,
\qquad
\widetilde{\sfS}_aW=0,
\]
where
\[
\widetilde{\sfT}_i:=\Phi^{-1}\sfT_i\Phi,
\qquad
\widetilde{\sfS}_a:=\Phi^{-1}\sfS_a\Phi.
\]

We recall \(\Omega_i\) and \(\Sigma_a\) from
\eqref{eq:edge-Omega-Sigma-clean} and set
\begin{align}\label{eq:edge-mu-nu}
\mu_i
:=
\partial_{t_i}\log\Phi
=
q_i+\Omega_i,
\qquad
\nu_a
:=
\partial_{s_a}\log\Phi
=
p_a+\Sigma_a.
\end{align}
Then
\begin{align}\label{eq:edge-first-conj}
\Phi^{-1}\partial_{t_i}\Phi
=
\partial_{t_i}+\mu_i,
\qquad
\Phi^{-1}\partial_{s_a}\Phi
=
\partial_{s_a}+\nu_a,
\end{align}
and
\begin{align}
\begin{split}\label{eq:edge-second-conj}
\Phi^{-1}\partial_{t_i}^2\Phi
&=
(\partial_{t_i}+\mu_i)^2
=
\partial_{t_i}^2
+2\mu_i\partial_{t_i}
+\partial_{t_i}\mu_i+\mu_i^2,
\\
\Phi^{-1}\partial_{s_a}^2\Phi
&=
(\partial_{s_a}+\nu_a)^2
=
\partial_{s_a}^2
+2\nu_a\partial_{s_a}
+\partial_{s_a}\nu_a+\nu_a^2.
\end{split}
\end{align}

We first compute \(\widetilde{\sfT}_i\). Substituting
\eqref{eq:edge-first-conj} and \eqref{eq:edge-second-conj} into \(\sfT_i\),
recalled in \eqref{eq:edgeTi-def}, we obtain
\begin{align}
\widetilde{\sfT}_i
&=
(\partial_{t_i}+\mu_i)^2-t_i
+\frac{2}{\beta}\sum_{j\neq i}
\frac{(\partial_{t_i}+\mu_i)-(\partial_{t_j}+\mu_j)}{t_i-t_j}
-\sum_{a=1}^m
\frac{(\partial_{t_i}+\mu_i)+(2/\beta)(\partial_{s_a}+\nu_a)}
{t_i-s_a}.
\label{eq:edge-Ttilde-expanded-clean}
\end{align}
Using \(\mu_i=q_i+\Omega_i\) from \eqref{eq:edge-mu-nu}, the derivative terms
in \eqref{eq:edge-Ttilde-expanded-clean} are
\begin{align}\label{eq:edge-T-derivative-part}
\partial_{t_i}^2
+2q_i\partial_{t_i}
+2\Omega_i\partial_{t_i}
+\frac{2}{\beta}\sum_{j\neq i}
\frac{\partial_{t_i}-\partial_{t_j}}{t_i-t_j}
-\sum_{a=1}^m
\frac{\partial_{t_i}+(2/\beta)\partial_{s_a}}{t_i-s_a}.
\end{align}
Thus the first two terms give \(\sfL_i^{(t)}\), and the remaining first-order
terms give the differential part of \(\sfR_i^{(t)}\).

The scalar part is
\begin{align}
\begin{split}\label{eq:edge-Vti-hat-clean}
&\partial_{t_i}\Omega_i
+
q_i' + q_i^2 - t_i
+
\Omega_i^2
+2q_i\Omega_i
+\frac{2}{\beta}\sum_{j\neq i}
\frac{q_i-q_j+\Omega_i-\Omega_j}{t_i-t_j}
-\sum_{a=1}^m
\frac{q_i+(2/\beta)p_a+\Omega_i+(2/\beta)\Sigma_a}{t_i-s_a}.
\end{split}
\end{align}
Using \eqref{e:Riccati}, the term \(q_i'+q_i^2-t_i\) vanishes. Using the decomposition \eqref{eq:edge-r-v},  the scalar part \eqref{eq:edge-Vti-hat-clean} becomes
\begin{align}
\begin{split}\label{eq:edge-Vti-split-clean}
&V_i^{(t)}
+
2\epsilon_i\sqrt{t_i}\,\Omega_i
+\frac{2}{\beta}\sum_{j\neq i}
\frac{\epsilon_i\sqrt{t_i}-\epsilon_j\sqrt{t_j}}{t_i-t_j}
-\sum_{a=1}^m
\frac{\epsilon_i\sqrt{t_i}+\eta_a\sqrt{s_a}}{t_i-s_a}.
\end{split}
\end{align}
For the remaining square-root terms, use
\[
\frac{\epsilon_i\sqrt{t_i}-\epsilon_j\sqrt{t_j}}{t_i-t_j}
=
\frac{1}{\epsilon_i\sqrt{t_i}+\epsilon_j\sqrt{t_j}},
\qquad
\frac{\epsilon_i\sqrt{t_i}+\eta_a\sqrt{s_a}}{t_i-s_a}
=
\frac{1}{\epsilon_i\sqrt{t_i}-\eta_a\sqrt{s_a}}.
\]
Substituting the definition of \(\Omega_i\) from
\eqref{eq:edge-Omega-Sigma-clean}, we get
\begin{align*}
&2\epsilon_i\sqrt{t_i}\,\Omega_i
+\frac{2}{\beta}\sum_{j\neq i}
\frac{\epsilon_i\sqrt{t_i}-\epsilon_j\sqrt{t_j}}{t_i-t_j}
-\sum_{a=1}^m
\frac{\epsilon_i\sqrt{t_i}+\eta_a\sqrt{s_a}}{t_i-s_a}
\\
&=
-\frac{2}{\beta}
\sum_{j\neq i}
\frac{1}{\epsilon_i\sqrt{t_i}+\epsilon_j\sqrt{t_j}}
+
\sum_{a=1}^m
\frac{1}{\epsilon_i\sqrt{t_i}-\eta_a\sqrt{s_a}}
+\frac{2}{\beta}
\sum_{j\neq i}
\frac{1}{\epsilon_i\sqrt{t_i}+\epsilon_j\sqrt{t_j}}
-
\sum_{a=1}^m
\frac{1}{\epsilon_i\sqrt{t_i}-\eta_a\sqrt{s_a}}
=0.
\end{align*}
Hence the scalar part is exactly \(V_i^{(t)}\), and
\eqref{eq:edge-conj-t-clean} follows by combining
\eqref{eq:edge-Ttilde-expanded-clean} and
\eqref{eq:edge-T-derivative-part}.

The computation for \(\widetilde{\sfS}_a\) is analogous. Substituting
\eqref{eq:edge-first-conj} and \eqref{eq:edge-second-conj} into \(\sfS_a\),
recalled in \eqref{eq:edgeSa-def}, we obtain
\begin{align}
\widetilde{\sfS}_a
&=
(\partial_{s_a}+\nu_a)^2-\frac{\beta^2}{4}s_a
+\frac{\beta}{2}\sum_{b\neq a}
\frac{(\partial_{s_a}+\nu_a)-(\partial_{s_b}+\nu_b)}{s_a-s_b}
-\sum_{i=1}^n
\frac{(\partial_{s_a}+\nu_a)+(\beta/2)(\partial_{t_i}+\mu_i)}
{s_a-t_i}.
\label{eq:edge-Stilde-expanded-clean}
\end{align}
Using \(\nu_a=p_a+\Sigma_a\) from \eqref{eq:edge-mu-nu}, the derivative terms
in \eqref{eq:edge-Stilde-expanded-clean} are
\begin{align}\label{eq:edge-S-derivative-part}
\partial_{s_a}^2
+2p_a\partial_{s_a}
+2\Sigma_a\partial_{s_a}
+\frac{\beta}{2}\sum_{b\neq a}
\frac{\partial_{s_a}-\partial_{s_b}}{s_a-s_b}
-\sum_{i=1}^n
\frac{\partial_{s_a}+(\beta/2)\partial_{t_i}}{s_a-t_i}.
\end{align}
Thus the first two terms give \(\sfL_a^{(s)}\), and the remaining first-order
terms give the differential part of \(\sfR_a^{(s)}\).

The scalar part is
\begin{align}
\begin{split}\label{eq:edge-Vsa-hat-clean}
&\partial_{s_a}\Sigma_a
+
p_a' + p_a^2-\frac{\beta^2}{4}s_a
+
\Sigma_a^2
+2p_a\Sigma_a
+\frac{\beta}{2}\sum_{b\neq a}
\frac{p_a-p_b+\Sigma_a-\Sigma_b}{s_a-s_b}
-\sum_{i=1}^n
\frac{p_a+(\beta/2)q_i+\Sigma_a+(\beta/2)\Omega_i}{s_a-t_i}.
\end{split}
\end{align}
Using \eqref{e:Riccati}, 
the term \(p_a'+p_a^2-\beta^2 s_a/4\) vanishes. Using the decomposition \eqref{eq:edge-r-v}, the scalar part \eqref{eq:edge-Vsa-hat-clean} becomes
\begin{align}
\begin{split}\label{eq:edge-Vsa-split-clean}
&V_a^{(s)}
+
\beta\eta_a\sqrt{s_a}\,\Sigma_a
+\frac{\beta^2}{4}\sum_{b\neq a}
\frac{\eta_a\sqrt{s_a}-\eta_b\sqrt{s_b}}{s_a-s_b}
-\frac{\beta}{2}\sum_{i=1}^n
\frac{\eta_a\sqrt{s_a}+\epsilon_i\sqrt{t_i}}{s_a-t_i}.
\end{split}
\end{align}
For the remaining square-root terms, use
\[
\frac{\eta_a\sqrt{s_a}-\eta_b\sqrt{s_b}}{s_a-s_b}
=
\frac{1}{\eta_a\sqrt{s_a}+\eta_b\sqrt{s_b}},
\qquad
\frac{\eta_a\sqrt{s_a}+\epsilon_i\sqrt{t_i}}{s_a-t_i}
=
-\frac{1}{\epsilon_i\sqrt{t_i}-\eta_a\sqrt{s_a}}.
\]
Substituting the definition of \(\Sigma_a\) from
\eqref{eq:edge-Omega-Sigma-clean}, we get
\begin{align*}
&\beta\eta_a\sqrt{s_a}\,\Sigma_a
+\frac{\beta^2}{4}\sum_{b\neq a}
\frac{\eta_a\sqrt{s_a}-\eta_b\sqrt{s_b}}{s_a-s_b}
-\frac{\beta}{2}\sum_{i=1}^n
\frac{\eta_a\sqrt{s_a}+\epsilon_i\sqrt{t_i}}{s_a-t_i}
\\
&=
-\frac{\beta^2}{4}
\sum_{b\neq a}
\frac{1}{\eta_a\sqrt{s_a}+\eta_b\sqrt{s_b}}
-\frac{\beta}{2}
\sum_{i=1}^n
\frac{1}{\epsilon_i\sqrt{t_i}-\eta_a\sqrt{s_a}}
+\frac{\beta^2}{4}
\sum_{b\neq a}
\frac{1}{\eta_a\sqrt{s_a}+\eta_b\sqrt{s_b}}
+\frac{\beta}{2}
\sum_{i=1}^n
\frac{1}{\epsilon_i\sqrt{t_i}-\eta_a\sqrt{s_a}}
=0.
\end{align*}
Hence the scalar part is exactly \(V_a^{(s)}\), and
\eqref{eq:edge-conj-s-clean} follows by combining
\eqref{eq:edge-Stilde-expanded-clean} and
\eqref{eq:edge-S-derivative-part}. This proves the proposition.
\end{proof}

\begin{lemma}[Size of the conjugated coefficients]
\label{lem:edge-coeff-estimates}
Let \(\sfC\) be a domain of the form \eqref{e:edgeC}, with \(D\)
sufficiently large, let \((\bmt,\bms)\in\sfC\), and set
$
R=\sfR(\bmt,\bms)
$
recall from \eqref{e:defsfR}. Then, uniformly on \(\sfC\),
\begin{align}
&\Omega_i=\OO(R^{-1}),
\qquad
\Sigma_a=\OO(R^{-1}),
\label{eq:edge-Om-Sig-est}
\\
&q_i-\epsilon_i\sqrt{t_i}=\OO(R^{-1}),
\qquad
p_a-\eta_a\frac{\beta}{2}\sqrt{s_a}=\OO(R^{-1}),
\label{eq:edge-q-p-est}
\\
&V_i^{(t)}=\OO(R^{-2}),
\qquad
V_a^{(s)}=\OO(R^{-2}).
\label{eq:edge-V-est}
\end{align}
Moreover, if \(\sfD\) is any differential monomial of total order
\(\ell\), then
\begin{align}
&\sfD\Omega_i,
\quad
\sfD\Sigma_a,
\quad
\sfD\bigl(q_i-\epsilon_i\sqrt{t_i}\bigr),
\quad
\sfD\left(p_a-\eta_a\frac{\beta}{2}\sqrt{s_a}\right)
=
\OO\bigl(R^{-1-\ell}\bigr),
\label{eq:edge-coeff-all-derivatives-1}
\\
&\sfD V_i^{(t)},
\quad
\sfD V_a^{(s)}
=
\OO\bigl(R^{-2-\ell}\bigr).
\label{eq:edge-coeff-all-derivatives-2}
\end{align}
The same all-order estimates hold for every complete first-order
coefficient in \(\sfR_i^{(t)}\) and \(\sfR_a^{(s)}\), with
\(R^{-1-\ell}\) on the right-hand side. The implicit constants may
depend on \(\beta,n,m\) and on the differential monomial \(\sfD\), but
are uniform over \(\sfC\).
\end{lemma}

\begin{proof}
We begin with the zeroth-order estimates. We use the elementary inequality
\begin{equation}\label{eq:uv-estimate}
|u(u\pm v)|
\gtrsim
\min\bigl(|u|^2,|u^2-v^2|\bigr),
\qquad
u,v\in\bC.
\end{equation}
Indeed, if \(|v|\ge 2|u|\), then
\[
|u\pm v|\ge |v|-|u|\ge |u|,
\]
so the left-hand side of \eqref{eq:uv-estimate} is at least
\(|u|^2\). If \(|v|\le 2|u|\), then
\[
|u\mp v|\le |u|+|v|\le 3|u|,
\]
and therefore
\[
|u^2-v^2|
=
|u+v|\,|u-v|
\le
3|u|\,|u\pm v|.
\]
This proves \eqref{eq:uv-estimate}.

Applying \eqref{eq:uv-estimate} to the terms in
\eqref{eq:edge-Omega-Sigma-clean}, we obtain, for \(j\neq i\),
\[
\left|
\sqrt{t_i}
\bigl(\epsilon_i\sqrt{t_i}+\epsilon_j\sqrt{t_j}\bigr)
\right|
\gtrsim
\min\bigl(|t_i|,|t_i-t_j|\bigr)
\gtrsim R.
\]
Similarly, for \(b\neq a\),
\[
\left|
\sqrt{s_a}
\bigl(\eta_a\sqrt{s_a}+\eta_b\sqrt{s_b}\bigr)
\right|
\gtrsim
\min\bigl(|s_a|,|s_a-s_b|\bigr)
\gtrsim R.
\]
For the mixed terms,
\[
\left|
\sqrt{t_i}
\bigl(\epsilon_i\sqrt{t_i}-\eta_a\sqrt{s_a}\bigr)
\right|
\gtrsim
\min\bigl(|t_i|,|t_i-s_a|\bigr)
\gtrsim R,
\]
and
\[
\left|
\sqrt{s_a}
\bigl(\epsilon_i\sqrt{t_i}-\eta_a\sqrt{s_a}\bigr)
\right|
\gtrsim
\min\bigl(|s_a|,|t_i-s_a|\bigr)
\gtrsim R.
\]
Thus every complete summand in
\eqref{eq:edge-Omega-Sigma-clean} is \(\OO(R^{-1})\). Since the
numbers of summands depend only on \(n\) and \(m\), this proves
\eqref{eq:edge-Om-Sig-est}.

The Airy asymptotics \eqref{eq:edge-r-v} give
\[
q_i-\epsilon_i\sqrt{t_i}
=
-\frac{1}{4t_i}+\OO(t_i^{-5/2}),
\]
and
\[
p_a-\eta_a\frac{\beta}{2}\sqrt{s_a}
=
-\frac{1}{4s_a}+\OO(s_a^{-5/2}),
\]
uniformly in the sectors under consideration. Since
$
|t_i|\ge R
$, $
|s_a|\ge R
$,
and \(R\) is sufficiently large, this proves
\eqref{eq:edge-q-p-est}.

We next establish the derivative estimates. The angular restrictions in
\eqref{e:edgeC} place every coordinate in a sector having a fixed positive
angular distance from the real axis and from the branch cut of the
principal square root. Consequently, there exists a sufficiently small
constant \(c>0\) such that the polydisc obtained by moving each coordinate
by at most \(4cR\) remains in the same half-plane and in a fixed slightly
larger sector on which the square-root branches and the Airy asymptotics
are valid.

After decreasing \(c\) if necessary, every coordinate modulus and every
pairwise difference on this polydisc remain bounded below by a constant
multiple of \(R\). Indeed, the definition of \(R\) gives
\[
|t_i|,\ |s_a|,\ |t_i-t_j|,\ |t_i-s_a|,\ |s_a-s_b|
\ge R
\]
at the center, while moving each coordinate by at most \(4cR\) changes
a coordinate modulus by at most \(4cR\) and a pairwise difference by at
most \(8cR\).

It follows that the zeroth-order estimates proved above hold uniformly
throughout this polydisc. In particular,
\[
\Omega_i,\quad
\Sigma_a,\quad
q_i-\epsilon_i\sqrt{t_i},\quad
p_a-\eta_a\frac{\beta}{2}\sqrt{s_a}
=
\OO(R^{-1})
\]
there. These functions are holomorphic on the polydisc. Applying the
multivariable Cauchy estimates on concentric smaller polydiscs therefore
gives
\[
\sfD\Omega_i,
\quad
\sfD\Sigma_a,
\quad
\sfD\bigl(q_i-\epsilon_i\sqrt{t_i}\bigr),
\quad
\sfD\left(p_a-\eta_a\frac{\beta}{2}\sqrt{s_a}\right)
=
\OO(R^{-1-\ell}),
\]
which is \eqref{eq:edge-coeff-all-derivatives-1}.

We now use \eqref{eq:edge-Vt-simplified} and
\eqref{eq:edge-Vs-simplified}. Recall that
\[
r_i=q_i-\epsilon_i\sqrt{t_i},
\qquad
v_a=p_a-\eta_a\frac{\beta}{2}\sqrt{s_a}.
\]
On a smaller concentric polydisc,
\[
r_i,\ v_a,\ \Omega_i,\ \Sigma_a=\OO(R^{-1}),
\]
and \eqref{eq:edge-coeff-all-derivatives-1} gives
\[
\partial_{t_i}\Omega_i=\OO(R^{-2}),
\qquad
\partial_{s_a}\Sigma_a=\OO(R^{-2}).
\]
Hence
\[
\partial_{t_i}\Omega_i,
\quad
\Omega_i^2,
\quad
r_i\Omega_i
=
\OO(R^{-2}),
\]
and, for \(j\neq i\),
\[
\frac{r_i-r_j+\Omega_i-\Omega_j}{t_i-t_j}
=
\OO(R^{-2}).
\]
Likewise,
\[
\frac{
r_i+(2/\beta)v_a+\Omega_i+(2/\beta)\Sigma_a
}{t_i-s_a}
=
\OO(R^{-2}).
\]
Substitution into \eqref{eq:edge-Vt-simplified} gives
$
V_i^{(t)}=\OO(R^{-2}).
$
The same estimates applied to \eqref{eq:edge-Vs-simplified} give
$
V_a^{(s)}=\OO(R^{-2}).
$
This proves \eqref{eq:edge-V-est}, uniformly on the smaller polydisc.

The functions \(V_i^{(t)}\) and \(V_a^{(s)}\) are holomorphic on that
polydisc. A further application of the multivariable Cauchy estimates
therefore yields
\[
\sfD V_i^{(t)},
\quad
\sfD V_a^{(s)}
=
\OO(R^{-2-\ell}),
\]
which proves \eqref{eq:edge-coeff-all-derivatives-2}.

Finally, after collecting the coefficient of each first-order derivative
in \eqref{eq:edge-Rt-clean} and \eqref{eq:edge-Rs-clean}, every such
coefficient is a finite sum of terms of the form
\[
\Omega_i,\qquad
\Sigma_a,\qquad
\frac{1}{t_i-t_j},\qquad
\frac{1}{t_i-s_a},\qquad
\frac{1}{s_a-s_b},
\]
multiplied by constants depending only on \(\beta\). Each complete
coefficient is therefore \(\OO(R^{-1})\) on the same polydisc and is
holomorphic there. Cauchy's estimates give the asserted
\(\OO(R^{-1-\ell})\) bound for all of its derivatives.
\end{proof}

\subsection{Solutions attached to the branch factors}
\label{s:edge_construct_solution}

\begin{proposition}\label{p:direction-sector}
Let \(\sfC=\sfC(\bm\tau,D)\) be the domain
\eqref{e:edgeC}, and let
\[
\bm\sigma=(\sigma_1,\dots,\sigma_{n+m})\in\{\pm1\}^{n+m}.
\]
There exist a vector
\(\bm\omega=\bm\omega(\bm\tau,\bm\sigma)\), independent of the base
point \(\bmz\in\sfC\), and a constant \(c>0\), depending only on \(n+m\),
such that, for every \(\bmz\in\sfC\) and every \(r\ge0\), we have $\bmz+r\bm\omega\in\sfC$ and 
\begin{align}
&\sfR(\bmz+r\bm\omega)
\ge
c\bigl(\sfR(\bmz)+r\bigr),
\label{eq:dir-sector-R}
\\
&\left|
\arg\left(
-\sigma_j\omega_j\sqrt{z_j+r\omega_j}
\right)
\right|
\le
\frac{23\pi}{48},
\qquad
1\le j\le n+m,
\label{eq:dir-sector-angle}
\\
&\Re\left[
-\sigma_j\omega_j\sqrt{z_j+r\omega_j}
\right]
\ge
c\bigl(\sfR(\bmz)+r\bigr)^{1/2},\qquad
1\le j\le n+m.
\label{eq:dir-sector-phase}
\end{align}
Here the square root is the principal square root.
\end{proposition}
\begin{proof}
Choose \(K>1\), depending only on \(n+m\), sufficiently large, and define
\begin{equation}\label{e:edge-fixed-direction}
\omega_j
:=
K^{n+m+1-j}
\exp\!\left[
\ri\tau_j
\left(
\frac{\pi}{3}
+
\sigma_j\frac{\pi}{16}
\right)
\right],
\qquad
1\le j\le n+m.
\end{equation}
For every \(j\),
\[\left|\arg\omega_j-\frac{\pi}{3}\tau_j\right|\leq \frac{\pi}{16}.
\]
Thus \(z_j\) and \(\omega_j\) belong to the same convex cone appearing in
the definition of \(\sfC\), and hence
$
(z_j+r\omega_j)_{1\leq j\leq n+m}
$
remains in that cone for every \(r\ge0\).

Moreover,
\[
\Re\omega_j
\ge
\cos\!\left(\frac{19\pi}{48}\right)
K^{n+m+1-j},\quad |\omega_j|\leq K^{n+m+1-j}.
\]
Consequently, after taking \(K\) sufficiently large,
\[
\Re\omega_j-\Re\omega_{j+1}
\ge
K^{n+m-j}
\left[
K\cos\!\left(\frac{19\pi}{48}\right)-1
\right]
>0,
\qquad
1\le j<n+m,
\]
and
$
\Re\omega_{n+m}>0.
$
It follows that
\begin{align*}
\Re(z_j+r\omega_j)-\Re(z_{j+1}+r\omega_{j+1})
>
D,\quad 
1\le j<n+m,\quad 
\Re(z_{n+m}+r\omega_{n+m})
>
D.
\end{align*}
Thus all the inequalities defining \(\sfC\) are preserved along the ray,
which proves $\bmz+r\bm\omega\in\sfC$.

We next prove \eqref{eq:dir-sector-R}. Since \(z_j\) and \(\omega_j\)
lie in the same cone of opening \(\pi/6\),
\[
|z_j+r\omega_j|
\ge
c\bigl(|z_j|+r|\omega_j|\bigr)
\ge
c\bigl(\sfR(\bmz)+r\bigr).
\]
For \(1\le i<j\le n+m\), the geometric growth in
\eqref{e:edge-fixed-direction} gives
\[
\omega_i-\omega_j
=
\omega_i\bigl(1+\OO(K^{-1})\bigr),
\]
uniformly in \(i<j\). Increasing \(K\), if necessary, we may therefore
assume that
\[
\left|\arg(\omega_i-\omega_j)\right|
<
\frac{11\pi}{24},
\qquad
|\omega_i-\omega_j|\ge1.
\]
On the other hand, the ordering inequalities defining \(\sfC\) give
$
\Re(z_i-z_j)>0,
$
so the angle between \(z_i-z_j\) and \(\omega_i-\omega_j\) is at most
\(23\pi/24\). Hence
\begin{align*}
\left|
(z_i-z_j)+r(\omega_i-\omega_j)
\right|
&\ge
\sin\!\left(\frac{\pi}{48}\right)
\left(
|z_i-z_j|+r|\omega_i-\omega_j|
\right)\ge
c\bigl(\sfR(\bmz)+r\bigr).
\end{align*}
Taking the minimum over all coordinates and all pairs proves
\eqref{eq:dir-sector-R}.

It remains to prove the angular estimate. By
$\bmz+r\bm\omega\in\sfC$,
\[
\frac{\pi}{4}
<
\tau_j\arg(z_j+r\omega_j)
<
\frac{5\pi}{12}.
\]
Since the square root is the principal square root, its argument is one-half
of the argument of \(z_j+r\omega_j\). If \(\sigma_j=-1\), then
\[
\arg\left(
-\sigma_j\omega_j\sqrt{z_j+r\omega_j}
\right)
=
\tau_j
\left(
\frac{\pi}{3}
-\frac{\pi}{16}
+
\frac{\tau_j}{2}\arg(z_j+r\omega_j)
\right).
\]
If \(\sigma_j=+1\), its principal argument is
\[
-\tau_j
\left(
\pi
-\frac{\pi}{3}
-\frac{\pi}{16}
-
\frac{\tau_j}{2}\arg(z_j+r\omega_j)
\right).
\]
In either case,
\[
\frac{19\pi}{48}
<
\left|
\arg\left(
-\sigma_j\omega_j\sqrt{z_j+r\omega_j}
\right)
\right|
<
\frac{23\pi}{48},
\]
which proves \eqref{eq:dir-sector-angle}. Consequently,
\begin{align*}
\Re\left[
-\sigma_j\omega_j\sqrt{z_j+r\omega_j}
\right]\ge
\sin\!\left(\frac{\pi}{48}\right)
|\omega_j|\,|z_j+r\omega_j|^{1/2}\ge
c\,\sfR(\bmz+r\bm\omega)^{1/2}\ge
c\bigl(\sfR(\bmz)+r\bigr)^{1/2},
\end{align*}
where we used \(|\omega_j|\ge1\) and
\eqref{eq:dir-sector-R}. This proves
\eqref{eq:dir-sector-phase}.
\end{proof}

\begin{proposition}[Radial block system for the edge problem]
\label{prop:edge-radial-block}
Fix a sign pattern
\[
(\bm\epsilon,\bm\eta)\in\{\pm1\}^n\times\{\pm1\}^m.
\]
Let \(\sfC\) be a domain of the form \eqref{e:edgeC}, fix
\[
\bm z=(\bmt,\bms)\in\sfC,
\qquad
R_0\leq \sfR(\bm z)\leq 2R_0,
\]
and let \(\bm\omega\) be the fixed vector provided by
\Cref{p:direction-sector} for
\[
\bm\sigma
:=
(\epsilon_1,\dots,\epsilon_n,\eta_1,\dots,\eta_m).
\]
Define
\[
\gamma_{\bm z,\bm\omega}(\lambda)
:=
\bm z+(\lambda-R_0)\bm\omega,
\qquad
\lambda\in[R_0,\infty).
\]

Let \(W\) be a holomorphic solution of the conjugated edge system
\eqref{eq:edge-conj-t-clean}--\eqref{eq:edge-conj-s-clean}. Let
\(\mathcal J\) be the set of square-free index pairs
\[
(I,A),
\qquad
I\subseteq\{1,\dots,n\},
\quad
A\subseteq\{1,\dots,m\},
\]
and write
\[
\mathcal J^\ast:=\mathcal J\setminus\{(\varnothing,\varnothing)\}.
\]
For \((I,A)\in\mathcal J\), define the scaled square-free derivatives along
\(\gamma_{\bm z,\bm\omega}\) by
\[
U_{I,A}(\lambda;\bm z,\bm\omega)
:=
\lambda^{|I|+|A|}
\partial_{t_I}\partial_{s_A}W
\bigl(\gamma_{\bm z,\bm\omega}(\lambda)\bigr).
\]
Write
\[
U=(U_0,U_\ast),
\qquad
U_0:=U_{\varnothing,\varnothing}
=
W\bigl(\gamma_{\bm z,\bm\omega}(\lambda)\bigr),
\]
where \(U_\ast\) collects the components indexed by
\(\mathcal J^\ast\).

Set
\[
\zeta:=\frac23\lambda^{3/2},
\qquad
\zeta_0:=\frac23R_0^{3/2}.
\]
Then there exist matrix-valued functions
\[
B_{01}(\zeta;\bm z,\bm\omega),
\qquad
B_{10}(\zeta;\bm z,\bm\omega),
\qquad
B_{11}(\zeta;\bm z,\bm\omega),
\]
and a matrix
\[
\Gamma(\zeta;\bm z,\bm\omega)
\]
on the non-empty derivative block. The entries of
\(B_{01},B_{10},B_{11}\) are uniformly bounded for
\(\zeta\ge\zeta_0\), and the restriction of the square-free derivative
system to \(\gamma_{\bm z,\bm\omega}\) has the form
\begin{align}
\frac{\rd}{\rd\zeta}U_0
&=
\frac1\zeta B_{01}(\zeta;\bm z,\bm\omega)U_\ast,
\label{eq:edge-radial-top}
\\[1mm]
\frac{\rd}{\rd\zeta}U_\ast
&=
\Gamma(\zeta;\bm z,\bm\omega)U_\ast
+
\frac1\zeta B_{10}(\zeta;\bm z,\bm\omega)U_0
+
\frac1\zeta B_{11}(\zeta;\bm z,\bm\omega)U_\ast.
\label{eq:edge-radial-bottom}
\end{align}

Order the components of \(U_\ast\) by increasing total derivative order
\(|I|+|A|\). Then \(\Gamma(\zeta;\bm z,\bm\omega)\) is block lower
triangular. More precisely,
\begin{equation}
\label{eq:edge-Gamma-triangular}
\Gamma_{(I,A),(J,B)}(\zeta;\bm z,\bm\omega)=0
\end{equation}
unless either \((I,A)=(J,B)\), or
\[
J\subseteq I,
\qquad
B\subseteq A,
\qquad
1\le |J|+|B|<|I|+|A|.
\]
For \((I,A)\in\mathcal J^\ast\), the corresponding diagonal entry is
\begin{align}
\Gamma_{I,A}(\zeta;\bm z,\bm\omega)
:={}&
\frac{1}{\lambda^{1/2}}
\left(
-2\sum_{i\in I}\epsilon_i\omega_i
\sqrt{t_i+(\lambda-R_0)\omega_i}
-\beta\sum_{a\in A}\eta_a\omega_{n+a}
\sqrt{s_a+(\lambda-R_0)\omega_{n+a}}
\right).
\label{eq:edge-mu-IA-statement}
\end{align}
Moreover, whenever \((I,A)\ne(J,B)\) and the entry on the left-hand side
of \eqref{eq:edge-Gamma-triangular} is nonzero,
\begin{equation}
\label{eq:edge-Gamma-relative-bound}
\left|
\Gamma_{(I,A),(J,B)}(\zeta;\bm z,\bm\omega)
\right|
\le
C\,\Re\left[
\Gamma_{I,A}(\zeta;\bm z,\bm\omega)
-
\Gamma_{J,B}(\zeta;\bm z,\bm\omega)
\right].
\end{equation}
In particular, there exists \(c_\delta>0\) such that
\begin{equation}
\label{eq:edge-spectral-gap}
\Re\Gamma_{I,A}(\zeta;\bm z,\bm\omega)
\ge
c_\delta(2|I|+\beta|A|)
\ge c_\delta
\end{equation}
for every \((I,A)\in\mathcal J^\ast\) and every
\(\zeta\ge\zeta_0\), after decreasing \(c_\delta\) if necessary.
\end{proposition}

\begin{proof}
We prove the statement in three steps.

\medskip
\noindent
\emph{Step 1: Coefficient estimates along the shifted ray.}

By \Cref{p:direction-sector}, along the shifted ray
\(\gamma_{\bm z,\bm\omega}\) we have
\[
\sfR\bigl(\gamma_{\bm z,\bm\omega}(\lambda)\bigr)
\gtrsim\lambda,
\qquad
\Re\left[
-\sigma_j\omega_j
\sqrt{z_j+(\lambda-R_0)\omega_j}
\right]
\gtrsim\lambda^{1/2},\quad \lambda\geq R_0
\]
Therefore \Cref{lem:edge-coeff-estimates} gives, uniformly along the ray,
\[
\Omega_i,\Sigma_a,r_i,v_a=\OO(\lambda^{-1}),
\qquad
\partial\Omega_i,\partial\Sigma_a=\OO(\lambda^{-2}),
\]
and the first-order coefficients in \(\sfR_i^{(t)}\) and
\(\sfR_a^{(s)}\) are \(\OO(\lambda^{-1})\), whereas the scalar
coefficients \(V_i^{(t)}\) and \(V_a^{(s)}\) are
\(\OO(\lambda^{-2})\). 
More generally, every additional derivative falling
on one of these coefficients improves its decay by one further power of
\(\lambda^{-1}\).

The conjugated equations give the reduction rules
\begin{align}
\partial_{t_i}^2W
&=
-2q_i\partial_{t_i}W-\sfR_i^{(t)}W,
\label{eq:edge-red-ti-clean}
\\
\partial_{s_a}^2W
&=
-2p_a\partial_{s_a}W-\sfR_a^{(s)}W.
\label{eq:edge-red-sa-clean}
\end{align}
We shall use these identities repeatedly to reduce every repeated derivative
back to square-free derivatives.

We also record a comparison that will be used for the principal off-diagonal
terms. Let \(u\) and \(v\) be any two coordinates along the shifted ray. Since
$
|u|,\ |v|,\ |u-v|\gtrsim\lambda,
$
one has
\begin{equation}
\label{eq:edge-square-root-transfer}
\frac{\lambda|\sqrt v|}{|u-v|}
\le
C|\sqrt u|.
\end{equation}
Indeed, if \(|v|\le2|u|\), then
\[
\frac{\lambda|\sqrt v|}{|u-v|}
\le C|\sqrt u|
\]
because \(|u-v|\gtrsim\lambda\). If \(|v|>2|u|\), then
\(|u-v|\ge |v|/2\), and hence
\[
\frac{\lambda|\sqrt v|}{|u-v|}
\le
\frac{2\lambda}{|\sqrt v|}
\le
C|\sqrt u|,
\]
where we used \(|u|,|v|\gtrsim\lambda\). The same estimate applies to every
pair of \(t\)- and \(s\)-coordinates.

\medskip
\noindent
\emph{Step 2: The radial system in the \(\lambda\)-variable.}

First,
\[
U_0(\lambda;\bm z,\bm\omega)
=
W\bigl(\gamma_{\bm z,\bm\omega}(\lambda)\bigr).
\]
Since, along the shifted ray,
\[
\frac{\rd}{\rd\lambda}
=
\sum_{i=1}^n\omega_i\partial_{t_i}
+
\sum_{a=1}^m\omega_{n+a}\partial_{s_a},
\]
we obtain
\begin{equation}
\label{eq:edge-lambda-top-clean}
\frac{\rd}{\rd\lambda}U_0
=
\frac1\lambda C_{01}(\lambda;\bm z,\bm\omega)U_\ast,
\end{equation}
where \(C_{01}\) is uniformly bounded.

We next derive the equations for the non-empty square-free derivatives. Fix
\[
(I,A)\in\mathcal J^\ast.
\]
By the chain rule,
\begin{align}
\frac{\rd}{\rd\lambda}U_{I,A}
={}&
\frac{|I|+|A|}{\lambda}U_{I,A}
+
\lambda^{|I|+|A|}
\left(
\sum_{i=1}^n
\omega_i\partial_{t_i}\partial_{t_I}\partial_{s_A}W
+
\sum_{a=1}^m
\omega_{n+a}\partial_{s_a}\partial_{t_I}\partial_{s_A}W
\right)
\bigl(\gamma_{\bm z,\bm\omega}(\lambda)\bigr).
\label{eq:edge-lambda-chain-clean}
\end{align}

If \(i\notin I\), then the derivative remains square-free and
\[
\lambda^{|I|+|A|}
\partial_{t_i}\partial_{t_I}\partial_{s_A}W
\bigl(\gamma_{\bm z,\bm\omega}(\lambda)\bigr)
=
\frac1\lambda U_{I\cup\{i\},A}.
\]
Similarly, if \(a\notin A\), then
\[
\lambda^{|I|+|A|}
\partial_{s_a}\partial_{t_I}\partial_{s_A}W
\bigl(\gamma_{\bm z,\bm\omega}(\lambda)\bigr)
=
\frac1\lambda U_{I,A\cup\{a\}}.
\]
Thus all terms in \eqref{eq:edge-lambda-chain-clean} with
\(i\notin I\) or \(a\notin A\) belong to the
\(\lambda^{-1}\)-part of the radial system.

Suppose now that \(i\in I\), and write
\[
I':=I\setminus\{i\}.
\]
Using \eqref{eq:edge-red-ti-clean}, and observing that
\(\partial_{t_{I'}}\partial_{s_A}\) does not act on \(q_i\), we obtain
\begin{align}
\begin{split}
\lambda^{|I|+|A|}
\partial_{t_i}^2\partial_{t_{I'}}\partial_{s_A}W
={}&
\lambda^{|I|+|A|}
\partial_{t_{I'}}\partial_{s_A}
\left(
-2q_i(t_i+(\lambda-R_0)\omega_i)\partial_{t_i}
-\sfR_i^{(t)}
\right)W
\\
={}&
-2q_i(t_i+(\lambda-R_0)\omega_i)U_{I,A}
-
\lambda^{|I|+|A|}
\partial_{t_{I'}}\partial_{s_A}
\bigl(\sfR_i^{(t)}W\bigr),
\end{split}
\label{eq:edge-two-derivative-reduction}
\end{align}
where all functions are evaluated along the shifted ray. Since
\[
q_i(t_i+(\lambda-R_0)\omega_i)
=
\epsilon_i\sqrt{t_i+(\lambda-R_0)\omega_i}
+
\OO(\lambda^{-1}),
\]
the first term on the right-hand side of
\eqref{eq:edge-two-derivative-reduction}, after multiplication by
\(\omega_i\), contributes
\[
-2\epsilon_i\omega_i
\sqrt{t_i+(\lambda-R_0)\omega_i}\,U_{I,A}
+
\frac1\lambda\OO(1)U_{I,A}.
\]
The first term is the \(t_i\)-part of the diagonal entry in
\eqref{eq:edge-mu-IA-statement}.

We now analyze the terms arising from \(\sfR_i^{(t)}\). Consider first a
first-order summand involving \(\partial_{t_j}W\). After applying the
Leibniz rule, a typical term is
\begin{align}
\lambda^{|I|+|A|}
\bigl(
\partial_{t_{I_1}}\partial_{s_{A_1}}
\text{(coefficient)}
\bigr)
\partial_{t_{I_2}}\partial_{s_{A_2}}\partial_{t_j}W,
\label{eq:edge-first-order-typical-term}
\end{align}
where
\[
I_1\sqcup I_2=I',
\qquad
A_1\sqcup A_2=A.
\]
By Step 1, the differentiated coefficient in
\eqref{eq:edge-first-order-typical-term} is
\[
\OO\bigl(\lambda^{-1-|I_1|-|A_1|}\bigr).
\]

If \(j\notin I_2\), then the derivative acting on \(W\) is already
square-free. Passing to the corresponding scaled component gives
\begin{align*}
&
\lambda^{|I|+|A|}
\lambda^{-1-|I_1|-|A_1|}
\lambda^{-|I_2|-|A_2|-1}
=
\lambda^{-1}.
\end{align*}
Hence this term belongs to the \(\lambda^{-1}\)-part of the radial system.
The same conclusion holds for a first-order summand involving
\(\partial_{s_b}W\) when \(b\notin A_2\).

Suppose instead that \(j\in I_2\). Then
\[
\partial_{t_{I_2}}\partial_{s_{A_2}}\partial_{t_j}W
=
\partial_{t_{I_2\setminus\{j\}}}\partial_{s_{A_2}}
\partial_{t_j}^2W.
\]
Using \eqref{eq:edge-red-ti-clean}, we may either use the leading square-root
term in \(q_j\), use the remainder \(r_j\), or continue the reduction through
\(-\sfR_j^{(t)}W\). If we use the leading square-root term, then the
resulting square-free derivative is \(\partial_{t_{I_2}}\partial_{s_{A_2}}W\),
and its coefficient in the \(\lambda\)-system is bounded by
\begin{align*}
C\lambda^{|I|+|A|}
\lambda^{-1-|I_1|-|A_1|}
\left|
\sqrt{t_j+(\lambda-R_0)\omega_j}
\right|
\lambda^{-|I_2|-|A_2|}
=
C\left|
\sqrt{t_j+(\lambda-R_0)\omega_j}
\right|.
\end{align*}
Moreover,
\[
|I_2|+|A_2|
=
|I|+|A|-1-|I_1|-|A_1|
<
|I|+|A|.
\]
Thus this is a principal off-diagonal contribution from a strictly lower
square-free derivative order. After the change of variable
\(\zeta=2\lambda^{3/2}/3\), its coefficient acquires the factor
\(\lambda^{-1/2}\). Such a coefficient need not be uniformly bounded when
some coordinates are much larger than the separation scale; it will instead
be controlled by the relative estimate below.

If we use \(r_j\) rather than the leading square-root part of \(q_j\), then
\(r_j=\OO(\lambda^{-1})\), and the resulting term belongs to the
\(\lambda^{-1}\)-part. If we use \(-\sfR_j^{(t)}W\), the reduction may
continue. At this stage the remaining square-free derivatives no longer
contain \(t_j\). Consequently, the term
\(\Omega_j\partial_{t_j}W\), as well as the \(\partial_{t_j}\)-part of a
pair-interaction operator, cannot create another repeated derivative and
therefore terminates the reduction with a \(\lambda^{-1}\)-term. A scalar
term also terminates the reduction and contributes only to the
\(\lambda^{-1}\)-part. The only terms that can continue a principal
reduction are pair-interaction terms whose first-order derivative has an
index still present among the remaining square-free derivatives. Such a term
creates a new repeated derivative, to which the corresponding reduction rule
is applied.

The same discussion applies when the first-order summand involves
\(\partial_{s_b}W\). In particular, a principal reduction chain may pass
between the \(t\)- and \(s\)-variables, but at every step the next repeated
index is already present among the remaining square-free derivatives. Thus no
new index can occur in a principal off-diagonal term. If a principal
contribution from \(U_{J,B}\) appears in the equation for \(U_{I,A}\), then
\[
J\subseteq I,
\qquad
B\subseteq A,
\qquad
|J|+|B|<|I|+|A|.
\]
This proves the strict block-lower-triangular structure of the principal
matrix.

We next prove the relative estimate for the principal off-diagonal entries.
It is enough to consider a chain beginning with the radial
\(t_i\)-derivative; the argument for a chain beginning with an
\(s_a\)-derivative is identical. Fix \(i\in I\), and retain the notation
\(I'=I\setminus\{i\}\). Consider first the pair-interaction term in \eqref{eq:edge-first-order-typical-term} involving
\[
\frac{\partial_{t_j}}{t_i-t_j},
\qquad
j\in I'.
\]
If no derivative falls on its coefficient, then this term produces, up to a
fixed constant,
\[
\omega_i\lambda^{|I|+|A|}
\frac{1}{
 t_i-t_j+(\lambda-R_0)(\omega_i-\omega_j)
}
\partial_{t_{I'\setminus\{j\}}}\partial_{s_A}
\partial_{t_j}^2W.
\]
Using the leading square-root part of the reduction of
\(\partial_{t_j}^2W\), and then passing to the \(\zeta\)-variable, gives a
contribution to \(U_{I',A}\) bounded by
\begin{align*}
&C|\omega_i|
\frac{
\lambda^{1/2}
\left|
\sqrt{t_j+(\lambda-R_0)\omega_j}
\right|
}{
\left|
 t_i-t_j+(\lambda-R_0)(\omega_i-\omega_j)
\right|
}
=
C\frac{|\omega_i|}{\lambda^{1/2}}
\frac{
\lambda
\left|
\sqrt{t_j+(\lambda-R_0)\omega_j}
\right|
}{
\left|
 t_i-t_j+(\lambda-R_0)(\omega_i-\omega_j)
\right|
}.
\end{align*}
By \eqref{eq:edge-square-root-transfer}, this is bounded by
\[
C
\frac{
|\omega_i|
\left|
\sqrt{t_i+(\lambda-R_0)\omega_i}
\right|
}{\lambda^{1/2}}.
\]
Thus the square-root factor associated with the terminal repeated index has
been transferred to the initial index \(i\), which is deleted from the final
square-free pair.

The same mechanism applies to an arbitrary principal reduction chain. Every
additional pair-interaction step contributes one shifted difference
denominator and one compensating factor of \(\lambda\) coming from the scaled
square-free derivatives. Starting from the square-root factor produced at the
last reduction and applying \eqref{eq:edge-square-root-transfer} successively,
in the reverse order of the interaction chain, transfers that square-root
factor back across all the interaction denominators. Therefore every
principal chain beginning with the radial \(t_i\)-derivative is bounded, in
the \(\zeta\)-system, by
\[
C
\frac{
|\omega_i|
\left|
\sqrt{t_i+(\lambda-R_0)\omega_i}
\right|
}{\lambda^{1/2}}.
\]
Likewise, every principal chain beginning with the radial
\(s_a\)-derivative is bounded by
\[
C
\frac{
|\omega_{n+a}|
\left|
\sqrt{s_a+(\lambda-R_0)\omega_{n+a}}
\right|
}{\lambda^{1/2}}.
\]

Derivatives falling on a pair-interaction coefficient do not change this
conclusion. They produce additional inverse powers of the corresponding
shifted difference. After passage to the scaled derivatives, each such
additional inverse difference is accompanied by one additional factor of
\(\lambda\). By \Cref{p:direction-sector}, every shifted difference is
bounded below by a constant multiple of \(\lambda\), so these additional
factors are uniformly bounded. They only change the constant in the preceding
estimates.

Since there are only finitely many reduction chains, we conclude that a
principal contribution from \(U_{J,B}\) to the equation for \(U_{I,A}\) is
bounded by
\begin{align*}
C\Bigg(
&\sum_{i\in I\setminus J}
\frac{
|\omega_i|
\left|
\sqrt{t_i+(\lambda-R_0)\omega_i}
\right|
}{\lambda^{1/2}}
+
\sum_{a\in A\setminus B}
\frac{
|\omega_{n+a}|
\left|
\sqrt{s_a+(\lambda-R_0)\omega_{n+a}}
\right|
}{\lambda^{1/2}}
\Bigg).
\end{align*}
By the angular estimate \eqref{eq:dir-sector-phase} in \Cref{p:direction-sector},
\begin{align*}
|\omega_i|
\left|
\sqrt{t_i+(\lambda-R_0)\omega_i}
\right|
&\le
C\Re\left[
-\epsilon_i\omega_i
\sqrt{t_i+(\lambda-R_0)\omega_i}
\right],
\\
|\omega_{n+a}|
\left|
\sqrt{s_a+(\lambda-R_0)\omega_{n+a}}
\right|
&\le
C\Re\left[
-\eta_a\omega_{n+a}
\sqrt{s_a+(\lambda-R_0)\omega_{n+a}}
\right].
\end{align*}
All real parts in these displays are nonnegative. Since
\begin{align*}
&\Gamma_{I,A}(\zeta;\bm z,\bm\omega)
-
\Gamma_{J,B}(\zeta;\bm z,\bm\omega)
\\
&\qquad=
\frac1{\lambda^{1/2}}
\left(
-2\sum_{i\in I\setminus J}
\epsilon_i\omega_i
\sqrt{t_i+(\lambda-R_0)\omega_i}
-
\beta\sum_{a\in A\setminus B}
\eta_a\omega_{n+a}
\sqrt{s_a+(\lambda-R_0)\omega_{n+a}}
\right),
\end{align*}
the preceding estimate yields
\[
\left|
\Gamma_{(I,A),(J,B)}(\zeta;\bm z,\bm\omega)
\right|
\le
C\Re\left[
\Gamma_{I,A}(\zeta;\bm z,\bm\omega)
-
\Gamma_{J,B}(\zeta;\bm z,\bm\omega)
\right].
\]
This proves the relative estimate for the terms arising from the
\(t_i\)-equations. The proof for the \(s_a\)-equations is identical.

For completeness, we collect the remaining terms. If a first-order summand
does not create a repeated derivative, then the preceding scaling calculation
shows that it contributes \(\OO(\lambda^{-1})\). If a reduction uses one of
the remainders \(r_i\) or \(v_a\), then it again contributes
\(\OO(\lambda^{-1})\). A scalar coefficient is
\(\OO(\lambda^{-2})\); if
\(|I_1|+|A_1|\) derivatives fall on it, then its contribution after passage
to scaled square-free derivatives is bounded by
\begin{align*}
&
\lambda^{|I|+|A|}
\lambda^{-2-|I_1|-|A_1|}
\lambda^{-(|I|+|A|-1-|I_1|-|A_1|)}
=
\lambda^{-1}.
\end{align*}
The terms involving \(\Omega_i\partial_{t_i}\) or
\(\Sigma_a\partial_{s_a}\) also terminate the reduction and belong to the
\(\lambda^{-1}\)-part, because the remaining square-free derivatives do not
contain the currently repeated index. Thus every term which is not part of a
principal pair-interaction chain terminating through a leading square-root
term is included in the matrices \(C_{10}\) and \(C_{11}\).

The case \(a\in A\) is now identical. Writing
$
A':=A\setminus\{a\},
$
and using \eqref{eq:edge-red-sa-clean}, the leading diagonal contribution is
\[
-\beta\eta_a\omega_{n+a}
\sqrt{s_a+(\lambda-R_0)\omega_{n+a}}\,U_{I,A},
\]
up to a term of the form \(\lambda^{-1}\OO(1)U_{I,A}\). The remaining
principal terms satisfy the same strict-inclusion and relative estimates, and
all other terms belong to the \(\lambda^{-1}\)-part.

Collecting all contributions in \eqref{eq:edge-lambda-chain-clean}, we obtain
\begin{align}
\frac{\rd}{\rd\lambda}U_0
&=
\frac1\lambda C_{01}(\lambda;\bm z,\bm\omega)U_\ast,
\label{eq:edge-lambda-top-block-clean}
\\[1mm]
\frac{\rd}{\rd\lambda}U_\ast
&=
\lambda^{1/2}\Gamma(\zeta;\bm z,\bm\omega)U_\ast
+
\frac1\lambda C_{10}(\lambda;\bm z,\bm\omega)U_0
+
\frac1\lambda C_{11}(\lambda;\bm z,\bm\omega)U_\ast,
\label{eq:edge-lambda-bottom-block-clean}
\end{align}
where the entries of \(C_{01},C_{10},C_{11}\) are uniformly bounded. Under
the ordering by increasing \(|I|+|A|\), the matrix \(\Gamma\) is block lower
triangular, its diagonal entries are given by
\eqref{eq:edge-mu-IA-statement}, and its off-diagonal entries satisfy
\eqref{eq:edge-Gamma-relative-bound}. Notice also that a principal reduction
always leaves at least one derivative. Hence no principal term couples
\(U_0\) into \(U_\ast\); all couplings from \(U_0\) into \(U_\ast\) are
contained in \(\lambda^{-1}C_{10}U_0\).

\medskip
\noindent
\emph{Step 3: Change of variable and the diagonal spectral gap.}

Since
\[
\zeta=\frac23\lambda^{3/2},
\qquad
\frac{\rd}{\rd\zeta}
=
\lambda^{-1/2}\frac{\rd}{\rd\lambda},
\qquad
\lambda^{-3/2}=\frac{2}{3\zeta},
\]
applying \(\lambda^{-1/2}\rd/\rd\lambda\) to
\eqref{eq:edge-lambda-top-block-clean}--
\eqref{eq:edge-lambda-bottom-block-clean} gives
\eqref{eq:edge-radial-top}--\eqref{eq:edge-radial-bottom}. The constant
factor \(2/3\) is absorbed into \(B_{01},B_{10},B_{11}\). The triangularity
property \eqref{eq:edge-Gamma-triangular} is unchanged by this change of
variable.

Finally, \Cref{p:direction-sector} gives
\[
\Re\left[
-\epsilon_i\omega_i
\sqrt{t_i+(\lambda-R_0)\omega_i}
\right]
\ge
c_\delta\lambda^{1/2},
\quad
\Re\left[
-\eta_a\omega_{n+a}
\sqrt{s_a+(\lambda-R_0)\omega_{n+a}}
\right]
\ge
c_\delta\lambda^{1/2}.
\]
Substituting these estimates into the diagonal formula
\eqref{eq:edge-mu-IA-statement} gives
\[
\Re\Gamma_{I,A}(\zeta;\bm z,\bm\omega)
\ge
c_\delta(2|I|+\beta|A|).
\]
Since \((I,A)\ne(\varnothing,\varnothing)\), the right-hand side is bounded
below by a positive constant after decreasing \(c_\delta\). This proves
\eqref{eq:edge-spectral-gap}.
\end{proof}

\begin{proof}[Proof of the second statement in
\Cref{thm:edge-sectorial-solution}]

Fix a sign pattern \((\bm\epsilon,\bm\eta)\), and let \(\bm\omega\) be the
fixed vector from \Cref{p:direction-sector}, with
\[
\bm\sigma
=(\epsilon_1,\dots,\epsilon_n,\eta_1,\dots,\eta_m).
\]
We use this same vector at every point of \(\sfC\). Fix first
\(\bm z\in\sfC\), set \(R_0:=\sfR(\bm z)\), consider
\[
\gamma_{\bm z,\bm\omega}(\lambda)
:=
\bm z+(\lambda-R_0)\bm\omega,
\qquad
\lambda\ge R_0.
\]
and set
\[
\zeta:=\frac23\lambda^{3/2},
\qquad
\zeta_0:=\frac23R_0^{3/2}.
\]
Along this ray, \Cref{prop:edge-radial-block} gives
\begin{align}
\frac{\rd}{\rd\zeta}U_0
&=
\frac1\zeta B_{01}(\zeta;\bm z,\bm\omega)U_\ast,
\label{eq:edge-radial-top-thm}
\\[1mm]
\frac{\rd}{\rd\zeta}U_\ast
&=
\Gamma(\zeta;\bm z,\bm\omega)U_\ast
+
\frac1\zeta B_{10}(\zeta;\bm z,\bm\omega)U_0
+
\frac1\zeta B_{11}(\zeta;\bm z,\bm\omega)U_\ast.
\label{eq:edge-radial-bottom-thm}
\end{align}
After increasing \(M\) if necessary, we may assume that
\[
\|B_{01}(\zeta;\bm z,\bm\omega)\|
+
\|B_{10}(\zeta;\bm z,\bm\omega)\|
+
\|B_{11}(\zeta;\bm z,\bm\omega)\|
\le M,
\qquad
\zeta\ge\zeta_0,
\]
and that the constant in \eqref{eq:edge-Gamma-relative-bound} is at most
\(M\). Recall that \(\Gamma\) is block lower triangular under the ordering by
increasing \(|I|+|A|\). Its entries need not be uniformly bounded when some
coordinates are much larger than the separation scale; the relative estimate
\eqref{eq:edge-Gamma-relative-bound} is used instead.

\medskip
\noindent
\emph{Step 1: Exponential decay of the principal backward evolution.}

For \(\xi\ge\zeta\), let
\(\mathcal E(\zeta,\xi;\bm z,\bm\omega)\) be the backward evolution generated
by the full triangular matrix \(\Gamma\), namely
\begin{equation}
\partial_\zeta\mathcal E(\zeta,\xi;\bm z,\bm\omega)
=
\Gamma(\zeta;\bm z,\bm\omega)
\mathcal E(\zeta,\xi;\bm z,\bm\omega),
\qquad
\mathcal E(\xi,\xi;\bm z,\bm\omega)=I.
\label{eq:edge-principal-evolution}
\end{equation}
By \eqref{eq:edge-spectral-gap}, the backward evolution generated by the
diagonal part of \(\Gamma\) satisfies
\begin{equation}
\left\|
\exp\left(
-\int_\zeta^\xi
\operatorname{diag}\Gamma(u;\bm z,\bm\omega)\,\rd u
\right)
\right\|
\le
 e^{-c_\delta(\xi-\zeta)}.
\label{eq:edge-diagonal-evolution}
\end{equation}

By \eqref{eq:edge-Gamma-triangular}, every off-diagonal factor strictly
increases the total derivative order from its column to its row. Therefore,
as in the bulk expansion \eqref{eq:finite-variation-expansion}, the
variation-of-constants expansion terminates after at most \(n+m-1\)
off-diagonal factors.

We estimate the terms in this finite expansion. Consider a strictly increasing
chain of non-empty square-free pairs, with \(1\le r\le n+m-1\),
\[
(I_0,A_0)
\subsetneq
(I_1,A_1)
\subsetneq
\cdots
\subsetneq
(I_r,A_r),
\]
and transition times
$
\xi\ge\tau_1\ge\cdots\ge\tau_r\ge\zeta.
$
The diagonal factors appearing in the term associated with this chain are
\begin{align*}
&
\exp\left(
-\int_{\tau_1}^{\xi}
\Gamma_{I_0,A_0}(u;\bm z,\bm\omega)\,\rd u
\right)
\prod_{j=1}^{r-1}
\exp\left(
-\int_{\tau_{j+1}}^{\tau_j}
\Gamma_{I_j,A_j}(u;\bm z,\bm\omega)\,\rd u
\right)
\exp\left(
-\int_{\zeta}^{\tau_r}
\Gamma_{I_r,A_r}(u;\bm z,\bm\omega)\,\rd u
\right).
\end{align*}
Taking absolute values and using the telescoping identity, we obtain
\begin{align*}
&\exp\left(
-\int_\zeta^\xi
\Re\Gamma_{I_0,A_0}(u;\bm z,\bm\omega)\,\rd u
\right)
\prod_{j=1}^r
\exp\left(
-\int_\zeta^{\tau_j}
\Re\left[
\Gamma_{I_j,A_j}(u;\bm z,\bm\omega)
-
\Gamma_{I_{j-1},A_{j-1}}(u;\bm z,\bm\omega)
\right] \,\rd u
\right).
\end{align*}
By \eqref{eq:edge-Gamma-relative-bound}, the off-diagonal factor at
\(\tau_j\) is bounded by
\[
M\Re\left[
\Gamma_{I_j,A_j}(\tau_j;\bm z,\bm\omega)
-
\Gamma_{I_{j-1},A_{j-1}}(\tau_j;\bm z,\bm\omega)
\right].
\]
The real part in the last display is nonnegative by the angular estimate
proved in \Cref{p:direction-sector} and the diagonal formula
\eqref{eq:edge-mu-IA-statement}. Hence, after dropping the ordering
restriction on the transition times, the absolute value of the integral
associated with the chain is bounded by
\begin{align*}
&M^r
\exp\left(
-\int_\zeta^\xi
\Re\Gamma_{I_0,A_0}(u;\bm z,\bm\omega)\,\rd u
\right)
\prod_{j=1}^r
\int_\zeta^\xi
\Re\left[
\Gamma_{I_j,A_j}(\tau;\bm z,\bm\omega)
-
\Gamma_{I_{j-1},A_{j-1}}(\tau;\bm z,\bm\omega)
\right]
\\
&\hspace{28mm}\times
\exp\left(
-\int_\zeta^\tau
\Re\left[
\Gamma_{I_j,A_j}(u;\bm z,\bm\omega)
-
\Gamma_{I_{j-1},A_{j-1}}(u;\bm z,\bm\omega)
\right] \,\rd u
\right)\rd\tau.
\end{align*}
Each integral in the product is at most \(1\). Indeed, since the real
part in question is nonnegative, for every \(j=1,\dots,r\),
\begin{align*}
&\int_\zeta^\xi
\Re\left[
\Gamma_{I_j,A_j}(\tau;\bm z,\bm\omega)
-
\Gamma_{I_{j-1},A_{j-1}}(\tau;\bm z,\bm\omega)
\right]
\\
&\quad\times
\exp\left(
-\int_\zeta^\tau
\Re\left[
\Gamma_{I_j,A_j}(u;\bm z,\bm\omega)
-
\Gamma_{I_{j-1},A_{j-1}}(u;\bm z,\bm\omega)
\right] \,\rd u
\right)\rd\tau
\\
&=
1-
\exp\left(
-\int_\zeta^\xi
\Re\left[
\Gamma_{I_j,A_j}(u;\bm z,\bm\omega)
-
\Gamma_{I_{j-1},A_{j-1}}(u;\bm z,\bm\omega)
\right] \,\rd u
\right)
\le 1.
\end{align*}
Moreover, \eqref{eq:edge-spectral-gap} gives
\[
\exp\left(
-\int_\zeta^\xi
\Re\Gamma_{I_0,A_0}(u;\bm z,\bm\omega)\,\rd u
\right)
\le
 e^{-c_\delta(\xi-\zeta)}.
\]
There are only finitely many strictly increasing chains of square-free pairs.
Summing the preceding bounds and enlarging \(C\), we obtain
\begin{equation}
\|\mathcal E(\zeta,\xi;\bm z,\bm\omega)\|
\le
C e^{-c_\delta(\xi-\zeta)},
\qquad
\xi\ge\zeta\ge\zeta_0.
\label{eq:edge-propagator-est}
\end{equation}
Thus the full principal backward evolution has uniform exponential decay.

\medskip
\noindent
\emph{Step 2: Reduction to a Volterra equation.}

We impose the normalization
\[
U_0(\zeta)\to1,
\qquad
U_\ast(\zeta)\to0,
\qquad
\zeta\to\infty.
\]
Integrating \eqref{eq:edge-radial-top-thm} from \(\zeta\) to infinity gives
\begin{equation}
U_0(\zeta)
=
1+(\mathcal K U_\ast)(\zeta),
\qquad
(\mathcal K U_\ast)(\zeta)
:=
-\int_\zeta^\infty
\frac1\xi
B_{01}(\xi;\bm z,\bm\omega)U_\ast(\xi)\,\rd\xi.
\label{eq:edge-U0-volterra}
\end{equation}
Substituting this into \eqref{eq:edge-radial-bottom-thm} and using variation
of constants with the full triangular propagator \(\mathcal E\) gives
\begin{equation}
U_\ast=\mathcal T(U_\ast),
\label{eq:edge-volterra}
\end{equation}
where
\begin{align}
(\mathcal T U_\ast)(\zeta)
:={}&
-\int_\zeta^\infty
\mathcal E(\zeta,\xi;\bm z,\bm\omega)
\frac1\xi
\Bigl[
B_{10}(\xi;\bm z,\bm\omega)
\bigl(1+\mathcal K U_\ast(\xi)\bigr)
+
B_{11}(\xi;\bm z,\bm\omega)U_\ast(\xi)
\Bigr] \,\rd\xi.
\label{eq:edge-T-def}
\end{align}

\medskip
\noindent
\emph{Step 3: The fixed-point space and the basic estimates.}

Let \(\mathcal B_{\bm z}\) be the Banach space of continuous functions
\[
U_\ast:[\zeta_0,\infty)\longrightarrow\bC^{|\mathcal J^\ast|}
\]
with norm
\[
\|U_\ast\|_{\mathcal B_{\bm z}}
:=
\sup_{\zeta\ge\zeta_0}\zeta\|U_\ast(\zeta)\|_2.
\]
If \(U_\ast\in\mathcal B_{\bm z}\), then by
\eqref{eq:edge-U0-volterra},
\begin{align*}
|(\mathcal K U_\ast)(\zeta)|
&\le
M\int_\zeta^\infty
\frac{\|U_\ast(\xi)\|_2}{\xi}\,\rd\xi
\le
M\|U_\ast\|_{\mathcal B_{\bm z}}
\int_\zeta^\infty\xi^{-2}\,\rd\xi
=
\frac{M}{\zeta}\|U_\ast\|_{\mathcal B_{\bm z}}.
\end{align*}
Hence
\begin{equation}
|(\mathcal K U_\ast)(\zeta)|
\le
\frac{C}{\zeta}\|U_\ast\|_{\mathcal B_{\bm z}}.
\label{eq:edge-K-estimate}
\end{equation}

Using \eqref{eq:edge-propagator-est}, \eqref{eq:edge-T-def}, and
\eqref{eq:edge-K-estimate}, we have
\begin{align*}
\zeta\|(\mathcal T U_\ast)(\zeta)\|_2
&\le
M\zeta\int_\zeta^\infty
\|\mathcal E(\zeta,\xi;\bm z,\bm\omega)\|
\frac1\xi
\Bigl(
1+|(\mathcal K U_\ast)(\xi)|+\|U_\ast(\xi)\|_2
\Bigr)\,\rd\xi
\\
&\le
CM\zeta\int_\zeta^\infty
 e^{-c_\delta(\xi-\zeta)}
\Bigl(
\xi^{-1}
+
\|U_\ast\|_{\mathcal B_{\bm z}}\xi^{-2}
\Bigr)\,\rd\xi
\\
&\le
C_0+C_1\zeta_0^{-1}\|U_\ast\|_{\mathcal B_{\bm z}}.
\end{align*}
Therefore
\begin{equation}
\|\mathcal T U_\ast\|_{\mathcal B_{\bm z}}
\le
C_0+C_1\zeta_0^{-1}\|U_\ast\|_{\mathcal B_{\bm z}}.
\label{eq:edge-T-map}
\end{equation}

Similarly, for \(U_\ast,V_\ast\in\mathcal B_{\bm z}\),
\[
|(\mathcal K U_\ast)(\zeta)-(\mathcal K V_\ast)(\zeta)|
\le
\frac{C}{\zeta}
\|U_\ast-V_\ast\|_{\mathcal B_{\bm z}}.
\]
Hence
\begin{align*}
\zeta\|(\mathcal T U_\ast-\mathcal T V_\ast)(\zeta)\|_2
&\le
M\zeta\int_\zeta^\infty
\|\mathcal E(\zeta,\xi;\bm z,\bm\omega)\|
\frac1\xi
\Bigl(
|(\mathcal K U_\ast)(\xi)-(\mathcal K V_\ast)(\xi)|
+
\|U_\ast(\xi)-V_\ast(\xi)\|_2
\Bigr)\,\rd\xi
\\
&\le
CM\zeta\int_\zeta^\infty
 e^{-c_\delta(\xi-\zeta)}\xi^{-2}\,\rd\xi
\,\|U_\ast-V_\ast\|_{\mathcal B_{\bm z}}
\\
&\le
C_2\zeta_0^{-1}
\|U_\ast-V_\ast\|_{\mathcal B_{\bm z}}.
\end{align*}
Thus
\begin{equation}
\|\mathcal T U_\ast-\mathcal T V_\ast\|_{\mathcal B_{\bm z}}
\le
C_2\zeta_0^{-1}
\|U_\ast-V_\ast\|_{\mathcal B_{\bm z}}.
\label{eq:edge-T-contraction}
\end{equation}

\medskip
\noindent
\emph{Step 4: Existence, asymptotics, and raywise uniqueness.}

If \(R_0\), equivalently \(\zeta_0\), is sufficiently large, then
\(C_2\zeta_0^{-1}<1\), so by \eqref{eq:edge-T-contraction}
\(\mathcal T\) is a contraction on \(\mathcal B_{\bm z}\). Hence
\eqref{eq:edge-volterra} admits a unique fixed point
\(U_\ast\in\mathcal B_{\bm z}\). Defining \(U_0\) by
\eqref{eq:edge-U0-volterra}, we obtain a unique solution
\((U_0,U_\ast)\) of
\eqref{eq:edge-radial-top-thm}--\eqref{eq:edge-radial-bottom-thm}
satisfying
\[
U_0(\zeta)\to1,
\qquad
U_\ast(\zeta)\to0,
\qquad
\zeta\to\infty.
\]
Moreover, \eqref{eq:edge-K-estimate} and \eqref{eq:edge-T-map} imply
\begin{equation}
U_0(\zeta)=1+\OO(\zeta^{-1}),
\qquad
U_\ast(\zeta)=\OO(\zeta^{-1}),
\qquad
\zeta\ge\zeta_0.
\label{eq:edge-ray-asymptotics}
\end{equation}

Now define
\[
W(\bm z):=U_0(\zeta_0).
\]
Since \(\zeta_0=2R_0^{3/2}/3\), we have
\[
W(\bm z)
=
1+\OO(\zeta_0^{-1})
=
1+\OO\bigl(\mathsf R(\bm z)^{-3/2}\bigr).
\]
For every \((I,A)\in\mathcal J^\ast\),
\eqref{eq:edge-ray-asymptotics} also gives
\[
U_{I,A}(\zeta_0)
=
\OO\bigl(\mathsf R(\bm z)^{-3/2}\bigr).
\]
At this stage these are estimates for the components of the constructed
square-free jet. In Step 5 we verify that this jet is horizontal for the full
flat system; its components are then identified with the square-free
derivatives of \(W\), yielding the desired derivative estimates.

The same argument applied to the difference of two normalized solutions has
no constant term in \eqref{eq:edge-U0-volterra}. Therefore the corresponding
Volterra map is a strict contraction with zero forcing, and the difference
vanishes. This proves uniqueness along every admissible shifted ray.

\medskip
\noindent
\emph{Step 5: Holomorphic dependence and gluing on \(\sfC\).}

Fix \(\bm z_0\in\sfC\). Choose a fixed number \(R_0>0\), and shrink to a sufficiently small neighborhood
\(\fU\subset\sfC\) of \(\bm z_0\), so that
\[
R_0\le\sfR(\bm z)\le2R_0,
\qquad
\bm z\in\fU.
\]
For every \(\bm z\in\fU\), use the rays
\[
\bm z+(\lambda-R_0)\bm\omega,
\qquad
\lambda\ge R_0.
\]
The separation and phase estimates of \Cref{p:direction-sector} hold
uniformly on this family of rays. Since both \(R_0\) and \(\bm\omega\) are
fixed on \(\fU\), all coefficients of the Volterra equation depend
holomorphically on \(\bm z\), and its contraction estimate is uniform.
Hence the normalized fixed point depends holomorphically on \(\bm z\).

We have so far solved the first-order system only in the radial direction. To
check the remaining coordinate directions, fix one coordinate and compare two
ways of differentiating the constructed square-free jet: differentiation with
respect to the base point, and the corresponding coordinate equation of the
full square-free first-order system. Their difference satisfies the
homogeneous radial system by flatness. The estimates above hold uniformly for
nearby base points, so this difference has zero normalization at infinity.
Raywise uniqueness forces it to vanish. Thus the constructed jet is a
holomorphic horizontal section of the full flat system. Its empty component
is a holomorphic solution
\(W_{\bm\epsilon,\bm\eta}^{(\fU)}\) of the conjugated edge system, and
\[
W_{\bm\epsilon,\bm\eta}^{(\fU)}
=
1+\OO\bigl(\sfR^{-3/2}\bigr),
\qquad
\partial_{\bmt_I}\partial_{\bms_A}
W_{\bm\epsilon,\bm\eta}^{(\fU)}
=
\OO\bigl(\sfR^{-|I|-|A|-3/2}\bigr).
\]

The local solutions agree on overlaps. Indeed, all neighborhoods use the
same vector \(\bm\omega\). At a point \(\bm z\in\fU\cap\fV\), undo the
radial scaling and write both constructions on the same geometric ray
\(\bm z+r\bm\omega\), \(r\ge0\). They solve the same radial system and have
the same normalization at infinity. Raywise uniqueness therefore gives
\[
W_{\bm\epsilon,\bm\eta}^{(\fU)}
=
W_{\bm\epsilon,\bm\eta}^{(\fV)}
\qquad
\text{on }\fU\cap\fV.
\]
The local solutions glue to a holomorphic function
\(W_{\bm\epsilon,\bm\eta}\) on all of \(\sfC\). Setting
\[
F_{\bm\epsilon,\bm\eta}
:=
\Phi_{\bm\epsilon,\bm\eta}
W_{\bm\epsilon,\bm\eta}
\]
gives a holomorphic solution of the original edge system satisfying
\eqref{eq:edge-W-basic-est}--\eqref{eq:edge-jet-basic-est}.

\medskip
\noindent
\emph{Step 6: Linear independence.}

For a square-free pair \((I,A)\), the Airy asymptotics, the explicit
logarithmic derivatives of \(\Phi_{\bm\epsilon,\bm\eta}\), and
\eqref{eq:edge-W-basic-est}--\eqref{eq:edge-jet-basic-est} give, uniformly on
\(\sfC\),
\begin{align}\label{eq:edge-normalized-jet-matrix}
&\left(
\prod_{i\in I}\sqrt{t_i}
\prod_{a\in A}\frac{\beta}{2}\sqrt{s_a}
\right)^{-1}
\Phi_{\bm\epsilon,\bm\eta}^{-1}
\partial_{\bmt_I}\partial_{\bms_A}
F_{\bm\epsilon,\bm\eta}
=
\prod_{i\in I}\epsilon_i
\prod_{a\in A}\eta_a
+
\OO\bigl(D^{-3/2}\bigr).
\end{align}
The leading matrix in \eqref{eq:edge-normalized-jet-matrix}, with rows indexed
by \((I,A)\) and columns indexed by \((\bm\epsilon,\bm\eta)\), is the
character table
\[
\left(
\prod_{i\in I}\epsilon_i
\prod_{a\in A}\eta_a
\right)_{(I,A),(\bm\epsilon,\bm\eta)}
\]
of \(\{\pm1\}^{n+m}\). It is invertible. Increasing \(D\), if necessary,
the exact normalized jet matrix is also invertible. Therefore the
\(2^{n+m}\) solutions are linearly independent.
\end{proof}

\appendix

\section{Airy functions}
\label{a:Airy}

We recall some standard facts about the Airy function \(\Ai\). For real
\(x\), it is given by the conditionally convergent integral
\[
\Ai(x)
=
\frac{1}{\pi}
\lim_{R\to\infty}
\int_0^R
\cos\left(\frac{t^3}{3}+xt\right)\,\rd t .
\]
The function \(\Ai\) extends to an entire function on \(\bC\), and it solves
the Airy equation
\begin{align}\label{e:Airy_eq}
\Ai''(w)-w\Ai(w)=0 .
\end{align}
It is characterized by the decay condition
\[
\Ai(w)\to0
\qquad \text{as } w\to+\infty
\]
along the positive real axis. All zeros of \(\Ai\) are real and negative; we
write them as
\[
0>\fa_1>\fa_2>\fa_3>\cdots .
\]
It is known that the Airy zeros satisfy
$
\fa_i\sim -\left({3\pi i}/{2}\right)^{2/3}$.
More precisely, for every \(i\in\bN\), see for example
\cite[eqs. (9.9.6) and (9.9.18)]{NIST:DLMF},
\begin{equation}\label{e:airy-zero-location0}
\left|
\fa_i+\left(\frac{3\pi i}{2}\right)^{2/3}
\right|
\lesssim i^{-1/3}.
\end{equation}
We refer to \cite[Chapter 9]{NIST:DLMF} for further background on Airy
functions.

\subsection{Nevanlinna representation}

The logarithmic derivative
\[
-\frac{\Ai'(w)}{\Ai(w)}
\]
is a particle-generated Nevanlinna function. Its associated configuration is
the Airy-zero configuration
\[
\Gamma_{\Ai}:=\sum_{i=1}^{\infty}\delta_{\fa_i}.
\]
More precisely, one has the partial fraction expansion
\begin{equation}\label{eq:weairy}
-\frac{\Ai'(w)}{\Ai(w)}
=
\sum_{i=1}^{\infty}
\left(
\frac{1}{\fa_i-w}
-
\frac{1}{\fa_i}
\right)
-
\frac{\Ai'(0)}{\Ai(0)} .
\end{equation}
We shall also use the square-root asymptotic
\begin{equation}\label{e:aiasymp}
\left|
\frac{\Ai'(w)}{\Ai(w)}
+
\sqrt w
\right|
\lesssim |w|^{-1},
\end{equation}
valid for \(|w|\) sufficiently large and
\[
|\arg w|<\pi-|w|^{-9/7}.
\]
The estimates \eqref{eq:weairy} and \eqref{e:aiasymp} are proved in
\cite[Appendix A.1]{huang2024convergence}.

\subsection{Asymptotics for \(\Ai\)}

We recall the standard Airy asymptotic
\begin{align}\label{e:Ai_asymp}
\Ai(w)
=
\frac{1}{2\sqrt{\pi}}w^{-1/4}
\exp\left(-\frac23 w^{3/2}\right)
\left(1+\OO_\delta(|w|^{-3/2})\right),
\qquad |w|\to\infty,
\end{align}
uniformly for \(|\arg w|\le \pi-\delta\), see
\cite[Section 9.7(iv)]{NIST:DLMF} and \cite[Appendix B]{nemes2017error}. Here and below, fractional powers such as \(\sqrt w\), \(w^{3/2}\),
and \(w^{-1/4}\) are taken with the principal branch on $\bC\setminus \bR_{<0}$. 

Let
\[
\omega:=e^{-2\pi\ri/3}.
\]
If \(\alpha^3=1\), then \(\Ai(\alpha z)\) also solves the Airy equation
\eqref{e:Airy_eq}. In particular,
\[
\ri\omega \Ai(\omega z),
\qquad
-\ri\omega^2 \Ai(\omega^2 z)
\]
are also solutions of \eqref{e:Airy_eq}. 

Moreover, in the upper half-plane \(0<\arg z<\pi\), the solution
\(\ri\omega \Ai(\omega z)\) has the pure \(e^{2z^{3/2}/3}\) asymptotic
\begin{equation}\label{e:Bi-minus-iAi-upper}
\ri\omega \Ai(\omega z)
=
\frac{1}{2\sqrt{\pi}}z^{-1/4}
\exp\left(\frac23 z^{3/2}\right)
\left(1+\OO(|z|^{-3/2})\right),
\qquad |z|\to\infty.
\end{equation}
Similarly, in the lower half-plane \(-\pi<\arg z<0\),
\(-\ri\omega^2 \Ai(\omega^2 z)\) has the pure \(e^{2z^{3/2}/3}\)
asymptotic
\begin{equation}\label{e:Bi-plus-iAi-lower}
-\ri\omega^2 \Ai(\omega^2 z)
=
\frac{1}{2\sqrt{\pi}}z^{-1/4}
\exp\left(\frac23 z^{3/2}\right)
\left(1+\OO(|z|^{-3/2})\right),
\qquad |z|\to\infty.
\end{equation}

\section{Loop equations of $\beta$-ensembles}
\label{s:beta_ensemble}

Here we recall the notation, estimates, and loop equations for the
\(\beta\)-ensemble from \cite{bourgade2022optimal}. We use the notation
introduced in \Cref{s:beta_loop}. As in \cite[Section 2.1]{bourgade2022optimal},
\(m_V\) satisfies the fixed-point equation
\begin{align}\label{e:fix}
	m_V(z)^2+V'(z)m_V(z)+h(z)=0,
\end{align}
where
\[
	h(z)
	:=
	\int_A^B
	\frac{V'(\lambda)-V'(z)}{\lambda-z}\,
	\varrho_V(\lambda)\,\rd\lambda .
\]
Moreover,
\begin{align}\label{e:MV2}
	2m_V(z)+V'(z)=2r(z)\sqrt{z-A}\sqrt{z-B},
\end{align}
with the usual branch choice. In particular, at the right edge,
\begin{equation}
	\label{eq:right_edge_identity}
	2m_B+V'(B)=0.
\end{equation}

We denote the empirical Stieltjes transform as
\[
	\fs_N(z)
	:=
	\frac{1}{N}\sum_{j=1}^N\frac{1}{\lambda_j-z}.
\]

\begin{proposition}[Loop equation {\cite[Proposition A.3]{bourgade2022optimal}}]
For any \(p\geq 1\) and
\(z, z_2,z_3,\dots,z_p\in\bC\setminus\bR\),
\begin{equation}
	\label{eqn:loop1}
	\begin{aligned}
	\bE\left[
	\left(
		\fs_N(z)^2
		+\frac{1}{N}\left(\frac{2}{\beta}-1\right)\fs_N'(z)
		+\frac{1}{N}\sum_{k=1}^N
		\frac{V'(\lambda_k)}{\lambda_k-z}
	\right)\prod_{i=2}^{p}\fs_N(z_i)\right]\\
	+\frac{2}{\beta N^2}\bE\left[
	\sum_{j=2}^{p}
	\partial_{z_j}
	\frac{\fs_N(z)-\fs_N(z_j)}{z-z_j}
	\prod_{\substack{i=2\\ i\neq j}}^p \fs_N(z_i)
	\right]=0
	\end{aligned}
\end{equation}
\end{proposition}

Introduce
\begin{equation}\label{eq:def_Delta}
	P_N(z):=\fs_N(z)^2+V'(z)\fs_N(z)+h(z),\quad
	\Delta_N(z)
	:=
	\frac{1}{N}\sum_{k=1}^N
	\frac{V'(\lambda_k)-V'(z)}{\lambda_k-z}
	-
	h(z).
\end{equation}
Then \eqref{eqn:loop1} becomes
\begin{equation}
	\label{eqn:loop2}
\bE\left[
		\left(
			P_N(z)
			+\frac{1}{N}\left(\frac{2}{\beta}-1\right)\fs_N'(z)
			+\Delta_N(z)
		\right)
		\prod_{i=2}^{p}\fs_N(z_i)
	\right]
	+\frac{2}{\beta N^2}
	\bE\left[
		\sum_{j=2}^{p}
		\partial_{z_j}
		\frac{\fs_N(z)-\fs_N(z_j)}{z-z_j}
		\prod_{\substack{i=2\\ i\neq j}}^p\fs_N(z_i)
	\right]=0.
\end{equation}

\begin{proposition}[ {\cite[Lemma 2.1 and Proposition 3.5]{bourgade2022optimal}}]
For any compact set $K\subset \bC$, there exists $C>0$ such that for any 
\(q\geq 1\),
\begin{equation}
	\label{eq:Delta_moment_bound}
	\bE\left[\abs{\Delta_N(z)}^{2q}\right]
	\leq
	\frac{(Cq)^{2q}}{N^{2q}}.
\end{equation}
Moreover, there exist $\eta_0>0$ and $C>0$ such that for any $q\geq 1$, 
any \(z=E+\ri\eta\) with \(0<\eta\leq \eta_0\) and
\(A-\eta_0\leq E\leq B+\eta_0\),
\begin{equation}
	\label{eq:local_law_beta}
	\bE\left[\abs{\fs_N(z)-m_V(z)}^q\right]
	\leq
	\frac{(Cq)^{q/2}}{(N\eta)^q}
	+
	\frac{(Cq)^{2q}}{(N\eta)^q}.
\end{equation}
\end{proposition}

\subsection{Bulk loop equations for $\beta$-ensembles}

\begin{proof}[Proof of \Cref{p:bulk-loop-beta}]
Set
\begin{align}\label{e:defz}
	z=E+\frac{w}{N\pi\varrho_E},
	\qquad
	z_j=E+\frac{w_j}{N\pi\varrho_E}.
\end{align}
By linearity of the loop equation, we may replace
$
	\prod_{j=2}^p \fs_N(z_j)
$
by
$
	\prod_{j=2}^p
	(
		\fs_N(z_j)+V'(E)/2).
$
Starting from the loop equation in the form \eqref{eqn:loop2}, we therefore get
\begin{align}
	\label{e:bulk-loop-before-scaling-new}
&\bE\left[
\left(
P_N(z)
+\frac1N\left(\frac2\beta-1\right)\partial_z\fs_N(z)
+\Delta_N(z)
\right)
\prod_{j=2}^{p}
\left(
	\fs_N(z_j)+\frac{V'(E)}{2}
\right)
\right]
\notag\\
&\quad
+
\frac{2}{\beta N^2}
\bE\left[
\sum_{j=2}^{p}
\partial_{z_j}
\frac{\fs_N(z)-\fs_N(z_j)}{z-z_j}
\prod_{\substack{i=2\\ i\neq j}}^p
\left(
	\fs_N(z_i)+\frac{V'(E)}{2}
\right)
\right]=0.
\end{align}

We recall from \eqref{e:defsN0}
\begin{align}\label{e:defsN2}
	\fs_N(z)+\frac{V'(E)}{2}
	=
	\pi\varrho_E s_N(w).
\end{align}

We divide \eqref{e:bulk-loop-before-scaling-new} by
$(\pi\varrho_E)^{p+2}$. First consider the derivative term. Using the relation \eqref{e:defz} and \eqref{e:defsN2}
we have
\[
	\partial_w s_N(w)
	=
	\frac{1}{\pi\varrho_E}
	\frac{1}{N\pi\varrho_E}
	\partial_z\fs_N(z)
	=
	\frac{1}{N(\pi\varrho_E)^2}
	\partial_z\fs_N(z).
\]
Therefore
\[
	\frac{1}{(\pi\varrho_E)^2}
	\frac1N
	\left(\frac2\beta-1\right)
	\partial_z\fs_N(z)
	=
	\frac{2-\beta}{\beta}
	\partial_w s_N(w).
\]

Next consider the second loop term. Using the relation \eqref{e:defz} and \eqref{e:defsN2}, we get
\[
	\frac{\fs_N(z)-\fs_N(z_j)}{z-z_j}
	=
	N(\pi\varrho_E)^2
	\frac{s_N(w)-s_N(w_j)}{w-w_j}.
\]
We further notice that
$
	\partial_{z_j}
	=
	N\pi\varrho_E\,\partial_{w_j},
$
and it follows 
\[
	\frac{1}{N^2}
	\partial_{z_j}
	\frac{\fs_N(z)-\fs_N(z_j)}{z-z_j}
	=
	(\pi\varrho_E)^3
	\partial_{w_j}
	\frac{s_N(w)-s_N(w_j)}{w-w_j}.
\]
After division by \((\pi\varrho_E)^{p+2}\), and using \eqref{e:defsN2}
\[
	\prod_{\substack{i=2\\ i\neq j}}^p
	\left(
		\fs_N(z_i)+\frac{V'(E)}{2}
	\right)
	=
	(\pi\varrho_E)^{p-1}
	\prod_{\substack{i=2\\ i\neq j}}^ps_N(w_i),
\]
the second loop term becomes exactly
\[
	\partial_{w_j}
	\frac{s_N(w)-s_N(w_j)}{w-w_j}
	\prod_{\substack{\ell=2\\ \ell\neq j}}^p s_N(w_\ell).
\]

It remains to identify the leading part of \(P_N(z)\). Recall from \eqref{eq:def_Delta}
\[
	P_N(z)=\fs_N(z)^2+V'(z)\fs_N(z)+h(z).
\]
Using \eqref{e:defsN2}, 
we have
\begin{align}
	P_N(z)
	=
	(\pi\varrho_E)^2s_N(w)^2
	+
	\bigl(V'(z)-V'(E)\bigr)\pi\varrho_E s_N(w)
	+
	h(z)
	+\frac{V'(E)^2}{4}
	-\frac12 V'(z)V'(E).
	\label{e:bulk-P-expansion-new}
\end{align}
At the bulk point \(E\), the boundary value of \(m_V\) satisfies
\begin{align}\label{e:MV}
	m_V(E+\ri0)
	=
	-\frac12 V'(E)+\ri\pi\varrho_E.
\end{align}
Using the fixed point equation \eqref{e:fix}
\[
	m_V(z)^2+V'(z)m_V(z)+h(z)=0
\]
at \(E+\ri0\), together with \eqref{e:MV}, we obtain
\[
	h(E)
	=
	\frac{V'(E)^2}{4}
	+
	\pi^2\varrho_E^2.
\]
Since \(V\) and \(h\) are analytic near \(E\), and $|z-E|=\OO_K(1/N)$,
we have uniformly for \(w\in K\),
\[
	V'(z)-V'(E)=\OO_K(N^{-1}),
	\qquad
	h(z)-h(E)=\OO_K(N^{-1}).
\]
Thus \eqref{e:bulk-P-expansion-new} gives
\[
	P_N(z)
	=
	(\pi\varrho_E)^2
	\bigl(s_N(w)^2+1\bigr)
	+
	\OO_K(N^{-1})
	\bigl(1+|s_N(w)|\bigr).
\]
Equivalently,
\begin{equation}
	\label{e:bulk-P-final-new}
	\frac{P_N(z)}{(\pi\varrho_E)^2}
	=
	s_N(w)^2+1
	+
	\OO_K(N^{-1})
	\bigl(1+|s_N(w)|\bigr).
\end{equation}

We now record the moment bounds needed to absorb the error terms. The
local law \eqref{eq:local_law_beta} gives, for every fixed \(q\ge1\),
\begin{align}\label{e:smomentbb}
	\sup_{w\in K}
	\bE\left[
		|s_N(w)|^q
	\right]
	\leq C_{q,K}.
\end{align}
For points in the lower half-plane, this follows by conjugation. Using \eqref{e:sbound} and \eqref{e:fest}, the derivative terms have uniformly bounded moments as well. For \(K\subset\bC\setminus\bR\),
\[
	\abs{\partial_w s_N(w)}
	\leq
	C_K\abs{s_N(w)},\quad 
	\left|
	\partial_{w_j}
	\frac{s_N(w)-s_N(w_j)}{w-w_j}
	\right|
	\leq
	C_K\bigl(\abs{s_N(w)}+\abs{s_N(w_j)}\bigr).
\]

Finally, the \(\Delta_N\)-term is negligible. In fact \eqref{eq:Delta_moment_bound}, \eqref{e:smomentbb} and 
Hölder's inequality gives,
\[
	\sup_{w,w_2,\dots,w_p\in K}
	\left|
	\bE\left[
		\frac{\Delta_N(z)}{(\pi\varrho_E)^2}
		\prod_{j=2}^p s_N(w_j)
	\right]
	\right|
	=
	\OO(N^{-1})
	=
	\oo_N(1).
\]

Substituting \eqref{e:bulk-P-final-new} into the scaled form of
\eqref{e:bulk-loop-before-scaling-new}, and using the preceding moment bounds,
we obtain
\begin{align*}
	\bE\left[
		\left(
			\frac{2-\beta}{\beta}\partial_w s_N(w)
			+s_N(w)^2+1
		\right)
		\prod_{j=2}^p s_N(w_j)
	\right]
	+
	\frac{2}{\beta}
	\bE\left[
		\sum_{j=2}^p
		\partial_{w_j}
		\frac{s_N(w)-s_N(w_j)}{w-w_j}
		\prod_{\substack{i=2\\ i\neq j}}^ps_N(w_i)
	\right]
	=
	\oo_N(1),
\end{align*}
uniformly for \(w,w_2,w_3,\dots,w_p\in K\). This is
\eqref{e:approx-loopeq2-bulk}.
\end{proof}

\subsection{Edge loop equations for $\beta$-ensembles}

\begin{proof}[Proof of \Cref{p:edge-loop-beta}]
We write
\[
	z=B+\gamma N^{-2/3}w,
	\qquad
	z_j=B+\gamma N^{-2/3}w_j.
\]

First we record the deterministic edge expansion of \(m_V\). We recall from \eqref{e:MV2} and \eqref{eq:right_edge_identity}
\[
	2m_V(z)+V'(z)=2r(z)\sqrt{z-A}\sqrt{z-B}, \quad 2m_B+V'(B)=0.
\]
We have, uniformly for \(w\in K\),
\[
	m_V(B+\gamma N^{-2/3}w)-m_B
	=
	r(B)\sqrt{B-A}\,\gamma^{1/2}N^{-1/3}\sqrt w
	+
	\OO_K(N^{-2/3}).
\]
The branch of \(\sqrt w\) is the principal branch. By the definition of
\(\gamma\) in \eqref{e:defgamma},
\[
	\gamma^{3/2}r(B)\sqrt{B-A}=1,
\]
and therefore
\begin{equation}
	\label{e:edge-m-expansion}
	\gamma N^{1/3}
	\left(
		m_V(B+\gamma N^{-2/3}w)-m_B
	\right)
	=
	\sqrt w+\OO_K(N^{-1/3}).
\end{equation}
Consequently, we can rewrite \eqref{e:defsN3} as
\begin{align}\begin{split}
	\label{e:edge-centering-relation}
	&s_N(w)
	=
	\gamma N^{1/3}
	\left(
		\fs_N(B+\gamma N^{-2/3}w)-m_B
	\right),\\
	&s_N(w)-\sqrt w
	=
	\gamma N^{1/3}
	\left(
		\fs_N(B+\gamma N^{-2/3}w)
		-
		m_V(B+\gamma N^{-2/3}w)
	\right)
	+
	\OO_K(N^{-1/3}).
\end{split}\end{align}

By the local law \eqref{eq:local_law_beta}, for every fixed \(q\geq1\),
\begin{equation}
	\label{e:scaled-s-moment}
	\sup_{w\in K}
	\bE\left[\abs{s_N(w)}^q\right]
	\leq
	C_{q,K}.
\end{equation}

We now apply the loop equation. By linearity, \eqref{eqn:loop2} remains valid
when
$
	\prod_{j=z}^p \fs_N(z_j)
$
is replaced by the centered product
$
	\prod_{j=2}^p \bigl(\fs_N(z_j)-m_B\bigr)$.
Thus,
\begin{align}
	\label{e:unscaled-loop-centered}
	&\bE\left[
		\left(
			P_N(z)
			+\frac{1}{N}\left(\frac{2}{\beta}-1\right)\partial_z\fs_N(z)
			+\Delta_N(z)
		\right)
		\prod_{j=2}^p
		\bigl(\fs_N(z_j)-m_B\bigr)
	\right]
	\notag\\
	&\qquad
	+
	\frac{2}{\beta N^2}
	\bE\left[
		\sum_{j=2}^p
		\partial_{z_j}
		\frac{\fs_N(z)-\fs_N(z_j)}{z-z_j}
		\prod_{\substack{i=2\\ i\neq j}}^p
		\bigl(\fs_N(z_i)-m_B\bigr)
	\right]=0.
\end{align}

Substituting \(z=B+\gamma N^{-2/3}w\) and
\(z_j=B+\gamma N^{-2/3}w_j\), and multiplying
\eqref{e:unscaled-loop-centered} by
$
	\bigl(\gamma N^{1/3}\bigr)^{p+2}$,
we obtain
\begin{align}
	\label{e:scaled-loop-before-expansion}
	&\bE\left[
		\left(
			\gamma^2N^{2/3}P_N(z)
			+
			\left(\frac{2}{\beta}-1\right)\partial_w s_N(w)
			+
			\gamma^2N^{2/3}\Delta_N(z)
		\right)
		\prod_{j=2}^p s_N(w_j)
	\right]
	\notag\\
	&\qquad
	+
	\frac{2}{\beta}
	\bE\left[
		\sum_{j=2}^p
		\partial_{w_j}
		\frac{s_N(w)-s_N(w_j)}{w-w_j}
		\prod_{\substack{i=2\\ i\neq j}}^p s_N(w_i)
	\right]=0.
\end{align}
Indeed, by \eqref{e:edge-centering-relation}
\[
	\partial_w s_N(w)
	=
	\gamma^2N^{-1/3}\partial_z\fs_N(z)
	=
	\frac{\gamma^2N^{2/3}}{N}\partial_z\fs_N(z),
\]
and the same change of variables gives the second derivative term in
\eqref{e:scaled-loop-before-expansion}.

It remains to identify the leading part of
\(\gamma^2N^{2/3}P_N(z)\). We recall from \eqref{eq:def_Delta} and \eqref{e:edge-centering-relation}
\[
	P_N(z)=\fs_N(z)^2+V'(z)\fs_N(z)+h(z),\quad \fs_N(z)=m_B+\frac{s_N(w)}{\gamma N^{1/3}},
\]
then
\begin{align}
	\label{e:scaled-P-expansion}
	\gamma^2N^{2/3}P_N(z)
	=
	s_N(w)^2
	+
	\gamma N^{1/3}\bigl(2m_B+V'(z)\bigr)s_N(w)
	+
	\gamma^2N^{2/3}
	\bigl(m_B^2+V'(z)m_B+h(z)\bigr).
\end{align}
 By \eqref{eq:right_edge_identity},
\[
	2m_B+V'(z)
	=
	V'(B+\gamma N^{-2/3}w)-V'(B)
	=
	\OO_K(N^{-2/3}),
\]
and hence the middle term in \eqref{e:scaled-P-expansion} is negligible,
\begin{equation}
	\label{e:middle-term-bound}
	\gamma N^{1/3}\bigl(2m_B+V'(z)\bigr)
	=
	\OO_K(N^{-1/3}).
\end{equation}

For the final term in \eqref{e:scaled-P-expansion}, use the fixed point equation
for \(m_V\) in \eqref{e:fix}:
\[
	m_V(z)^2+V'(z)m_V(z)+h(z)=0.
\]
Subtracting the expression at \(m_B\) gives
\[
	m_B^2+V'(z)m_B+h(z)
	=
	-\bigl(2m_B+V'(z)\bigr)
	\bigl(m_V(z)-m_B\bigr)
	-
	\bigl(m_V(z)-m_B\bigr)^2.
\]
Therefore, using \eqref{e:edge-m-expansion} and
\eqref{e:middle-term-bound},
\begin{align}
	\label{e:edge-drift}
	\gamma^2N^{2/3}
	\bigl(m_B^2+V'(z)m_B+h(z)\bigr)
	=
	-w+\OO_K(N^{-1/3}).
\end{align}

Combining \eqref{e:scaled-P-expansion}, \eqref{e:middle-term-bound}, and
\eqref{e:edge-drift}, we have
\begin{equation}
	\label{e:scaled-P-final}
	\gamma^2N^{2/3}P_N(z)
	=
	s_N(w)^2-w
	+
	\OO_K(N^{-1/3})\bigl(1+\abs{s_N(w)}\bigr).
\end{equation}

Using \eqref{e:sbound} and \eqref{e:fest},  the derivative terms have uniformly bounded moments as well. For any compact set \(K\subset\bC\setminus\bR\),
\[
	\abs{\partial_w s_N(w)}
	\leq
	C_K\abs{s_N(w)},\quad 
	\left|
	\partial_{w_j}
	\frac{s_N(w)-s_N(w_j)}{w-w_j}
	\right|
	\leq
	C_K\bigl(\abs{s_N(w)}+\abs{s_N(w_j)}\bigr).
\]
Together with \eqref{e:scaled-s-moment} and Hölder's inequality, this gives
uniform moment bounds for all random variables appearing in
\eqref{e:scaled-loop-before-expansion}.

Finally, the \(\Delta_N\)-term is negligible. From
\eqref{eq:Delta_moment_bound},
\[
	\bE\left[\abs{\Delta_N(z)}^{2q}\right]
	\leq
	\frac{(Cq)^{2q}}{N^{2q}},
\]
uniformly for \(z=B+\gamma N^{-2/3}w\), \(w\in K\). Hence, by Hölder's
inequality and \eqref{e:scaled-s-moment},
\[
	\sup_{w,w_2,\dots,w_p\in K}
	\left|
	\bE\left[
		\gamma^2N^{2/3}\Delta_N(z)
		\prod_{j=2}^p s_N(w_j)
	\right]
	\right|
	=
	\OO(N^{-1/3})
	=
	\oo_N(1).
\]

Substituting \eqref{e:scaled-P-final} into
\eqref{e:scaled-loop-before-expansion}, and using the preceding moment bounds,
we obtain, uniformly for \(w,w_2, w_3\dots,w_p\in K\),
\begin{align*}
	\bE\left[
		\left(
			\frac{2-\beta}{\beta}\partial_w s_N(w)
			+s_N(w)^2
			-w
		\right)
		\prod_{j=2}^p s_N(w_j)
	\right]
	+
	\frac{2}{\beta}
	\bE\left[
		\sum_{j=2}^p
		\partial_{w_j}
		\frac{s_N(w)-s_N(w_j)}{w-w_j}
		\prod_{\substack{i=2\\ i\neq j}}^ps_N(w_i)
	\right]
	=
	\oo_N(1).
\end{align*}
This is \eqref{e:edge-loop-from-beta}.
\end{proof}

\section{Proof of the edge local law}\label{s:asymp_edge}

In this section, we prove \Cref{p:edge_concentration}. The proof is parallel
to the bulk case. The main difference from the bulk case is that the edge loop equation is centered
around the algebraic relation \(s(w)^2=w\), rather than \(s(w)^2=-1\). Thus the
natural quantity to estimate is
\[
P(w):=s(w)^2-w,
\]
and one must use a stability estimate to convert moment bounds on \(P(w)\) into
bounds on \(s(w)-\sqrt w\). 

\subsection{Proof of \Cref{eqn:edge} and \Cref{eqn:edge2}}

We start with a deterministic stability estimate.

\begin{lemma}\label{lem:edge-stability}
Let
\begin{align}\label{e:defP2}
P(w):=s(w)^2-w .
\end{align}
Then, with the principal branch of \(\sqrt w\),
\begin{equation}\label{eq:edge-stability}
\min\{|s(w)-\sqrt w|,\ |s(w)+\sqrt w|\}
\le
\frac{|P(w)|}{\sqrt{|w|}}\wedge |P(w)|^{1/2}.
\end{equation}
\end{lemma}

\begin{proof}
Since
\[
P(w)=\bigl(s(w)-\sqrt w\bigr)\bigl(s(w)+\sqrt w\bigr),
\]
we have
\[
\max\{|s(w)-\sqrt w|,\ |s(w)+\sqrt w|\}
\ge
\frac12\left|(s(w)+\sqrt w)-(s(w)-\sqrt w)\right|
=
|\sqrt w|
=
\sqrt{|w|}.
\]
This gives
\[
\min\{|s(w)-\sqrt w|,\ |s(w)+\sqrt w|\}
\le
\frac{|P(w)|}{\sqrt{|w|}}.
\]
The second bound follows from
\[
\min\{|s(w)-\sqrt w|,\ |s(w)+\sqrt w|\}
\le
\sqrt{|s(w)-\sqrt w|\,|s(w)+\sqrt w|}
=
|P(w)|^{1/2}.
\]
\end{proof}

\begin{proof}[Proof of \eqref{eqn:edge} and \eqref{eqn:edge2}]
Throughout the proof, \(C\) denotes a constant depending only on \(\beta\), whose
value may change from line to line. Recall \(P(w)\) from \eqref{e:defP2}. By the
edge loop equation, exactly as in the bulk case, we have
\begin{equation}\label{eq:edge-Xq-second}
\mathbb E\bigl[|P(w)|^{2q}\bigr]
\le
\left(\frac{Cq}{\eta^2}\right)^q
\mathbb E\bigl[\Im [s(w)]^q |s(w)|^q\bigr],
\qquad w=E+\ri\eta .
\end{equation}

We first derive a general moment bound for \(P(w)\). Since \(s\) is a
Nevanlinna function, \(\Im[s(w)]\ge0\), and hence \(\Im[s(w)]\le |s(w)|\).
Therefore,
\begin{align}\label{e:Pmoment}
\mathbb E\bigl[|P(w)|^{2q}\bigr]
\le
\left(\frac{Cq}{\eta^2}\right)^q
\mathbb E\bigl[|s(w)|^{2q}\bigr].
\end{align}
Using \(s(w)^2=w+P(w)\), we get
\begin{align}\label{e:smoment}
|s(w)|^{2q}=|w+P(w)|^q
\le C^q\bigl(|w|^q+|P(w)|^q\bigr).
\end{align}
Thus \eqref{e:Pmoment} and \eqref{e:smoment} together imply
\[
\mathbb E\bigl[|P(w)|^{2q}\bigr]
\le
\left(\frac{Cq}{\eta^2}\right)^q |w|^q
+
\left(\frac{Cq}{\eta^2}\right)^q
\mathbb E\bigl[|P(w)|^q\bigr]\leq \left(\frac{Cq}{\eta^2}\right)^q |w|^q
+
\left(\frac{Cq}{\eta^2}\right)^q
\mathbb E\bigl[|P(w)|^{2q}\bigr]^{1/2}.
\]
Solving the above quadratic inequality yields
\begin{equation}\label{eq:edge-P-moment}
\mathbb E\bigl[|P(w)|^{2q}\bigr]
\le
\left(\frac{Cq}{\eta^2}\right)^{2q}
+
\left(\frac{Cq}{\eta^2}\right)^q |w|^q .
\end{equation}

We now prove \eqref{eqn:edge}. Since \(\Im[s(w)]\ge0\), we have
\[
|s(w)+\sqrt w|\ge \Im[s(w)+\sqrt w]\ge \Im[\sqrt w].
\]
Consequently,
\[
|s(w)-\sqrt w|
\le
|s(w)+\sqrt w|+2|\sqrt w|
\le
C\frac{\sqrt{|w|}}{\Im[\sqrt w]} |s(w)+\sqrt w|.
\]
Therefore the stability estimate \eqref{eq:edge-stability} gives,
\begin{equation}\label{eq:edge-branch-stability}
|s(w)-\sqrt w|
\le
C\frac{\sqrt{|w|}}{\Im[\sqrt w]}
\min\{|s(w)-\sqrt w|,\ |s(w)+\sqrt w|\}
\le
C\frac{\sqrt{|w|}}{\Im[\sqrt w]}\left(
\frac{|P(w)|}{\sqrt{|w|}}\wedge |P(w)|^{1/2}
\right).
\end{equation}

If \(|w|\ge q/\eta^2\), then by \eqref{eq:edge-branch-stability} and
\eqref{eq:edge-P-moment},
\begin{align*}
\mathbb E\bigl[|s(w)-\sqrt w|^{2q}\bigr]\le
\frac{C^q}{(\Im[\sqrt{w}])^{2q}}
\mathbb E\bigl[|P(w)|^{2q}\bigr]  \le
\frac{C^q}{(\Im[\sqrt{w}])^{2q}}
\left[
\left(\frac{Cq}{\eta^2}\right)^{2q}
+
\left(\frac{Cq}{\eta^2}\right)^q |w|^q
\right] \le
\frac{(Cq)^q}{\eta^{2q}}
\frac{|w|^q}{(\Im[\sqrt{w}])^{2q}} .
\end{align*}
If \(|w|<q/\eta^2\), then again by \eqref{eq:edge-branch-stability} and
\eqref{eq:edge-P-moment},
\begin{align*}
\mathbb E\bigl[|s(w)-\sqrt w|^{2q}\bigr]
&\le
C^q \frac{|w|^q}{(\Im[\sqrt{w}])^{2q}}
\mathbb E\bigl[|P(w)|^q\bigr]  \le
C^q \frac{|w|^q}{(\Im[\sqrt{w}])^{2q}}
\mathbb E\bigl[|P(w)|^{2q}\bigr]^{1/2} \\
&\le
\frac{|w|^q}{(\Im[\sqrt{w}])^{2q}}
\left[
\left(\frac{Cq}{\eta^2}\right)^q
+
\left(\frac{Cq}{\eta^2}\right)^{q/2}|w|^{q/2}
\right] \le
\frac{|w|^q}{(\Im[\sqrt{w}])^{2q}}
\frac{(Cq)^q}{\eta^{2q}} .
\end{align*}
This proves \eqref{eqn:edge}.

It remains to prove \eqref{eqn:edge2}. Assume $w=E+\ri \eta$ and \(E\ge\eta\). Then
\(|w|\asymp E\). Moreover,
\[
\left|\Im\left[s(w)-\sqrt w\right]\right|
\le
|s(w)-\sqrt w|.
\]
Since \(\Im[s(w)]\ge0\) and \(\Im[\sqrt w]\ge0\), we also have
\[
\left|\Im\left[s(w)-\sqrt w\right]\right|
=
|\Im[s(w)]-\Im[\sqrt w]|
\le
\Im[s(w)]+\Im[\sqrt w]
=
\Im\left[s(w)+\sqrt w\right]
\le
|s(w)+\sqrt w|.
\]
Hence, with
\begin{align}\label{e:defdelta}
\delta(w):=\min\{|s(w)-\sqrt w|,\ |s(w)+\sqrt w|\},
\end{align}
we have
\begin{equation}\label{eq:edge-im-delta}
\left|\Im\left[s(w)-\sqrt w\right]\right|
\le
\delta(w).
\end{equation}
By \Cref{lem:edge-stability},
\begin{equation}\label{eq:edge-im-stability}
\left|\Im\left[s(w)-\sqrt w\right]\right|
\le
\frac{|P(w)|}{\sqrt{|w|}} .
\end{equation}

We next derive a moment bound for \(P(w)\) adapted to the regime \(E\ge\eta\).
In this case,
\begin{align}\label{e:wwbound}
|\sqrt w|\le C\sqrt E,
\qquad
\Im[\sqrt w]\le C\frac{\eta}{\sqrt E}.
\end{align}
Using \eqref{e:defdelta}, \eqref{eq:edge-im-delta} and \eqref{e:wwbound}, we get
\[
\Im[s(w)]\le \Im[\sqrt w]+\delta(w)
\le
C\frac{\eta}{\sqrt E}+\delta(w),\quad 
|s(w)|\le |\sqrt{w}|+\delta(w)\le C\sqrt E+\delta(w),
\]
because \(s(w)\) is within distance \(\delta(w)\) of either \(\sqrt w\) or
\(-\sqrt w\). Therefore,
\[
\Im [s(w)]|s(w)|
\le
C\left(\frac{\eta}{\sqrt E}+\delta(w)\right)
\left(\sqrt E+\delta(w)\right)
\le
C\bigl(\eta+\sqrt E\,\delta(w)+\delta(w)^2\bigr).
\]
By \Cref{lem:edge-stability} and \(|w|\asymp E\),
\[
\sqrt E\,\delta(w)\le C|P(w)|,
\qquad
\delta(w)^2\le |P(w)|.
\]
Thus
\begin{equation}\label{e:key}
\Im [s(w)]|s(w)|
\le
C\bigl(\eta+|P(w)|\bigr).
\end{equation}
Raising \eqref{e:key} to the \(q\)-th power gives
\[
\Im [s(w)]^q |s(w)|^q
\le
C^q\bigl(\eta^q+|P(w)|^q\bigr).
\]
Substituting this into \eqref{eq:edge-Xq-second}, we obtain
\begin{align*}
\mathbb E\bigl[|P(w)|^{2q}\bigr]
\le
\left(\frac{Cq}{\eta^2}\right)^q
\left(
\eta^q+\mathbb E[|P(w)|^q]
\right) \le
\left(\frac{Cq}{\eta^2}\right)^q
\left(
\eta^q+\mathbb E[|P(w)|^{2q}]^{1/2}
\right).
\end{align*}
Solving this quadratic inequality yields
\begin{equation}\label{eq:edge-P-moment2}
\mathbb E\bigl[|P(w)|^{2q}\bigr]
\le
\frac{(Cq)^q}{\eta^q}
+
\frac{(Cq)^{2q}}{\eta^{4q}} .
\end{equation}

Finally, using \(|w|\asymp E\), \eqref{eq:edge-im-stability}, and
\eqref{eq:edge-P-moment2}, we get
\begin{align*}
\mathbb E\left[
\left|\Im\left[s(w)-\sqrt w\right]\right|^{2q}
\right]\le
\frac{C}{E^q}\mathbb E\bigl[|P(w)|^{2q}\bigr] \le
\frac{(Cq)^q}{\eta^q E^q}
+
\frac{(Cq)^{2q}}{\eta^{4q}E^q}.
\end{align*}
This proves \eqref{eqn:edge2}.
\end{proof}

\subsection{Proof of \Cref{e:concentration2}}

\begin{lemma}[Edge counting estimate]\label{lem:edge-counting}
Adopt \Cref{a:edge}. There exists an almost surely finite random variable
\(M\ge 2\) such that, almost surely,
\begin{equation}\label{e:x1-bound}
	x_1\leq M,
\end{equation}
and, for every \(a\ge 2\),
\begin{equation}\label{e:edge-counting}
\left|
\#\{x\in P:-a\le x\le M\}
-\frac{2}{3\pi}a^{3/2}
\right|
\le M(\log a)^2.
\end{equation}
\end{lemma}

\begin{proof}[Proof of \eqref{e:x1-bound}]
We first prove that there exists an almost surely finite random integer \(M\)
such that
\begin{equation}\label{e:small-im-positive-axis}
\Im[s(E+\ri)]\le \frac14,
\qquad E\ge M-1,\quad E\in \mathbb Z .
\end{equation}

Fix \(q=2\). By \eqref{eqn:edge2}, with \(\eta=1\), for all sufficiently large
integers \(E\),
\begin{equation}\label{e:edge2-q2}
\mathbb E\left[
\left|\Im\left[s(E+\ri)-\sqrt{E+\ri}\right]\right|^4
\right]
\le
\frac{C}{E^2}.
\end{equation}
Here we used that the assumptions of \eqref{eqn:edge2} are satisfied for
\(\eta=1\), \(q=2\), and all sufficiently large \(E\).

Moreover, for all sufficiently large \(E\),
\[
\Im[ \sqrt{E+\ri}]\le \frac{1}{2\sqrt{E}}\leq \frac18.
\]
Therefore,
\[
\left\{\Im[s(E+\ri)]>\frac14\right\}
\subset
\left\{
\left|\Im\left[s(E+\ri)-\sqrt{E+\ri}\right]\right|>\frac18
\right\}.
\]
By Markov's inequality and \eqref{e:edge2-q2},
\[
\mathbb P\left(\Im[s(E+\ri)]>\frac14\right)
\le
8^4
\mathbb E\left[
\left|\Im\left[s(E+\ri)-\sqrt{E+\ri}\right]\right|^4
\right]
\le
\frac{C}{E^2}.
\]
Consequently,
\[
\sum_{E=1}^{\infty}
\mathbb P\left(\Im[s(E+\ri)]>\frac14\right)
<\infty .
\]
By the Borel--Cantelli lemma, almost surely only finitely many of the events
\[
\left\{\Im[s(E+\ri)]>\frac14\right\},
\qquad E\in\mathbb Z_{\ge1},
\]
occur. Thus there exists an almost surely finite random integer \(M\) such that
\eqref{e:small-im-positive-axis} holds.

We now show that this implies boundedness of the top particle. Suppose, for
contradiction, that there exists a particle \(x\in P\) with \(x\ge M\). Choose
an integer \(E\ge M-1\) such that \(x\in [E,E+1]\). Evaluating at
\(z=E+\ri\), the contribution of this single particle gives
\[
\Im[s(E+\ri)]
=
c+\sum_{y\in P}\frac{1}{(y-E)^2+1}
\ge
\frac{1}{(x-E)^2+1}
\ge
\frac12,
\]
since \(c\ge0\) and \(0\le x-E\le 1\). This contradicts
\eqref{e:small-im-positive-axis}. Therefore no particle can lie in
\([M,\infty)\), and hence
\[
x_1\le M .
\]
After enlarging \(M\), we may also assume \(M\ge 2\). This proves
\eqref{e:x1-bound}.
\end{proof}

\begin{proof}[Proof of \eqref{e:edge-counting}]
The edge concentration estimates \eqref{eqn:edge} and \eqref{eqn:edge2},  together with the net argument, imply that,
almost surely, after possibly enlarging \(M\), for every dyadic scale \(2^n\),
\begin{align}
\begin{split}\label{e:edge-as-bound}
&|s(x+\ri y)-\sqrt{x+\ri y}|
\le \frac{Mn}{y},
\qquad
2^{-n}\le y\le 2^n,\quad -2^n\leq x\leq y,
\\
&|\Im[s(x+\ri y)-\sqrt{x+\ri y}]|
\le \frac{Mn}{y},
\qquad
0< y\le 2^n,\quad |x|\le 2^n .
\end{split}
\end{align}
This is the edge analogue of the almost-sure rectangle bound in the bulk case  \Cref{prop:as-uniform-rectangles}, and can be proven in the same way. So we omit its proof.

Fix \(a\ge 2\), and choose \(n\) such that
\[
2^{n-1}\le a\le 2^n.
\]
Choose a smooth nonnegative function \(f_+\) such that
\[
f_+=1 \quad \text{on }[-a,M],
\qquad
f_+=0 \quad \text{outside }[-a-a^{-1/2},M+1],
\]
with
\[
|f_+'|\lesssim a^{1/2},\qquad |f_+''|\lesssim a
\]
on the interval \([-a-a^{-1/2},-a]\), and
\[
|f_+'|+|f_+''|\lesssim 1
\]
on the interval \([M,M+1]\).

Similarly, choose a smooth nonnegative function \(f_-\) such that
\[
f_-=1 \quad \text{on }[-a+a^{-1/2},M-1],
\qquad
f_-=0 \quad \text{outside }[-a,M],
\]
with
\[
|f_-'|\lesssim a^{1/2},\qquad |f_-''|\lesssim a
\]
on the interval \([-a,-a+a^{-1/2}]\), and
\[
|f_-'|+|f_-''|\lesssim 1
\]
on the interval \([M-1,M]\). In particular,
\[
f_-\le \mathbf 1_{[-a,M]}\le f_+ .
\]

We claim that for either \(f=f_+\) or \(f=f_-\),
\begin{equation}\label{e:edge-HS-sharp-bound}
\left|
\sum_{x\in P}f(x)
-\frac{1}{\pi}\int_{-\infty}^0 f(x)\sqrt{-x}\,\rd x
\right|
\le CMn^2.
\end{equation}
We now prove \eqref{e:edge-HS-sharp-bound}. Let
\(\chi\in C_c^\infty(\mathbb R)\) satisfy
\[
\chi=1 \quad \text{on }[-a,a],
\qquad
\chi=0 \quad \text{outside }[-a-1,a+1],
\qquad
|\chi'|\lesssim 1,
\]
and set
\[
\widetilde f(x+\ri y):=(f(x)+\ri y f'(x))\chi(y).
\]
We apply \Cref{lem:HS} with
\[
s_\varrho(z)=\sqrt z,
\qquad
\varrho(x)=\frac{\sqrt{-x}}{\pi}\,\mathbf 1_{x\le 0}.
\]
For the \(f''\)-term near the fixed endpoint \(M\), we split the \(y\)-integral
at \(\eta=M+1\). For the \(f''\)-term near the moving endpoint \(-a\), we split
the \(y\)-integral at \(\eta=a^{-1/2}\). Using
\eqref{e:edge-as-bound} as input, \Cref{lem:HS} gives
\eqref{e:edge-HS-sharp-bound}.

Now apply \eqref{e:edge-HS-sharp-bound} to \(f_+\). Since
\(f_+\ge \mathbf 1_{[-a,M]}\),
\[
\#\{x\in P:-a\le x\le M\}
\le
\sum_{x\in P}f_+(x).
\]
Moreover,
\[
\frac{1}{\pi}\int_{-\infty}^0 f_+(x)\sqrt{-x}\,\rd x
=
\frac{2}{3\pi}a^{3/2}+\OO(1),
\]
because the smoothing interval near \(-a\) has length \(a^{-1/2}\), and hence
its contribution to the reference integral is \(\OO(1)\). Therefore
\[
\#\{x\in P:-a\le x\le M\}
\le
\frac{2}{3\pi}a^{3/2}+CMn^2+\OO(1).
\]

Similarly, applying \eqref{e:edge-HS-sharp-bound} to \(f_-\), and using
\(f_-\le \mathbf 1_{[-a,M]}\), gives
\[
\#\{x\in P:-a\le x\le M\}
\ge
\frac{2}{3\pi}a^{3/2}-CMn^2-\OO(1).
\]
Since \(n\asymp \log a\), after enlarging \(M\) once more we obtain
\begin{equation}\label{e:numberbb}
\left|
\#\{x\in P:-a\le x\le M\}
-\frac{2}{3\pi}a^{3/2}
\right|
\le M(\log a)^2 .
\end{equation}
This proves \eqref{e:edge-counting}.
\end{proof}

\begin{proof}[Proof of \eqref{e:concentration2}]
Since
\[
\#\{x\in P:-a\le x\le M\}\to\infty
\qquad \text{as } a\to\infty,
\]
the poles may be enumerated in decreasing order as
\[
x_1\ge x_2\ge x_3\ge \cdots,
\qquad x_j\to-\infty.
\]
The estimate \eqref{e:concentration2} then follows by rearranging the counting
estimate \eqref{e:edge-counting}.
\end{proof}

\subsection{Proof of \Cref{e:normalized}}

\begin{proof}[Proof of \eqref{e:normalized}]
We recall from \eqref{e:airy-zero-location0},
\begin{equation}\label{e:airy-zero-location}
\left|
\fa_i+\left(\frac{3\pi i}{2}\right)^{2/3}
\right|
\lesssim i^{-1/3}.
\end{equation}
Combining \eqref{e:concentration2} with \eqref{e:airy-zero-location}, we obtain
\begin{equation}\label{e:xi-ai-close}
|x_i-\fa_i|
\le
M\frac{(\log(2+i))^2}{i^{1/3}},
\qquad i\ge 1.
\end{equation}
Moreover,
\begin{equation}\label{e:xaasymp}
|x_i|\asymp i^{2/3},
\qquad
|\fa_i|\asymp i^{2/3},
\qquad i\ge 1.
\end{equation}

We first prove that the right-hand side of \eqref{e:normalized} is
well-defined. Let \(w\) lie in a compact subset of \(\bC\setminus\bR\). Then,
for \(i\) large enough,
$
|x_i-w|\asymp i^{2/3},
|\fa_i|\asymp i^{2/3}.
$
Hence
\[
\left|
\frac{1}{x_i-w}-\frac{1}{\fa_i}
\right|
=
\left|
\frac{\fa_i-x_i+w}{(x_i-w)\fa_i}
\right|
\le
C_w\left(
\frac{1}{i^{4/3}}
+
\frac{(\log(2+i))^2}{i^{5/3}}
\right).
\]
The right-hand side is summable in \(i\). Thus the series in
\eqref{e:normalized} converges absolutely and locally uniformly on
\(\bC\setminus\bR\). Define
\[
\widetilde s(w)
:=
\sum_{i=1}^{\infty}
\left(
\frac{1}{x_i-w}-\frac{1}{\fa_i}
\right)
-\frac{\Ai'(0)}{\Ai(0)}.
\]

We compare \(\widetilde s\) with the Airy Nevanlinna function. Recall the
standard partial fraction expansion \eqref{eq:weairy}:
\begin{equation}\label{e:airy-partial-fraction}
-\frac{\Ai'(w)}{\Ai(w)}
=
\sum_{i=1}^{\infty}
\left(
\frac{1}{\fa_i-w}-\frac{1}{\fa_i}
\right)
-\frac{\Ai'(0)}{\Ai(0)}.
\end{equation}
Therefore
\begin{equation}\label{e:tilde-minus-airy}
\widetilde s(w)+\frac{\Ai'(w)}{\Ai(w)}
=
\sum_{i=1}^{\infty}
\left(
\frac{1}{x_i-w}-\frac{1}{\fa_i-w}
\right).
\end{equation}

We claim that the right-hand side tends to \(0\) along the imaginary axis.
Indeed, let \(w=\ri y\), with \(y\ge 2\), and split the sum into
\(i\le y^{3/2}\) and \(i>y^{3/2}\). For \(i\le y^{3/2}\), we have
$
|x_i-\ri y|\,|\fa_i-\ri y|\ge cy^2$,
and therefore, by \eqref{e:xi-ai-close},
\[
\sum_{i\le y^{3/2}}
\left|
\frac{1}{x_i-\ri y}-\frac{1}{\fa_i-\ri y}
\right|
\le
\frac{C}{y^2}
\sum_{i\le y^{3/2}}
\frac{(\log(2+i))^2}{i^{1/3}}
\le
\frac{C(\log y)^2}{y}.
\]
For \(i>y^{3/2}\), using \eqref{e:xaasymp}, we get
\[
\sum_{i>y^{3/2}}
\left|
\frac{1}{x_i-\ri y}-\frac{1}{\fa_i-\ri y}
\right|
\le
C\sum_{i>y^{3/2}}
\frac{(\log(2+i))^2}{i^{5/3}}
\le
\frac{C(\log y)^2}{y}.
\]
Thus
\begin{equation}\label{e:tilde-airy-axis}
\widetilde s(\ri y)+\frac{\Ai'(\ri y)}{\Ai(\ri y)}
\longrightarrow 0,
\qquad y\to\infty.
\end{equation}

By the Airy asymptotic \eqref{e:aiasymp},
\[
-\frac{\Ai'(w)}{\Ai(w)}=\sqrt w+\OO(|w|^{-1}),
\]
uniformly in fixed sectors away from the negative real axis. Together with
\eqref{e:tilde-airy-axis}, this gives
\begin{equation}\label{e:tilde-s-axis}
\widetilde s(\ri y)-\sqrt{\ri y}\longrightarrow 0,
\qquad y\to\infty.
\end{equation}

On the other hand, \eqref{e:edge-as-bound} implies
\begin{equation}\label{e:slimit}
|s(\ri y)-\sqrt{\ri y}|
\le
\frac{M\log(2+y)}{y}
\longrightarrow 0,
\qquad y\to+\infty.
\end{equation}
It follows from \eqref{e:tilde-s-axis} and \eqref{e:slimit} that
\[
s(\ri y)-\widetilde s(\ri y)\longrightarrow 0,
\qquad y\to\infty.
\]

It remains to identify \(s\) and \(\widetilde s\). Both are Nevanlinna
functions with the same representing measure \(\sum_i\delta_{x_i}\). Hence
their difference is an affine function:
\[
s(w)-\widetilde s(w)=b+cw,
\qquad b\in\mathbb R,\quad c\ge0.
\]
Since \(s(\ri y)-\widetilde s(\ri y)\to0\) as \(y\to\infty\), we must have
\(c=0\) and \(b=0\). Therefore \(s=\widetilde s\), that is,
\[
s(w)
=
\sum_{i\geq 1}
\left(
\frac{1}{x_i-w}-\frac{1}{\fa_i}
\right)
-\frac{\Ai'(0)}{\Ai(0)}.
\]
This proves \eqref{e:normalized}.
\end{proof}

\section{Explicit solutions for $\beta=2$ bulk system}
\label{s:b2solution} 
We prove \Cref{thm:beta2-explicit} in this section. The key idea is to remove
the mixed interaction terms by factoring out the mixed Vandermonde product
\[
P(\bmt,\bms)=\prod_{i=1}^n\prod_{a=1}^m(t_i-s_a).
\]
After writing \(F=P\,G\), the conjugated equations for \(G\) become the
ordinary scalar \(A_{n+m-1}\)-type system in the combined variables
\[
x_1=t_1,\dots,x_n=t_n,
\qquad
x_{n+a}=s_a.
\]
Thus, when \(\beta=2\), the deformed system is gauge-equivalent, through
multiplication by \(P(\bmt,\bms)\), to the standard scalar \(A\)-type CMS
system \eqref{eq:beta2-A-system}--\eqref{eq:beta2-A-operator}; see
\cite{olshanetsky1983quantum}.
\subsection{Reduction to the \(A\)-type CMS system}
\label{s:beta2-reduction}

Let $(\bmt,\bms)=(t_1,\cdots, t_n, s_1,\cdots, s_m)$, and
\[
P(\bmt,\bms):=\prod_{i=1}^n\prod_{a=1}^m (t_i-s_a),
\qquad
F(\bmt,\bms)=P(\bmt,\bms)\,G(\bmt,\bms).
\]
Introduce
\[
A_i:=\partial_{t_i}\log P
=
\sum_{a=1}^m \frac{1}{t_i-s_a},
\qquad
B_a:=\partial_{s_a}\log P
=
\sum_{i=1}^n \frac{1}{s_a-t_i}.
\]
Then
\begin{align}\begin{split}\label{e:trelation}
\partial_{t_i}F
=
P(\partial_{t_i}+A_i)G,
\quad
\partial_{t_i}^2F
=
P\Bigl(
\partial_{t_i}^2
+2A_i\partial_{t_i}
+\partial_{t_i}A_i+A_i^2
\Bigr)G,
\end{split}\end{align}
and similarly
\begin{align}\label{e:srelation}
\partial_{s_a}F
=
P(\partial_{s_a}+B_a)G,
\quad
\partial_{s_a}^2F
=
P\Bigl(
\partial_{s_a}^2
+2B_a\partial_{s_a}
+\partial_{s_a}B_a+B_a^2
\Bigr)G.
\end{align}

\begin{theorem}
\label{thm:beta2-reduction}
Assume \(\beta=2\). Then \(F=P\,G\) satisfies the original \(\beta=2\)
system \eqref{e:bulkbeta=2} if and only if \(G\) satisfies
\begin{align}
\label{eq:beta2-G-ti}
\left(
\partial_{t_i}^2+1
+\sum_{j\neq i}\frac{\partial_{t_i}-\partial_{t_j}}{t_i-t_j}
+\sum_{a=1}^m\frac{\partial_{t_i}-\partial_{s_a}}{t_i-s_a}
\right)G=0,
\qquad i=1,\dots,n,
\end{align}
and
\begin{align}
\label{eq:beta2-G-sa}
\left(
\partial_{s_a}^2+1
+\sum_{b\neq a}\frac{\partial_{s_a}-\partial_{s_b}}{s_a-s_b}
+\sum_{i=1}^n\frac{\partial_{s_a}-\partial_{t_i}}{s_a-t_i}
\right)G=0,
\qquad a=1,\dots,m.
\end{align}
Equivalently, after setting
\[
x_1=t_1,\dots,x_n=t_n,
\qquad
x_{n+1}=s_1,\dots, x_{n+m}=s_m,
\]
the function \(G\) satisfies
\begin{align}
\label{eq:beta2-A-system}
\mathcal L_r G=0,
\qquad r=1,\dots,n+m,
\end{align}
where
\begin{align}
\label{eq:beta2-A-operator}
\mathcal L_r
:=
\partial_{x_r}^2+1
+\sum_{u\neq r}
\frac{\partial_{x_r}-\partial_{x_u}}{x_r-x_u}.
\end{align}
\end{theorem}

\begin{proof}
We work on the complement of the collision hyperplanes, i.e. \eqref{e:defX}. We first treat the
\(t_i\)-equation; the \(s_a\)-equation is analogous.

Substituting \(F=PG\) into the \(t_i\)-equation of the original \(\beta=2\)
system \eqref{e:bulkbeta=2} and dividing by \(P\), using \eqref{e:trelation} and
\eqref{e:srelation}, gives
\begin{align*}
0
=
\Bigl(
\partial_{t_i}^2
+2A_i\partial_{t_i}
+\partial_{t_i}A_i+A_i^2
+1
\Bigr)G
+
\sum_{j\neq i}
\frac{
(\partial_{t_i}+A_i)-(\partial_{t_j}+A_j)
}{t_i-t_j}G
-
\sum_{a=1}^m
\frac{
(\partial_{t_i}+A_i)+(\partial_{s_a}+B_a)
}{t_i-s_a}G .
\end{align*}
Collecting the first-order terms involving the mixed variables, we find
\[
2A_i\partial_{t_i}
-
\sum_{a=1}^m
\frac{\partial_{t_i}+\partial_{s_a}}{t_i-s_a}
=
\sum_{a=1}^m
\frac{\partial_{t_i}-\partial_{s_a}}{t_i-s_a},
\]
since \(A_i=\sum_a(t_i-s_a)^{-1}\). Therefore
\[
0=
\left(
\partial_{t_i}^2+1
+\sum_{j\neq i}\frac{\partial_{t_i}-\partial_{t_j}}{t_i-t_j}
+\sum_{a=1}^m\frac{\partial_{t_i}-\partial_{s_a}}{t_i-s_a}
\right)G
+Z_iG,
\]
where
\[
Z_i
=
\partial_{t_i}A_i+A_i^2
+\sum_{j\neq i}\frac{A_i-A_j}{t_i-t_j}
-\sum_{a=1}^m\frac{A_i+B_a}{t_i-s_a}.
\]

It remains to show that \(Z_i=0\). Set
\[
c_{ia}:=\frac{1}{t_i-s_a}.
\]
Then
\[
A_i=\sum_{a=1}^m c_{ia},
\qquad
B_a=-\sum_{i=1}^n c_{ia}.
\]
We compute the three contributions to \(Z_i\). First,
\[
\partial_{t_i}A_i+A_i^2
=
-\sum_a c_{ia}^2
+
\left(\sum_a c_{ia}\right)^2
=
2\sum_{a<b}c_{ia}c_{ib}.
\]
Second,
\[
\frac{A_i-A_j}{t_i-t_j}
=
\sum_{a=1}^m
\frac{c_{ia}-c_{ja}}{t_i-t_j}
=
-\sum_{a=1}^m c_{ia}c_{ja},
\]
and hence
\[
\sum_{j\neq i}\frac{A_i-A_j}{t_i-t_j}
=
-\sum_{j\neq i}\sum_{a=1}^m c_{ia}c_{ja}.
\]
Third, since
\[
A_i+B_a
=
\sum_{b\neq a}c_{ib}
-
\sum_{j\neq i}c_{ja},
\]
we have
\[
\sum_{a=1}^m
\frac{A_i+B_a}{t_i-s_a}
=
\sum_{a=1}^m\sum_{b\neq a}c_{ia}c_{ib}
-
\sum_{a=1}^m\sum_{j\neq i}c_{ia}c_{ja}
=
2\sum_{a<b}c_{ia}c_{ib}
-
\sum_{j\neq i}\sum_{a=1}^m c_{ia}c_{ja}.
\]
Combining these identities gives \(Z_i=0\), and therefore
\eqref{eq:beta2-G-ti} follows.

The \(s_a\)-equation is handled in the same way. Substituting \(F=PG\) into the
\(s_a\)-equation and collecting first-order terms gives
\[
0=
\left(
\partial_{s_a}^2+1
+\sum_{b\neq a}\frac{\partial_{s_a}-\partial_{s_b}}{s_a-s_b}
+\sum_{i=1}^n\frac{\partial_{s_a}-\partial_{t_i}}{s_a-t_i}
\right)G
+W_aG,
\]
where
\[
W_a
=
\partial_{s_a}B_a+B_a^2
+\sum_{b\neq a}\frac{B_a-B_b}{s_a-s_b}
-\sum_{i=1}^n\frac{B_a+A_i}{s_a-t_i}.
\]
The same algebra as above gives \(W_a=0\). Hence \eqref{eq:beta2-G-sa}
follows.

Finally, the unified form \eqref{eq:beta2-A-system} is immediate after the
identification
\[
x_1=t_1,\dots,x_n=t_n,
\qquad
x_{n+1}=s_1,\dots, x_{n+m}=s_m.
\]
This proves the theorem.
\end{proof}

\subsection{Proof of \Cref{thm:beta2-explicit}}

\begin{proof}[Proof of \Cref{thm:beta2-explicit}]
Set
\[
x_1=t_1,\dots,x_n=t_n,
\qquad
x_{n+1}=s_1,\dots, x_{n+m}=s_m,
\]
and write
\[
\bm\sigma=(\sigma_1,\dots,\sigma_{n+m})
:=
(\epsilon_1,\dots,\epsilon_n,\eta_1,\dots,\eta_m)
\in\{\pm1\}^{n+m}.
\]
Let
\[
P(\bmt,\bms):=\prod_{i=1}^n\prod_{a=1}^m(t_i-s_a).
\]
By \Cref{thm:beta2-reduction}, a function of the form
\[
F(\bmt,\bms)=P(\bmt,\bms)\,G(\bmx)
\]
solves the original \(\beta=2\) system if and only if \(G\) solves the reduced
system
\begin{equation}\label{eq:proof-beta2-Lr}
\mathcal L_rG=0,
\qquad r=1,\dots,n+m,
\end{equation}
where
\[
\mathcal L_r
:=
\partial_{x_r}^2+1
+
\sum_{u\neq r}
\frac{\partial_{x_r}-\partial_{x_u}}{x_r-x_u}.
\]

\medskip
\noindent
\emph{Step 1. Explicit solutions of the reduced system.}
For \(\bm\sigma\in\{\pm1\}^{n+m}\), define
\begin{equation}\label{eq:proof-Delta-sigma}
\Delta_{\bm\sigma}(\bmx)
:=
\prod_{1\le r<u\le n+m}
(x_r-x_u)^{(\sigma_r\sigma_u-1)/2}.
\end{equation}
Since \((\sigma_r\sigma_u-1)/2\in\{0,-1\}\), the function
\(\Delta_{\bm\sigma}\) is rational. Set
\begin{equation}\label{eq:proof-Gsigma}
G_{\bm\sigma}(\bmx)
:=
\exp\!\left(\ri\sum_{r=1}^{n+m}\sigma_rx_r\right)
\Delta_{\bm\sigma}(\bmx).
\end{equation}
We claim that
\[
\mathcal L_rG_{\bm\sigma}=0,
\qquad r=1,\dots,n+m.
\]

Fix \(r\), and introduce
\[
J_r:=\{u\neq r:\sigma_u\neq\sigma_r\},
\qquad
K_r:=\{u\neq r:\sigma_u=\sigma_r\}.
\]
Then
\[
\psi_r
:=
\partial_{x_r}\log\Delta_{\bm\sigma}
=
-\sum_{u\in J_r}\frac{1}{x_r-x_u}.
\]
Hence
\[
\frac{\partial_{x_r}G_{\bm\sigma}}{G_{\bm\sigma}}
=
\ri\sigma_r+\psi_r,
\qquad
\frac{\partial_{x_r}^2G_{\bm\sigma}}{G_{\bm\sigma}}
=
(\ri\sigma_r+\psi_r)^2+\partial_{x_r}\psi_r.
\]
Therefore
\begin{align}\label{eq:proof-LrGsigma}
\frac{\mathcal L_rG_{\bm\sigma}}{G_{\bm\sigma}}
&=
(\ri\sigma_r)^2+1
+\partial_{x_r}\psi_r+\psi_r^2
+2\ri\sigma_r\psi_r
+
\sum_{u\neq r}
\frac{\ri(\sigma_r-\sigma_u)+\psi_r-\psi_u}{x_r-x_u}.
\end{align}
Since \((\ri\sigma_r)^2+1=0\), it remains to check that the rational terms
cancel.

First, the exponential contributions cancel:
\[
2\ri\sigma_r\psi_r
+
\sum_{u\in J_r}
\frac{\ri(\sigma_r-\sigma_u)}{x_r-x_u}
=
-2\ri\sigma_r\sum_{u\in J_r}\frac{1}{x_r-x_u}
+
2\ri\sigma_r\sum_{u\in J_r}\frac{1}{x_r-x_u}
=0.
\]
For the remaining terms, we compute
\[
\partial_{x_r}\psi_r+\psi_r^2
=
2\sum_{a\in J_r}\frac{1}{(x_r-x_a)^2}
+
\sum_{\substack{a,b\in J_r\\ a\neq b}}
\frac{1}{(x_r-x_a)(x_r-x_b)}.
\]
Moreover,
\[
\sum_{u\in K_r}
\frac{\psi_r-\psi_u}{x_r-x_u}
=
\sum_{u\in K_r}
\sum_{a\in J_r}
\frac{1}{(x_r-x_a)(x_u-x_a)},
\]
while
\[
\sum_{a\in J_r}
\frac{\psi_r-\psi_a}{x_r-x_a}
=
-2\sum_{a\in J_r}\frac{1}{(x_r-x_a)^2}
-
\sum_{\substack{a,b\in J_r\\ a\neq b}}
\frac{1}{(x_r-x_a)(x_r-x_b)}
-
\sum_{a\in J_r}
\sum_{u\in K_r}
\frac{1}{(x_r-x_a)(x_u-x_a)}.
\]
These three expressions cancel term by term. Thus
\[
\mathcal L_rG_{\bm\sigma}=0,
\qquad r=1,\dots,n+m.
\]

\medskip
\noindent
\emph{Step 2. Lifting back to the original system.}
Define
\[
F_{\bm\sigma}(\bmt,\bms):=P(\bmt,\bms)\,G_{\bm\sigma}(\bmt,\bms).
\]
Since \(G_{\bm\sigma}\) solves the reduced system \eqref{eq:proof-beta2-Lr},
\Cref{thm:beta2-reduction} implies that \(F_{\bm\sigma}\) solves the original
\(\beta=2\) system.

We now identify \(F_{\bm\sigma}\). By definition,
\[
\Delta_{\bm\sigma}(x)
=
\prod_{1\le i<j\le n}
(t_i-t_j)^{(\epsilon_i\epsilon_j-1)/2}
\prod_{1\le a<b\le m}
(s_a-s_b)^{(\eta_a\eta_b-1)/2}
\prod_{i=1}^n
\prod_{a=1}^m
(t_i-s_a)^{(\epsilon_i\eta_a-1)/2}.
\]
Multiplication by
\[
P(t,s)=\prod_{i=1}^n\prod_{a=1}^m(t_i-s_a)
\]
changes the mixed exponent from
$
(\epsilon_i\eta_a-1)/{2}
$
to
$
(1+\epsilon_i\eta_a)/{2}.
$
Therefore
\begin{align*}
F_{\bm\sigma}(\bmt,\bms)
&=
\exp\!\left(
\ri\sum_{i=1}^n\epsilon_i t_i
+
\ri\sum_{a=1}^m\eta_a s_a
\right)
\prod_{1\le i<j\le n}
(t_i-t_j)^{(\epsilon_i\epsilon_j-1)/2}\\
&\qquad \times\prod_{1\le a<b\le m}
(s_a-s_b)^{(\eta_a\eta_b-1)/2}
\prod_{i=1}^n
\prod_{a=1}^m
(t_i-s_a)^{(1+\epsilon_i\eta_a)/2}.
\end{align*}
This is exactly \eqref{eq:beta2-explicit-F}. Since all exponents are integers,
these solutions are single-valued.

\medskip
\noindent
\emph{Step 3. Linear independence.}
Choose \(\bmv=(v_1,\cdots,v_{n+m})\in\bC^{n+m}\) such that
\[
v_r\neq v_u\quad(r\neq u),
\]
and
\[
\Re[\ri\,\bm\sigma\cdot \bmv]
\neq
\Re[\ri\,\bm\sigma'\cdot \bmv]
\qquad
\text{for } \bm\sigma\neq\bm\sigma'.
\]
Such a choice is possible after excluding finitely many proper real hyperplanes.

Along the ray \(\bmx=R\bmv\), as \(R\to+\infty\), we have
\[
G_{\bm\sigma}(R\bmv)
=
c_{\bm\sigma}
R^{-N_+(\bm\sigma)N_-(\bm\sigma)}
\exp\!\left(\ri R\,\bm\sigma\cdot \bmv\right)
\left(1+\OO(R^{-1})\right),
\]
where
\[
c_{\bm\sigma}
=
\prod_{1\le r<u\le n+m}
(v_r-v_u)^{(\sigma_r\sigma_u-1)/2}
\neq0,
\quad
N_\pm(\bm\sigma):=\#\{r:\sigma_r=\pm1\}.
\]
The numbers \(\Re[\ri\,\bm\sigma\cdot \bmv]\) are pairwise distinct, so the
exponentials
$
\exp\!\left(\ri R\,\bm\sigma\cdot\bmv\right)
$
have distinct growth rates along this ray. Hence no nontrivial linear
combination of the \(G_{\bm\sigma}\) can vanish identically.

Therefore the functions \(G_{\bm\sigma}\) are linearly independent. Since
\(F_{\bm\sigma}=P\,G_{\bm\sigma}\) and \(P\) is not identically zero, the functions
\(F_{\bm\sigma}\) are also linearly independent. This proves the theorem.
\end{proof}

\section{Explicit solutions for \(\beta=2\) edge system}
\label{s:edge-b2solution}

\subsection{Reduction to the edge \(A\)-type CMS system}
\label{s:edge-beta2-reduction}

Let
\[
(\bmt,\bms)=(t_1,\dots,t_n,s_1,\dots,s_m),
\]
and define
\[
P(\bmt,\bms):=\prod_{i=1}^n\prod_{a=1}^m(t_i-s_a),
\qquad
F(\bmt,\bms)=P(\bmt,\bms)\,G(\bmt,\bms).
\]

\begin{theorem}[Reduction to the scalar edge \(A_{n+m-1}\)-system]
\label{thm:edge-beta2-reduction}
Assume \(\beta=2\). Then \(F=P\,G\) satisfies the original edge system
\eqref{e:edgebeta=2} if and only if \(G\) satisfies
\begin{align}
\label{eq:edge-reduced-ti}
\left(
\partial_{t_i}^2-t_i
+\sum_{j\neq i}\frac{\partial_{t_i}-\partial_{t_j}}{t_i-t_j}
+\sum_{a=1}^m\frac{\partial_{t_i}-\partial_{s_a}}{t_i-s_a}
\right)G=0,
\qquad i=1,\dots,n,
\end{align}
and
\begin{align}
\label{eq:edge-reduced-sa}
\left(
\partial_{s_a}^2-s_a
+\sum_{b\neq a}\frac{\partial_{s_a}-\partial_{s_b}}{s_a-s_b}
+\sum_{i=1}^n\frac{\partial_{s_a}-\partial_{t_i}}{s_a-t_i}
\right)G=0,
\qquad a=1,\dots,m.
\end{align}
Equivalently, after setting
\[
x_1=t_1,\dots,x_n=t_n,
\qquad
x_{n+1}=s_1,\dots,x_{n+m}=s_m,
\]
the function \(G\) satisfies
\begin{align}
\label{eq:edge-reduced}
 \sfL_rG=0,
\qquad r=1,\dots,n+m,
\end{align}
where
\begin{align}
\label{eq:edge-A-operator}
\sfL_r 
:=
\partial_{x_r}^2-x_r
+\sum_{u\neq r}
\frac{\partial_{x_r}-\partial_{x_u}}{x_r-x_u}.
\end{align}
\end{theorem}

\begin{proof}
This is the same conjugation as in the bulk \(\beta=2\) case. Indeed, after
writing \(F=P\,G\), the logarithmic derivatives
\[
A_i:=\partial_{t_i}\log P
=
\sum_{a=1}^m\frac{1}{t_i-s_a},
\qquad
B_a:=\partial_{s_a}\log P
=
\sum_{i=1}^n\frac{1}{s_a-t_i}
\]
are exactly the same as in the proof of \Cref{thm:beta2-reduction}. Hence the
first-order terms conjugate in the same way, and the same \(c_{ia}\)-algebra
shows that all zero-order terms coming from \(P\) cancel.

The only difference from the bulk computation is that the constant term \(+1\)
is replaced by the linear potential \(-t_i\) or \(-s_a\). These terms are
unaffected by conjugation with \(P\). Therefore the bulk proof carries over
verbatim and yields \eqref{eq:edge-reduced-ti} and
\eqref{eq:edge-reduced-sa}. The unified form \eqref{eq:edge-reduced} is then
immediate from the identification of the \(x\)-variables.
\end{proof}

\subsection{Proof of \Cref{thm:edge-beta2-explicit}}
The proof is parallel in spirit to the bulk case. The reduced scalar equation is
now the edge \(A\)-type Calogero system \eqref{eq:edge-reduced}, so the role of
the plane waves is played by Airy solutions and a Slater determinant.
\begin{proof}[Proof of \Cref{thm:edge-beta2-explicit}]
We write
\[
\bmx=(x_1,\dots,x_{n+m}),
\qquad
x_1=t_1,\dots,x_n=t_n,
\qquad
x_{n+1}=s_1,\dots,x_{n+m}=s_m.
\]
Let
\[
\Delta(\bmx):=\prod_{1\le r<u\le {n+m}}(x_r-x_u).
\]

For each sign pattern
\[
\bm\sigma=(\sigma_1,\dots,\sigma_{n+m})\in\{\pm1\}^{n+m},
\]
choose Airy solutions \(\phi_{\sigma_r}\) (recall from \eqref{e:defphi}) of
\[
\phi_{\sigma_r}''(x)-x\,\phi_{\sigma_r}(x)=0.
\]
Define
\[
D_{\bm\sigma}(\bmx)
:=
\det\!\left[
\partial_x^{\,j-1}\phi_{\sigma_r}(x_r)
\right]_{1\le r,j\le {n+m}},
\]
and
\begin{equation}\label{eq:G-sigma-edge}
G_{\bm\sigma}(\bmx)
:=
\frac{D_{\bm\sigma}(\bmx)}{\Delta(\bmx)}.
\end{equation}

The above determinant quotient, with an arbitrary row-by-row choice of
one-variable Airy solutions, gives local solutions of
\eqref{eq:edge-reduced} on the configuration space \(x_r\neq x_u\).
The special symmetric solution obtained by taking
\(\phi_1=\cdots=\phi_{n+m}=\Ai\) is regular across the diagonals and agrees, up to normalization, with the
standard \(\beta=2\) multivariate Airy function, equivalently the
Kontsevich matrix Airy function; see
\cite{kontsevich1992intersection, desrosiers2014asymptotics}.

\medskip
\noindent
\emph{Step 1. Solving the reduced system.}
We first prove that
\[
\sfL_r G_{\bm\sigma}=0,
\qquad r=1,\dots,{n+m}.
\]
Equivalently,
\begin{equation}\label{eq:direct-edge-reduced}
\left(
\partial_{x_r}^2-x_r
+
\sum_{u\neq r}\frac{\partial_{x_r}-\partial_{x_u}}{x_r-x_u}
\right)G_{\bm\sigma}=0.
\end{equation}

Let
\[
\mathfrak X:=\{\bmx\in\bC^{n+m}:\ x_r\neq x_u \text{ for } r\neq u\}.
\]
We prove the identity on \(\mathfrak X\), where the operator is well defined.

Since the Airy equation has a two-dimensional space of solutions spanned by
Airy contour branches, and since \(D_{\bm\sigma}\) is multilinear in its rows,
it suffices to prove the claim when each \(\phi_{\sigma_a}\) is represented by
a single admissible Airy contour integral. Thus, for each \(a=1,\dots,{n+m}\), choose
an admissible Airy contour \(\Gamma_{\sigma_a}\) such that
\[
\phi_{\sigma_a}(x)
=
\int_{\Gamma_{\sigma_a}} e^{z^3/3-xz}\,\rd z .
\]
Then
\[
\partial_x^{\,j-1}\phi_{\sigma_a}(x_a)
=
(-1)^{j-1}
\int_{\Gamma_{\sigma_a}} z^{j-1}e^{z^3/3-x_a z}\,\rd z .
\]
By multilinearity of the determinant in the rows,
\begin{align}
D_{\bm\sigma}(\bmx)
&=
(-1)^{\frac{(n+m)(n+m-1)}2}
\int_{\Gamma_{\sigma_1}\times\cdots\times\Gamma_{\sigma_{n+m}}}
\det[z_a^{\,j-1}]_{a,j=1}^{n+m}
\exp\!\left(\sum_{a=1}^{n+m}\left(\frac{z_a^3}{3}-x_a z_a\right)\right)
\,\rd\bmz
\nonumber\\
&=
c_{n+m}
\int_{\Gamma_{\bm\sigma}}
\Delta(\bmz)e^{S(\bmz;\bmx)}\,\rd\bmz,
\label{eq:D-contour}
\end{align}
where
\[
c_{n+m}:=(-1)^{\frac{(n+m)({n+m}-1)}2},
\qquad
\Gamma_{\bm\sigma}:=\Gamma_{\sigma_1}\times\cdots\times\Gamma_{\sigma_{n+m}},
\]
\[
S(\bmz;\bmx):=
\sum_{a=1}^{n+m}\left(\frac{z_a^3}{3}-x_a z_a\right),
\qquad
\rd\bmz:=\rd z_1\cdots \rd z_{n+m} .
\]
Hence, for \(\bmx\in\mathfrak X\),
\begin{equation}
G_{\bm\sigma}(\bmx)
=
c_{n+m}
\int_{\Gamma_{\bm\sigma}}
K(\bmx,\bmz)\,\rd\bmz,
\qquad
K(\bmx,\bmz)
:=
\frac{\Delta(\bmz)}{\Delta(\bmx)}e^{S(\bmz;\bmx)}.
\label{eq:G-contour}
\end{equation}

Fix \(r\in\{1,\dots,{n+m}\}\), and set
\[
\sfL_r
:=
\partial_{x_r}^2-x_r
+
\sum_{u\neq r}
\frac{\partial_{x_r}-\partial_{x_u}}{x_r-x_u}.
\]
For \(r=1,\dots,{n+m}\), write
\[
A_r(\bmx):=
\partial_{x_r}\log\Delta(\bmx)
=
\sum_{u\neq r}\frac{1}{x_r-x_u}.
\]
Then
\[
\partial_{x_r}K(\bmx,\bmz)
=
-(z_r+A_r(\bmx))K(\bmx,\bmz).
\]
Therefore,
\[
\partial_{x_r}^2K
=
\left((z_r+A_r)^2-\partial_{x_r}A_r\right)K,
\]
and hence
\begin{align}
\sfL_rK
&=
\Biggl(
z_r^2-x_r
+2z_rA_r
+A_r^2-\partial_{x_r}A_r
+
\sum_{u\neq r}\frac{z_u-z_r}{x_r-x_u}
+
\sum_{u\neq r}\frac{A_u-A_r}{x_r-x_u}
\Biggr)K
\nonumber\\
&=
\Biggl(
z_r^2-x_r
+
\sum_{u\neq r}\frac{z_r+z_u}{x_r-x_u}
+
C_r(\bmx)
\Biggr)K,
\label{eq:LrK-pre}
\end{align}
where
\[
C_r(\bmx)
:=
A_r^2-\partial_{x_r}A_r
+
\sum_{u\neq r}\frac{A_u-A_r}{x_r-x_u}.
\]

We now simplify \(C_r(\bmx)\). Write
\[
d_u:=x_r-x_u,
\qquad u\neq r.
\]
Then
\[
A_r=\sum_{u\neq r}\frac1{d_u},
\qquad
-\partial_{x_r}A_r=\sum_{u\neq r}\frac1{d_u^2},
\]
and therefore
\begin{equation}
A_r^2-\partial_{x_r}A_r
=
2\sum_{u\neq r}\frac1{d_u^2}
+
2\sum_{\substack{u<v\\u,v\neq r}}\frac1{d_ud_v}.
\label{eq:Ar-part}
\end{equation}
On the other hand, for \(u\neq r\),
\[
A_u
=
-\frac1{d_u}
+
\sum_{\substack{v\neq r\\v\neq u}}\frac1{x_u-x_v}.
\]
Thus
\[
\frac{A_u-A_r}{d_u}
=
-\frac{2}{d_u^2}
+
\sum_{\substack{v\neq r\\v\neq u}}
\frac{1}{d_u(x_u-x_v)}
-
\sum_{\substack{v\neq r\\v\neq u}}
\frac{1}{d_ud_v}.
\]
Summing over \(u\neq r\) and pairing the ordered pairs \((u,v)\) and \((v,u)\),
we get
\begin{align}
\sum_{u\neq r}\frac{A_u-A_r}{d_u}
&=
-2\sum_{u\neq r}\frac1{d_u^2}
+
\sum_{\substack{u<v\\u,v\neq r}}
\left(
\frac{1}{d_u(x_u-x_v)}
+
\frac{1}{d_v(x_v-x_u)}
\right)
-
2\sum_{\substack{u<v\\u,v\neq r}}\frac1{d_ud_v}
\nonumber\\
&=
-2\sum_{u\neq r}\frac1{d_u^2}
+
\sum_{\substack{u<v\\u,v\neq r}}
\frac{1}{x_u-x_v}
\left(\frac1{d_u}-\frac1{d_v}\right)
-
2\sum_{\substack{u<v\\u,v\neq r}}\frac1{d_ud_v}
\nonumber\\
&=
-2\sum_{u\neq r}\frac1{d_u^2}
-
\sum_{\substack{u<v\\u,v\neq r}}\frac1{d_ud_v}.
\label{eq:Au-part}
\end{align}
In the last step we used
\[
\frac1{x_u-x_v}\left(\frac1{d_u}-\frac1{d_v}\right)
=
\frac1{d_ud_v}.
\]
Combining \eqref{eq:Ar-part} and \eqref{eq:Au-part}, we find
\[
C_r(\bmx)
=
\sum_{\substack{u<v\\u,v\neq r}}
\frac1{(x_r-x_u)(x_r-x_v)}.
\]
Thus \eqref{eq:LrK-pre} becomes
\begin{equation}
\sfL_rK(\bmx,\bmz)
=
B_r(\bmx,\bmz)K(\bmx,\bmz),
\label{eq:LrK}
\end{equation}
where
\begin{equation}
B_r(\bmx,\bmz)
:=
z_r^2-x_r
+
\sum_{u\neq r}\frac{z_r+z_u}{x_r-x_u}
+
\sum_{\substack{u<v\\u,v\neq r}}
\frac1{(x_r-x_u)(x_r-x_v)}.
\label{eq:Br}
\end{equation}

Therefore, by \eqref{eq:G-contour},
\[
\sfL_rG_{\bm\sigma}(\bmx)
=
\frac{c_{n+m}}{\Delta(\bmx)}
\int_{\Gamma_{\bm\sigma}}
B_r(\bmx,\bmz)\Delta(\bmz)e^{S(\bmz;\bmx)}\,\rd\bmz.
\]
It remains to prove that
\begin{equation}
\int_{\Gamma_{\bm\sigma}}
B_r(\bmx,\bmz)\Delta(\bmz)e^{S(\bmz;\bmx)}\,\rd\bmz
=0.
\label{eq:target-vanishing}
\end{equation}

\medskip
\noindent
\emph{Step 1a. A family of integration-by-parts identities.}
Set
\[
\Phi(\bmz;\bmx):=\Delta(\bmz)e^{S(\bmz;\bmx)},
\qquad
U_r:=\{1,\dots,{n+m}\}\setminus\{r\}.
\]
For each subset \(I\subseteq U_r\), define
\[
J_I:=I\cup\{r\},
\qquad
\Delta_{J_I}(\bmz)
:=
\prod_{\substack{a<b\\a,b\in J_I}}(z_a-z_b),
\qquad
\Omega_I(\bmz;\bmx)
:=
\frac{\Phi(\bmz;\bmx)}{\Delta_{J_I}(\bmz)}.
\]
Since \(\Delta_{J_I}(\bmz)\) divides \(\Delta(\bmz)\), the function
\(\Omega_I\) is entire in \(\bmz\).

If \(J_I=\{j_0<j_1<\cdots<j_m\}\), define
\begin{equation}
\mathcal D_I
:=
\sum_{\ell=0}^m
(-1)^\ell
\Delta_{J_I\setminus\{j_\ell\}}(\bmz)\,
\partial_{z_{j_\ell}}.
\label{eq:DI-def}
\end{equation}
For each \(\ell\), the coefficient
\(\Delta_{J_I\setminus\{j_\ell\}}(\bmz)\) is independent of \(z_{j_\ell}\).
Moreover, \(\Omega_I\) is a polynomial in \(\bmz\) times
\(e^{S(\bmz;\bmx)}\), and it decays exponentially at the ends of the admissible
Airy contours in the variable being differentiated. Therefore integration by
parts in \(z_{j_\ell}\) yields
\begin{equation}
\int_{\Gamma_{\bm\sigma}} \mathcal D_I\Omega_I\,\rd\bmz=0,
\qquad
I\subseteq U_r.
\label{eq:IBP}
\end{equation}

We now introduce
\[
c_I(\bmx):=\prod_{u\in I}\frac1{x_r-x_u},
\qquad
c_\varnothing(\bmx):=1.
\]
The key identity is
\begin{equation}
\sum_{I\subseteq U_r} c_I(\bmx)\,\mathcal D_I\Omega_I
=
B_r(\bmx,\bmz)\,\Phi(\bmz;\bmx).
\label{eq:key-identity}
\end{equation}
Once \eqref{eq:key-identity} is proved, integrating it over
\(\Gamma_{\bm\sigma}\) and using \eqref{eq:IBP} gives
\eqref{eq:target-vanishing}. Then \eqref{eq:LrK} implies
\(\sfL_r G_{\bm\sigma}=0\).

\medskip
\noindent
\emph{Step 1b. Proof of the key identity.}
For a finite set \(J\subseteq\{1,\dots,{n+m}\}\) and \(a\in J\), define
\[
w_a^J(\bmz)
:=
\frac{1}{\prod_{b\in J\setminus\{a\}}(z_a-z_b)}.
\]
If \(J=\{j_0<\cdots<j_m\}\) and \(a=j_\ell\), then a direct sign check gives
\begin{equation}
(-1)^\ell
\frac{\Delta_{J\setminus\{a\}}(\bmz)}
{\Delta_J(\bmz)}
=
w_a^J(\bmz).
\label{eq:w-ratio}
\end{equation}
Since
\[
\frac{\partial_{z_a}\Omega_I}{\Omega_I}
=
z_a^2-x_a
+
\sum_{v\notin J_I}\frac1{z_a-z_v},
\qquad a\in J_I,
\]
it follows from \eqref{eq:DI-def} and \eqref{eq:w-ratio} that
\begin{equation}
\frac{\mathcal D_I\Omega_I}{\Phi}
=
\sum_{a\in J_I}
w_a^{J_I}(\bmz)
\left(
z_a^2-x_a
+
\sum_{v\notin J_I}\frac1{z_a-z_v}
\right).
\label{eq:DI-over-Phi}
\end{equation}

We use two elementary identities.

\smallskip
\noindent
\emph{(i) Lagrange interpolation identity.}
For every finite set \(J\),
\begin{equation}
\sum_{a\in J} w_a^J(\bmz)\,z_a^m
=
\begin{cases}
0,& 0\le m\le |J|-2,\\[1mm]
1,& m=|J|-1.
\end{cases}
\label{eq:Lagrange}
\end{equation}
This is the standard Lagrange interpolation formula applied to \(t^m\).

\smallskip
\noindent
\emph{(ii) Partial fraction identity.}
If \(v\notin J\), then
\begin{equation}
\sum_{a\in J}
\frac{w_a^J(\bmz)}{z_a-z_v}
=
-\frac1{\prod_{a\in J}(z_v-z_a)}
=
-w_v^{J\cup\{v\}}(\bmz).
\label{eq:partial-frac}
\end{equation}
This follows from the partial fraction decomposition of
\[
\frac1{\prod_{a\in J}(t-z_a)}
=
\sum_{a\in J}\frac{w_a^J(\bmz)}{t-z_a},
\]
evaluated at \(t=z_v\).

We now sum \eqref{eq:DI-over-Phi} over all \(I\subseteq U_r\) with
coefficients \(c_I(\bmx)\).

First, consider the terms involving \(z_a^2\). By \eqref{eq:Lagrange}, only the
cases \(|J_I|=1,2,3\) contribute:
\[
\sum_{a\in J_I}w_a^{J_I}(\bmz)z_a^2
=
\begin{cases}
z_r^2,& I=\varnothing,\\[1mm]
z_r+z_u,& I=\{u\},\\[1mm]
1,& I=\{u,v\},\ u<v,\ u,v\neq r,\\[1mm]
0,& |I|\ge3.
\end{cases}
\]
Hence the total contribution of the \(z_a^2\)-terms is
\begin{equation}
z_r^2
+\sum_{u\neq r}\frac{z_r+z_u}{x_r-x_u}
+\sum_{\substack{u<v\\u,v\neq r}}
\frac1{(x_r-x_u)(x_r-x_v)}.
\label{eq:z2-contrib}
\end{equation}

It remains to treat the \(-x_a\)-terms and the terms containing
\((z_a-z_v)^{-1}\). These cancel telescopically.

Suppose \(I\neq\varnothing\). Since \(|J_I|\ge2\), identity
\eqref{eq:Lagrange} with \(m=0\) gives
\[
\sum_{a\in J_I} w_a^{J_I}(\bmz)=0.
\]
Therefore
\begin{align}
-c_I(\bmx)
\sum_{a\in J_I}w_a^{J_I}(\bmz)x_a
&=
-c_I(\bmx)
\sum_{b\in I}w_b^{J_I}(\bmz)(x_b-x_r)
\nonumber\\
&=
\sum_{b\in I}
c_{I\setminus\{b\}}(\bmx)w_b^{J_I}(\bmz),
\label{eq:x-terms}
\end{align}
because
\[
-c_I(\bmx)(x_b-x_r)
=
\prod_{u\in I\setminus\{b\}}\frac1{x_r-x_u}
=
c_{I\setminus\{b\}}(\bmx).
\]

Now fix \(b\in I\). Consider the contribution of the term \(v=b\) in the last
sum of \eqref{eq:DI-over-Phi}, but for the smaller subset \(I\setminus\{b\}\).
Since
\[
J_{I\setminus\{b\}}\cup\{b\}=J_I,
\]
the partial fraction identity \eqref{eq:partial-frac} gives
\begin{align}
c_{I\setminus\{b\}}(\bmx)
\sum_{a\in J_{I\setminus\{b\}}}
\frac{w_a^{J_{I\setminus\{b\}}}(\bmz)}{z_a-z_b}
&=
-\,c_{I\setminus\{b\}}(\bmx)w_b^{J_I}(\bmz).
\label{eq:partner-term}
\end{align}
This cancels exactly with the corresponding summand in \eqref{eq:x-terms}.
Thus every \(-x_a\)-term with \(I\neq\varnothing\) is cancelled by a unique
term coming from the \((z_a-z_v)^{-1}\)-part for a smaller subset.

The only unmatched \(x\)-term is the one coming from \(I=\varnothing\), namely
\[
-\sum_{a\in\{r\}}w_a^{\{r\}}(\bmz)x_a=-x_r.
\]
Therefore, after all cancellations, the total contribution of the
\(-x_a\)-terms together with the \((z_a-z_v)^{-1}\)-terms is exactly \(-x_r\).
Combining this with \eqref{eq:z2-contrib}, we obtain
\[
\sum_{I\subseteq U_r}
c_I(\bmx)\frac{\mathcal D_I\Omega_I}{\Phi}
=
z_r^2-x_r
+\sum_{u\neq r}\frac{z_r+z_u}{x_r-x_u}
+\sum_{\substack{u<v\\u,v\neq r}}
\frac1{(x_r-x_u)(x_r-x_v)}.
\]
By the definition of \(B_r(\bmx,\bmz)\) in \eqref{eq:Br}, this proves
\eqref{eq:key-identity}.

Integrating \eqref{eq:key-identity} over \(\Gamma_{\bm\sigma}\) and using
\eqref{eq:IBP} give
\[
\int_{\Gamma_{\bm\sigma}}
B_r(\bmx,\bmz)\Phi(\bmz;\bmx)\,\rd\bmz=0.
\]
Thus \(\sfL_r G_{\bm\sigma}=0\), which proves \eqref{eq:direct-edge-reduced}.

\medskip
\noindent
\emph{Step 2. Lifting back to the original system.}
By \Cref{thm:edge-beta2-reduction},
\[
F_{\bm\sigma}=P\,G_{\bm\sigma}
\]
solves the original edge system \eqref{e:edgebeta=2}. Since
\[
\Delta(\bmx)=\Delta(\bmt)\Delta(\bms)P(\bmt,\bms),
\]
formula \eqref{eq:F-sigma-edge} follows directly.

\medskip
\noindent
\emph{Step 3. Linear independence.}
In the sector \(0<|\arg x|<\pi-\delta\),
\[
\partial_x^{\,j-1}\phi_{\pm}(x)
\sim
(\pm\sqrt{x})^{j-1}\phi_{\pm}(x).
\]
Hence, along any admissible ray in this sector,
\[
D_{\bm\sigma}(\bmx)
\sim
\left(\prod_{r=1}^{n+m} \phi_{\sigma_r}(x_r)\right)
\det\!\left[(\sigma_r\sqrt{x_r})^{j-1}\right]_{1\le r,j\le {n+m}}
=
\left(\prod_{r=1}^{n+m} \phi_{\sigma_r}(x_r)\right)
\Delta(\sigma_1\sqrt{x_1},\dots,\sigma_{n+m}\sqrt{x_{n+m}}).
\]
Therefore
\[
F_{\bm\sigma}(\bmt,\bms)
\sim
\frac{
\Delta(\sigma_1\sqrt{x_1},\dots,\sigma_{n+m}\sqrt{x_{n+m}})
}{
\Delta(\bmt)\Delta(\bms)
}
\prod_{r=1}^{n+m} \phi_{\sigma_r}(x_r).
\]
Distinct sign vectors \(\bm\sigma\) produce distinct leading exponential
factors
\[
\exp\!\left(
\sum_{r=1}^{n+m}
\sigma_r\frac{2}{3}x_r^{3/2}
\right).
\]
Thus the same large-ray argument as in the bulk proof shows that the functions
\(\{F_{\bm\sigma}\}\) are linearly independent.
\end{proof}

\section{The equilibrium BBGKY hierarchy and the loop equations}\label{App:BBGKY}
In this appendix we explain that the equilibrium Bogoliubov–Born–Green–Kirkwood–Yvon equations for the one-dimensional logarithmic interaction imply the loop equations (\ref{e:loopeq2-bulk}), and therefore 
the particle system is the ${\rm Sine}_{\beta}$ ensemble.  This provides one example of a point process characterized by the BBGKY hierarchy in the case of long-range and singular interactions.
For the proof,  we assume a priori mild estimates on the point process,  see \eqref{eqn:weakRigid}.

The definition and analysis of the BBGKY hierarchy for particle systems are delicate for long-range interactions. 
We first discuss a general setting  and then specialize to the log interaction in dimension $1$. 
Consider a simple point process
\[
\mu=\sum_{i\in I}\delta_{x_i},\ x_i\in X=\mathbb{R}^d,
\]
and its correlation functions defined as the measures $(\rho_p)_{p\geq 1}$ such that
\[
\mathbb{E}\!\left[
\sum_{i_1,\dots,i_p\in I}^{*}
f(x_{i_1},\dots,x_{i_p})
\right]
=
\int_{X^p}
f(q_1,\dots,q_p)\,
\rd\rho_p(q_1,\dots,q_p),
\]
for all suitable test functions $f$, where $\sum^*$ means the summation is over all ordered $p$-tuples of distinct indices.  Assume $\mu$ is the local limit of a finite interacting particle system
with the same charges and symmetric,  singular,  long-range interaction potential $\varphi$ (e.g.  consider either  Coulombic $\varphi=(-\Delta)^{-1}$ or logarithmic $\varphi=-\log$ below), at inverse temperature $\beta$: defining
\begin{equation}\label{eqn:finiteN}
\mathbb{P}_N(y_1,\dots,y_N)=\frac{1}{Z_{N,\beta}}e^{-\beta N^\alpha\big(\frac{1}{2}\sum_{i\neq j}\varphi(y_i-y_j)-N\sum_{k=1}^NV(y_k)\big)}\rd y_1\dots\rd y_N
\end{equation}
with confining $V$ (e.g.  $V(y)=|y|^2$) we assume that the following convergence holds in distribution with respect to the weak topology, as $N\to\infty$ (we only consider convergence around the energy level $E=0$ to simplify the discussion):
\[\sum_{i=1}^N\delta_{N^{1/d} y_i}\to \mu \quad \text{as } N\to\infty.\]
Depending on $\alpha\in\mathbb{R}$, the point process $\mu$ is expected to be either Poisson,  crystallized,  or a non-trivial indermediate depending on $\beta$. 
When $\mu$ is non-trivial,
one expects the following Yvon–Born–Green or YBG hierarchy (which corresponds to the BBGKY  hierarchy at equilibrium) for the correlation functions of $\mu$, see \cite[Equation (2.10)]{Martin}: 
assuming some decay of correlations (called a clustering condition in  \cite{Martin}), we have
for distinct points $q_1,\dots, q_p$, 
\begin{multline}
\frac{1}{\beta}\nabla_{q_1}\rho_p(q_1,\dots,q_p)
+
\big(\sum_{j=2}^p \nabla\varphi(q_1-q_j)\big)\rho_p(q_1,\dots,q_p)\\
+
\int_{X} \nabla \varphi(q_1-y)\left(\rho_{p+1}(q_1,\dots,q_p,y)
-\rho_{p}(q_1,\dots,q_p)\rho_1(y)
\right){\rm d}y=0.\label{eqn:BBGKY}
\end{multline}
For singular $\varphi$, this pointwise identity \eqref{eqn:BBGKY} holds only away from the collision diagonals.  It does not by itself imply the following weak formulation, as this would require \emph{a priori} information on the behavior of $\rho_p$ near collisions. We therefore adopt the following distributional BBGKY hierarchy as the definition.

\begin{definition}[Equilibrium BBGKY]\label{def:bbgky}
We say that $\mu$ satisfies the equilibrium BBGKY hierarchy for the interaction $\varphi$ if
for any $p\geq 1$ and any test function $\Phi\in C_c^\infty(X^p)$, we have
\begin{equation}\label{eqn:IBP}
\mathbb E
\sum_{i_1,\dots,i_p\in I}
\left[
\frac{1}{\beta}
\nabla_{x_{i_1}}\Phi(x_{i_1},\dots,x_{i_p})
-
\Phi(x_{i_1},\dots,x_{i_p})\int
\nabla\varphi(x_{i_1}-u)
(\rd\mu_{i_1}-\rd\rho_1)(u)
\right]=0,
\end{equation}
where the indices $i_1,\dots,i_p$ are not required to be distinct and we denote $\mu_i=\mu-\delta_{x_i}$.
\end{definition}

\begin{remark}
For symmetric interactions,  constant density of states  $\rho_1$, and $\nabla\varphi\in{\rm L}^1$ we have $\int\nabla\varphi(q-y)\rho_1(y)=0$, so $\rho_1$ above and the $\rho_p\rho_1$ term in \eqref{eqn:BBGKY} can be omitted.  However, in our long-range and singular context
$\nabla\varphi\not\in{\rm L}^1$ and this counter-term is needed.
\end{remark}

\begin{remark} The above definition is natural: 
integrating by parts with respect to the measure
(\ref{eqn:finiteN}) gives, for any $F\in C_c^\infty(X^p)$
\[
{\mathbb E}_N
\left[
\frac{1}{\beta}
\nabla_{y_{1}}F
-
F\big(\sum_{k\neq 1}
\nabla\varphi(y_{1}-y_k)
+N\nabla V(y_1)
\big)
\right]=0.
\]
For example, picking $V(y)=|y|^2$ and $F(y)=\Phi(N^{1/d}y)$,   the contribution from $V$ can be shown to be negligible as $N\to\infty$.  
Defining the equilibrium density through $N^{-1}\sum\delta_{y_i}\to\rho_{\rm eq}$, up to a negligible error one can replace  $\sum_{k\neq 1}
\nabla\varphi(y_{1}-y_k)$ with $\sum_{k\neq 1}
\nabla\varphi(y_{1}-y_k)-N\int\nabla \varphi(y_1-u)\rd\rho_{\rm eq}(u)$. Under some mild assumption on the discrepancy of the point process,  dominated convergence can then be justified and one obtains \eqref{eqn:IBP} with constant $\rho_1$.
\end{remark}

We now show that the distributional formulation \Cref{def:bbgky} implies the pointwise
BBGKY equation \eqref{eqn:BBGKY} away from the collision diagonals. To see this,  consider distinct $q_i$'s and $\Phi(x_1,\dots,x_p)=\prod_{i=1}^p\xi^i_\varepsilon(x_i)$ with $\xi^i_\varepsilon(x)=\varepsilon^{-d}\xi((x-q_i)/\varepsilon)$, $\xi$ nonnegative, smooth,  supported on $|x|<1$ and $\int\xi=1$.  When
 $\varepsilon<\min_{i\neq j}|q_i-q_j|/2$, the supports of
\(\xi_\varepsilon^1,\dots,\xi_\varepsilon^p\) are pairwise disjoint.  Applying Definition \ref{def:bbgky} with this test function, we observe that, since the supports of the
\(\xi_\varepsilon^a\)'s are pairwise disjoint, every term with repeated
indices vanishes.  Indeed, if \(i_a=i_b\) for some \(a\neq b\), then the
factor
\[
\xi_\varepsilon^a(x_{i_a})\xi_\varepsilon^b(x_{i_b})
=
\xi_\varepsilon^a(x_{i_a})\xi_\varepsilon^b(x_{i_a})
\]
is zero.  The same argument applies to the derivative term, since
\(\operatorname{supp}\nabla\xi_\varepsilon^i\subset
\operatorname{supp}\xi_\varepsilon^i\).  Therefore, for this particular choice of \(\Phi\), the contribution of the unrestricted sum is equal to the contribution of the distinct-index sum. Hence \Cref{def:bbgky} gives
\begin{multline*}
\mathbb E
\sum_{i_1,\dots,i_p\in I}^*
\left[
\frac{1}{\beta}
(\nabla_{x_{i_1}}\xi_\varepsilon^1(x_{i_1}))\prod_{j=2}^p\xi_\varepsilon^j(x_{i_j})
-
\prod_{j=1}^p\xi_\varepsilon^j(x_{i_j})\sum_{k=2}^p\nabla\varphi(x_{i_1}-x_{i_k})\right.\\
\left.-
\prod_{j=1}^p\xi_\varepsilon^j(x_{i_j})\left(\sum_{\ell\in I\backslash\{i_1,\dots,i_p\}}\nabla\varphi(x_{i_1}-x_{\ell})
-\int\nabla\varphi(x_{i_1}-u)\rd\rho_1(u)
\right)
\right]=0,
\end{multline*}
so
\begin{multline*}
\int_{X^p}
\frac{1}{\beta}
(\nabla_{r_1}\xi_\varepsilon^1(r_1))\prod_{j=2}^p\xi_\varepsilon^j(r_j)\rho_p(r_1,\dots,r_p)\rd r_1\dots\rd r_p
-\int_{X^p}
\prod_{j=1}^p\xi_\varepsilon^j(r_j)\sum_{k=2}^p\nabla\varphi(r_1-r_k)\rho_p(r_1,\dots,r_p)\rd r_1\dots\rd r_p\\
-\int_{X^{p+1}}
\prod_{j=1}^p\xi_\varepsilon^j(r_j)\nabla\varphi(r_1-x)\big(\rho_{p+1}(r_1,\dots,r_p,x)-\rho_{p}(r_1,\dots,r_p)\rho_1(x)\big)\rd r_1\dots\rd r_p\rd x
=0.
\end{multline*}
Integrating by parts in \(r_1\), taking \(\varepsilon \to 0\), and using that
\(\prod_{j=1}^p\xi_\varepsilon^j\) converges to a Dirac mass at
\((q_1,\dots,q_p)\), we obtain \eqref{eqn:BBGKY}.

We now consider the case \(d=1\), \(\varphi(x)=-\log|x|\), and
\[
\nabla\varphi(x-y)=-\frac{1}{x-y}.
\]
We also assume that the density of states is constant, with normalization
\(\rho_1\equiv\pi^{-1}\). Finally, we will need a very weak control on the
eigenvalue counts, which can possibly be weakened (we have not tried to
optimize this assumption): there exists \(\kappa>0\) such that, for any
\(p\geq 1\), the following holds. There exists \(C_p>0\) such that, for any
interval \(I\) with \(|I|\geq 1\),
\begin{equation}\label{eqn:weakRigid}
\mathbb{E}\big[|\mu(I)-|I|/\pi|^p\big]^{1/p}
\leq C_p|I|^{1-\kappa}.
\end{equation}
Under the above assumption, the following properties are elementary to prove:
for any \(w\notin\mathbb{R}\),
\[
s(w)={\rm P.V.}\sum_i\frac{1}{x_i-w}
\]
is a.s. well-defined and belongs to \(L^p\) for every \(p\geq 1\), and it is
a Nevanlinna function with parameters
$
c=0$, 
$b={\rm P.V.}\sum_i{x_i}/({x_i^2+1})$.
The main result of this section is the following.

\begin{proposition}\label{prop:bbgky}
Fix \(\beta>0\) rational. Let \(\mu\) satisfy the above assumption
\eqref{eqn:weakRigid} and the equilibrium BBGKY hierarchy \eqref{eqn:IBP}.
Then \(s\) satisfies the loop equations \eqref{e:loopeq2-bulk}, and \(\mu\)
is the \({\rm Sine}_\beta\) point process.
\end{proposition}

\begin{proof}
Let \(\varepsilon>0\) be small enough, for example
\(\varepsilon=\kappa/10\), where \(\kappa\) is defined in
\eqref{eqn:weakRigid}. Let \(R>0\), which will eventually tend to infinity,
and let \(\chi_a:\mathbb{R}\to\mathbb{R}_+\) be smooth and even, with
\[
\chi_a=1 \quad \text{on } [-a,a],
\qquad
\chi_a=0 \quad \text{on }
[-(a+R^{1-\varepsilon}),a+R^{1-\varepsilon}]^{\complement},
\]
and
$
\|\chi_a'\|_\infty\leq R^{-1+\varepsilon}$.
We will consider \(a=R\) and \(a=R+10R^{1-\varepsilon}\) in the proof.

We write \(Y=Y(R)\) for a random variable, changing from line to line, which
is uniformly bounded in \(L^q\), for every fixed \(q\geq 1\).

We introduce
\begin{align}
s_R(w)=\sum_i\frac{\chi_R(x_i)}{x_i-w}
,\quad 
\Phi(x_1,\dots,x_p)
=
\frac{\chi_R(x_1)}{x_1-w_1}\cdots
\frac{\chi_R(x_p)}{x_p-w_p}.
\end{align}
Then \eqref{eqn:IBP} gives
\begin{equation}\label{eqn:inter}
\mathbb E
\left[
\frac{1}{\beta}\sum_{i_1,\dots,i_p}
\partial_{x_{i_1}}\left(
\frac{\chi_R(x_{i_1})}{x_{i_1}-w_1}
\prod_{j=2}^p\frac{\chi_R(x_{i_j})}{x_{i_j}-w_j}
\right)
+\sum_{i_1,\dots,i_p}
\prod_{j=1}^p \frac{\chi_R(x_{i_j})}{x_{i_j}-w_j}
\int
\frac{1}{x_{i_1}-x}(\rd \mu_{i_1}-\rd \rho_1)(x)
\right]=0.
\end{equation}

We first consider the \(\partial_{x_{i_1}}\) term:
\begin{align*}
\partial_{x_{i_1}}\left(
\frac{\chi_R(x_{i_1})}{x_{i_1}-w_1}
\prod_{j=2}^p\frac{\chi_R(x_{i_j})}{x_{i_j}-w_j}
\right)
&=
\left(
-\partial_{w_1}\frac{\chi_R(x_{i_1})}{x_{i_1}-w_1}
+\frac{\chi_R'(x_{i_1})}{x_{i_1}-w_1}
\right)
\prod_{j=2}^p\frac{\chi_R(x_{i_j})}{x_{i_j}-w_j}
\\
&-
\frac{\chi_R(x_{i_1})}{x_{i_1}-w_1}
\sum_{k=2}^p\mathds{1}_{i_k=i_1}
\left(
\frac{\chi_R(x_{i_k})}{(x_{i_k}-w_k)^2}
-\frac{\chi_R'(x_{i_k})}{x_{i_k}-w_k}
\right)
\prod_{\substack{2\leq j\leq p\\ j\neq k}}
\frac{\chi_R(x_{i_j})}{x_{i_j}-w_j},
\end{align*}
so that
\begin{align*}
\sum_{i_1,\dots,i_p}
\partial_{x_{i_1}}\left(
\frac{\chi_R(x_{i_1})}{x_{i_1}-w_1}
\prod_{j=2}^p\frac{\chi_R(x_{i_j})}{x_{i_j}-w_j}
\right)
&=
-\partial_{w_1}s_R(w_1)\prod_{j=2}^p s_R(w_j)
-
\sum_{k=2}^p
\sum_{i_1}
\frac{\chi_R(x_{i_1})^2}
{(x_{i_1}-w_1)(x_{i_1}-w_k)^2}
\prod_{\substack{2\leq j\leq p\\ j\neq k}} s_R(w_j)
\\
&+
\sum_{i_1}\frac{\chi_R'(x_{i_1})}{x_{i_1}-w_1}
\prod_{j=2}^p s_R(w_j)
+
\sum_{k=2}^p
\sum_{i_1}
\frac{\chi_R(x_{i_1})}{x_{i_1}-w_1}
\frac{\chi_R'(x_{i_1})}{x_{i_1}-w_k}
\prod_{\substack{2\leq j\leq p\\ j\neq k}} s_R(w_j).
\end{align*}
From \eqref{eqn:weakRigid}, all terms above involving \(\chi_R'\) are of
order \(Y/R\). For example,
\[
\left|
\sum_{i_1}\frac{\chi_R'(x_{i_1})}{x_{i_1}-w_1}
\right|
\leq
\|\chi_R'\|_\infty
\sum_{|x_{i_1}|\in[R,R+R^{1-\varepsilon}]}
\frac{1}{|x_{i_1}-w_1|}
\lesssim
R^{-1+\varepsilon}
\frac{\#\{|x_i|\in[R,R+R^{1-\varepsilon}]\}}{R}
\leq
\frac{Y}{R}.
\]
Note that
\begin{multline*}
\sum_i
\frac{\chi_R(x_i)^2}{(x_i-w_1)(x_i-w_k)^2}
=
\partial_{w_k}
\frac{s_R(w_1)-s_R(w_k)}{w_1-w_k}
+
\sum_i
\frac{\chi_R(x_i)^2-\chi_R(x_i)}
{(x_i-w_1)(x_i-w_k)^2}
\\
=
\partial_{w_k}
\frac{s_R(w_1)-s_R(w_k)}{w_1-w_k}
+
{\rm O}\left(
\frac{\#\{|x_i|\in[R,R+R^{1-\varepsilon}]\}}{R^3}
\right)
=
\partial_{w_k}
\frac{s_R(w_1)-s_R(w_k)}{w_1-w_k}
+
\frac{Y}{R^{2+\varepsilon}}.
\end{multline*}
We obtain
\begin{align}\begin{split}\label{eqn:1stterm}
&\phantom{{}={}}\sum_{i_1,\dots,i_p}
\partial_{x_{i_1}}\left(
\frac{\chi_R(x_{i_1})}{x_{i_1}-w_1}
\prod_{j=2}^p\frac{\chi_R(x_{i_j})}{x_{i_j}-w_j}
\right)
\\
&=
-\partial_{w_1}s_R(w_1)\prod_{j=2}^p s_R(w_j)
-
\sum_{k=2}^p
\partial_{w_k}
\frac{s_R(w_1)-s_R(w_k)}{w_1-w_k}
\prod_{\substack{2\leq j\leq p\\ j\neq k}}s_R(w_j)
+
\frac{Y}{R}.
\end{split}\end{align}

For the remaining term in \eqref{eqn:inter}, using the symmetry under
exchange of \(i_1\) and \(k\),
\begin{align*}
&\phantom{{}={}}\sum_{i_1\neq k}
\frac{\chi_R(x_{i_1})}{x_{i_1}-w_1}
\frac{\chi_R(x_k)}{x_{i_1}-x_k}
 =
\frac12
\sum_{i_1\neq k}
\frac{\chi_R(x_{i_1})\chi_R(x_k)}{x_{i_1}-x_k}
\left(
\frac{1}{x_{i_1}-w_1}
-
\frac{1}{x_k-w_1}
\right)
\\
&=
-\frac12
\sum_{i_1\neq k}
\frac{\chi_R(x_{i_1})\chi_R(x_k)}
{(x_{i_1}-w_1)(x_k-w_1)} =
-\frac12
\left(
s_R(w_1)^2-\partial_{w_1}s_R(w_1)
\right)
+
\frac{Y}{R^{1+\varepsilon}},
\end{align*}
so that
\begin{multline}\label{eqn:2ndterm}
\sum_{i_1,\dots,i_p}
\prod_{j=1}^p \frac{\chi_R(x_{i_j})}{x_{i_j}-w_j}
\int
\frac{1}{x_{i_1}-x}(\rd \mu_{i_1}-\rd \rho_1)(x)
=
-\frac12
\left(
s_R(w_1)^2-\partial_{w_1}s_R(w_1)
\right)
\prod_{j=2}^p s_R(w_j)
\\
+
\sum_{i_1,\dots,i_p}
\prod_{j=1}^p \frac{\chi_R(x_{i_j})}{x_{i_j}-w_j}
\int
\frac{1}{x_{i_1}-x}
\big((1-\chi_R)\rd \mu_{i_1}-\rd \rho_1\big)(x)
+
\frac{Y}{R}.
\end{multline}

For \(|x_{i_1}|<R-R^{1-\varepsilon/2}\), the point \(x_{i_1}\) is far from
the transition region of \(\chi_R\). Hence
\[
\int
\frac{\chi_R(x)}{x_{i_1}-x}\,\rd x
=
{\rm P.V.}\!\int_{-R}^{R}
\frac{\rd x}{x_{i_1}-x}
+
{\rm O}(R^{-\varepsilon/2})
=
\log\frac{R+x_{i_1}}{R-x_{i_1}}
+
{\rm O}(R^{-\varepsilon/2}).
\]
The remaining particles, satisfying
$
R-R^{1-\varepsilon/2}
<
|x_{i_1}|
<
R+R^{1-\varepsilon},
$
lie in a boundary layer. Their contribution is negligible. Indeed, by
\eqref{eqn:weakRigid}, the number of such particles is
\(YR^{1-\varepsilon/2}\), while
$
\left|
{\chi_R(x_{i_1})}/({x_{i_1}-w_1})
\right|
\lesssim 1/R
$
in this region, and the integral against
\(\chi_R(x)\rd x/(x_{i_1}-x)\) is at most logarithmic. Thus the total
boundary-layer contribution is \(Y/R^{\varepsilon/4}\).

Combining the interior and boundary estimates gives
\[
\sum_{i_1}
\frac{\chi_R(x_{i_1})}{x_{i_1}-w_1}
\int
\frac{\chi_R(x)}{x_{i_1}-x}\,\rd\rho_1(x)
=
\frac1\pi
\sum_{|x_{i_1}|\leq R-R^{1-\varepsilon/2}}
\frac{\chi_R(x_{i_1})}{x_{i_1}-w_1}
\log\frac{R+x_{i_1}}{R-x_{i_1}}
+
\frac{Y}{R^{\varepsilon/4}}.
\]
We now replace the particle sum by the corresponding density integral, using
the weak rigidity estimate \eqref{eqn:weakRigid}. This gives
\[
\frac1\pi
\sum_{|x_{i_1}|\leq R-R^{1-\varepsilon/2}}
\frac{\chi_R(x_{i_1})}{x_{i_1}-w_1}
\log\frac{R+x_{i_1}}{R-x_{i_1}}
=
\frac1{\pi^2}
{\rm P.V.}\!\int_{-R}^{R}
\frac{1}{x-w_1}
\log\frac{R+x}{R-x}\,\rd x
+
\frac{Y}{R^{\varepsilon/4}}.
\]
After the change of variables \(x=Ru\), the deterministic integral becomes
\[
\frac1{\pi^2}
\int_{-1}^{1}
\frac{1}{u-w_1/R}
\log\frac{1+u}{1-u}\,\rd u.
\]
Letting \(R\to\infty\), this converges to
\[
\frac1{\pi^2}
\int_{-1}^{1}
\frac1u
\log\frac{1+u}{1-u}\,\rd u
=
\frac{1}{2}.
\]
Consequently,
\[
\sum_{i_1}
\frac{\chi_R(x_{i_1})}{x_{i_1}-w_1}
\int
\frac{\chi_R(x)}{x_{i_1}-x}\,\rd\rho_1(x)
=
\frac12+\frac{Y}{R^{\varepsilon/4}}.
\]
Substituting the above equation and the long-range cancellation
Lemma \ref{lem:LongRange} below into \eqref{eqn:2ndterm} gives
\begin{equation}\label{eqn:2ndtermbis}
\sum_{i_1,\dots,i_p}
\prod_{j=1}^p \frac{\chi_R(x_{i_j})}{x_{i_j}-w_j}
\int
\frac{1}{x_{i_1}-x}(\rd \mu_{i_1}-\rd \rho_1)(x)
=
\left(
-\frac12
\left(
s_R(w_1)^2-\partial_{w_1}s_R(w_1)
\right)
-\frac{1}{2}
\right)
\prod_{j=2}^p s_R(w_j)
+
\frac{Y}{R^{\varepsilon/4}}.
\end{equation}
From the estimates \eqref{eqn:1stterm} and \eqref{eqn:2ndtermbis},
\eqref{eqn:inter} can be written as
\[
\bE\left[
\left(
\frac{2-\beta}{\beta}\,\partial_{w_1}s_R(w_1)
+s_R(w_1)^2+1
\right)
\prod_{j=2}^p s_R(w_j)
\right]
+
\frac{2}{\beta}\,
\bE\left[
\sum_{j=2}^p
\partial_{w_j}
\frac{s_R(w_1)-s_R(w_j)}{w_1-w_j}
\prod_{\substack{\ell=2\\ \ell\neq j}}^p s_R(w_\ell)
\right]
=
{\rm O}(R^{-\varepsilon/4}).
\]
Taking \(R\to\infty\) gives the loop equation
\eqref{e:loopeq2-bulk} by dominated convergence.
Since the principal value \(s\) is a Nevanlinna function, we can apply
Theorem \ref{t:bulk-characterization}, and \(\mu\) is the
\({\rm Sine}_\beta\) point process.
\end{proof}

\begin{lemma}\label{lem:LongRange}
We have
\begin{align}\label{e:long_range}
\sum_{i}
\frac{\chi_R(x_i)}{x_i-w_1}
\int
\frac{1-\chi_R(x)}{x_i-x}
(\rd\mu_i-\rd\rho_1)(x)
=
\frac{Y}{R^{\varepsilon/2}},
\end{align}
where \(Y=Y(R)\) is uniformly bounded in \(L^p\) for every fixed
\(p\geq 1\).
\end{lemma}

\begin{proof}
Let \(R'=R+10R^{1-\varepsilon}\).  We decompose \eqref{e:long_range} as
\begin{align}\begin{split}\label{e:long_range2}
\eqref{e:long_range}
&=\sum_i
\frac{\chi_R(x_i)}{x_i-w_1}
\int
\frac{1-\chi_{R'}(x)}{x_i-x}
(\rd\mu_i-\rd\rho_1)(x)\\
&+\sum_{i_1}\frac{\chi_R(x_{i_1})}{x_{i_1}-w_1}\int
\frac{(\chi_{R'}-\chi_{R})(x)}{x_{i_1}-x}(\rd \mu_{i_1}-\rd \rho_1)(x).
\end{split}\end{align}

We start with the first term on the right-hand side of \eqref{e:long_range2}.
Note that, for any \(f\) defined on \([0,\infty)\) such that \(f(0)=0\),
\(|f(u)|\lesssim u^{-1}\), and \(|f'(u)|\lesssim u^{-2}\) as \(u\to\infty\), we have
\[
\int_0^\infty f(x)(\rd\mu-\rd\rho_1)(x)
=
-\int_0^\infty f'(x)(\mu[0,x]-x/\pi)\rd x.
\]
For \(f(u)=\frac{1-\chi_{R'}(u)}{x_{i_1}-u}\), for any
\(|x_{i_1}|<R+R^{1-\varepsilon}=:b\) and \(u>R+5R^{1-\varepsilon}\), we have
\[
|f'(u)|
\leq
\frac{1-\chi_{R'}(u)}{|x_{i_1}-u|^2}
+
\frac{|\chi_{R'}'(u)|}{|x_{i_1}-u|}
\leq
\frac{1-\chi_{R'}(u)}{|b-u|^2}
+
\frac{|\chi_{R'}'(u)|}{|b-u|}
=:g(u).
\]
Denoting \(\Delta(u)=|\mu[0,u]-u/\pi|\), we have
\[
\left|
\int_0^\infty
\frac{1-\chi_{R'}(x)}{x_{i_1}-x}
(\rd\mu_{i_1}-\rd\rho_1)(x)
\right|
\leq
\int_0^\infty g(u)\Delta(u)\rd u
=:Y.
\]
Moreover, for any positive \(\ell(u)\),
\begin{multline*}
|Y|^p
\leq
\int_{[0,\infty)^p}
g(u_1)\dots g(u_p)\Delta(u_1)\dots\Delta(u_p)
\rd u_1\dots\rd u_p\\
\lesssim
\int_{[0,\infty)^p}
g(u_1)\dots g(u_p)\ell(u_1)\dots\ell(u_p)
\left(
\frac{\Delta(u_1)^p}{\ell(u_1)^p}
+\dots+
\frac{\Delta(u_p)^p}{\ell(u_p)^p}
\right)
\rd u_1\dots\rd u_p.
\end{multline*}
Choosing \(\ell(u)=|b-u|^{1-\frac{\kappa}{2}}\), and using the weak local law
estimate \eqref{eqn:weakRigid}, we obtain
\[
\mathbb{E}\left[|Y|^p\right]
\lesssim
\frac{1}{R^{p(1-\varepsilon)\frac{\kappa}{2}}}.
\]
Thus, for some \(Y\) uniformly bounded in \(L^p\), \(p\geq 1\), we have
\[
\left|
\sum_{i_1}
\frac{\chi_R(x_{i_1})}{x_{i_1}-w_1}
\int_0^\infty
\frac{1-\chi_{R'}(x)}{x_{i_1}-x}
(\rd \mu_{i_1}-\rd \rho_1)(x)
\right|
\lesssim
\frac{Y}{R^{\kappa/4}}
\sum_{|x_{i_1}|<2R}
\frac{1}{|x_{i_1}-w_1|}.
\]
The sum on the right-hand side is at most \(Y\log R\), and the same estimate
holds over \((-\infty,0]\). We have therefore proved
\begin{equation}\label{eqn:st1}
\left|
\sum_{i_1}
\frac{\chi_R(x_{i_1})}{x_{i_1}-w_1}
\int_{-\infty}^\infty
\frac{1-\chi_{R'}(x)}{x_{i_1}-x}
(\rd \mu_{i_1}-\rd \rho_1)(x)
\right|
\lesssim
\frac{Y}{R^{\kappa/8}}.
\end{equation}

Next we bound the second term on the right-hand side of \eqref{e:long_range2},
\[
\sum_{i_1}
\frac{\chi_R(x_{i_1})}{x_{i_1}-w_1}
\int
\frac{(\chi_{R'}-\chi_{R})(x)}{x_{i_1}-x}
(\rd \mu_{i_1}-\rd \rho_1)(x),
\]
by considering the \(\mu_{i_1}\) and \(\rho_1\) contributions separately,
starting with \(\mu_{i_1}\). Denoting \(h=\chi_{R'}-\chi_R\), we notice that
\(h(x_j)\neq 0\) only if \(|x_j|\in [R,R+20R^{1-\varepsilon}]\). In the
following we restrict to the case \(x_j\in [R,R+20R^{1-\varepsilon}]\); the
case \(x_j<0\) can be bounded in the same way.

We discuss the two cases
$
|x_{i_1}-x_j|<100R^{1-\varepsilon}$
and
$|x_{i_1}-x_j|\geq 100R^{1-\varepsilon}$
separately.

For the sum over \(|x_{i_1}-x_j|<100R^{1-\varepsilon}\), by symmetrization we
have
\begin{multline*}
\left|
\sum_{\substack{i_1\neq j\\ |x_{i_1}-x_j|<100R^{1-\varepsilon}}}
\frac{\chi_R(x_{i_1})}{x_{i_1}-w_1}
\frac{h(x_j)}{x_{i_1}-x_j}
\right|
\lesssim
\sum_{\substack{i_1\neq j\\ |x_{i_1}-x_j|<100R^{1-\varepsilon}}}
\left|
\frac{\chi_R(x_{i_1})h(x_j)}{x_{i_1}-w_1}
-
\frac{\chi_R(x_j)h(x_{i_1})}{x_j-w_1}
\right|
\frac{1}{|x_{i_1}-x_j|}\\
\lesssim
\frac{\|h'\|_{\infty}+\|\chi_{R'}'\|_{\infty}+R^{-1}}{R}
\cdot
\left(\#\{x_i\in[R-1000R^{1-\varepsilon},R+1000R^{1-\varepsilon}]\}\right)^2
\lesssim
\frac{Y}{R^\varepsilon},
\end{multline*}
where in the second line we used that
$
|x_{i_1}-w_1|,\ |x_j-w_1|\asymp R$.

For the sum over \(|x_{i_1}-x_j|\geq 100R^{1-\varepsilon}\), if
\(\chi_R(x_{i_1})\neq 0\) and \(h(x_j)\neq0\), then
$
|x_{i_1}-x_j|\gtrsim |x_{i_1}-R|$.
Thus we also have
\begin{multline*}
\left|
\sum_{\substack{i_1\neq j\\ |x_{i_1}-x_j|\geq 100R^{1-\varepsilon}}}
\frac{\chi_R(x_{i_1})}{x_{i_1}-w_1}
\frac{h(x_j)}{x_{i_1}-x_j}
\right|
\lesssim
\#\{x_j\in [R,R+20R^{1-\varepsilon}]\}
\cdot
\sum_{|x_i|<R-10R^{1-\varepsilon}}
\frac{1}{|x_i-w_1|\cdot|x_i-R|}\\
\lesssim
Y\,R^{1-\varepsilon}
\left(
\frac{1}{R}\sum_{|x_i|<R/2}\frac{1}{|x_i-w_1|}
+
\frac{1}{R}
\sum_{|x_i|\in[R/2,R-10R^{1-\varepsilon}]}
\frac{1}{|x_i-R|}
\right)
\leq
\frac{Y}{R^{\varepsilon/2}}.
\end{multline*}
Together with similar, easier estimates for the contribution of \(\rho_1\), we have obtained
\[
\left|
\sum_{i_1}
\frac{\chi_R(x_{i_1})}{x_{i_1}-w_1}
\int
\frac{(\chi_{R'}-\chi_{R})(x)}{x_{i_1}-x}
(\rd \mu_{i_1}-\rd \rho_1)(x)
\right|
\lesssim
\frac{Y}{R^{\varepsilon/2}}.
\]
This estimate, combined with \eqref{eqn:st1} and the choice
\(\varepsilon\leq \kappa/4\), concludes the proof.
\end{proof}

Finally, we explain how the loop equations \eqref{e:loopeq2-bulk} allow us to obtain the equilibrium hierarchy
\eqref{eqn:BBGKY}. The reasoning below can be made fully rigorous by testing against local test functions and passing to limits, but we focus only on the algebraic aspects.
For a function \(F\) analytic off the real axis, define the jump operator
\[
\Delta_E F
:=
\frac1{2\pi {\rm i}}\left(F(E+\ri 0)-F(E-\ri 0)\right).
\]

\begin{lemma}\label{lem:basisEq}
Let
\[
s(z)=\operatorname{P.V.} \sum_i \frac{1}{x_i-z},
\qquad 
\mu=\sum_i \delta_{x_i},
\qquad 
H(E)=\operatorname{P.V.}\sum_i \frac{1}{x_i-E}.
\]
Then, in the sense of distributions,
\begin{align}\label{e:DE}
\Delta_E s = \mu(E),\quad \Delta_E s' =\partial_E \mu(E),\quad 
\Delta_E (s^2)
=
2\mu(E)H(E)+\partial_E\mu(E).
\end{align}
Here \(\mu(E)H(E)\) is understood with the pole at \(E=x_i\) omitted, namely
\[
\mu(E)H(E)
=
\sum_i
\left(
\operatorname{P.V.}\sum_{j\neq i}\frac{1}{x_j-x_i}
\right)
\delta_{x_i}(E).
\]
Moreover, for \(E\neq E'\),
\begin{align}\label{e:DDEE}
\Delta_E\Delta_{E'}\partial_{w'}
\frac{s(w)-s(w')}{w-w'}
=0,
\end{align}
where \(\Delta_E\) acts on the variable \(w\), and \(\Delta_{E'}\) acts on the variable \(w'\).
\end{lemma}

\begin{proof}[Proof of \Cref{lem:basisEq}]
The identity \(\Delta_E s=\mu(E)\) follows directly from the Sokhotski--Plemelj formula. Indeed,
\[
s(E\pm \mathrm i0)
=
H(E)\pm \mathrm i\pi \mu(E),
\]
and therefore
\[
\Delta_E s
=
\frac{1}{2\pi \mathrm i}
\left(
s(E+\mathrm i0)-s(E-\mathrm i0)
\right)
=
\mu(E).
\]

We also have, in the distributional sense,
\[
\Delta_E
\frac{1}{(x_i-z)^2}
=
\partial_E\delta_{x_i}(E).
\]
Summing over \(i\) gives
\begin{align}\label{e:ddest}
\Delta_E
\sum_i
\frac{1}{(x_i-z)^2}
=
\Delta_E s'
=
\sum_i \partial_E\delta_{x_i}(E)
=
\partial_E\mu(E).
\end{align}

We next prove the identity for \(\Delta_E(s^2)\). Split \(s^2\) into its off-diagonal and diagonal parts:
\[
s(z)^2
=
\sum_{i\neq j}
\frac{1}{(x_i-z)(x_j-z)}
+
\sum_i
\frac{1}{(x_i-z)^2}.
\]

Consider first the off-diagonal terms. For \(i\neq j\), the function
\[
\frac{1}{(x_i-z)(x_j-z)}
\]
has simple poles at \(x_i\) and \(x_j\). Hence its jump is
\[
\Delta_E
\frac{1}{(x_i-z)(x_j-z)}
=
\frac{\delta_{x_i}(E)}{x_j-x_i}
+
\frac{\delta_{x_j}(E)}{x_i-x_j}.
\]
Summing over \(i\neq j\), we obtain
\[
\Delta_E
\sum_{i\neq j}
\frac{1}{(x_i-z)(x_j-z)}
=
2\sum_i
\left(
\sum_{j\neq i}\frac{1}{x_j-x_i}
\right)
\delta_{x_i}(E)
=
2\mu(E)H(E).
\]
The diagonal contribution is given by \eqref{e:ddest}. Combining the off-diagonal and diagonal parts proves
\[
\Delta_E(s^2)
=
2\mu(E)H(E)+\partial_E\mu(E).
\]

Finally, we prove \eqref{e:DDEE}. We note that
\[
\partial_{w'}
\frac{s(w)-s(w')}{w-w'}
=
\sum_i
\frac{1}{(x_i-w)(x_i-w')^2}.
\]
Taking the jump first in \(w'\) and then in \(w\) gives
\[
\Delta_E\Delta_{E'}
\partial_{w'}
\frac{s(w)-s(w')}{w-w'}
=
\sum_i
\delta_{x_i}(E)\,
\partial_{E'}\delta_{x_i}(E').
\]
This distribution is supported on the diagonal \(E=E'\). Hence it vanishes on the region \(E\neq E'\), which proves \eqref{e:DDEE}.
\end{proof}

We now start from the loop equation \eqref{e:loopeq2-bulk}
\[
\mathbb E\left[
\left(
\frac{2-\beta}{\beta}\partial_{w_1}s(w_1)
+s(w_1)^2+1
\right)
\prod_{j=2}^p s(w_j)
\right]
+
\frac2\beta
\sum_{j=2}^p
\mathbb E\left[\partial_{w_j}\left(
\frac{s(w_1)-s(w_j)}{w_1-w_j}\right)
\prod_{\substack{\ell=2\\ \ell\neq j}}^p s(w_\ell)
\right]
=0.
\]
We now apply
\[
\Delta_{E_1}\cdots\Delta_{E_p}
\]
with \(w_a=E_a\pm {\rm i}0\), where the \(E_a\)'s are distinct.
Lemma \ref{lem:basisEq} gives, for the \(\partial_{w_1}s\) term,
\[
\Delta_{E_1}\cdots\Delta_{E_p}
\mathbb E\left[
\partial_{w_1}s(w_1)
\prod_{j=2}^p s(w_j)
\right]
=
\partial_{E_1}\rho_p(E_1,\dots,E_p);
\]
for the quadratic term,
\[
\Delta_{E_1}\cdots\Delta_{E_p}
\mathbb E\left[
s(w_1)^2
\prod_{j=2}^p s(w_j)
\right]
=
2\,{\rm P.V.}\!\int_{\mathbb R}
\frac{\rho_{p+1}(E_1,x,E_2,\dots,E_p)}{x-E_1}\rd x
+
2\sum_{j=2}^p
\frac{\rho_p(E_1,\dots,E_p)}{E_j-E_1}
+
\partial_{E_1}\rho_p(E_1,\dots,E_p).
\]
The analytic term \(+1\) has no jump in \(E_1\) and hence disappears.
When the \(E_a\)'s are distinct, the jumps of the last term also vanish:
\begin{align}
\Delta_{E_1}\cdots\Delta_{E_p}\mathbb E\left[\partial_{w_j}\left(
\frac{s(w_1)-s(w_j)}{w_1-w_j}\right)
\prod_{\substack{\ell=2\\ \ell\neq j}}^p s(w_\ell)
\right]=0.
\end{align}
Note that this term would be important in the equilibrium BBGKY hierarchy that includes contact terms, but it does not contribute to deriving \eqref{eqn:BBGKY}.

Therefore, away from coincident external points, we have proved
\[
\frac1\beta\partial_{E_1}\rho_p
+
{\rm P.V.}\!\int
\frac{\rho_{p+1}(x,E_1,E_2,\dots,E_p)}{x-E_1}\,\rd x
+
\sum_{j=2}^p
\frac{\rho_p(E_1,\dots,E_p)}{E_j-E_1}
=0.
\]
This is the equilibrium BBGKY hierarchy for the one-dimensional logarithmic gas as stated in \eqref{eqn:BBGKY}; the background term \(\rho_p\rho_1\) vanishes after integration in the principal-value sense.

\bibliography{References_capitalized2.bib}
\bibliographystyle{abbrv}

\end{document}